\documentclass[11pt]{amsart}
\usepackage{dynkin-diagrams}
 \usepackage{tikz}

\usepackage{listings}[2013/08/05]

\usepackage{bigints}

\usepackage{multirow}

\definecolor{darkgreen}{rgb}{.2,.6,.2}
\definecolor{MyDarkBlue}{rgb}{0.1,0,0.55}

\lstdefinelanguage{GAP}{%
  morekeywords={%
    Assert,Info,IsBound,QUIT,%
    TryNextMethod,Unbind,and,break,%
    continue,do,elif,%
    else,end,false,fi,for,%
    function,if,in,local,%
    mod,not,od,or,%
    quit,rec,repeat,return,%
    then,true,until,while%
  },%
  sensitive,%
  morecomment=[l]\#,%
  morestring=[b]",%
  morestring=[b]',%
}[keywords,comments,strings]

\usepackage[T1]{fontenc}
\usepackage[variablett]{lmodern}
\usepackage{xcolor}
\usepackage{lmodern}

\usepackage{color}
\usepackage[all,cmtip,color,matrix,arrow]{xy}

\usepackage[new]{old-arrows}
\usepackage{geometry}
\usepackage{graphicx}
\usepackage{lipsum}
\newsavebox{\myimage}

\usepackage{adjustbox}

\usepackage{hyperref}

\newcommand{\sk}{\smallskip}
\newcommand{\mk}{\medskip}
\newcommand{\bk}{\bigskip}

\usepackage{shuffle}

\usepackage{bigints}

\usepackage{tikz}
\makeatletter
\newcommand{\xleftrightarrow}[2][]{\ext@arrow 3359\leftrightarrowfill@{#1}{#2}}
\makeatother

\newcommand{\xdasharrow}[2][->]{
\tikz[baseline=-\the\dimexpr\fontdimen22\textfont2\relax]{
\node[anchor=south,font=\scriptsize, inner ysep=1.5pt,outer xsep=2.2pt](x){#2};
\draw[shorten <=3.4pt,shorten >=3.4pt,dashed,#1](x.south west)--(x.south east);
}
}

\usepackage{amsmath}
\usepackage{amssymb}
\usepackage{amsthm}
\usepackage{latexsym}
\usepackage{pstricks}
\usepackage{amsbsy}
\usepackage{amscd}
\usepackage{mathrsfs}
\usepackage{tipa}
\usepackage{txfonts}
\usepackage{amsfonts}
\usepackage{graphicx}
\usepackage{tabularx}
\usepackage{listings}
\usepackage{tikz}
\usepackage{hyperref}

\usepackage{scalerel,stackengine}
\stackMath
\newcommand\reallywidehat[1]{%
\savestack{\tmpbox}{\stretchto{%
  \scaleto{%
    \scalerel*[\widthof{\ensuremath{#1}}]{\kern-.6pt\bigwedge\kern-.6pt}%
    {\rule[-\textheight/2]{1ex}{\textheight}}
  }{\textheight}%
}{0.5ex}}%
\stackon[1pt]{#1}{\tmpbox}%
}
\newtheorem{thm}{Theorem}[section]
\newtheorem{cor}[thm]{Corollary}
\newtheorem{lem}[thm]{Lemma}

\newtheorem{exm}{Example}

\newtheorem{prop}[thm]{Proposition}

\newtheorem{defn}[thm]{Definition}
\newtheorem{rem}[thm]{Remark}
\newtheorem{defn-prop}[thm]{Definition-Proposition}

\newtheorem{question}[thm]{Question}

\usepackage{graphicx,epsfig,amscd,graphics,xypic}

 \newcommand{\eq}[1][r]
   {\ar@<-3pt>@{-}[#1]
    \ar@<-1pt>@{}[#1]|<{}="gauche"
    \ar@<+0pt>@{}[#1]|-{}="milieu"
    \ar@<+1pt>@{}[#1]|>{}="droite"
    \ar@/^2pt/@{-}"gauche";"milieu"
    \ar@/_2pt/@{-}"milieu";"droite"}

  \newcommand{\incl}[1][r]
  {\ar@<-0.2pc>@{^(-}[#1] \ar@<+0.2pc>@{-}[#1]}

 \newcommand{\WdPq}{\boldsymbol{\mathcal W}_{{\bf dP}_4}}
 
 \newcommand{\WdPqq}{\boldsymbol{\mathcal W}_{ \hspace{-0.07cm}{\rm dP}_4}}

\author[L. Pirio]{\href{mailto:luc.pirio@uvsq.fr}{Luc Pirio}\textsuperscript{$\dagger$}}
\thanks{${}^{}$\hspace{-0.4cm}\textsuperscript{$\dagger$}
\href{https://lmv.math.cnrs.fr}{Laboratoire de Math\'ematiques de Versailles}, Univ.\,Paris-Saclay \& CNRS (UMR 8100), 78000 Versailles, France.}

\newcommand{\mc}{\mathfrak{c}}

\newcommand{\cK}{\mathcal{K}}

\def\ra{\rightarrow}

\title[The 10-web by conics on the quartic del Pezzo surface]
 {On the 10-web by conics on the quartic del Pezzo surface}

\begin{document}

\maketitle

\begin{abstract}
We study and compare the webs $\boldsymbol{\mathcal W}_{{\rm dP}_d}$ defined by the conic fibrations on a 
smooth  
del Pezzo surface ${\rm dP}_d$ of degree $d$ for $d=4$ and $d=5$. 
 In a previous paper,  we proved that for any positive $d\leq 6$, the web 
$\boldsymbol{\mathcal W}_{{\rm dP}_d}$ 
 carries a particular abelian relation 
${\bf HLog}_d$, whose components all are weight $7-d$ antisymmetric hyperlogarithms.  
The web $\boldsymbol{\mathcal W}_{{\rm dP}_5}$ is a geometric model of the exceptional Bol's web and the relation ${\bf HLog}_5$ corresponds to the famous `Abel's identity' $(\boldsymbol{{\mathcal A}b})$ of the dilogarithm. 
Bol's web together with 
 $(\boldsymbol{{\mathcal A}b})$
enjoy several remarkable properties of different kinds.
We show that almost all of them admit natural generalizations
to the pair $\big( \boldsymbol{\mathcal W}_{{\rm dP}_4}, {\bf HLog}_4\big)$. 
\end{abstract}
\bk 
\mk 

In the whole paper, we work over the field $\mathbf C$ of complex numbers, in the analytic or algebraic category depending on the context.
We will often work with the affine coordinates $x,y$ on the complex projective plane $\mathbf P^2$ 
which correspond to the embedding $\mathbf C^2\hookrightarrow \mathbf P^2,\, (x,y)\mapsto [x:y:1]$.

\section{\bf Introduction}
\label{S:Introduction}
`{\it Cauchy's identity}' 
\begin{equation}
\label{Eq:EFA-3-terms-Log}
{\rm Log}(x)+{\rm Log}(y)-{\rm Log}(xy) = 0
\end{equation}
satisfied by the logarithm 
${\rm Log}(\cdot)$
is certainly one of the most important functional identities in mathematics. It admits a `weight 2 ' generalization, the so-called `{\it Abel's}' or `{\it five-term identity}'
$$
\boldsymbol{\big(\mathcal Ab\big)}
\hspace{3cm}
R(x)-R(y)-R\left(\frac{x}{y}\right)-R\left(\frac{1-y}{1-x}\right)
+R\left(\frac{x(1-y)}{y(1-x)}\right)=0\,
\hspace{3cm} {}^{}
$$
which is satisfied for  any $(x,y)\in \mathbf R^2$ such that $0<x<y<1$, by \href{https://mathworld.wolfram.com/RogersL-Function.html}{\it Rogers' dilogarithm} that is the function defined by 
$R(x)= {\bf L}{\rm i}_2(x) + \frac{1}{2}{\rm Log}(x){\rm Log}(1 - x) - 
{\pi^2}/{6}$
for $x\in ]0,1[$.\footnote{Here ${\bf L}{\rm i}_2$ stands for the classical bilogarithm, defined as the sum of the series  ${\bf L}{\rm i}_2(z)=\sum_{n=1}^{+\infty} z^n/n^2$ for $\lvert z \lvert <1$.}
Abel's dilogarithmic identity 
$\boldsymbol{\big(\mathcal Ab\big)}$ appears in several domains of mathematics hence connecting them.\footnote{To cite a few:  hyperbolic geometry, $K$-theory of number fields, web geometry, theory of cluster algebras, mirror symmetry of log CY manifolds,  scattering amplitudes, etc.}  

In \cite{CP}\footnote{See also the longer preliminary version \cite{Hlog}.}, 
we described a series of hyperlogarithmic functional identities ${\bf HLog}_d$ for $d$ ranging from 1 to 6, constructed from del Pezzo surfaces in a uniform way and such that ${\bf HLog}_6$ and ${\bf HLog}_5$ precisely coincide with the 
classical identities \eqref{Eq:EFA-3-terms-Log} and $\boldsymbol{\big(\mathcal Ab\big)}$ respectively.  
For any $d\in \{1,\ldots,6\}$, there is one  identity ${\bf HLog}_d$ for each smooth degree $d$ del Pezzo surface  ${\rm dP}_d$, and the former functional equation 
can be described geometrically as follows: up to post-compositions by projective automorphisms, there exists only a finite number $\kappa_d$ of `conic fibrations'    $U_i: {\rm dP}_d\rightarrow \mathbf P^1$ ($i=1,\ldots,\kappa_d$). For each $U_i$, 
the set $\Sigma_i$ of  values $\sigma\in   \mathbf P^1$ such that the fiber $U_i^{-1}(\sigma)$ decomposes in the some of two `lines' has cardinality $8-d$. 
To each $\Sigma_i$ is canonically associated  a certain `complete antisymmetric hyperlogarithm  $AH_{\Sigma_i}$' on $\mathbf P^1$, of weight $7-d$, which is well-defined up to sign.  
In \cite{CP}, we proved that for a suitable choice of the $AH_{\Sigma_i}$'s\footnote{In \cite{CP}, we give an effective construction of the suitable hyperlogarithms $AH_{\Sigma_i}$'s  such that  the functional identity $\sum_{i=1}^{\kappa_d}AH_{\Sigma_i}\big(U_i\big) =0$ be indeed satisfied. Moreover, we give a conceptual interpretation of this construction in terms of representations of the Weyl group associated to the considered del Pezzo surface.},  the following identity holds true at the generic point of the considered del Pezzo surface ${\rm dP}_d$: 
$$
\boldsymbol{\big({\bf HLog}_d \big)}
\hspace{5cm}
\sum_{i=1}^{\kappa_d}AH_{\Sigma_i}\big(U_i\big) \equiv 0\,. 
\hspace{8cm} {}^{}
$$

Considering the importance of the functional identities $(1)\simeq {\bf HLog}_6$ and  $\boldsymbol{\big(\mathcal Ab\big)}\simeq {\bf HLog}_6$, it is natural to wonder about the case of the other $\boldsymbol{\big({\bf HLog}_d \big)}$'s. 
In this paper, 
we deal with $\boldsymbol{\big({\bf HLog}_4 \big)}$ that we investigate through the prism of web geometry.
\mk

Web geometry is the study of webs which are geometric objects formed (locally) by the configuration of a finite number of foliations whose leaves intersect transversaly. It proved to be an interesting tool to study geometrically functional identities such as the $\boldsymbol{\big({\bf HLog}_d \big)}$'s.  Let $V_1,\ldots,V_k$ be holomorphic submersions on a complex manifold $M$, of dimension 2 say, such that $dV_i\wedge dV_j$ does not vanish identically for any $i\neq j$. The level sets of the $V_i$'s define a $k$-web denoted by  
$$ \boldsymbol{\mathcal W}= \boldsymbol{\mathcal W}\big( V_1,\ldots,V_{k}\big) \, .$$ 

Via projective duality, one classically associates a linear web $\boldsymbol{\mathcal W}_C$ to any reduced algebraic curve $C$, giving rise to the important notions of `{\it algebraic and algebraizable webs'}.
It is relevant to try mimicking for a general web, all the  constructions of classical algebraic geometry which apply to algebraic curves. This led to the notion of {\it `abelian relation'} (ab.\,AR) for an arbitrary web such as 
the $k$-web $\boldsymbol{\mathcal W}$ above: it is a $k$-tuple $(F_i)_{i=1}^k$ of (germs of) holomorphic functions such that the following functional identity is satisfied:
\begin{equation}
\label{Eq:AR}
{}^{} 
\hspace{0.5cm}
\sum_{i=1}^{k}F_i\big(V_i\big) =0\,. 
\end{equation}
The abelian relations of $\boldsymbol{\mathcal W}$ form a vector space\footnote{Actually, globally the ARs form a local system but we will not elaborate further on this  here.} $\boldsymbol{AR}( \boldsymbol{\mathcal W})$ whose dimension is known as the {\it rank} of the web and denoted by 
${\rm rk}( \boldsymbol{\mathcal W} )=\dim\,\boldsymbol{AR}( \boldsymbol{\mathcal W})$. By a result of Bol, we always have 
${\rm rk}( \boldsymbol{\mathcal W} )\leq (k-1)(k-2)/2$. 
For an algebraic web $\boldsymbol{\mathcal W}_C$, 
the space $\boldsymbol{AR}( \boldsymbol{\mathcal W}_C) $ linearly identifies with  ${\bf H}^0(C,\omega^1_C)$ hence the rank of $\boldsymbol{\mathcal W}_C$ is equal to the arithmetic genus of $C$:  one has ${\rm rk}( \boldsymbol{\mathcal W}_C )=p_a(C)=(d-1)(d-2)/2$ with $d=\deg(C)$  therefore Bol's bound actually is an equality in this case. 
\sk

The identity $\boldsymbol{\big({\bf HLog}_d \big)}$ above precisely is of the form 
\eqref{Eq:AR} which naturally leads to consider the associated web, namely the web 
$$ \boldsymbol{\mathcal W}_{ {\rm dP}_d }= \boldsymbol{\mathcal W}\big( U_1,\ldots,U_{\kappa_d}\big)$$ 
on ${\rm dP}_d$, defined by all the fibrations in conics on it.  For $d=6$, taking ${\rm dP}_6$ as the blow-up of $\mathbf P^2$ at the vertices $p_1,p_2,p_3$ of the standard projective frame relative to the choice of the affine embedding $\mathbf C^2\subset \mathbf P^2$ given by $(x,y)\mapsto [x:y:1]$ in coordinates, we obtain that 
$ \boldsymbol{\mathcal W}_{ {\rm dP}_6 }= \boldsymbol{\mathcal W}\big( \, x \,  , \, y
\, , \, x/y \,
\big)$ 
which is one of the two classical normal forms of the so-called 
{\it hexagonal 3-web}.\footnote{The other normal form we are referring to here being 
$\boldsymbol{\mathcal W}( \, x \,  , \, y
\, , \, x+y\,  )\, $.} 

   The case of $ \boldsymbol{\mathcal W}_{ {\rm dP}_5 }$ is much more interesting: 
viewing ${\rm dP}_5$  as the blow up of the plane at the points $p_i$ for $i=1,\ldots,4$ with $p_1,p_2,p_3$ the same as the points defined above and $p_4=[1:1:1]$, we obtain that with respect to the standard affine coordinates $(x,y)$, one has 
$$ \boldsymbol{\mathcal W}_{ {\rm dP}_5}\simeq  \boldsymbol{\mathcal W}\Bigg( \, x \,  , \, y
\, , \, \frac{x}{y} 
\, , \, \frac{1-y}{1-x} 
\, , \, \frac{x(1-y)}{y(1-x)} 
\, 
\Bigg)\, .$$

The five rational functions appearing as first integrals for $\boldsymbol{\mathcal W}_{ {\rm dP}_5}$ in the chosen coordinates are precisely those appearing as arguments of Rogers' dilogarithm in Abel's identity $\boldsymbol{\big(\mathcal Ab\big)}$. We recognize the five first integrals of the famous Bol's web $\boldsymbol{\mathcal B}$. 
This web is quite famous in web geometry since first, it is the first example  of an exceptional web ever discovered and second, because it enjoys several remarkable properties we are going to discuss below. Considering this, it is natural to investigate 
the webs $ \boldsymbol{\mathcal W}_{ {\rm dP}_d}$ for $d=4,3,2,1$. 

Each web $ \boldsymbol{\mathcal W}_{ {\rm dP}_d}$ carries the hyperlogarithmic AR corresponding to $({\bf HLog}_d)$ hence it is interesting to look at the whole space of the ARs and more generally to its properties. 
In what follows, we first discuss several properties of Bol's web then discuss in analogy the case of $ \boldsymbol{\mathcal W}_{ {\rm dP}_4}$ which is the main object of study of this text. \sk

\noindent{\bf Notation:} 
{\it In what follows, 
we will also denote by 
${\bf HLog}^{7-d}$ the identity ${\bf HLog}_{d}$ (for any $d=1,\ldots,6$), this in order to emphasize
the weight $7-d$ of the hyperlogarithms involved in it.}

\noindent{\bf Warning:} 
{\it 
Some of the statements in the next two subsections may look a bit cryptic at first sight, even for 
people familiar with web geometry. Everything will be carefully explained further.}

\subsection{Many remarkable properties of Bol's web}
\label{S:Web-WdP5-properties}
Abel's identity $\boldsymbol{(\mathcal Ab)}$ gives rise to an abelian relation for $\boldsymbol{\mathcal B}$, that we will denote by  $\boldsymbol{\mathcal Ab}$. 

Given a planar web  $\boldsymbol{\mathcal W}$, 
one sets  $\boldsymbol{AR}\big(\boldsymbol{\mathcal W}\big)$  for 
the space of its abelian relations.  Inspired by the terminology introduced by Damiano in  
\cite{Damiano}, we will say that an AR of $\boldsymbol{\mathcal W}$ 
is  {\it `combinatorial'} if it is of minimal length, that is if only 3 of its components 
are not trivial, and we will denote by $\boldsymbol{AR}_C\big(
\boldsymbol{\mathcal W}\big)$ the subspace of $\boldsymbol{AR}\big(
\boldsymbol{\mathcal W}\big)$ spanned by the ARs of this kind.

It has been known for a long time that Bol's web 
 is very particular as a planar web since it enjoys all the following  
 properties:  
\begin{itemize}
\item[] ${}^{}$ 
\hspace{-1cm}
{\bf 1.\,[\,Geometric definition\hspace{0.03cm}].} {\it 
$\boldsymbol{\mathcal B}\simeq \boldsymbol{\mathcal W}_{{\rm dP}_5}$: 
Bol's web is equivalent to the web by conics on ${\rm dP}_5$.}
\bk 
\item[] ${}^{}$ 
\hspace{-1cm}
{\bf 2.\,[\,Non linearizability\hspace{0.03cm}].} {\it 
Bol's web is not linearizable hence not algebraizable.}
\bk  
\item[] ${}^{}$ 
\hspace{-1cm}
{\bf 3.\,[\,Abelian relations\hspace{0.03cm}].}
{\it $\boldsymbol{a.}$ All the ARs of $\boldsymbol{\mathcal B}$ are polylogarithmic, of weight 1 or 2:}\sk \\
${}^{}$ \hspace{0.9cm} 
{\it   $-$ the subspace  $\boldsymbol{AR}_{log}\big(
\boldsymbol{\mathcal B}
\big)$ of logarithmic ARs of $\boldsymbol{\mathcal B}$ has dimension 5. It coincides}\\
${}^{}$ \hspace{0.9cm} 
{\it   \textcolor{white}{$-$}  with
 the space $\boldsymbol{AR}_C\big(
\boldsymbol{\mathcal B}\big)$ of combinatorial ARs of $\boldsymbol{\mathcal B}$;}\sk \\
${}^{}$ \hspace{0.9cm} 
{\it   $-$ the subspace  $\boldsymbol{AR}^2\big(
\boldsymbol{\mathcal B}
\big)$ of dilogarithmic ARs of $\boldsymbol{\mathcal B}$ 
is spanned by $\boldsymbol{\mathcal Ab}$; and} \sk \\
${}^{}$ \hspace{0.9cm} 
  {\it $-$ there is a decomposition in direct sum}\sk \\
${}^{}$ \hspace{5cm} 
$\boldsymbol{AR}\big(\boldsymbol{\mathcal B}\big)=\boldsymbol{AR}_{log}\big(\boldsymbol{\mathcal B}\big)\oplus \big\langle \boldsymbol{\mathcal Ab} \big\rangle  \, .
$ 
\hfill 
 $(\star)$\bk \\
 ${}^{}$ \hspace{-0.4cm} 
 $\boldsymbol{b.}$ {\it Abel's AR $ \boldsymbol{\mathcal Ab}$ spans the whole space 
$ \boldsymbol{AR}_{log}\big(\boldsymbol{\mathcal B}\big)$ by residues/monodromy.}\sk\\
${}^{}$ \hspace{-0.4cm} 
 $\boldsymbol{c.}$ {\it Conversely,  one can reconstruct $ \boldsymbol{\mathcal Ab}$  
from a basis of the space $\boldsymbol{AR}^{1}\big(\boldsymbol{\mathcal B}\big)$.}
\mk
\item[] ${}^{}$ 
\hspace{-1cm}
{\bf 4.\,[\,Rank\hspace{0.03cm}].} {\it Bol's web has maximal rank hence is exceptional}.
\mk 
\item[] ${}^{}$ 
\hspace{-1cm}
{\bf 5.\,[\,`Algebraization'\hspace{0.03cm}].} 
{\it  Let $\boldsymbol{\mathcal W}$ be a germ of 5-web 
 at the origin of $\mathbf C^2$,  
equivalent to Bol's web. There are two distinct canonical ways to recover 
$\boldsymbol{\mathcal B}$ (and ${\rm dP}_5$) from 
$\boldsymbol{\mathcal W}$:}
\sk \\
${}^{}$ \hspace{-0.1cm} 
{\it   $-$ the first is from the 
space $\boldsymbol{AR}_C\big(
\boldsymbol{\mathcal W} \big)$ of 
combinatorial ARs of $\boldsymbol{\mathcal W}$;}
\sk \\
${}^{}$ \hspace{-0.1cm} 
{\it   $-$ the second approach is via the canonical 
map $\varphi_{\boldsymbol{\mathcal W}}: 
(\mathbf C^2,0)
\rightarrow \mathcal M_{0,5}$ associated to  $\boldsymbol{\mathcal W}$;}
 \sk
\vspace{-0.26cm}
  \item[] ${}^{}$ 
\hspace{-1cm}
{\bf 6.\,[\,Weyl group action\hspace{0.03cm}].} {\it $\boldsymbol{a.}$ There is a natural action of 
 $W(A_4)=\mathfrak S_5$ on  the space of abelian 
  relations for which  $(\star)$ is the decomposition into $\mathfrak S_5$-irreducible representations.}\sk
  
\noindent{\it $\boldsymbol{b.}$ In particular,  $ \big\langle \boldsymbol{\mathcal Ab} \big\rangle$ is $\mathfrak S_5$-stable and 
isomorphic to the signature  as a $\mathfrak S_5$-representation.}
  \mk 
\vspace{-0.5cm}
\item[] ${}^{}$ 
\hspace{-1cm}
 {\bf 7.\,[\,Hexagonality \& Characterization\hspace{0.03cm}].} {\it 
 Bol's web is hexagonal and is essentially characterized by this property 
 since 
 `for any $k\geq 3$, a  hexagonal planar $k$-web  either is linearizable and   equivalent to a web formed by $k$ pencils of lines  or $k=5$ and it is equivalent to $\boldsymbol{\mathcal B}$'.}
 \mk
\item[] ${}^{}$ 
\hspace{-1cm}
  {\bf 8.\,[\,Construction 
  \textit{\textbf{\`a la}}
  GM\hspace{0.03cm}].} 
  {\it  
 $\boldsymbol{\mathcal B}$  
   can be constructed following the approach of Gelfand and MacPherson \cite{GM}:  it can be seen as  the quotient, under the action of the Cartan torus      $H_4\subset {\rm SL}_5(\mathbf C)$, of a natural $H_4$-equivariant web defined on the grassmannian 
 $G_2(\mathbf C^5)$.}
  \mk
\item[] ${}^{}$ 
\hspace{-1cm} 
{\bf 9.\,[\,Modularity\hspace{0.03cm}].} 
{\it Bol's web is modular: it is equivalent to the web 
$\boldsymbol{\mathcal W}_{{0,5}}$ on $\overline{\mathcal M}_{0,5}$
defined by the five forgetful maps 
$\overline{\mathcal M}_{0,5}\rightarrow \overline{\mathcal M}_{0,4}$.}
\mk
\item[] ${}^{}$ 
\hspace{-1cm} 
{\bf 10.\,[\,Cluster web\hspace{0.03cm}].} 
{\it Bol's web is of cluster type: up to equivalence, it can be defined by means of the $\mathscr  X$-cluster variables of the finite type cluster algebra of type $A_2$.}
 \end{itemize}

Some of the statements above are obvious/tautological ({\it e.g.}\,the first), some others are very classical ({\bf 2}, {\bf 3} and {\bf 7} for instance) or have been proved rather recently (such as {\bf 6}, {\bf 8} or {\bf 9}).  The fifth statement
 is new (and actually provocative without further explanations, as it seems to contradict the fourth!). In any case, all of these are well-known or easy to prove and it is the fact that all are satisfied by   $(\boldsymbol{\mathcal B}, \boldsymbol{\mathcal Ab} \big)$ which makes this pair such an interesting object. 

Although rather long, the list above is not complete since an interesting feature of  
$( \boldsymbol{\mathcal Ab})$, related to {\bf 8}, is  missing in it. Namely that the real identity  
$ (\boldsymbol{\mathcal Ab})$
can be obtained geometrically within Gelfand-MacPherson formalism \cite{GM}, by means of the invariant representative of the first Pontryagin class of the tautological bundle on the real grassmannian $G_2(\mathbf R^5)$. Since this is a phenomenon taking place in a real setting and because we do not have any counterpart for ${\bf HLog}^3$ yet, this result has not been included in the list above (although it truly deserves to be part of it).  We will discuss this matter  a bit further in \S\ref{SS:HLog3-a-la-GM}. 
\begin{center}
\vspace{-0.0cm}
$\star$
\end{center}

Since $\boldsymbol{\mathcal B}$ is equivalent to $\boldsymbol{\mathcal W}_{ {\rm dP}_5}$ and because what can be considered as its most important feature, namely the fact it carries the dilogarithmic abelian relation  
$\boldsymbol{\mathcal Ab}$, generalizes to $\boldsymbol{\mathcal W}_{ {\rm dP}_4}$, one can wonder whether the properties listed above admit analogs for this 10-web or not.
It turns out that it is indeed the case as we are going to discuss now.

\subsection{Many remarkable properties of  ${\mathcal W} \hspace{-0.48cm}{\mathcal W}_{{\bf dP}_4}$}
\label{S:Web-WdP4-properties}
The main aim of this text is to prove that the following facts hold true, where 
${\rm dP}_4$ stands for an arbitrary (fixed)  smooth del Pezzo quartic surface. 
\sk

\vspace{-0.4cm}
\begin{itemize}
\item[] ${}^{}$ 
\hspace{-1cm}
{\bf 1.\,[\,Geometric definition\hspace{0.03cm}].} {\it 
$ \boldsymbol{\mathcal W}_{{\rm dP}_4}$ 
is the web defined by the 10 pencils of conics on ${\rm dP}_4$.}
\mk \sk
\item[] ${}^{}$ 
\hspace{-1cm}
{\bf 2.\,[\,Non linearizability\hspace{0.03cm}].} 
{\it $\boldsymbol{\mathcal W}_{ {\rm dP}_4}$ is not linearizable hence not equivalent to an algebraic web.}\mk  \sk
\item[] ${}^{}$ 
\hspace{-1cm}
{\bf 3.\,[\,Abelian relations\hspace{0.03cm}].}
$\boldsymbol{a.}$  
{\it All the ARs of $\boldsymbol{\mathcal W}_{ {\rm dP}_4}$ are hyperlogarithmic, of weight 1, 2 or 3:\mk}\\
${}^{}$ \hspace{-0.7cm} 
{\it   $-$ the subspace  $\boldsymbol{AR}^1\big(
\boldsymbol{\mathcal W}_{ {\rm dP}_4}
\big)$ of logarithmic ARs of $\boldsymbol{\mathcal W}_{ {\rm dP}_4}$ has dimension 20. It coincides with\\ ${}^{}$ \hspace{-0.4cm}  the space $\boldsymbol{AR}_C\big(
\boldsymbol{\mathcal W}_{ {\rm dP}_4}
\big)$ of combinatorial ARs of $\boldsymbol{\mathcal W}_{ {\rm dP}_4}$;}\mk \\
 ${}^{}$ \hspace{-0.7cm}  {\it  
  $-$ the subspace  $\boldsymbol{AR}^2\big(
\boldsymbol{\mathcal W}_{ {\rm dP}_4}
\big)$ of weight 2 hyperlogarithmic ARs has dimension 15: it is the} \\
 ${}^{}$ \hspace{-0.8cm}   {\it \textcolor{white}{$-$} direct sum
 of two subspaces,  
a first one,  $\boldsymbol{AR}^{2}_{\rm sym}$,  of `symmetric' ARs, of dimension 5 and }
\\
 ${}^{}$ \hspace{-0.8cm}   {\it \textcolor{white}{$-$} another one denoted by $ \boldsymbol{AR}^{2}_{\rm asym}$ of `antisymmetric' ARs, of dimension 10;} \mk\\ 
${}^{}$ \hspace{-0.7cm} 
{\it   $-$ the subspace  $\boldsymbol{AR}^3\big(
\boldsymbol{\mathcal W}_{ {\rm dP}_4}
\big)$ of weight 3 hyperlogarithmic ARs of $\boldsymbol{\mathcal W}_{ {\rm dP}_4}$ 
is spanned by ${\bf Hlog^3}$;}\mk \\
${}^{}$ \hspace{-0.7cm} 
  {\it $-$ there is a decomposition in direct sum} 
 \\ 
${}^{}$ \hspace{3cm} 
$\boldsymbol{AR}\big( 
\boldsymbol{\mathcal W}_{ {\rm dP}_4} 
 \big)=\boldsymbol{AR}^{1}\big(\boldsymbol{\mathcal W}_{ {\rm dP}_4}\big)\oplus 
 \Big(
 \boldsymbol{AR}^{2}_{\rm sym}\oplus 
 \boldsymbol{AR}^{2}_{\rm asym}
 \Big)
 \oplus
 \big\langle  \,
 {\bf Hlog^3}\, 
 \big\rangle \, .
$  \qquad  $(\clubsuit)$\mk \\
${}^{}$ \hspace{-0.7cm} 
 $\boldsymbol{b.}$ {\it ${\bf Hlog^3}$
 spans the whole space 
 of antisymmetric ARs 
 of $\boldsymbol{\mathcal W}_{ {\rm dP}_4}$ 
by residues/monodromy.}
\mk \\
${}^{}$ \hspace{-0.7cm} 
 $\boldsymbol{c.}$ {\it Conversely,  one can reconstruct ${\bf Hlog^3}$
from a basis of the space $\boldsymbol{AR}^{2}_{\rm asym}$.}
\bk
\item[] ${}^{}$ 
\hspace{-1cm}
{\bf 4.\,[\,Rank\hspace{0.03cm}].} 
{\it The web $\boldsymbol{\mathcal W}_{ {\rm dP}_4}$ has maximal rank hence is exceptional}. 
\bk 
\item[] ${}^{}$ 
\hspace{-1cm}
{\bf 5.\,[\,`Algebraization'\hspace{0.03cm}].} 
{\it  Let $\boldsymbol{\mathcal W}$ be a germ of 10-web 
 at the origin of $\mathbf C^2$,  
equivalent to $\boldsymbol{\mathcal W}_{ {\rm dP}_4}$. There are two distinct canonical ways to recover 
the latter (and ${\rm dP}_4$) from 
$\boldsymbol{\mathcal W}$:}
\mk \\
${}^{}$ \hspace{-0.1cm} 
{\it   $-$ the first is from the 
space $\boldsymbol{AR}_C\big(
\boldsymbol{\mathcal W} \big)$ of 
combinatorial ARs of $\boldsymbol{\mathcal W}$;}
\mk \\
${}^{}$ \hspace{-0.1cm} 
{\it   $-$ the second approach is via the canonical 
map $\varphi_{\boldsymbol{\mathcal H\mathcal W}_5}: 
(\mathbf C^2,0)
\rightarrow \mathcal M_{0,5}$ associated to}
\sk \\
${}^{}$ \hspace{-0.1cm} 
{\it \textcolor{white}{$-$} any linearizable hexagonal 5-subweb 
$\boldsymbol{\mathcal H\mathcal W}_5$ of 
$\boldsymbol{\mathcal W}$. }
 \bk   
\item[] ${}^{}$ 
\hspace{-1cm}
{\bf 6.\,[\,Weyl group action\hspace{0.03cm}].} 
{\it $\boldsymbol{a.}$ The Weyl group $W=W(D_5)$ acts naturally on $\boldsymbol{AR}\big(\boldsymbol{\mathcal W}_{ {\rm dP}_4} \big)$ and $(\clubsuit)$
 actually is  its decomposition into irreducible $W$-modules.}\mk
 
\noindent{\it $\boldsymbol{b.}$ In particular,  $ \big\langle {\bf Hlog^3} \big\rangle$ is $W$-stable and 
isomorphic to the signature  as a $W$-representation.} 
  \bk 
\vspace{-0.5cm}
\item[] ${}^{}$ 
\hspace{-1cm}
 {\bf 7.\,[\,Hexagonality \& Characterization\hspace{0.03cm}].} 
{\it $\boldsymbol{\mathcal W}_{ {\rm dP}_4}$ is characterized by  its hexagonal 3-subwebs.}
 \bk
\item[] ${}^{}$ 
\hspace{-1.05cm}
  {\bf 8.\,[\,Construction 
   \textit{\textbf{\`a la}}
   GM\hspace{0.03cm}].} 
 {\it  
  $\boldsymbol{\mathcal W}_{ {\rm dP}_4}$ can be seen as  the quotient, under the action of the Cartan torus   $H_5\subset {\rm Spin}_{10}(\mathbf C)$, of a natural $H_5$-equivariant web defined on the 
spinor variety $\mathbb S_{5}$.
}
 \bk
\vspace{-0.4cm} 
   \item[] ${}^{}$ 
\hspace{-1cm} 
{\bf 9.\,[\,Modularity\hspace{0.03cm}].} 
{\it  The web 
$\boldsymbol{\mathcal W}_{ {\rm dP}_4}$ is modular:  
it can be obtained as the pull-back under a natural map ${\rm dP}_4\dashrightarrow {\rm Conf}_6(\mathbf P^2)$ of 
the web induced by the 30 rational maps ${\rm Conf}_6(\mathbf P^2)\dashrightarrow 
{\rm Conf}_4(\mathbf P^1)=
{\mathcal M}_{0,4}\simeq \mathbf P^1$
on the space ${\rm Conf}_6(\mathbf P^2)$ of projective configurations of six points on 
$\mathbf P^2$. 
}
\bk
\item[] ${}^{}$ 
\hspace{-1.05cm}
{\bf 10.\,[\,Cluster web\hspace{0.03cm}].} 
{\it Del pezzo's web $\boldsymbol{\mathcal W}_{ {\rm dP}_4}$ is of cluster type: it can be obtained by means of some of the $\mathscr X$-cluster  variables of the finite type cluster algebra of type $D_4$.} 
  \mk
 \end{itemize}

A few remarks about these statements are in order.
\sk
\begin{itemize}
\item As in the case of Bol's web  in \S\ref{S:Web-WdP5-properties}, all the assertions above do not have the same status. For instance, the first is tautological and the second is almost obvious.  The fourth point was known before (see just below), $\boldsymbol{6.b}$ has been established in \cite{Hlog}.  All the other statements (namely {\bf 3}, {\bf 5}, $\boldsymbol{6.a}$, and from {\bf 7} to {\bf 10}) are new.
\mk
\item 
Classically in web geometry, a maximal rank web is said to be exceptional precisely when it is not algebraizable: for such a web, one of these two adjectives means exactly the opposite of the other.  
Since both $\boldsymbol{\mathcal W}_{ {\rm dP}_5}\simeq \boldsymbol{\mathcal B}$
and $\boldsymbol{\mathcal W}_{ {\rm dP}_4}$ are exceptional in the classical 
sense, the use of the term `algebraization'  about them could appear completely 
 inconsistent
at first reading.  
The terminology that we introduce here is a bit provocative and has to be understood as follows: if these two webs are not linearizable, 
each of them  admits for natural model a web formed by the pencils of conics on a certain rational surface. The term `algebraization' used here in relation to them refers to the process of constructing an equivalence to this `natural algebraic model' when starting from any other web analytically equivalent to it. 
\mk
\item  Birationally equivalent models of $\boldsymbol{\mathcal W}_{ {\rm dP}_4}$ were considered in the author's PhD thesis \cite{PThese}, in the unpublished text \cite{Robert}  and also later in the article \cite{Pereira}.  That this web is exceptional (which is equivalent to the points {\bf 2.} and {\bf 4.}\,together) is mentioned in \cite{PThese} and \cite{Robert}  and more thoroughly investigated by Pereira who computed the hyperlogarithmic ranks  of a birational model of $\boldsymbol{\mathcal W}_{ {\rm dP}_4}$ 
({\it cf.}\,the proof of Theorem 5.1 in  \cite{Pereira}).
\mk
\item Many of the results established in the present text were previously announced without proof in our preprint \cite{Hlog} (in particular, in \S4.2 therein). 
\end{itemize}

What one has to have in mind is that,  considering that almost all (if not all) the remarkable properties of the pair Abel's identity $\boldsymbol{\big(\mathcal Ab\big)}\simeq {\bf HLog}_5$ / Bol's web $\boldsymbol{\mathcal B}\simeq \boldsymbol{\mathcal W}_{ {\rm dP}_5}$  admit formally very similar counterparts for the pair $ {\bf HLog}_4$ /  $\boldsymbol{\mathcal W}_{ {\rm dP}_4}$, then is is not 
unreasonable to see 
 the latter  as the most natural weight 3 generalization of the former.  
 
 The quest of natural generalizations to any weight of Abel's dilogarithmic identity is a long standing one which has received many contributions by several authors. But except until very recently and for very few cases, all these contributions were sticking to the case of polylogarithms or more precisely to iterated integrals on $\mathbf P^1\setminus \{0,1,\infty\}$ constructed from words in the alphabet with 
the two logarithmic forms 
$d{\log}(u)=du/u$ and $d{\log}(u-1)=du/(u-1)$ as letters.  The discovery of the identities ${\bf HLog}_d$ for $d=4,3,2,1$  as well as the striking fact that $\boldsymbol{\mathcal B}\simeq \boldsymbol{\mathcal W}_{ {\rm dP}_5}$ and $\boldsymbol{\mathcal W}_{ {\rm dP}_4}$ look so similar regarding the numerous remarkable properties they satisfy 
may be an indication that natural generalizations to higher weights of the five terms relation of the dilogarithm have to be looked for in the more general setting of hyperlogarithms. 

However, what makes the functional equations of polylogarithms so interesting is the role that these identities are playing in other areas of mathematics, such as in hyperbolic geometry (volume of hyperbolic polytopes) or in the K-theory of number fields (regarding to regulators and Zagier's polylogarithmic conjecture). For the moment, we are not aware of any occurence of the identity ${\bf HLog}^3$ within other fields of mathematics, which would be necessary for it to be fully recognized as the genuine  weight 3 generalization of Abel's identity.

\newpage
\subsection{Structure of the paper}
The rest of the paper is organized as follows.

 In {\bf Section \S\ref{S:Preliminary-Material}}, 
we introduce the material about hyperlogarithms, webs and del Pezzo surfaces which will be used in the rest of the text.   Everything here is well-known or straightforward, hence no proof has been included but only some references when they are needed. {\bf Section \S\ref{S:Bol-web}} is where we discuss  all the properties of Bol's web cited in \S\ref{S:Web-WdP5-properties} above with more details.  
  That $\boldsymbol{\mathcal W}_{{\rm dP}_4}$ satisfies all the properties 
 listed in \S\ref{S:Web-WdP4-properties} is proved in 
  {\bf Section \S\ref{S:WdP4}}, which constitues the main part of the paper. 
 In the last {\bf Section \S\ref{S:What-next}}, we discuss  some questions left open by the results obtained here that we find interesting, such as that of constructing 
 several other versions of 
 the identity ${\bf HLog}^3$. 
  \begin{center}
\vspace{-0.0cm}
$\star$\sk
\end{center}
 
 To help the reader to navigate throughout the text, we offer 
 Table \ref{Table:map}  which indicates 
  where in the text are discussed the different properties of the webs under consideration.
\begin{table}[!h]
\begin{tabular}{|l|c|c|}
\hline
\begin{tabular}{c} ${}^{}$\vspace{-0.27cm}  \\
${}^{}$ \hspace{-0.2cm} {\bf Property}\vspace{0.09cm} \\
\end{tabular}
 &
 $\boldsymbol{\mathcal B}\simeq \boldsymbol{\mathcal W}_{ {\bf dP}_5 }$  & 
$\boldsymbol{\mathcal W}_{ {\bf dP}_4 }$
\\
\hline\hline 
 \begin{tabular}{c}${}^{}$\vspace{-0.33cm}  \\
 {\bf Description of the space of ARs as a $\boldsymbol{W}$-module} 
\vspace{0.09cm} \\
\end{tabular}
  &  
 \S\ref{SS:ARs-WdP5}
  \& 
  \S\ref{SS:Bir-Symmetries-W-action} & 
    \S\ref{SS:Symbolic-ARs-WdP4}
\\
\hline
 \begin{tabular}{c}${}^{}$\vspace{-0.33cm}  \\
 {\bf Canonical algebraizations }
\vspace{0.09cm} \\
\end{tabular}
  & \S\ref{S:Algebraization-WdP5} & 
   \S\ref{SS:WdP4-algebraization}
\\
\hline
 \begin{tabular}{c}${}^{}$\vspace{-0.33cm}  \\
 {\bf Combinatorial characterization}
\vspace{0.09cm} \\
\end{tabular}
  & \S\ref{SS:Combinatorial-characterization} & 
  \S\ref{SS:WdP4-Combinatorial-characterization}
\\
\hline
\begin{tabular}{c}${}^{}$\vspace{-0.33cm}  \\
{\bf Description as a modular web} 
\vspace{0.09cm} \\
\end{tabular} &  
p.\,\pageref{Modular-W-M05}
&   \S\ref{SS:WdP4-as-a-modular-web}
\\
\hline
\begin{tabular}{c}${}^{}$\vspace{-0.33cm}  \\
{\bf Description \`a la Gelfand-MacPherson} 
\vspace{0.09cm} \\
\end{tabular} &  \S\ref{SS:Gelfand-MacPherson--construction-WdP5} & 
\S\ref{SS:WdP4-GM}
\\
\hline
\begin{tabular}{c}${}^{}$\vspace{-0.33cm}  \\
{\bf Description as  a cluster web}
\vspace{0.09cm} \\
\end{tabular} & 
\S\ref{SS:Clsuter-Web-type-A2}
 & 
 \S\ref{SS:WdP4-as-a-cluster-web}
\\
\hline 
\end{tabular}
\bk
\caption{}
\label{Table:map}
\end{table}

 
 \vspace{-0.3cm}
\subsection{\bf Acknowledgements}
 During the preparation of this text, the author benefited of interesting discussions with 
 Ana-Maria Castravet  and Nicolas Perrin. He is very grateful to them for that.  
  The author would also like to thank Brubru for her proofreading and for all the corrections she was able to point out. Noucnouc did a few remarks about the English too, we thank him for that.

\newpage
\section{\bf Hyperlogarithms, webs and del Pezzo surfaces in a nutshell}
\label{S:Preliminary-Material}
Here we first give a quick account of the classical theory of hyperlogarithms  then 
we discuss some points 
of web geometry.  After short reminders on webs, we discuss some new (although rather elementary) notions concerning abelian relations, related to hyperlogarithms. We also introduce a simple formalism which will prove to be convenient to investigate the rich combinatorial structure of the spaces of ARs of the del Pezzo webs we will study in the next sections. 
Then we recall some well-known facts about the geometry of del Pezzo surfaces and the Weyl groups actions on their Picard lattices. 
\sk

Almost no examples are given in this section. The two cases of ${\boldsymbol{\mathcal W}}_{{\rm dP}_5}$ and 
${\boldsymbol{\mathcal W}}_{{\rm dP}_4}$ studied in detail in the following sections will serve to show, among other things, the relevance of the notions introduced here.  

\subsection{\bf Hyperlogarithms}                                                                                                                                                                                                                                                                                                                                                                                                                                                                                                                                                                                                                                                                                                                                                                                                                                                                                                                        
Hyperlogarithms are special functions considered first by Poincar\'e (1884) and Lappo-Danilevski (1928) in the realm of one variable complex differential equations. 
Some modern references are \cite{Brown2004}, \cite{VanhoveZerbini2018} and \cite{EZ}. For more specific references on 
 hyperlogarithms in relation with the notion of abelian relation in web geometry, see  
\cite[\S1.2]{ClusterWebs}, \cite[\S2.1]{Hlog} or  \cite{CP}. 

Let $Y$ be a connected complex manifold. 
We fix a  finite dimensional subspace ${\bf H} \subset {\bf H}^0\big( Y, \Omega_Y^1\big)$ we are going to work with.   
We also assume that ${\bf H}$ 
has a basis $(\omega_1, \ldots,\omega_m)$  whose elements satisfy the following assumptions: 
these forms are closed and their wedge products all vanish, {\it i.e.}
\begin{equation}
\label{Eq:Conditions-II}
d\omega_i=0\qquad \mbox{ and } \qquad \omega_i\wedge \omega_j=0\qquad 
\mbox{ for all}\quad i,j=1,\ldots,m\,.
\end{equation}

Given a base point $y\in Y$, we consider 
 continuous paths  $\gamma^z: [0,1]\rightarrow Y$ joigning $y$ to another point $z\in Y$, all supported in a small open ball $B_y$ around $y$ in $Y$.  For any `weight' $w\geq 1$,  we identify any pure tensor $\underline{\omega}=\omega_{i_1}\otimes \cdots \otimes \omega_{i_w}\in {\bf H}^{\otimes w}$ with the word $\omega_{i_1}\cdots \omega_{i_w}$. 
Setting inductively 
\begin{equation}
\label{Eq:Integration-formulas}
L_{\omega_{i_1}}: z\mapsto  \int_{\gamma^z} \omega_{i_1} \, , 
\quad 
L_{\omega_{i_1}\omega_{i_2}}: z\mapsto  \int_{\gamma^z} L_{\omega_{i_1}}\, \omega_{i_1}
\quad 
\mbox{ and }\quad 
L_{\omega_{i_1}\cdots \omega_{i_w}}: z\mapsto  
 \int_{\gamma^z} L_{\omega_{i_2}\cdots \omega_{i_w}}\, \omega_{i_1}\, , 
\end{equation}
we obtain iterated integrals which, thanks to the conditions \eqref{Eq:Conditions-II},  do not depend on the path $\gamma^z$ but only on its endpoint hence   give rise to holomorphic germs  $L_{\underline{\omega}}\in \mathcal O_{Y,y}$. Extending  $\underline{\omega}\mapsto L_{\underline{\omega}}$ linearly to the whole tensorial power ${\bf H}^{\otimes w}$, we define a realization map
\begin{equation}
\label{Eq:IIyw}
{\bf II}_y^w : {\bf H}^{\otimes w}\longrightarrow \mathcal O_{Y,y}
\end{equation}
which is obviously $\mathbf C$-linear and turns out to be injective.  The 
shuffle product $\shuffle$ on words in the letters $\omega_1,\ldots,\omega_m$ 
gives rise to another product on the full tensorial algebra  
${\bf H}^{\otimes \bullet}=\oplus_{w\geq 1} 
{\bf H}^{\otimes w}$, again denoted by $\shuffle$ and which enjoys the nice property of making of 
$$
{\bf II}_y=\oplus_{w\geq 1}{\bf II}_y^w  : \Big(\, {\bf H}^{\otimes \bullet}, \shuffle \, \Big)
\longrightarrow \mathcal O_{Y,y}
$$ an injective map of complex algebras.\footnote{The product on $\mathcal O_{Y,y}$ is of course the one induced by the usual (pointwise) product of holomophic functions.} For $L\in {\rm Im}( {\bf II}_y)$, one defines its `{\it weight}' as the minimum 
$w(L)$  of the set of
 $w$'s such that $L\in  {\bf II}_y\big( \oplus_{w'\leq w} 
{\bf H}^{\otimes w'}\big)$ and the {\it `symbol of $L$'} (at $y$ a priori) is the element $\mathcal S_y(L)\in{\bf H}^{\otimes w(L)}$ such that 
the weight of $L-L_{\mathcal S_y(L)}$ is strictly less than that of $L$. 

Using the formulas \eqref{Eq:Integration-formulas}, one sees that any 
iterated integral $L \in {\rm Im}({\bf II}_y)$ extends holomorphically along any path in $Y$ and gives rise to a holomorphic multivalued function on the whole $Y$, again denoted by $L$.  If $L^\gamma$ stands for the analytic continuation 
along a path $\gamma: [0,1]\rightarrow Y $ joining $y$ to the other extremity $y'=\gamma(1)\in Y$, then 
$L^\gamma\in  {\rm Im}({\bf II}_{y'})\subset \mathcal O_{Y,y'}$ and moreover 
 $w(L)= 
w\big(L^\gamma \big)$ 
and $\mathcal S_y(L)=\mathcal S_{y'}(L^\gamma)$.  This shows that the 
notions of weight and of symbol of an iterated integral constructed from ${\bf H}$ do not depend on the chosen base point but only on its analytic extension as a multivalued function on the whole $Y$.  

One difficulty when working with multivalued holomorphic functions, 
is that it is often painful to be rigorous and precise about which branch of them has to be considered so that a certain property be satisfied. This can be avoided for 
iterated integrals by using their symbols which 
 offer an efficient way to deal with them algebraically. 
\begin{center}
\vspace{-0.15cm}
$\star$
\end{center}

The hyperlogarithms which are involved in this text (in relation with the abelian relations of planar webs) are among the simplest, and are those obtained when $Y$ is the complement of a finite set $\Sigma$ in $\mathbf P^1$ and one takes ${\bf H}={\bf H}^0\big( \mathbf P^1, \mathscr O_{\mathbf P^1}({\rm Log}\,\Sigma)\big)$. In this case, if $\sigma_1,\ldots,\sigma_{m+1}$ stand for the elements of $\Sigma$ (assumed to be pairwise distinct), one takes an affine coordinate $z$ such that $\sigma_{m+1}=\infty$ and 
the basis of ${\bf H}$ we are going to work with is that whose elements are the logarithmic 1-forms 
$\omega_i=d{\rm Log}(z-\sigma_i)=dz/(z-\sigma_i)$ for $i=1,\ldots,m$. 

Hyperlogarithms on $\mathbf P^1$ have been considered by many authors, ancient or contemporary. The properties of these functions,  for instance their monodromy, are well-known: as explained in 
\cite{Brown2004} ({\it cf.} Corollary 6.5 therein), given an element 
 ${\varpi}$ being an element of ${\bf H}^{w-1}$ (for some weight $w\geq 1$), then 
modulo ${\bf HLog}^{\leq w}$, for any $i\in \{1,\ldots,k\}$, one has 
\begin{equation}
\label{Eq:Monodromy-hyperlog-Lw}
\Big(\mathcal M_{\gamma_k} - {\bf I}{\rm d}\Big) \big( L_{\omega_i {\varpi}}  
\big)
\equiv 
\begin{cases}
 \hspace{0.3cm} 0 \hspace{0.78cm} \mbox{ if } \, k\neq i \, \\
 \hspace{0.1cm} 
 2i\pi \,L_{\varpi} 
 \hspace{0.15cm} \mbox{ if } \, k=i.
\end{cases}
\end{equation}

\subsection{\bf Webs}
\label{SS:Webs}
As references on webs, we mention the books \cite{BB} and \cite{PP}. The first part of our memoire  \cite{ClusterWebs} may be useful too. \sk

\subsubsection{}In this text, we only deal with a very restrictive class of planar webs, our definition will be restricted to them.  Here a $d$-web $\boldsymbol{\mathcal W}$ for $d\geq 1$ on a (connected and smooth) algebraic surface $X$ will designate  a $d$-tuple  $(\mathcal F_{U_i})_{i=1}^d$ of foliations formed by the level sets of non constant rational functions $U_i:X\dashrightarrow \mathbf P^1$ which are such that $dU_i\wedge dU_j\neq 0$ at the generic point of $ X$, for any $i,j=1,\ldots,d$ distinct.  One can and will assume that each of the $U_i$'s is non composite (or equivalently, the generic fibers of the $U_i$'s are connected).  In this case we use the following notation to denote $\boldsymbol{\mathcal W}$:
$$\boldsymbol{\mathcal W}=\boldsymbol{\mathcal W}\big(\,U_1
\, , \, 
\ldots
\, , \, 
U_k\, \big)\, .$$

  Any non constant map $\widetilde U_i$ defined on an open subset of $X$ such that $dU_i\wedge d \widetilde U_i=0$ is a first integral (for the $i$-th foliation) of $\boldsymbol{\mathcal W}$.  Another $d$-web $\boldsymbol{\mathcal W}'$ defined on another surface $X'$ is said to be {\it `equivalent'} to $\boldsymbol{\mathcal W}$ if there exists a (local) biholomophism $\varphi$ between two open subsets in $X'$ and $X$ such that the pull-backs $\varphi^*(U_i)=U_i\circ \varphi$ for $i=1,\ldots,d$ form a complet set of first integrals for $\boldsymbol{\mathcal W}'$ (possibly up to a relabelling of the foliations). 

\subsubsection{} An {\it `abelian relation'} (ab.\,AR) for $\boldsymbol{\mathcal W}(\,U_1,\ldots,U_k)$ is a tuple of exact holomorphic 1-forms $(dF_i)_{i=1}^d$ such that $\sum_{i=1}^d U_i^*(dF_i)=0$ or equivalently, such that 
$
\sum_{i=1}^d F_i(U_i)=0
$
holds true identically up to the constants (on a suitable open domain in $X$). Any AR extends holomorphically along any path supported in the regular set $Y=Y_{\boldsymbol{\mathcal W}}$ defined as the biggest open subset in $X$ on which all the foliations of the web are regular and intersect transversely. For that reason, the abelian relations form a local system $\boldsymbol{AR}( \boldsymbol{\mathcal W})$ on $Y$ rather than a vector space. However we will use, a bit abusively, the same notation to designate the restriction of $\boldsymbol{AR}( \boldsymbol{\mathcal W})$ on any open domain in $Y$, restriction which canonically has the structure of a vector space. 

\subsubsection{} For any $i$, let $\mathfrak R_i$ be the finite subset of $ \mathbf P^1$ whose elements are the $\lambda$'s such that there exists an irreducible component of $U_i^{-1}(\lambda)$ which is invariant by another foliation $\mathcal F_{U_j}$  ($j\neq i$) of the web.  Then the open subset $Z=Z_{\boldsymbol{\mathcal W}}=X\setminus \cup_{i=1}^d U_i^{-1}( \mathfrak R_i)$ contains $Y$ and one verifies that any germ of AR $(F_i(U_i))_{i=1}^d$ at some point $y\in Y$ extends holomorphicaly along any continuous path supported in $Z$. Equivalently, for any $i$, setting $y_i=U_i(y)$,   the germ $F_i\in \mathcal O_{ \mathbf P^1, y_i}$ extends to a global but a priori multivalued holomorphic  function on $\mathbf P^1\setminus \mathfrak R_i$.

\subsubsection{\bf Combinatorics of the space of abelian relations}
\label{SS:Combinatorial-characterization}
Let us introduce general notions which we believe are relevant 
for studying webs. Let $\boldsymbol{\mathcal W}$ be a planar $d$-web, formed by $d$ distinct foliations $\mathcal F_1,\ldots,\mathcal F_d$ on its definition domain in $\mathbf C^2$.

For any $k=3,\ldots, d$, one sets/defines
\begin{itemize}
\vspace{-0.15cm}
\item ${\rm SubW}_k(\boldsymbol{\mathcal W})$ is the set of $k$-subwebs of  $\boldsymbol{\mathcal W}$;
\item ${\rm SubW}(\boldsymbol{\mathcal W})$ is the set of subwebs of  $\boldsymbol{\mathcal W}$ \big({\it i.e.}\,${\rm SubW}(\boldsymbol{\mathcal W})=\cup_{k=3}^d {\rm SubW}_k(\boldsymbol{\mathcal W})$\big);
\item $\boldsymbol{AR}_3(\boldsymbol{\mathcal W})$  denotes the subspace of $\boldsymbol{AR}(\boldsymbol{\mathcal W})$ spanned by the ARs of the 3-subwebs of $\boldsymbol{\mathcal W}$;  
\item ${r}_3(\boldsymbol{\mathcal W})$ stands for the dimension of $\boldsymbol{AR}_3(\boldsymbol{\mathcal W})$;  
\item  $r_{ \boldsymbol{\mathcal W}} : {\rm SubW}(\boldsymbol{\mathcal W}) \rightarrow \mathbf N$ is the function associating $r_3(\boldsymbol{W})$ to any subweb $\boldsymbol{W}$ of 
$\boldsymbol{\mathcal W}$;
\item $r_{ \boldsymbol{\mathcal W},k} : {\rm SubW}_k(\boldsymbol{\mathcal W}) \rightarrow \mathbf N$ is the   restriction of $r_{ \boldsymbol{\mathcal W}}$ to ${\rm SubW}_k(\boldsymbol{\mathcal W})$.
\end{itemize}

The functions $r_{ \boldsymbol{\mathcal W},k}$'s are combinatorial objects invariantly attached to  $\boldsymbol{\mathcal W}$. They allow to state in a nice and concise way some properties of webs.  For instance, let us consider the interesting case of Bol's web $\boldsymbol{\mathcal B}\simeq 
\boldsymbol{\mathcal W}_{{\rm dP}_5}$: since it is hexagonal, one has identically $r_{ \boldsymbol{\mathcal B},3}=1$ and $r_{ \boldsymbol{\mathcal B},4} =3$ and one verifies that  $r_{ \boldsymbol{\mathcal B},5}\equiv 5$.  On the other hand,  if $\boldsymbol{\mathcal L}$ is a planar  web formed by 5 pencils of lines, then $r_{ \boldsymbol{\mathcal L},3}\equiv 1$ (hexagonality), 
$r_{ \boldsymbol{\mathcal L},4} \equiv 3$ and $r_{ \boldsymbol{\mathcal L},5}\equiv 6$. Hence Bol's characterization of $\boldsymbol{\mathcal B}$ can be stated nicely as follows:

\noindent {\bf Theorem.} (Bol \cite{Bol})
{\it Bol's web is characterized by $r_{ \boldsymbol{\mathcal B}}$: 
any 5-web $\boldsymbol{\mathcal W}$ such that $r_{ \boldsymbol{\mathcal W},3}\equiv 1$
and $r_{ \boldsymbol{\mathcal W},5}\leq 5$ is equivalent to Bol's web.}
\label{Page:BolsTHM}

This result shows that the  $r_{ \boldsymbol{\mathcal W},k}$'s can be useful to study webs up to equivalence. Given two planar $d$-webs $ \boldsymbol{\mathcal W}'$ and $ \boldsymbol{\mathcal W}''$ as above, let us say that the
functions 
 $r_{\boldsymbol{\mathcal W}'}$ and $r_{\boldsymbol{\mathcal W}''}$ (or 
$r_{\boldsymbol{\mathcal W}',3}$ and $r_{\boldsymbol{\mathcal W}'',3}$) are equivalent, denoted by $r_{\boldsymbol{\mathcal W}'}\sim r_{\boldsymbol{\mathcal W}''}$ (and similarly for $r_{\boldsymbol{\mathcal W}',3}$ and $r_{\boldsymbol{\mathcal W}'',3}$) if both coincide,  
possibly up to a relabelling of the foliations composing one of them.

In \cite{Burau}, Burau first determines $r_{\boldsymbol{\mathcal W}_{{\rm dP}_3},3}$  and then proves the following characterization result:

\noindent
 {\bf Theorem.} (Burau \cite{Burau})
{\it The class of 27-webs $\boldsymbol{\mathcal W}_{{\rm dP}_3}$'s is characterized by $r_{\boldsymbol{\mathcal W}_{{\rm dP}_3},3}$: any 27-web 
$\boldsymbol{\mathcal W}$ such that $r_{ \boldsymbol{\mathcal W},3}\sim r_{\boldsymbol{\mathcal W}_{{\rm dP}_3},3}$
 is equivalent to  the 27-web in conics on a cubic surface in $\mathbf P^3$.}

In \S\ref{SS:WdP4-Combinatorial-characterization}, we will show that there is an  analogous result for the web $\boldsymbol{\mathcal W}_{{\rm dP}_4}$, this for any smooth 
quartic del Pezzo surface ${\rm dP}_4$.

\subsubsection{\bf Hyperlogarithmic abelian relations}
\label{SSS:Hyperlog-AR}
As we are going to prove, all the he abelian relations of $\boldsymbol{\mathcal W}_{ {\rm dP}_4}$ are hyperlogarithmic hence all can be found using the `symbolic' approach we are going to describe. This method is elementary and amounts to elementary linear algebra computations in tensorial spaces. It has been used to find some ARs of webs first by Robert \cite{Robert}, then by Pereira \cite{Pereira}. But the use of a symbolic approach for solving a functional  equation can even be found before, for instance in  \cite[\S4]{HM}.

We continue to use the notation used above. For any $i$, 
let $\boldsymbol{\mathcal H}_i$ stand for the space of rational 1-forms on $\mathbf P^1$, regular on 
$\mathbf P^1\setminus \mathfrak R_i$ and having at most logarithmic poles at  each point of $\mathfrak R_i$, {\it i.e.}
$${\bf H}_i={\bf H}^0\Big( \mathbf P^1,\Omega^1_{\mathbf P^1}\big( {\rm Log}\, \mathfrak R_i\big)
\Big)\, .$$ 
We set $\mathfrak R_i=\{ \mathfrak r_{i,s} \}_{s=0}^{m_i}$ with 
the $\mathfrak r_{i,s}$'s pairwise distinct and 
$\mathfrak r_{i,0}=\infty$. Then the 1-forms $\theta_{i,s}=d{\rm Log}(z-\mathfrak r_{i,s})=dz/(z-\mathfrak r_{i,s})$  for $s=1,\ldots,m_i$ form a basis of $\boldsymbol{\mathcal H}_i$ hence their pull-backs $\Theta_{i,s}=U_i^*(\eta_{i,s})=d{\rm Log}(U_i-\mathfrak r_{i,s})=dU_i/(U_i-\mathfrak r_{i,s})$  form a basis of  a subspace 
${\bf H}_i=U_i^*\big(\boldsymbol{\mathcal H}_i\big)$ of the 
space  $\Omega^1_{\mathbf C(X)}\cap {\bf H}^0\big(Y,\Omega_Y^1)$ of  rational 1-forms on $X$ whose restrictions on $Y$ are holomorphic. We set 
$$
{\bf H}={\bf H}_{\boldsymbol{\mathcal W}}=
 \sum_{i=1}^d {\bf H}_i\subset \Omega^1_{\mathbf C(X)}\cap {\bf H}^0\big(Y,\Omega_Y^1\big)\, .
$$
Then for any weight $w\geq 1$, 
we define a vector space of {\it `(symbolic) weigth $w$ hyperlogarithmic abelian relations'}, denoted by  $  {\bf HLogAR}^w= {\bf HLogAR}^w(\boldsymbol{\mathcal W})$, 
by requiring that 
the following sequence of vector spaces be exact: 
\begin{equation}
\label{Eq:ExactSequence-1}
0 \rightarrow {\bf HLogAR}^w\longrightarrow \oplus_{i=1}^{10} {\bf H}_i^{\otimes w} \stackrel{\Phi^w}{\longrightarrow} {\bf H}^{\otimes w}\, ,
\end{equation}
where the  map $\Phi^w$ is the one defined as follows: 
if $\boldsymbol{\Omega}=\big( \Omega_i\big)_{i=1}^d$ 
with $\Omega_i\in  {\bf H}_i^{\otimes w}$ for $i=1,\ldots,d$,  then $\Phi^w(\boldsymbol{\Omega})=
\sum_{i=1}^d  \iota_i^w(\Omega_i)$, where 
for any $i$, $\iota_i^w$ stands for the monomorphism ${\bf H}_i^{\otimes w} \hookrightarrow {\bf H}^{\otimes w}$ naturally induced by the inclusion 
${\bf H}_i \subset {\bf H}$.

We fix an arbitrary base point $y\in Y$ and we set $y_i=U_i(y)\in \mathbf P^1\setminus \mathfrak R_i$ for $i=1,\ldots,d$. 
For any $i$, since ${\bf H}_i=U_i^*\big(\boldsymbol{\mathcal H}_i\big)$ and because this is obvious for $\boldsymbol{\mathcal H}_i$, the $\Theta_{i,s}$ satisfy the conditions \eqref{Eq:Conditions-II}, hence there are well-defined iterated integrations morphisms
$$
{\bf II}^w_{Y,y} : {\bf H}_i^{\otimes w} \longrightarrow \mathcal O_{Y,y}
\qquad \mbox{and}\qquad 
{\bf II}^w_{\mathbf P^1,y_i} : \boldsymbol{\mathcal H}_i^{\otimes w} \longrightarrow \mathcal O_{\mathbf P^1,y_i}
$$
which are easily seen to satisfy ${\bf II}^w_{Y,y}=U_i^*\circ {\bf II}^w_{\mathbf P^1,y_i}$ where $U_i^*$ stands here for the natural map $
\boldsymbol{\mathcal H}_i^{\otimes w}\simeq {\bf H}_i^{\otimes w}$ 
induced by the pull-back map $U_i^* : 
\boldsymbol{\mathcal H}_i\simeq {\bf H}_i$. In other terms, 
for any $\boldsymbol{\Theta}=\Theta_{i,s_1}\otimes \cdots \otimes \Theta_{i,s_w}=U_i^*\big( \boldsymbol{\eta}\big) $ with $\boldsymbol{\eta}=
\eta_{i,s_1}\otimes \cdots \otimes \eta_{i,s_w}$, one has 
 $$ L_{\boldsymbol{\Theta}}={\bf II}^w_{Y,y}\big(\boldsymbol{\Theta} \big) 
={\bf II}^w_{\mathbf P^1,y_i}\big(\boldsymbol{\eta} \big) =L_{\boldsymbol{\eta}}\circ U_i
$$
as holomophic germs at $y$ and this extends by linearity to any element of ${\bf H}_i^{\otimes w}$. 

From the elementary facts above, it follows that for any element $\boldsymbol{\Omega}=\big( \Omega_i\big)_{i=1}^d\in  {\bf HLogAR}^w$, then 
there exists $\big( \omega_i\big)_{i=1}^d \in \oplus_{i=1}^d \boldsymbol{\mathcal H}_i^{\otimes w}$ such that one has $\Omega_i=U_i^*(\omega_i)$ for any $i=1,\ldots,d$.  Thus the $d$-tuple of 
 hyperlogarithms $(L_{\omega_i})_{i=1}^d$ satisfies $\sum_{i=1}^d L_{\omega_i}(U_i)=0$ (equality in $\mathcal O_{Y,y}$) hence $L_{\boldsymbol{\Omega}}=
\big(L_{\omega_i}(U_i)\big)_{i=1}^d $ is a hyperlogarithmic AR of weight $w$ for $\boldsymbol{\mathcal W}$ at $y$.  
Denoting by $\boldsymbol{AR}(\boldsymbol{\mathcal W})$ the space of ARs of $\boldsymbol{\mathcal W}$ on any fixed but arbitrary open domain containing $y$, we thus obtain a linear map ${\bf II}^w:    {\bf HLogAR}^w \rightarrow \boldsymbol{AR}(\boldsymbol{\mathcal W})$ whose injectivity follows at once from that of 
\eqref{Eq:IIyw}. Using the symbol, we obtain that all the spaces ${\bf II}^w( {\bf HLogAR}^w)$'s for $w\geq 1$ are in direct sum in $\boldsymbol{AR}(\boldsymbol{\mathcal W})$ from which we obtain a linear injective map 
\begin{equation}
\label{Eq:II-W}
{\bf II}={\bf II}_{ \boldsymbol{\mathcal W} } \, : \,  {\bf HLogAR}=\oplus_{w\geq1}  {\bf HLogAR}^w \longrightarrow \boldsymbol{AR}(\boldsymbol{\mathcal W})\, . 
\end{equation}
For each $w\geq 1$, we set $\boldsymbol{A\hspace{-0.03cm}R}^w(\boldsymbol{\mathcal W})=
{\bf II}^w( {\bf HLogAR}^w)$ and call its elements the weight $w$ hyperlogarithmic ARs  of
$\boldsymbol{\mathcal W}$ (at $y$).  A global perspective requires a bit of additional formalism: the $\boldsymbol{\hspace{-0.03cm}R}^w(\boldsymbol{\mathcal W})$'s defined locally above actually do not give rise to global local systems on $Y$, but the 
direct sums $\boldsymbol{A\hspace{-0.03cm}R}^{\leq w}(\boldsymbol{\mathcal W})
= \oplus_{w'\leq w} \boldsymbol{AR}^{w'}(\boldsymbol{\mathcal W})$ do 
(this follows from the fact that hyperlogarithms have unipotent monodromies). 
They are  the pieces of what we call the  {\it `hyperlogarithmic filtration'} 
$\boldsymbol{HF}_{\boldsymbol{\mathcal W}}=
\boldsymbol{AR}^{\leq \, \bullet }(\boldsymbol{\mathcal W})
$ of $\boldsymbol{\mathcal W}$. As one verifies easily, for any weight $w$, one recovers the 
symbolic space 
$ {\bf HLogAR}^w$ as the $w$-th piece of the graded space associated to 
$\boldsymbol{HF}_{\boldsymbol{\mathcal W}}$:
$$
{\rm Gr}^w \left(\boldsymbol{HF}_
{\boldsymbol{\mathcal W}}\right)={\bf HLogAR}^w\, .
$$

Since they are algebraic objects, it is much more convenient to work with the symbolic spaces ${\bf HLogAR}^w$ for $w\geq 1$, than with the associated genuine (multivalued) hyperlogarithmic ARs, because of their multivaluedness. This is what we will do with the del Pezzo webs 
${\boldsymbol{\mathcal W}}_{ {\rm dP}_5} $ and 
${\boldsymbol{\mathcal W}}_{ {\rm dP}_4} $ which 
we will study further.

We define the `{\it weight $w$} (resp.\,the {\it total}) {\it hyperlogarithmic rank} ${\rm HLrk}^w({\boldsymbol{\mathcal W}})$'  (resp.\,${\rm HLrk}({\boldsymbol{\mathcal W}})$) of the web ${\boldsymbol{\mathcal W}}$ as the dimension 
of $ {\bf HLogAR}^w$ (resp.\,of ${\bf HLogAR}$): one has 
$$
{\rm HLrk}^w({\boldsymbol{\mathcal W}})=\dim\,  {\bf HLogAR}^w\\
\quad \mbox{ and } \quad 
{\rm HLrk}({\boldsymbol{\mathcal W}})= \dim\,  {\bf HLogAR}=\sum_{w\geq 1}
{\rm HLrk}^w({\boldsymbol{\mathcal W}})\, . 
$$
We will say that $\boldsymbol{\mathcal W}$ has {\it `all its ARs hyperlogarithmic'} when the map \eqref{Eq:II-W} is an isomorphism or equivalently, when 
${\rm HLrk}({\boldsymbol{\mathcal W}})={\rm rk}({\boldsymbol{\mathcal W}})$.

For $w\geq 1$ fixed, 
the symmetric group $\mathfrak S_w$ acts naturally 
on the tensor spaces ${\bf H}_i^{\otimes w}$'s and ${\bf H}^{\otimes w}$ and the map
 $\Phi^w$ is equivariant  for these actions. 
 It follows that $\mathfrak S_w$ acts naturally on ${\bf HLogAR}^w$ and that 
  \eqref{Eq:ExactSequence-1} actually can be seen 
 as a  short exact sequence in the category of $\mathfrak S_w$-modules. 
As a consequence, it follows that there is a decomposition of 
${\bf HLogAR}^w$ as a direct sum of irreducible 
$\mathfrak S_w$-modules. These being encoded by the partitions $\boldsymbol{\lambda} $ of $w$, one obtains a $\mathfrak S_w$-equivariant decomposition 
$$
{\bf HLogAR}^w=\oplus_{   \boldsymbol{\lambda} \vdash w} 
{\bf HLogAR}^w_{  \boldsymbol{\lambda} }
$$
where for any partition $ \boldsymbol{\lambda} \vdash w$, ${\bf HLogAR}^w_{  \boldsymbol{\lambda} }$ stands for the sum of all the submodules of 
${\bf HLogAR}^w$ which are isomorphic to the `the' irreducible  
$\mathfrak S_w$-module $V_{\boldsymbol{\lambda}}$ associated with the partition $\boldsymbol{\lambda}$. 

There is nothing new  for $w=1$, the first interesting case is when $w=2$. There are only two partitions to consider,  $(2)$ and $(1,1)$ and the corresponding 
decomposition can be written 
$$
{\bf HLogAR}^2=
{\bf HLogAR}^2_{\rm asym} \oplus 
{\bf HLogAR}^2_{\rm sym} 
$$
where ${\bf HLogAR}^2_{\rm sym}$ (resp.\,${\bf HLogAR}^2_{\rm asym}$) stands for the space of $d$-tuples $(\Omega_i)_{i=1}^d\in 
 \oplus_{i=1}^d {\bf H}_{i}^{\otimes 2}$ such that $\sum_i \Omega_i=0$ 
 and 
 where 
 each symbol $\Omega_i$ is a linear combination of symmetrized elements 
 $\eta_{i,s}\circ \eta_{i,s'}=(\eta_{i,s}\otimes \eta_{i,s'}+\eta_{i,s'}\otimes \eta_{i,s})/2$ 
  (resp.\,antisymmetrized elements 
 $\eta_{i,s}\wedge \eta_{i,s'}=(\eta_{i,s}\otimes \eta_{i,s'}-\eta_{i,s'}\otimes \eta_{i,s})/2$) with $s,s'=1,\ldots,m_i$. Accordingly, the corresponding decomposition of $\boldsymbol{AR}^2
 ( \boldsymbol{\mathcal W})$ will be written
 $$
 \boldsymbol{AR}^2
 ( \boldsymbol{\mathcal W})=
 \boldsymbol{AR}^2_{\rm asym}
 ( \boldsymbol{\mathcal W})\oplus
 \boldsymbol{AR}^2_{\rm sym}
 ( \boldsymbol{\mathcal W})\, .
  $$

\subsubsection{\bf Monodromies and residues}
\label{SSS:Monodromies-Residues}
Here we describe two ways to get new ARs from a given one.  The first is topological and goes by analytic continuation and applies to any abelian relation.  The second is algebraic and only concerns hyperlogarithmic ARs. 

We continue to use the notation considered above. Let $y $ be in $ Y$ and let $\boldsymbol{A}$ be a germ of abelian relation at this point. As explained above, 
$\boldsymbol{A}$ extends holomorphically along any path $\gamma$ in $Y$ originating from $y$. 
If  $\gamma$ is a loop based at $y$, the analytic continuation 
$ \mathcal M_{[\gamma]}(\boldsymbol{A})= \boldsymbol{A}^\gamma$ of $\boldsymbol{A}$ along $\gamma$ 
is an AR of $\boldsymbol{\mathcal W}$ at $y$ which only depends on the homotopy class $[\gamma]$ of $\gamma$ as a based loop. 
We thus define a {\it `monodromy representation'}  
\begin{align}
\label{Eq:Monodromy}
\pi_1(Y,y) & \longrightarrow {\rm GL}\Big( \boldsymbol{AR}
 \big( \boldsymbol{\mathcal W}\big) \Big)
\footnotemark
  \\
 [\gamma] & \longmapsto \,  \mathcal M_{[\gamma]} :  \boldsymbol{A} \longmapsto 
  \boldsymbol{A}^\gamma
 \nonumber
\end{align}
 \footnotetext{More rigorously, the monodromy representation has values in 
 $ {\rm GL}\Big( \boldsymbol{AR}
 \big( \boldsymbol{\mathcal W}\big)_y \Big)$ where $\boldsymbol{AR}
 \big( \boldsymbol{\mathcal W}\big)_y$ stands for the vector space of {\it `germs of ARs of $\boldsymbol{\mathcal W}$ at $y$'}.  
  We take the liberty here and after 
   of not being very rigorous with the notation.}
whose orbits it could be interesting to study. For instance, given 
 a certain subspace $E\subset \boldsymbol{AR}
 ( \boldsymbol{\mathcal W})$, one will say that {\it `the abelian relation $\boldsymbol{A}$ spans $E$ by monodromy'} whenever one has 
 $${\rm Monod}\big(\boldsymbol{A}\big) \stackrel{\scalebox{0.8}{${\rm def}$}}{=} 
 \Big\langle \, \pi_1(Y,y) \cdot \boldsymbol{A}\,  \Big\rangle=E\, .$$  
In the extremal case when ${\rm Monod}\big(\boldsymbol{A}\big) =\boldsymbol{AR}
 ( \boldsymbol{\mathcal W})$, then $\boldsymbol{A} $ has to be considered as `the'   fundamental AR for the web $\boldsymbol{\mathcal W}$ since all the others can be obtained from it by analytic continuation (up to linear combinations). Abel's AR of Bol's web is an example of such a very specific AR.

If one limits oneself to considering only hyperlogarithmic ARs, there are even more structures associated with the action of monodromy.  Let 
$ 
{\bf HLogAR}^{\leq \bullet}$
 be the trivial filtration 
 defined by 
 ${\bf HLogAR}^{\leq w}=\oplus_{w'\leq w} {\bf HLogAR}^{w'}$  for any $w\geq 1$. Because hyperlogarithms have unipotent monodromy, the choice of any base point 
$y\in Y$ gives rise to a structure of unipotent $\pi_1(Y,y)$-module on the 
filtered space ${\bf HLogAR}^{\leq \bullet}$ given by 
\begin{equation}
\label{Eq:Pi1(Y,y)-module}
[\gamma]\cdot \boldsymbol{\Omega}= {\bf II}_y^{-1} \left( \gamma\cdot {\bf II}_y\big( 
\boldsymbol{\Omega}\big) 
\right)
\end{equation}
for any $[\gamma]\in \pi_1(Y,y)$ and any $\boldsymbol{\Omega}\in {\bf HLogAR}^{w}$. This structure of $\pi_1(Y,y)$-module defined by \eqref{Eq:Pi1(Y,y)-module}  depends on the base point $y$ but in a non essential way that we let the reader to make precise.  To \eqref{Eq:Pi1(Y,y)-module}   is associated a representation 
$\rho_{Y,y} : \pi_1(Y,y)\rightarrow {\rm Aut}\big( 
{\bf HLogAR}^{ \leq \bullet} \big) $  which depends on $y$. We let the reader verify that the image  
$\Pi_1={\rm Im}(\rho_{Y,y})$ is a  subgroup of the automorphism 
group of the filtration ${\bf HLogAR}^{\leq \bullet}$ which is well-defined 
up to conjugation.
\sk

For hyperlogarithmic ARs, there is another, more algebraic, way to get new ARs from a given one. Assume that all the elements of the space ${\bf H}=\sum_i {\bf H}_i$ have logarithmic poles along  any irreducible component of $Z=X\setminus Y$. For such a component $D$, there is a residue map ${\rm Res}_D: {\bf H} \mapsto \mathbf C$.\footnote{A priori, the image of an element of ${\bf H}$  by 
${\rm Res}_D$ is a holomorphic function on $D$ minus }
 For any pure tensor $ h_1\otimes \cdots \otimes h_w\in {\bf H}^{\otimes w}$, one sets 
\begin{equation}
\label{Eq:Res-D}
{\rm Res}_D\big( h_1\otimes \cdots \otimes h_w\big)={\rm Res}_D( h_1)\, \big(h_2\otimes \cdots \otimes h_w\big)
\end{equation}
 and extending it by linearity, we get a linear map ${\rm Res}_D: {\bf H}^{\otimes w}\rightarrow 
{\bf H}^{\otimes (w-1)}$. We define similarly 
 ${\rm Res}_{i,D}: {\bf H}_i^{\otimes w}\rightarrow 
{\bf H}_i^{\otimes (w-1)}$ for each $i$ and one verifies that the residue maps commute, which gives us a residue map at the level of symbolic hyperlogarithmic abelian relations:
\begin{align*}
{\rm Res}_D : {\bf HLogAR}^w& \longrightarrow  {\bf HLogAR}^{w-1}\\
\big( \Omega_i\big)_{i=1}^d & \longmapsto  \Big( {\rm Res}_{i,D}(\Omega_i)\Big)_{i=1}^d\, .
\end{align*}

Given a hyperlogarithmic  abelian relation $\boldsymbol{\Omega} \in {\bf HLogAR}^w $ and an irreducible component $D$ as above, the iterates $({\rm Res}_D)^{\circ k}(\boldsymbol{\Omega} )\in {\bf HLogAR}^{w-k}$ for $k=0,1,\ldots,w-1$ generate 
the smallest ${\rm Res}_D$-stable subspace of ${\bf HLogAR}$ containing $\boldsymbol{\Omega}$: by definition, it is the {\it `space spanned by $\boldsymbol{\Omega}$ by taking residues along $D$'}. 
When taking the span of these spaces for all the irreducible components $D$ of $Z$, we obtain the 
{\it `space obtained from $\boldsymbol{\Omega}$ by residues'}, denoted by 
$$
{\rm Res}(\boldsymbol{\Omega})= \Big\langle \, \big({\rm Res}_D\big)^{\circ k}(\boldsymbol{\Omega} ) \, \big\lvert
\, k\geq 0, \, D\subset Z\, \mbox{ irred}\, \Big\rangle  \,.  $$ 

For  hyperlogarithmic ARs, the operations to get new ARs for a given one by monodromy or by `residues' are related, at least `locally'. Indeed, let $y_D$ be a generic point of  an irreducible component $D$  of $Z$. We fix a small transversal  to $D$ at $y_D$, namely a holomorphic embedding $T_D: \Delta\rightarrow X$  of the unit disk $\Delta\subset \mathbf C$ such that $T_D(\Delta)$ is transverse to $D$ at $T_D(0)=y_D$.  We fix $z\in \Delta^*$, we set $y=T_D(z)$  and we consider the loop $\gamma_D: [0,1]\rightarrow \Delta, \, t\mapsto z\,e^{2i\pi t}$ whose homotopy class spans $\pi_1(\Delta^*,z)\simeq \mathbf Z$.  Then $\widetilde \gamma_D= T_D\circ \gamma_D$ spans the {\it `local fundamental group along $D$ (at $y_D$)'} and  the monodromy 
map $\mathcal M_{[\widetilde \gamma_D]}$ ({\it cf.}\,\eqref{Eq:Monodromy}) induces 
an endomorphism of $\boldsymbol{F}^w{\bf HLogAR}=\oplus_{w'\leq w} {\bf HLogAR}^{w'}$ for any weight $w\geq 1$, that we call the `{\it local monodromy operator around $D$ at $y_D$'} and denote by $\mathcal M_{D,y_D}^{\rm loc}$.

Considering the monodromy of hyperlogarithms which is 
well-known (see \eqref{Eq:Monodromy-hyperlog-Lw}) and in particular unipotent, one obtains first that the 
`{\it local variation around $D$ at $y_D$'}, which by definition is the linear operator 
$$V_{[\widetilde \gamma_D]}= \mathcal M_{[\widetilde \gamma_D]}-{\bf I}{\rm d}\, , $$
sends ${\bf HLogAR}^{w}$ into 
${\bf HLogAR}^{w-1}$.  A bit more work gives us that, as linear maps from ${\bf HLogAR}^{w}$ to 
${\bf HLogAR}^{w-1}$,   the following equality  holds true:
$$
V_{[\widetilde \gamma_D]}=2i\pi \,{\rm Res}_D \, .
$$

For $\boldsymbol{\Omega}\in {\bf HLogAR}^w $ and $y\in Y$ arbitrary, 
the space $\langle \pi_1(Y,y)\cdot \boldsymbol{\Omega}\rangle$ spanned by the global monodromy of course contains those spanned by the `local monodromies' around all the irreducible components of $Z$.\footnote{These claims are not fully correct as currently stated. We let the reader make this precise and fully rigorous.} It follows that 
$$  {\rm Res}\big(\boldsymbol{\Omega}\big)
\subset
{\rm Monod}\big(\boldsymbol{\Omega}\big)
\, .$$

When the local monodromies $\mathcal M_{D}^{\rm loc}$ 
 around the irreducible 
divisors $D$ of $Z$ span the global mondromy in  \eqref{Eq:Monodromy},\footnote{Details must be added to make this sentence valid. This is left to the reader.} 
constructing new ARs from 
a hyperlogarithmic abelian relation 
 by residues is just an algebraic 
counterpart to the topological method by monodromy and  in particular one has 
\begin{equation}
\label{Eq:Res=Monod}
{\rm Res}\big(\boldsymbol{\Omega}\big)
=
{\rm Monod}\big(\boldsymbol{\Omega}\big)
\end{equation}
 for any   $\boldsymbol{\Omega}\in {\bf HLogAR}$.  The case when $Y$ can be taken as the complement of a curve in $\mathbf P^2$ is interesting because it covers the cases of del Pezzo surfaces as we will see further 
(cf.\,\eqref{Eq:U-L}).  Given a reduced algebraic curve $C\subset \mathbf P^2$ of degree $d$, there is a classical description of the set of generators of the fundamental group of the complement $\mathbf P^2\setminus C$ due to Van Kampen. Let $L\subset \mathbf P^2$ be a line intersecting $C$ transversely in $d$ distinct points $q_1,\ldots,q_d$.  One sets $L^*=L\setminus C$, 
one chooses a base point $y\in L^*$ and for all $i=1,\ldots,d$, one takes a loop $\gamma_i$  supported in $ L^*$, starting from $y$ and going close to $q_i$, making a small circle around this point in the direct order, and then returning to $y$ by traveling along the same path as at the beginning but in the opposite direction. Moreover, one assumes that the supports of the loops $\gamma_i$ do not cross. 
Then the homotopy classes of the $\gamma_i$'s generate $\pi_1( L^*,y)$ hence generate $\pi_1( \mathbf P^2\setminus C,y)$ up to the natural embedding 
of the former fundamental group into the latter. It follows that, in this situation, equality  \eqref{Eq:Res=Monod} holds true. 

An  interesting and specific case is when $\boldsymbol{\mathcal W}$ has maximal rank with a peculiar AR $\boldsymbol{\mathcal A}$ such that 
$${\rm Res}\big(\boldsymbol{\mathcal A}\big)
=
{\rm Monod}\big(\boldsymbol{\mathcal A}\big)
= \boldsymbol{AR}\big( \boldsymbol{\mathcal W}\big)\, .$$
 In such a case, $\boldsymbol{\mathcal A}$ is essentially uniquely defined and has to be considered as `the' most fundamental abelian relation of the considered web $\boldsymbol{\mathcal W}$. 
 

\subsection{Del Pezzo surfaces and their webs by conics}
\label{SS:dP-surfaces-webs-conics}
We  first recall some well-known material about del Pezzo surfaces then we discuss the webs induced by the fibrations in conics on them. We end by remarking that the Weyl group naturally acting on the Picard lattice of such a surface acts also on the space of (symbolic) hyperlogarithms we have to work with in order to handle the hyperlogarithmic ARs of the del Pezzo webs. 
We are going to be fairly succinct in this subsection, referring to 
the fourth chapter of Manin's book 
\cite{Manin} or in the eigth of Dolgachev's one \cite{Dolgachev} for
everything to do with the basics of the theory of del Pezzo's surfaces.

\subsubsection{\bf Del Pezzo surfaces}
\label{SSS}
Here we denote by  $r$  an integer in $\{3, 4,5,6,7,8\}$ and $d$ stands  for $9-r\in \{1,2,3,4,5,6\}$. A del Pezzo surface is the 
total space of the 
blow-up $b: {\rm dP}_d\rightarrow \mathbf P^2$ of the projective plane at 
$r$ points $p_1,\ldots,p_r$ in general position.  To indicate the number of points in the blow up, we use the notation $X_r$ as well. 
The anticanonical  sheaf $-K_{X_r}$ is ample (even very ample if $r\leq 6$) with anticanonical degree $(-K_r)^2$ equal to $d$. 

The Picard lattice ${\bf Pic}(X_r)$ is freely spanned over $\mathbf Z$ by 
the class $h$ of the pull-back under $b$ of a general line in $\mathbf P^2$ together with 
the classes of the exceptional divisor $\ell_i=b^{-1}(p_i)$ for $i=1,\ldots,$: one has ${\bf Pic}(X_r)=\mathbf Z \boldsymbol{h}\oplus \bigoplus_{i=1}^r \mathbf Z \boldsymbol{\ell}_i\simeq \mathbf Z^{r+1}$. Moreover, the intersection pairing $(\cdot,\cdot): {\bf Pic}(X_r)^2\rightarrow \mathbf Z$ is non degenerate with signature $(1,r)$. The canonical class of $X_r$ is given by 
$$K_r=K_{X_r}=-3\boldsymbol{h}+\sum_{i=1}^r \boldsymbol{\ell}_i$$
 and is   such that 
$K^2_r=d$.  Its orthogonal $K_r^\perp=\big\{ \, \boldsymbol{\rho} \in {\bf Pic}(X_r)\, \big\lvert \, 
(\boldsymbol{\rho}\cdot K_r)=0\, \big\}$ is free of rank $r$ and spanned by the $r$ classes 
\begin{equation}
\label{Eq:Fundamental-Roots}
\boldsymbol{\rho}_i=\boldsymbol{\ell}_i-\boldsymbol{\ell}_{i+1}  \quad \big( \, i=1,\ldots,r-1\,\big)\qquad 
\mbox{ 
 and } \qquad \boldsymbol{\rho}_r=\boldsymbol{h}-\boldsymbol{\ell}_1-\boldsymbol{\ell}_2-\boldsymbol{\ell}_3\, .
 \end{equation}

Each of the class $\boldsymbol{\rho}_i$ is such that 
$(\boldsymbol{\rho}_i,K_r)=0$ and $ \boldsymbol{\rho}_i^2=-2$. Together with the positive definite symmetric form $ - \big( \cdot,\cdot)\lvert_{K_r^\perp}$, the $\rho_i$'s define a root system of type $E_r$, with the convention that $E_4=A_4$ and $E_5=D_5$. The corresponding set of roots is $\mathcal R_r=\{ \, \boldsymbol{\rho}\, \lvert \, (\boldsymbol{\rho},K_r)=0\, \mbox{ and } \, \boldsymbol{\rho}^2=-2\}$ which is contained in the associated 
root space $R_r=K_r^\perp\otimes_{\mathbf Z} \mathbf R=\oplus_{i=1}^r \mathbf R\,\boldsymbol{\rho}_i$. 
Endowed with (the restriction of)  $-(\cdot,\cdot)$, the latter is  a Euclidean vector space.
For each root $\boldsymbol{\rho}$, the map 
\begin{equation}
\label{Eq:s-rho}
s_{\boldsymbol{\rho}}: R_r\rightarrow R_r, \, d\mapsto d+(d,{\boldsymbol{\rho}})\,{\boldsymbol{\rho}}
\end{equation}
 is a hyperplane reflection of $R_r$, admitting ${\boldsymbol{\rho}}$ as $(-1)$-eigenvector hence with invariant hyperplane ${\boldsymbol{\rho}}^\perp\subset R_r$.  The subgroup of the group of orthogonal transformations of  $R_r$ 
spanned by the $s_i=s_{{\boldsymbol{\rho}}_i}$ for $i=1,\ldots,r$ is a Coxeter group (that is it is finite). It is the so-called `Weyl group of type $E_r$' and will be denoted by $W(E_r)$ or shortly by $W_r$. 

In the two cases we are interested in in this paper, namely when $r=4,5$,
there are isomorphisms 
$$
W_4=W(A_4)\simeq \mathfrak S_5 \qquad \mbox{and}\qquad 
W_5=W(D_5)\simeq \big( \mathbf Z/{2\mathbf Z}\big)^4  \ltimes \mathfrak S_5
$$
which can be described explicitly in terms of the generators $s_i$ of the two Weyl groups (this is left to the reader). 

\subsubsection{\bf Lines and conics} 
\label{SSS:Lines-Conics-properties}
We now discuss several notions, objects and facts related to the lines  and to the conics contained in del Pezzo surfaces.

By definition, a {\it `line'} is a class $\boldsymbol{\ell} \in {\bf Pic}(X_r)$ such that $(K_r,\boldsymbol{\ell})=\boldsymbol{\ell}^2=-1$. 
This terminology makes sense for the following reason: from $\boldsymbol{\ell}^2=-1$, one deduces first that the associated linear system $\lvert \,\boldsymbol{\ell} \,\lvert$ is a singleton which is a smooth rational curve embedded in $X_r$, denoted by ${\ell}$ (ie. by the same but unbolded notation). Second, when $-K$ is very ample (that is when $r \leq 6$), the condition $(K_r,\boldsymbol{\ell})=-1$ means that $\varphi_{\lvert -K\lvert}({\ell})$ is a projective line in the target projective space $ \lvert\, -K_r\lvert^\vee\simeq \mathbf P^d$. 

The set  $\boldsymbol{\mathcal L}_r$ of lines included in $X_r$ is finite and its elements all can be explicitly given. In the two cases we will consider in this paper, these sets are 
\begin{align*}
\boldsymbol{\mathcal L}_4=  \Big\{
\hspace{0.1cm} \boldsymbol{\ell}_i\, ,\,  
\boldsymbol{h}-\boldsymbol{\ell}_j-\boldsymbol{\ell}_k
\hspace{0.1cm}
\Big\} \qquad 
\mbox{and }\qquad  
\boldsymbol{\mathcal L}_5=  \Big\{
\hspace{0.1cm} \boldsymbol{\ell}_i\, ,\,  
\boldsymbol{h}-\boldsymbol{\ell}_j-\boldsymbol{\ell}_k\, , \, 
2\boldsymbol{h}-\boldsymbol{\ell}_{tot} \hspace{0.1cm}
\Big\}
\end{align*}
where for $r=4,5$, the indices $i,j,k$ are  such that $i=1,\ldots,r$ and $1\leq j<k\leq  r$ and where we use the notation $\boldsymbol{\ell}_{tot}=\sum_{i=1}^r 
\boldsymbol{\ell}_i$.  It follows that one has 
$\lvert \, \boldsymbol{\mathcal L}_4 \, \lvert=10$ and 
$\lvert \, \boldsymbol{\mathcal L}_5 \, \lvert=16$.

A `{\it conic class}' on $X_r$ is an element  $\boldsymbol{\mathfrak c}\in {\bf Pic}(X_r)$ such that $(-K_r,\boldsymbol{\mathfrak c})=2$ and $\boldsymbol{\mathfrak c}^2=0$.  From the second condition, one deduces that for any such class, 
the linear system $\lvert \,\boldsymbol{\mathfrak c} \,\lvert$ has dimension 1 and gives rise to a  fibration
 $ \varphi_{\boldsymbol{\mathfrak c}}
: X_r\rightarrow 
\lvert \,\boldsymbol{\mathfrak c} \,\lvert ^\vee\simeq \mathbf P^1$ whose fibers are conics ({\it i.e.} rational 
curves with anticanonical degree 2). The description of conic classes is classical and not difficult:  
for instance, see Table 2 in \cite{CP} for the case $r=8$ (that is, case of del Pezzo surface of degree 1), from which the corresponding descriptions for any $r$ can be deduced easily.  In the two cases $r=4$ and $r=5$ of interest in this paper, we 
recall that $ \boldsymbol{\ell}_{tot}= \sum_{i=1}^r \boldsymbol{\ell}_{i}$ 
and setting  $\boldsymbol{\ell}_{\hat \imath}=
\boldsymbol{\ell}_{tot}-\boldsymbol{\ell}_i=\sum_{j\neq i} \boldsymbol{\ell}_j$ 
for $i=1,\ldots,r$, one has:  
\begin{equation}
\label{Eq:K4-K5}
\boldsymbol{\mathcal K}_4= \Big\{  \hspace{0.1cm}  \boldsymbol{h}-\boldsymbol{\ell}_i\, , \, 2\boldsymbol{h}- \boldsymbol{\ell}_{tot}
\hspace{0.1cm}  \big\lvert \hspace{0.1cm}  
i=1,\ldots,4\, 
 \Big\}
\quad 
\mbox{ and } 
\quad 
\boldsymbol{\mathcal K}_5=  \Big\{  \hspace{0.1cm} 
\boldsymbol{h}-\boldsymbol{\ell}_i
\, , \, 
2\boldsymbol{h}- \boldsymbol{\ell}_{\hat \imath}
\hspace{0.1cm}  \big\lvert \hspace{0.1cm}  
i=1,\ldots,5\, 
\Big\}\, .
\end{equation}

 The Weyl group $W_r$ acts transitively on the set  of lines $\boldsymbol{\mathcal L}_r$ which gives rise to a structure of $W_r$-module on $\boldsymbol{R}^{\boldsymbol{\mathcal L}_r}$ for any ring $\boldsymbol{R}$ (such as the ring of integers $\mathbf Z$ or the field $\mathbf C$).  
Since the action of each generator $s_i$  of $W_r$ on the set of lines ${\boldsymbol{\mathcal L}_r}$ can be described explicitly, 
 determining the decomposition of $\mathbf C^{ \boldsymbol{\mathcal L}_r}$  into irreducible $W_r$-modules can be done by straightforward computations within the theory of characters of $W_r$ (see \S3.2 and the Appendix in our preprint \cite{Hlog}{\rm )}: 
\begin{prop}
\label{Prop:Key-C-Lr-Wr-modules}
For $r=4,\ldots,8$, one has the following decompositions of\, $\mathbf C^{ \boldsymbol{\mathcal L}_r}$ in irreducible $W_r$-modules:\footnote{See \cite[Remark 3.3]{Hlog} for some explanations about the notations used for the $W_r$-irreducibles.} 
\begin{align}
\label{Eq:Decomposition-Wr}
\mathbf C^{ \boldsymbol{\mathcal L}_4}=& \, {\bf 1}\oplus V_{[41]}^4\oplus V_{[32]}^5 \nonumber \\
\mathbf C^{ \boldsymbol{\mathcal L}_5}=&  \, 
{\bf 1}\oplus V_{[4,1]}^5\oplus V_{[3,2]}^{10} 
 \nonumber   \\
\mathbf C^{ \boldsymbol{\mathcal L}_6}=&  
\, {\bf 1}\oplus V^{6,1}\oplus V^{20,2} 
\\
\mathbf C^{ \boldsymbol{\mathcal L}_7}=&   \, 
{\bf 1} \oplus V^{7,1}\oplus V^{21,3}\oplus V^{27,2}
%
%
\nonumber 
\\
\mathbf C^{ \boldsymbol{\mathcal L}_8}=&   \,  
{\bf 1} \oplus V^{8,1} 
\oplus V^{35,2} 
\oplus V^{84,4} 
\oplus V^{112,3} \, .
\mk
 \nonumber 
 \sk
\end{align}
\end{prop}
 
 By its very definition, the set of lines $\boldsymbol{\mathcal L}_r$ is included in 
 the Picard group of $X_r$ hence there is a  natural $\mathbf Z$-linear map 
 \begin{equation}
 \label{Z-Lr->Pic(Xr)}
{\mathbf Z}^{\boldsymbol{\mathcal L}_r} \longrightarrow {\bf Pic}(X_r)
 \end{equation}
which is easily seen to be $W_r$-equivariant and surjective.  Because $W_r$ acts by permutations on $\boldsymbol{\mathcal L}_r$, it lets invariant the total sum $\sum_{\boldsymbol{\ell}\in \boldsymbol{\mathcal L}_r}\boldsymbol{\ell}$ hence it follows that the trivial $W_r$-representation ${\bf 1}$ is a subrepresentation of both spaces in \eqref{Z-Lr->Pic(Xr)}.  Because $K^\perp$ is a subrepresentation of 
${\bf Pic}(X_r)$ which is the fundamental reflection representation $V^{r,1}$ of $W_r$ (with 
$V^{4,1}=V_{[41]}^4$ and $V^{5,1}=V_{[4,1]}^5$ for $r=4,5$ respectively), we get that  for $r=4,\ldots,8$, 
as a $W_r$-module, one has 
$$ {\bf Pic}(X_r)={\bf 1} \oplus V^{r,1}\, .$$
 
For any $r$, one denotes by ${\bf K}_r$ the kernel of  \eqref{Z-Lr->Pic(Xr)} in the category of $W_r$-representations. 
\begin{cor} As $W_r$-representations, one has  
\begin{align}
\label{Eq:Kr}
\mathbf K_4=  &\, V_{[32]}^5\hspace{2cm} 
\mathbf K_5=V_{[3,2]}^{10} \hspace{3cm} 
\mathbf K_6=V^{20,2} \\
\nonumber 
\mathbf K_7=&\, V^{21,3}\oplus V^{27,2}
\quad \quad 
\mathbf K_8=
 V^{35,2} 
\oplus V^{84,4} 
\oplus V^{112,3} \, .
\end{align}
\end{cor}
\noindent (We will give a natural geometric interpretation of the $W_r$-module ${\bf K}_r$ 
below  in 
 \S\ref{SSS:Del-Pezzo-Properties}).

For each conic fibration $\varphi_{\boldsymbol{{\mathfrak c}}}: X_r\rightarrow  \mathbf P^1$, 
we denote by $\Sigma_{\boldsymbol{{\mathfrak c}}}$ its {\it `spectrum'} which by definition is the subset formed by the points $\sigma$ of the target projective line such that the preimage $\varphi_{\boldsymbol{{\mathfrak c}}}^{-1}(\sigma)$ is a non irreducible conic, namely the sum of two lines lines of $X_r$: one has
$$\Sigma_{\boldsymbol{{\mathfrak c}}}=\Big\{ 
\hspace{0.1cm} \sigma \in \mathbf P^1\hspace{0.1cm} \big\lvert 
\hspace{0.1cm} \exists \, \ell_{\boldsymbol{{\mathfrak c}}}
,  \ell_{\boldsymbol{{\mathfrak c}}}'
\in \boldsymbol{{\mathcal L}}_r \hspace{0.1cm}\mbox{ such that  }\hspace{0.1cm} 
\varphi_{\boldsymbol{{\mathfrak c}}}^{-1}(\sigma)=\ell_{\boldsymbol{{\mathfrak c}}} 
+ \ell_{\boldsymbol{{\mathfrak c}}}'
\hspace{0.1cm} \Big\}\, .
$$
It is known that $\Sigma_{\boldsymbol{{\mathfrak c}}}$ 
admits exactly $r-1$ elements, which can also be characterized as the 
singular values of $\varphi_{\boldsymbol{{\mathfrak c}}}$. 
One verifies easily that any tangency point between two distinct conic fibrations on $X_r$ 
necessarily lies on a line which is invariant by both fibrations. In particular, if one denotes by  $U_r=X_r\setminus L_r$ the complement of the union of lines $L_r=\cup_{\boldsymbol{\ell} \in \boldsymbol{\mathcal L}_r} \ell \subset X_r$, then  
\begin{equation}
\label{Eq:dPhic-wedge-dPhic'}
\forall\, \boldsymbol{\mathfrak c}, \boldsymbol{\mathfrak c}'\in \boldsymbol{\mathcal K}_r 
\hspace{0.2cm} : 
\hspace{0.2cm} 
 \boldsymbol{\mathfrak c} \neq \boldsymbol{\mathfrak c}' 
\hspace{0.2cm}  
 \Longrightarrow 
 \hspace{0.2cm}  
 \mbox{\it the wedge-product }\, 
 d\varphi_{\boldsymbol{{\mathfrak c}}}\wedge 
d \varphi_{\boldsymbol{{\mathfrak c}}'} 
\hspace{0.15cm} \mbox{\it does not vanish on }\, U_r\, .
\end{equation}
\begin{center}
$\star$
\end{center}

The objects and quantities just considered above regarding lines and conic classes are given in explicit form in the following table: 
\begin{table}[!h]
\scalebox{1}{\begin{tabular}{|l||c|c|c|c|c|c|}
\hline
${}^{}$ \hspace{0.4cm}  
\begin{tabular}{c}\vspace{-0.35cm}\\
$\boldsymbol{r}$
\vspace{0.13cm}
\end{tabular} & $\boldsymbol{3}$ &  $\boldsymbol{4}$ & $\boldsymbol{5}$  & $\boldsymbol{6}$ & $\boldsymbol{7}$  & $\boldsymbol{8}$\\ \hline \hline
${}^{}$ \quad \begin{tabular}{c}\vspace{-0.35cm}\\
$\boldsymbol{E_r}$
\vspace{0.1cm}
\end{tabular}
 & $A_2\times A_1$ & $A_4$ & $D_5$   & $E_6$ & $E_7$  & $E_8$  \\ \hline 
${}^{}$ \quad 
\begin{tabular}{c}\vspace{-0.35cm}\\
$\boldsymbol{W_r=W(E_{r})}$
\vspace{0.1cm}
\end{tabular}
& 
$\mathfrak S_3\times \mathfrak S_2$
 & $\mathfrak S_5$ & $\big(\mathbf Z/2\mathbf Z)^4
 \ltimes  \mathfrak S_5  
$   & $W(E_6)$ & $W(E_7)$ & $W(E_8)$  \\ \hline 
${}^{}$ \quad  
\begin{tabular}{c}\vspace{-0.35cm}\\
$\boldsymbol{\omega_r=
\lvert W_r\lvert}$
\vspace{0.1cm}
\end{tabular}
& $12$ & $5!$ & $2^4\cdot
5!
$   & $2^7\cdot 3^4\cdot 5$ & $2^{10}\cdot 3^4 \cdot 5\cdot7$ & 
$2^{14}\cdot 3^5 \cdot 5^2\cdot7$
  \\ \hline 
${}^{}$ \quad 
\begin{tabular}{c}\vspace{-0.35cm}\\
$\boldsymbol{l_r={\lvert \mathcal L_r \lvert}}$
\vspace{0.1cm}
\end{tabular}
&6 & 10 &  16  & 27 & 56 & 240  \\ \hline
${}^{}$ \quad 
\begin{tabular}{c}\vspace{-0.35cm}\\
$\boldsymbol{\kappa_r={\lvert \mathcal K_r \lvert}}$
\vspace{0.13cm}
\end{tabular}
& 3 & 5 &    10 & 27 & 126 & 2160 \\
\hline
\end{tabular}}
\bk 
\caption{}
\label{Table1!}
\end{table}

\subsubsection{\bf Some other properties of del Pezzo surfaces}
\label{SSS:Del-Pezzo-Properties}
We briefly discuss some properties of del Pezzo surfaces which we will use at some points further on.

Let $X_r={\rm dP}_d$ be a del Pezzo surface as above (with $d=9-r$). By combining Riemann-Roch theorem and Kodaira vanishing, it follows that for any ample divisor $A$ on 
$X_r$, then one has $h^0(\mathcal O_{X_r}(A))=\chi(\mathcal O_{X_r}(A))=(1/2)A^2-\big(A ,K\big)+1$.  Since $-K$ is ample, one gets 
$$
{}^{} \qquad 
h^0\big(\mathcal O_{X_r}( -m K)\big)=  \frac{1}{2}\big( m^2+m\big)\,K^2+1= 
\frac{1}{2} (m^2+m)d+1 \qquad \forall\, m\in \mathbf N^*\, . 
$$

Let us now discuss some differential 1-forms with logarithmic poles along lines on $X_r$. 
First, remark that there is a short exact sequence of sheaves
$$
0 \rightarrow \Omega_{X_r}^1\longrightarrow 
\Omega_{X_r}^1\big({\rm Log}\,L_r\big)
\stackrel{\rm Res}{\longrightarrow} \bigoplus_{ \ell \in \boldsymbol{\mathcal L}_r}  \mathcal O_\ell
\rightarrow 0
$$
where for any line $\ell$, the structure sheaf $\mathcal O_\ell$ is identified with its push-forward 
$i^\ell_*( \mathcal O_\ell )$ 
on $X_r$ by the inclusion map $i^\ell : \ell\hookrightarrow X_r$.  
 The first terms of the  associated  
long exact sequence of cohomology groups are
\begin{equation}
\label{Eq:mook}
0\rightarrow   {\bf H}^0\Big( X,\Omega_{X_r}^1\Big)  \longrightarrow 
{\bf H}^0\Big( X_r,\Omega_{X_r}^1\big( {\rm Log}\,L_r \big) \Big) 
\stackrel{\rm Res}{\longrightarrow}  \mathbf C^{ \boldsymbol{\mathcal L}_r}
 \longrightarrow {\bf H}^1\big(X_r,\Omega^1_{X_r} \big)
\end{equation}
where the map ${\rm Res}$ is the direct sum of the ${\rm Res}_\ell$ for $\ell\in \boldsymbol{\mathcal L}_r$, each map ${\rm Res}_\ell$ consisting in taking the residue along the line $\ell\subset X_r$. 

Since $X_r$ is rational, the space ${\bf H}^0\big( X_r,\Omega_{X_r}^1\big)$ is trivial 
hence the space of logarithmic 1-forms 
$$ {\bf H}_{X_r}={\bf H}^0\Big( X_r,\Omega_{X_r}^1\big( {\rm Log}\,L_r \big) \Big) $$ identifies with the kernel of the map 
$\mathbf C^{ \boldsymbol{\mathcal L}_r}
 \longrightarrow {\bf H}^1\big(X_r,\Omega^1_{X_r} \big)$.
 Because  of the identifications 
${\bf H}^1\big(X_r,\Omega^1_{X_r} \big)\simeq {\bf H}^{1,1}\big(X_r,\mathbf C\big)=
{\bf H}^{2}\big(X_r,\mathbf C\big){=}  {\bf Pic}(X_r)_{\mathbf C}={\bf Pic}(X_r)\otimes_{\mathbf Z}\mathbf C$, this map  
 can be interpreted as  the complexification of the surjective cycle map 
$\mathbf Z^{ \boldsymbol{\mathcal L}}\rightarrow {\bf Pic}(X)$, $\ell\mapsto \boldsymbol{\ell}$ 
considered above in \eqref{Z-Lr->Pic(Xr)}. Hence $ {\bf H}_{X_r}$ identifies naturally with the kernel ${\bf K}_r$ discussed above.   We deduce the following corollary for any $r=4,\ldots,8$: 
\begin{cor}
\label{Cor:H-Xr}
  By means of the  map  {\rm Res} in \eqref{Eq:mook}, 
 the space of logarithmic forms ${\bf H}_{X_r}$ is naturally a subrepresentation of $\mathbf C^{ \boldsymbol{\mathcal L}_r}$.  And as such, ${\bf H}_{X_r}$ is isomorphic to the $W_r$-module ${\bf K}_r$ in \eqref{Eq:Kr}. 
\end{cor}

 As far we are aware of, the description of ${\bf H}_{X_r}= {\bf H}^0\big( X_r,\Omega_{X_r}^1\big( {\rm Log}\,L_r \big) \big)$ as a $W_r$-module and its decomposition in irreducibles does not appear in the existing literature, except in the case when $r=4$ which is the subject of \cite[Lemma 2.1]{DFL}.

\begin{rem}
\label{Rem:rem-rem}
\noindent {\rm 1.}  The group 
${\rm Aut}(X_r)$ of 
automorphisms of $X_r$ naturally  acts  on 
${\bf H}_{X_r}$ as well. One verifies that 
this action is compatible with the $W_r$-action defined just above up to the natural monomophism of groups ${\rm Aut}(X_r) \hookrightarrow W_r$.  

When $r=4$, one has $X_4={\rm dP}_5\simeq \overline{\mathcal M}_{0,5}$  with ${\rm Aut}(X_4)={\rm Aut}( \overline{\mathcal M}_{0,5})= {\rm Aut}( {\mathcal M}_{0,5})=\mathfrak S_5$ and in this case  ${\rm Aut}(X_4) \hookrightarrow W_4$ happens to be an isomorphism.  
This does not longer holds true for $r>4$.  For instance when $r=5$, one has 
${\rm Aut}(X_5)\simeq (\mathbf Z/2\mathbf Z)^4$ for a generic del Pezzo quartic surface $X_5={\rm dP}_4$ ({\it cf.} \cite[\S8.6.4]{Dolgachev}{\rm )} and up to identifying $W_5$ with $(\mathbf Z/2\mathbf Z)^4\rtimes \mathfrak S_5$, the morphism ${\rm Aut}(X_5) \hookrightarrow W_5$ corresponds to the natural inclusion $(\mathbf Z/2\mathbf Z)^4\hookrightarrow (\mathbf Z/2\mathbf Z)^4\ltimes \mathfrak S_5$ of the first factor. 
\sk

\vspace{-0.5cm}
\noindent {\rm 2.} Because the  map $\mathbf Z^{\boldsymbol{\mathcal L} }\rightarrow   {\bf Pic}(X)$ is defined over $\mathbf Z$, \eqref{Eq:mook} can be lifted to a short exact sequence of $W_r$-representations $0\rightarrow   
{\bf H}_{X_r}^{\mathbf Z}
  \longrightarrow 
\mathbf Z^{ \boldsymbol{\mathcal L}}
 \longrightarrow   {\bf Pic}(X)\rightarrow 0$ defined over the ring of integers,  
 where 
 ${\bf H}_{X_r}^{\mathbf Z}
$ stands for the free $\mathbf Z$-module of elements in  ${\bf H}_{X_r}={\bf H}^0\big( X_r,\Omega_{X_r}^1\big( {\rm Log}\,L_r \big) \big)$ whose residues along the lines in $X_r$ all are integers. 
We will not use this arithmetical fact in what follows.
\end{rem}

\subsubsection{\bf The del Pezzo web ${\mathcal W} \hspace{-0.46cm}{\mathcal W}_{ {\bf dP }_d}$}
\label{SSS:Del-Pezzo-webs}
Here we just give the general definition of the webs $\boldsymbol{\mathcal W}_{ {\rm dP }_d}$ before focusing on the two cases we are interested in in this paper, namely $\boldsymbol{\mathcal W}_{ {\rm dP }_5}$ and $\boldsymbol{\mathcal W}_{ {\rm dP }_4}$ that we make explicit in some affine coordinates.
\mk 

By definition, the {\it `del Pezzo web'} of the considered del Pezzo surface $X_r={\rm dP}_d$, denoted by $\boldsymbol{\mathcal W}_{X_r}$ or $\boldsymbol{\mathcal W}_{{\rm dP}_d}$,  is the web formed by all the pencils of conics on the surface:  one has 
$$\boldsymbol{\mathcal W}_{{\rm dP}_d}= 
\boldsymbol{\mathcal W}_{X_r}
=\boldsymbol{\mathcal W}\Big(\hspace{0.05cm} 
\varphi_{ \boldsymbol{\mathfrak c}} \hspace{0.15cm} \big\lvert \hspace{0.15cm} 
\boldsymbol{\mathfrak c} \in 
\boldsymbol{\mathcal K}_r \hspace{0.05cm}  \Big)\, . 
$$ 
From \eqref{Eq:dPhic-wedge-dPhic'}, it follows that the web $\boldsymbol{\mathcal W}_{{\rm dP}_d}$ is regular on the complement $U_r\subset X_r$ of the union of all lines contained in the considered del Pezzo surface.  In the case when $r=4$, $U_4$ corresponds to $\mathcal M_{0,5}$ up to the natural identification $X_4={\rm dP}_5\simeq 
\overline{\mathcal M}_{0,5}$.

In the two cases $r=4$ and $r=5$ of interest in this paper, 
the sets $\boldsymbol{\mathcal K}_r$ are described explicitly  in \eqref{Eq:K4-K5} which allows to give simple birational models for the two del Pezzo webs 
$\boldsymbol{\mathcal W}_{{\rm dP}_5}$ and $\boldsymbol{\mathcal W}_{{\rm dP}_4}$. 
Let $p_1,\ldots,p_4$ and $p_5$ be the points (in general position) in $\mathbf P^2$ such that $X_r$ is the total space of the blow-up $b=b_r: X_r\rightarrow \mathbf P^2$. We assume 
that the first four points $p_i$  are 
the vertices of the fundamental tetrahedron in $\mathbf P^2$ (that is $p_1=[1:0:0]$, $p_2=[0:1:0],\,\ldots$, $p_4=[1:1:1]$ 
and $p_5=[a:b:1]$ with 
affine coordinate $(a,b) \in \mathbf C^2$ such that $ab(a-1)(b-1)(a-b)\neq 0$. Then there 
is no difficulty to see that the 
following holds true: 
\begin{prop}
\label{P:b*WdPr-on-P2}
  {\rm 1.} The direct image of $\boldsymbol{\mathcal W}_{{\rm dP}_5}$ by $b$ 
is the 5-web on $\mathbf P^2$ formed by the four pencils of lines with vertices $p_1,\ldots,p_4$ plus the pencil of conics passing through these four points. 

Analytically, one has $b_*\big( 
\boldsymbol{\mathcal W}_{{\rm dP}_5}
\big)=\boldsymbol{\mathcal W}\Big(\, x
\, , \, 
y 
\, , \, \frac{x}{y}
\, , \,  \frac{1-y}{1-x}
\, , \, 
 \frac{x(1-y)}{y(1-x)}\,\Big)\, . 
$\sk 


\noindent {\rm 2.} The direct image of $\boldsymbol{\mathcal W}_{{\rm dP}_4}$ by $b$ 
is the 10-web on $\mathbf P^2$  formed by the five pencils of lines with vertices $p_1,\ldots,p_5$ plus the five pencils of conics passing through four points among these five.

 Analytically, one has $ 
b_*\big( 
\boldsymbol{\mathcal W}_{{\rm dP}_5}
\big)=\boldsymbol{\mathcal W}\big(\, U_1,\ldots, U_{10}\,\big)
$ where the $U_i$'s are the following ten rational functions (with $P$ standing for the polynomial $(1-b )x - (1-a )y - (a - b)$): 
$$
\scalebox{1.1}
{
\begin{tabular}{lllll}
$U_1= x$ & $ U_2=  \frac1y$ 
& $ U_3=  \frac{y}{x} $ 
& $U_4= 
\frac{x-y}{x-1} $ 
& $  U_5= 
   \frac{b(a-x ) }{a y - b x} $ 
  \vspace{0.25cm}  \\
 $ U_6=    \frac{
 P}{ (x - 1)(y-b)}$ 
   & $ U_7=     \frac{ (x - y)(y-b)}{ y\,P}$ 
    &
$    U_8=  \frac{x\,P}{ (x - y)(x - a)} $
     & 
$ U_9= \,
          \frac{  
      y(x - a)}{ x (y-b)}  $ 
      & 
      $U_{10}=  \frac{
       x(y - 1)}{y (x - 1)} \, .$
\end{tabular}
}
$$
\end{prop}

\subsubsection{\bf The Weyl group action on the hyperlogarithmic ARs of a del Pezzo' web}
\label{SSS:Weyl-group-action-on-ARs}
Since the Weyl group $W_4\simeq \mathfrak S_5$ identifies with the group of automorphism of $U_4\simeq \mathcal M_{0,5}$, it is not surprising (and it hs been known  for a long time) that it acts linearly on the space of abelian relations of $\boldsymbol{\mathcal W}_{ {\rm dP}_5}$. Since the inclusion ${\rm Aut}(U_r)\hookrightarrow W_r$ is strict when $r>4$ (see Remark \ref{Rem:rem-rem}.1),  that a similar fact happens for all the other del Pezzo webs is not obvious.  This is what we discuss below.

Given a conic class $\boldsymbol{\mc}  \in \boldsymbol{\cK}_r$, we denote by $\boldsymbol{\mc}^{red}$ the set of $r-1$ non irreducible conics in the 
 linear system $\lvert \, \boldsymbol{\mc} \, \lvert \subset \mathbf P{\bf H}^0\big(  X_r, 
-2K_r\big)$. As a 1-cycle, each non irreducible conic $c\in \boldsymbol{\mc}^{red}$ can be written $c=l_c+l'_c$  for two lines  $l_c,l'_c \in \boldsymbol{\mathcal L}_r$. 
Then we can define a natural injection $\boldsymbol{\mc}^{red}\hookrightarrow \mathbf C^{\boldsymbol{\mathcal L}_r}$ which in turn 
allows us to consider $\mathbf C^{\boldsymbol{\mc}^{red}}$ as a subspace of 
$\mathbf C^{\boldsymbol{\mathcal L}_r}$. 
The stabilizer $W_{\boldsymbol{\mc}}$ of $\boldsymbol{\mc} $  in $W_r$ acts by permutations on $\boldsymbol{\mc}^{red}$ from which one gets a $W_{\boldsymbol{\mc}}$-action on 
$\mathbf C^{{\boldsymbol{\mc} }^{red}}$. Considering the action of the latter group on $\boldsymbol{\mathcal L}_r$ induced by restriction of  the action of $W_r$, one obtains that $i_{\boldsymbol{\mc} }$ is a morphism of 
$W_{\boldsymbol{\mc}}$-representations.  Since $W_{\boldsymbol{\mc}}$ acts by permutations on $\boldsymbol{\mathfrak c}^{red}$, the element $1_{\boldsymbol{\mc}}=\sum_{ c\in \boldsymbol{\mathfrak c}^{red}} c\in \mathbf C^{\mathfrak c^{red}}$ is $W_{\boldsymbol{\mc}}$-invariant  and 
admits as $W_{\boldsymbol{\mc}}$-invariant supplementary the subspace 
$\mathcal U_{\boldsymbol{\mc}}$ 
 spanned over $\mathbf C$ by the elements 
$c-c'\in  \mathbf C^{\boldsymbol{\mathfrak c}^{red}}$ for all $c,c'\in \boldsymbol{\mathfrak c}^{red}$.  We denote by ${\bf 1}_{\boldsymbol{\mathfrak c}}=\langle\, 1_{\boldsymbol{\mc}}\, \rangle$ the trivial representation spanned by 
$1_{\boldsymbol{\mc}}$. 
\begin{lem}
\label{L:Fc-module}
 {\rm 1.} As a group,  $W_{\boldsymbol{\mathfrak c}}$ is isomorphic to $W(D_{r-1})$.

\noindent ${}^{}$ \; {\rm 2.} The decomposition $ \mathbf C^{\boldsymbol{\mathfrak c}^{red}}={\bf 1}\oplus \boldsymbol{\mathcal U}_{\boldsymbol{\mathfrak c}}$ 
actually is a direct sum of irreducible $W_{\boldsymbol{\mc}}$-modules. 

\noindent ${}^{}$ \; {\rm 3.}
As $W_{\boldsymbol{\mc}}$-representations,  
 one has  $\boldsymbol{\mathcal U}_{\boldsymbol{\mathfrak c}} \simeq V_{[.(r-2)1]}$  
for $r=4,6$ and $ \boldsymbol{\mathcal U}_{\boldsymbol{\mathfrak c}} \simeq V_{[2.2]^+}$ for $r=5$.\footnote{We recall that the irreducible representations  
of $W(D_n)$ are encoded by bipartitions of $n$. This is a 1-1 correspondence when $n$ is odd. When $n$ is even, given a partition $\eta$ of $n/2$, there are two non isomorphic $W(D_n)$-irreps associated to the bipartition $(\eta,\eta)$ which are denoted by $[\eta.\eta]^+$ 
and $[\eta.\eta]^-$. For more details, see \cite[\S5.6]{GP}.}
\end{lem}
\begin{proof} 
 Since $W_r$ acts transitively on $\boldsymbol{\mathcal K}_r$, it suffices to prove the lemma when $\boldsymbol{\mathfrak c}= \boldsymbol{h}-\boldsymbol{\ell}_1$. 
 In this case, one verifies that $W_{\boldsymbol{\mc}_1}$ is isomorphic to the subroup of $W_r$ generated by 
 the $s_i$'s for $i=2,\ldots,r$. Thus $W_{\boldsymbol{\mc}_1}\simeq W(D_{r-1})$ (with the convention that $D_3=A_3$). 

That $ \mathbf C^{\boldsymbol{\mathfrak c}^{red}}={\bf 1}\oplus \boldsymbol{\mathcal U}_{\boldsymbol{\mathfrak c}}$  is a decomposition in $W_{\boldsymbol{\mc}_1}$-modules is clear.  To determine $\boldsymbol{\mathcal U}_{\boldsymbol{\mathfrak c}}$, we notice that the elements of $\boldsymbol{\mathfrak c}^{red}$ are the conics $c_j=[\boldsymbol{h}-\boldsymbol{l}_1-\boldsymbol{l}_j]+[\boldsymbol{l}_j]$'s for $j=2,\ldots,r$, where each element between brackets is an element  of $\boldsymbol{\mathcal L}_r$ and the notation means that $c_j$ has the two  lines $\boldsymbol{h}-\boldsymbol{l}_1-\boldsymbol{l}_j$ and $\boldsymbol{l}_j$ (viewed here as curves contained in $X_r$) as irreducible components. Direct computations give that as permutations of $\boldsymbol{\mathfrak c}^{red}=\{ c_2,\ldots,c_r\}$: $s_j$ is the transposition exchanging $c_j$ and $c_{j+1}$ for $j=2,\ldots,r-1$ and $s_r$  is the transposition exchanging $c_2$ and $c_{3}$. 
By elementary character theory of $W(D_{r-1})$ (see the Appendix), one gets 
3.  
\end{proof}

A point which will be  relevant for our purpose, is that there is a nice geometric interpretation of $\boldsymbol{\mathcal U}_{\boldsymbol{\mathfrak c}} $. 
We denote by $\varphi_{\boldsymbol{\mathfrak c}}: X_r\rightarrow \lvert \, \boldsymbol{\mc} \, \lvert \simeq \mathbf P^1$ the fibration in conics associated to 
$\boldsymbol{\mathfrak c}$. It is known that  
 the {\it spectrum'} $\Sigma_{ \boldsymbol{\mathfrak c}} 
 =\big\{ \hspace{0.1cm}  \sigma \in \mathbf P^1
\hspace{0.1cm}\big\lvert
\hspace{0.1cm}  \varphi_{\boldsymbol{\mathfrak c}}^{-1}(\sigma)=\ell_{\boldsymbol{\mathfrak c}}+
\ell'_{\boldsymbol{\mathfrak c}} \hspace{0.1cm}\mbox{ with } \hspace{0.1cm}\ell_{\boldsymbol{\mathfrak c}},
\ell'_{\boldsymbol{\mathfrak c}} \in \boldsymbol{\mathcal L}_r \hspace{0.1cm} \big\}
 $ of $\varphi_{\boldsymbol{\mathfrak c}}$ is 
 has cardinality $r-1$ from which it follows that the 
 space of logarithmic holomorphic 1-forms
$$\boldsymbol{\mathcal H}_{ \boldsymbol{\mathfrak c}}=
{\bf H}^0
\Big( \mathbf P^1, \Omega_{\mathbf P^1}\big( \, {\rm Log}\, \
\Sigma_{ \boldsymbol{\mathfrak c}}
 \big)\Big)$$ has dimension $r=2$. If $z$ stands for an affine coordinate on $\mathbf P^1$ with respect to which  $\infty$ belongs to $\Sigma_{ \boldsymbol{\mathfrak c}}$, then denoting by $\sigma_{ \boldsymbol{\mathfrak c}}^1,\ldots,\sigma_{ \boldsymbol{\mathfrak c}}^{r-2}$ the elements of $\Sigma_{ \boldsymbol{\mathfrak c}}\cap \mathbf C$, one has that 
 the logarithmic 1-forms $dz/(z-\sigma_{ \boldsymbol{\mathfrak c}}^i)$, $i=1,\ldots,r-2$, form a basis of $\boldsymbol{\mathcal H}_{ \boldsymbol{\mathfrak c}}$. Since $\varphi_{\boldsymbol{\mathfrak c}}$ is a morphism, it follows that 
$$
{\bf H}_{ \boldsymbol{\mathfrak c}} = 
\varphi_{\boldsymbol{\mathfrak c}}^*\Big(
 \boldsymbol{\mathcal H}_{ \boldsymbol{\mathfrak c}} 
 \Big)   =\left\langle\hspace{0.1cm}
 \frac{d\varphi_{\boldsymbol{\mathfrak c}}}{\varphi_{\boldsymbol{\mathfrak c}}-\sigma_{ \boldsymbol{\mathfrak c}}^i} \hspace{0.1cm} 
 \hspace{0.1cm} 
 \right\rangle_{i=1}^{r-2} \simeq \mathbf C^{r-2}
 $$
is naturally a subspace of  
${\bf H}_{X_r}={\bf H}^0\Big(  X_r, 
\Omega^1_{X_r}\big(  {\rm Log}\, L_r\big) \Big)$  hence of $\mathbf C^{\boldsymbol{\mathcal L}_r}$ (as it follows from the 
injectivity of the 
 map ${\rm Res}$ in \eqref{Eq:mook}).
For any $ \boldsymbol{\mathfrak c}$, we denote the linear inclusion of 
${\bf H}_{ \boldsymbol{\mathfrak c}}  $  into $\mathbf C^{\boldsymbol{\mathcal L}_r}$ by 
\begin{equation}
\label{Eq:iota-c}
\iota_{\boldsymbol{\mathfrak c} }  : {\bf H}_{ \boldsymbol{\mathfrak c}} \hookrightarrow \mathbf C^{\boldsymbol{\mathcal L}_r}\, .
\end{equation}

\begin{lem}
\label{Lem:Hc}
As subspaces of $\mathbf C^{\boldsymbol{\mathcal L}_r}$, 
${\bf H}_{ \boldsymbol{\mathfrak c}}  $ and $\boldsymbol{\mathcal U}_{ \boldsymbol{\mathfrak c}} $  coincide. In particular, ${\bf H}_{ \boldsymbol{\mathfrak c}}  $ is $W_{ \boldsymbol{\mathfrak c}}  $-stable and is an irreducible representation.
\end{lem}
\begin{proof}
That ${\bf H}_{ \boldsymbol{\mathfrak c}}  $ is included in $\boldsymbol{\mathcal U}_{ \boldsymbol{\mathfrak c}} $ (when both spaces are seen as subspaces of $\mathbf C^{\boldsymbol{\mathcal L}_r}$) immediately follows from the residues theorem applied to the 1-forms $dz/(z-\sigma_{ \boldsymbol{\mathfrak c}}^i)$ which form a basis of $
\boldsymbol{\mathcal H}_{ \boldsymbol{\mathfrak c}}$. One concludes by noticing that both ${\bf H}_{ \boldsymbol{\mathfrak c}}  $ and $\boldsymbol{\mathcal U}_{ \boldsymbol{\mathfrak c}} $ have dimension $r-2$. 
\end{proof}
From this lemma together with what has been said above about $\boldsymbol{\mathcal U}_{ \boldsymbol{\mathfrak c}} $ in Lemma \ref{L:Fc-module}, we get that 
${\bf H}_{ \boldsymbol{\mathfrak c}} $ naturally is an irreducible  representation of 
$W_{ \boldsymbol{\mathfrak c}}$.  For any $w$ in this group, we denote by $w\cdot \omega$ the action of  $w$ on any  $\omega\in 
{\bf H}_{ \boldsymbol{\mathfrak c}}$ when this space is viewed as a $W_{ \boldsymbol{\mathfrak c}}$-representation. We will use the same notation, with a dot $\cdot$, for the action of the whole Weyl group $W_r$ on ${\bf H}_{X_r}$ or $\mathbf C^{ \boldsymbol{\mathcal L}_r}$. Thanks to Lemma \ref{Lem:Hc}, all these notations are compatible. \sk

From now on we fix a base conic class in $\boldsymbol{\mathcal K}$ that we denote by 
$\boldsymbol{\mathfrak c}_0$. Since $W_r$ acts transitively on $\boldsymbol{\mathcal K}$, for each element $\boldsymbol{\mathfrak c}$ in it, there 
exists $w_{\boldsymbol{\mathfrak c }_0}^{\boldsymbol{\mathfrak c}} $ (with 
$w_{\boldsymbol{\mathfrak c }_0}^{\boldsymbol{\mathfrak c}_0}=1$) 
such that 
$w_{\boldsymbol{\mathfrak c }_0}^{\boldsymbol{\mathfrak c}} \cdot 
{\boldsymbol{\mathfrak c }_0}= {\boldsymbol{\mathfrak c}}$ for every 
conic class $\boldsymbol{\mathfrak c }$.  
Since $W_r/W_{\boldsymbol{\mathfrak c }_0}\simeq \boldsymbol{\mathcal K}_r$, it follows that 
\begin{equation}
\label{Eq:Wr-sqcup}
W_r=\bigsqcup_{ \boldsymbol{\mathfrak c} \in \boldsymbol{\mathcal K}} \,
w_{\boldsymbol{\mathfrak c }_0}^{\boldsymbol{\mathfrak c}} W_{ \boldsymbol{\mathfrak c }_0}\, . 
\end{equation}
Remark that  for any $\boldsymbol{\mathfrak c } \in 
   \boldsymbol{\mathcal K }_r$, one has $w_{\boldsymbol{\mathfrak c }_0}^{\boldsymbol{\mathfrak c}} W_{ \boldsymbol{\mathfrak c }_0}=\big\{ \, w \in W_r \, \lvert\, 
   w \cdot \boldsymbol{\mathfrak c }_0=\boldsymbol{\mathfrak c }\, \big\}$.

Given a fixed weight $\nu>0$, the inclusions \eqref{Eq:iota-c} actually land in 
${\bf H}_{X_r}$ hence 
give rise to inclusions between the $\nu$-th tensorial powers ${\bf H}_{ \boldsymbol{\mathfrak c} }^{\otimes \nu} \hookrightarrow {\bf H}_{X_r}^{\otimes \nu}$, 
 again denoted by 
$\iota_{ \boldsymbol{\mathfrak c} } $. 
For any $w \in  W_{r}$ and any $\eta\in {\bf H}_{ \boldsymbol{\mathfrak c }}^{\otimes \nu}$, 
letting $w$ act on $ \iota_{\boldsymbol{\mathfrak c }}( \eta)$ gives an element of  
${\rm Im}\big(  \iota_{w \boldsymbol{\mathfrak c }}
\big)\subset {\bf H}_{X_r}^{\otimes \nu} $ hence one can set 
\begin{equation}
w\bullet \eta= \iota_{w \boldsymbol{\mathfrak c }}^{-1}\big(   w \cdot 
 \iota_{\boldsymbol{\mathfrak c }}( \eta)\big)\in  {\bf H}_{ w \boldsymbol{\mathfrak c }}
 ^{ \otimes \nu}\, , 
\end{equation}
a definition which gives rise to a map 
\begin{equation}
\label{Eq:bullet}
\bullet  : W_r\times \left(\oplus_{ \boldsymbol{\mathfrak c } \in \boldsymbol{\cK}} {\bf H}_{\boldsymbol{\mathfrak c }}^{\otimes \nu}  \right)
\longrightarrow 
\oplus_{ \boldsymbol{\mathfrak c } \in \boldsymbol{\cK}} {\bf H}_{\boldsymbol{\mathfrak c }}^{\otimes \nu}\, .
\end{equation}

From now on, we no longer write the inclusions $\iota_{ \boldsymbol{\mathfrak c }}$ to simplify formulas.  For any $\boldsymbol{\mathfrak c } \in 
   \boldsymbol{\mathcal K }_r$ 
we have $w_{\boldsymbol{\mathfrak c }_0}^{\boldsymbol{\mathfrak c}} 
 {\bf H}_{ \boldsymbol{\mathfrak c }_0}^{\otimes \nu}= {\bf H}_{  \boldsymbol{\mathfrak c }}
 ^{\otimes \nu}$ and for any $w\in W$, setting 
 $\boldsymbol{\mathfrak c }'=w\boldsymbol{\mathfrak c }$, we have $(w w_{\boldsymbol{\mathfrak c }_0}^{\boldsymbol{\mathfrak c}}) {\boldsymbol{\mathfrak c }_0}={\boldsymbol{\mathfrak c }}'$ hence $w w_{\boldsymbol{\mathfrak c }_0}^{\boldsymbol{\mathfrak c}}\in  w_{\boldsymbol{\mathfrak c }_0}^{\boldsymbol{\mathfrak c}'}W_{ \boldsymbol{\mathfrak c }_0}$ thus  there exists ${w}'\in W_{ \boldsymbol{\mathfrak c }_0}$ such that $w w_{\boldsymbol{\mathfrak c }_0}^{\boldsymbol{\mathfrak c}}=w_{\boldsymbol{\mathfrak c }_0}^{\boldsymbol{\mathfrak c}'} {w}'$ in $W_r$. From this, one gets that given an element $\eta_{\boldsymbol{\mathfrak c }}$ 
 of 
$ {\bf H}_{\boldsymbol{\mathfrak c }}^{\otimes \nu}$, first one has 
$\eta_{\boldsymbol{\mathfrak c }}=
w_{\boldsymbol{\mathfrak c }_0}^{\boldsymbol{\mathfrak c}} 
\eta_{\boldsymbol{\mathfrak c }_0}$ for some 
$\eta_{\boldsymbol{\mathfrak c }_0}
\in  {\bf H}_{\boldsymbol{\mathfrak c}_0}^{\otimes \nu}$ hence 
$$w \bullet \eta_{\boldsymbol{\mathfrak c }}=w\cdot w_{\boldsymbol{\mathfrak c }_0}^{\boldsymbol{\mathfrak c}} 
\eta_{\boldsymbol{\mathfrak c }_0}= (w w_{\boldsymbol{\mathfrak c }_0}^{\boldsymbol{\mathfrak c}}) \,
\eta_{\boldsymbol{\mathfrak c }_0}= (w_{\boldsymbol{\mathfrak c }_0}^{\boldsymbol{\mathfrak c}'} {w}')\, \eta_{\boldsymbol{\mathfrak c }_0}
= w_{\boldsymbol{\mathfrak c }_0}^{\boldsymbol{\mathfrak c}'} ({w}'\, \eta_{\boldsymbol{\mathfrak c }_0})\, .
$$
As explained in \cite[\S3.3]{FultonHarris}, this defines a $W_r$-action on $\oplus_{ \boldsymbol{\mathfrak c } \in \boldsymbol{\cK}} {\bf H}_{\boldsymbol{\mathfrak c }}^{\otimes \nu} $, which is isomorphic to  the one induced by the $W_{\boldsymbol{\mathfrak c }_0}$-action on ${\bf H}_{\boldsymbol{\mathfrak c }_0}^{\otimes \nu}$.  This gives us the following

\begin{lem}
\label{Lem:yoyo}
 \hspace{0.2cm}{\rm 1.} The map $\bullet$ defined in \eqref{Eq:bullet} makes of the direct sum 
$ \oplus_{ \boldsymbol{\mathfrak c } \in \boldsymbol{\cK}} {\bf H}_{\boldsymbol{\mathfrak c }}^{\otimes \nu}$ a $W_r$-representation isomorphic to  the induced representation ${\rm Ind}_{W_{\boldsymbol{\boldsymbol{\mathfrak c} }_0}}^{W_r}\Big( 
{\bf H}_{\boldsymbol{\mathfrak c }_0}^{\otimes \nu}
\Big)$. 
\sk

\noindent \hspace{0.2cm}{\rm 2.} The linear map 
\begin{align}
\Phi^\nu_{r} =\Phi^\nu_{X_r} 
:  
\oplus_{ \boldsymbol{\mathfrak c } \in \boldsymbol{\cK}} {\bf H}_{\boldsymbol{\mathfrak c }}^{\otimes \nu} & \longrightarrow {\bf H}_{X_r}^{\otimes \nu} \\
 \big(\eta_{ \boldsymbol{\mathfrak c }}\big)_{ \boldsymbol{\mathfrak c } \in \boldsymbol{\cK}}& \longmapsto \sum_{ { \boldsymbol{\mathfrak c }
\in \boldsymbol{\cK}}} \eta_{ \boldsymbol{\mathfrak c }}
\nonumber
\end{align}
 is $W_r$-equivariant  
therefore requiring that 
\begin{equation}
\label{Eq:ARHlog-W-module}
0\ra {\bf ARHLog}^\nu\longrightarrow \oplus_{ \boldsymbol{\mathfrak c } \in \boldsymbol{\cK}} {\bf H}_{\boldsymbol{\mathfrak c }}^{\otimes \nu} 
\stackrel{
\Phi^\nu_{r} 
}{\longrightarrow} {\bf H}_{X_r}^{\otimes \nu}
\end{equation}
be an  exact sequence of $W_r$-representations naturally defines a $W_r$-action on 
the kernel ${\bf ARHLog}^\nu$ which identifies with 
the space of weight $\nu$  hyperlogarithmic abelian relations of the web $\boldsymbol{\mathcal W}_{{\rm dP}_d}$. 
\end{lem}

For any $\nu\geq 1$, let $h_r^\nu$ stand for the dimension of the image of ${\Phi}^\nu_{r} :  \oplus_{ \boldsymbol{\mathfrak c } \in \boldsymbol{\cK}} {\bf H}_{\boldsymbol{\mathfrak c }}^{\otimes \nu}
\longrightarrow {\bf H}_{X_r}^{\otimes \nu}$.  Obviously, one has 
$\dim_{\mathbf C}\, \big({\bf ARHLog}^\nu\big)=  \kappa_r\cdot (r-2)^\nu-h_r^\nu$.  
Thanks to the preceding lemma, one can easily compute the dimension $h_r^\nu$ in the simplest cases: 
because  $\Phi^\nu_{r} $ is $W_r$-equivariant and non trivial (since 
any ${\bf H}_{\boldsymbol{\mc}}^{\otimes \nu} \ra {\bf H}_{X_r}^{\otimes \nu}$ is an embedding hence non trivial), its image is a non zero $W_r$-submodule of $ {\bf H}_{X_r}^{\otimes \nu}$. If the latter tensor product is irreducible as a $W_r$-module then 
necessarily $\nu=1$ and 
$\Phi^1_{r} $ is surjective. This implies that  $h^1_r=\dim_{\mathbf C} {\bf H}_{X_r}=
l_r -r-1$ (with $l_r=\lvert \, \boldsymbol{\mathcal L }_r\, \lvert$). 
Then from Corollary \eqref{Cor:H-Xr} and \eqref{Eq:Kr} we deduce the 
\begin{lem}
\label{Lem:lolo}
For $r=4,5,6$, one has 
$ {\rm HLrk}^1_r=
\dim_{\mathbf C}\, \big({\bf ARHLog}^1\big)=  \kappa_r\cdot (r-2)-l_r+r+1$.\sk
%
\end{lem}

The  explicit values of $l_r=\lvert \, \boldsymbol{\mathcal L }_r\, \lvert$, $\kappa_r
=\lvert \, \boldsymbol{\mathcal K }_r\, \lvert$ and ${\rm HLrk}^1_r$ for $r=4,5,6$ are given in the following table: \sk
$$ \begin{tabular}{|c||c|c|c|}
\hline
\, $\boldsymbol{r}$  &   4& 5  & 6 \\ \hline \hline
\, $\boldsymbol{l_r}$ & 10 &  16  & 27   \\ \hline
\, $\boldsymbol{\kappa_r}$  & 5 &    10 & 27  \\ \hline 
\, \begin{tabular}{l}\vspace{-0.25cm} \\
$  \boldsymbol{{\bf HLrk}^1_r}$  
\vspace{0.15cm}
\end{tabular}
& 5 &    20 & 88 
 \\
\hline
\end{tabular}
$$

\subsubsection{\bf Some remarks about the space of  hyperlogarithmic ARs of del Pezzo webs}
Let $\boldsymbol{\mathcal W}$ be a web as in \S\ref{SS:Webs} and in 
\S\ref{SSS:Hyperlog-AR} whose notations we use here.  One assumes that this web  carries many  
hyperlogarithmic abelian relations, 
 so that its spaces ${\bf HLogAR}^w$ of symbolic hyperlogarithmic ARs  are not all trivial. 
From the preceding subsections, it comes that the following facts hold true (for any $y\in Y$): 
\begin{enumerate}
\item[$1.$]  each weighted piece ${\bf HLogAR}^w$ is naturally a $\mathfrak S_w$-module;
\sk 
\item[$2.$] the filtered space ${\bf HLogAR}^{\leq \bullet}$ naturally is a (unipotent) $\pi_1(Y,y)$-module;
\sk 
\item[$3.$]  there is a well-defined residue map 
${\rm Res}_D: {\bf HLogAR}^w\longrightarrow {\bf HLogAR}^{w-1}$
for any irreducible component $D\subset Z$.
\end{enumerate}
The case of a  
del Pezzo web $\boldsymbol{\mathcal W}_{{\rm dP}_d}$ (with $d=9-r$) is particularly interesting since  in addition to the above, the following holds true as well 
\begin{enumerate}
\item[$4.$]    the Weyl group $W(E_r)$ also acts on each space ${\bf HLogAR}^w$. 
\end{enumerate}

In \cite{CP}, we have proved that any del Pezzo web $\boldsymbol{\mathcal W}_{{\rm dP}_d}$ for $d=1,\ldots,6$ carries a hyperlogarithmic AR ${\bf HLog}^{7-d}$ of weight $7-d$. 
Considering the above four points, it follows that this web is likely to carry many other hyperlogarithmic ARs, possibly of smaller weights. In other words, it is natural to expect that the spaces ${\bf HLogAR}^w$ for $w\leq 7-d$ are 'rather big'.  Having spaces of a priori high dimension compatible with all the algebraic structures listed in the fourth points above is not at all trivial and makes of these spaces interesting objects of study, at least for us. We hope to come back to this in a future work.

\subsection{Cluster algebras and cluster webs}
In this subsection, we introduce and discuss some elements of the theory of cluster algebras and some interesting webs one can derive from some of them. 
Our treatment below is very succinct, we refer to \S\ref{SS:WdP4-as-a-cluster-web} below for more specific details 
and  to our long memoir \cite{ClusterWebs} for a more panoramic perspective.
\sk

We start with some very basic definitions of the theory of cluster algebras.  Let $n$ be a fixed positive integer. 
A {\it (labeled) seed of rank $n$} is a triple $\boldsymbol{\mathcal S}=(\boldsymbol{a},\boldsymbol{x},B)$ where $\boldsymbol{a}=(a_1,\ldots,a_n)$ is a $n$-tuple of indeterminates, 
$\boldsymbol{x}=(x_1,\ldots,x_n)$ is another one (independent of the former) and  $B = (b_{ij})_{i,j=1}^n$ is a skew-symmetrizable $n \times n$  integer matrix, called the {\it exchange matrix} of the seed.  
The associated 
{\it (labeled) $\boldsymbol{\mathcal A}$-seed}  (resp.\,{\it $\boldsymbol{\mathcal X}$-seed})  is the pair $(\boldsymbol{a},B)$ (resp.\,$(\boldsymbol{x},B)$).  The $n$-tuples $\boldsymbol{a}$ and $\boldsymbol{x}$ are called {\it `$\boldsymbol{\mathcal A}-$} (resp.\,{\bf $\boldsymbol{\mathcal X}-$}) {\it clusters'} and their elements  $a_i$ and $x_i$ 
 {\it `$\boldsymbol{\mathcal A}-$} (resp.\,{\bf $\boldsymbol{\mathcal X}-$}) {\it cluster variables'} respectively.  Note that all these objects are a priori associated to the considered seed. \mk 
  
 By definition, for $k\in \{1,\ldots,n\}$, the {\it (cluster) mutation} in the $k$-th direction  of the seed $\boldsymbol{\mathcal S}$ is the new seed 
 $$\boldsymbol{\mathcal S}'=\big(\,\boldsymbol{a}'\,  , \, \boldsymbol{x}'\,  , \, B'\,\big)=\Big(\, \mu_k\big(\boldsymbol{a},B\big)\,  , \, \mu_k\big(\boldsymbol{x},B\big)\,  , \,\mu_k(B)\, \Big)=\mu_k\big(\boldsymbol{\mathcal S}\big)
 $$ 
  with $B'=(b_{ij}')_{i,j=1}^n=\mu_k(B)$ being the {\it matrix mutation of $B$ in direction $k$'} 
whose coefficients $b_{ij}'$ are given by the following formulas: 
$$
b_{ij}'=\begin{cases}
\, - b_{ij} \hspace{1.95cm} \mbox{ if }\,  \,  k \in \{i,j\} ;  \\
\, \hspace{0.2cm} b_{ij}\hspace{2cm} \mbox{ if } \,  k\not \in \{i,j\}  \, \mbox{  and   } \,   b_{ik}b_{kj}\leq 0;\\
\, \hspace{0.2cm} b_{ij}+\lvert b_{ik}\lvert \, b_{kj}\hspace{0.5cm} \mbox{ if } \,  k\not \in \{i,j\} \, \mbox{ and  }\,   b_{ik}b_{kj}> 0, 
\end{cases}
$$
and where the new clusters $\boldsymbol{a}'=(a_1',\ldots,a_n')=\mu_k(\boldsymbol{a},B)$ and $\boldsymbol{x}'=(x_1',\ldots,x_n')=\mu_k(\boldsymbol{x},B)$ are defined by the following formulas for the corresponding cluster variables $a_j'$s and $x_j'$'s:
\begin{align}
\label{Eq:A-X-Mutation-formulae}
a_{j}'= & \, a_j  \hspace{-0.5cm} &\mbox{if }\, j\neq k   \qquad \quad    &\mbox{ and }\quad \qquad
a_k'=  \frac{1}{a_k}\left( \,   \prod_{
b_{\ell k}>0} a_\ell^{b_{\ell k}}+
 \prod_{
   b_{lk }<0} a_l^{-b_{l k  }}
\, \right) \hspace{0.6cm}  \mbox{if }\, j= k \sk \\
x_{j}'= & \, {x_j}^{-1}  \hspace{-0.5cm} &\mbox{ if }\, j= k   \qquad \quad    &\mbox{ and }\quad \qquad
x_j'=  x_j\,\left(1+x_k^{\max\big(0,- b_{kj}\big)}\right)^{-b_{kj}}
\hspace{0.59cm}  \quad  \mbox{if }\, j\neq  k\, .\nonumber
\end{align}
Taking  $\boldsymbol{S}=\big( \boldsymbol{a}, \boldsymbol{x}, B\big)$ as initial seed,  by considering 
all the seeds obtained by successive mutations from it, one can construct the associated `{\it cluster exchange pattern}', which is the countable family of seeds   $\boldsymbol{S}_t=\big( \boldsymbol{a}^t, \boldsymbol{x}^t, B^t\big)$ indexed by the vertices $t$ of the $n$-regular tree $\mathbb T^n$, with the initial seed 
$\boldsymbol{S}=\boldsymbol{S}_{t_0}$ associated to the
root $t_0$ of this tree, the edges of  $\mathbb T^n$ being labeled by elements of $\{1,\ldots,n\}$ in the natural consistent way:  two vertices are  linked as follows 
$ t
 \hspace{-0.1cm}
\begin{tabular}{c} 
$ \stackrel{k}{\line(1,0){25}}$ \vspace{-0.35cm}\\  ${}^{}$
 \end{tabular} 
\hspace{-0.2cm}
 t'$ in $\mathbb T^n$ if and only if the two corresponding seeds satisfy 
$\boldsymbol{S}_{t'}=\mu_k\big(\boldsymbol{S}_t\big) $ (or equivalently $\boldsymbol{S}_{t}=\mu_k\big(\boldsymbol{S}_{t'}\big) $)

For $\boldsymbol{\mathcal Z}$ standing for $\boldsymbol{\mathcal A}$ or $\boldsymbol{\mathcal X}$ 
and accordingly $\boldsymbol{z}=(z_i)_{i=1}^s$ standing for $\boldsymbol{a}$ or $\boldsymbol{x}$, the {\it `$\boldsymbol{\mathcal Z}$-cluster algebra $\boldsymbol{Z}_{\boldsymbol{\mathcal S}}$'} with initial seed $\boldsymbol{\mathcal S}$ is the subalgebra of rational functions in the initial cluster coordinates $z_i$'s
spanned by all the $\boldsymbol{\mathcal Z}$-cluster variables of  seeds obtained from $\boldsymbol{\mathcal S}$ by a finite sequence of successive mutations : one has 
$
\boldsymbol{Z}_{\boldsymbol{\mathcal S}}=\big\langle
\, z_i^t\, \big\lvert 
\begin{tabular}{l}
$t\in \mathbb T^n$, 
$i=1,\ldots,n$\end{tabular}
\big\rangle \subset \mathbf Q(\boldsymbol{z})= \mathbf Q(z_1,\ldots,z_n)
\, .
$

To each vertex $t$ on $\mathbb T^n$, 
 one sets $\boldsymbol{{\mathcal Z}{\bf T}}^t={\rm Spec}\big( 
 \mathbf Q[z_1^t,\ldots,z_n^t, (z_1^t)^{-1},\ldots,1/z_n^t]\big)$ for $\boldsymbol{\mathcal Z}\in \{ \boldsymbol{\mathcal A}\, , \, \boldsymbol{\mathcal X}\}$ and one considers the monomial map 
 $p^t: \boldsymbol{{\mathcal A}{\bf T}}^t\rightarrow \boldsymbol{{\mathcal X}{\bf T}}^t$ 
 characterized by the relations $(p^t)^*(x_i^t)= \prod_{j=1}^n (a_i^t)^{b_{ij}^t}$ for $i=1,\ldots,n$. 
The cluster tori $\boldsymbol{{\mathcal Z}{\bf T}}^t$ can be glued together by means of the corresponding (sequences of) mutations t
and give rise to the so-called  {\it $\boldsymbol{\mathcal Z}$-cluster variety}: 
\begin{equation}
 \label{Eq:Cluster-variety}
\boldsymbol{\mathcal Z}_{ \boldsymbol{\mathcal S}} =\Bigg(\,  \bigsqcup_{t\in \mathbb T^n}
\boldsymbol{\mathcal Z} \mathbf T^t\, \Bigg)\, 
 /_{\boldsymbol{\mathcal Z}{-mut}}\, 
 \end{equation}
which has the structure of scheme (not necessarily of finite type nor separated\footnote{The 
$\boldsymbol{\mathcal A}$-cluster variety is always separated but  in general not the associated $\boldsymbol{\mathcal X}$-cluster variety.}). 

The monomial maps $p^t$ commute with mutations hence can be seen as the restrictions of a global well-defined cluster map $p: \boldsymbol{\mathcal A}_{ \boldsymbol{\mathcal S}}\longrightarrow 
\boldsymbol{\mathcal X}_{ \boldsymbol{\mathcal S}}$  between the two corresponding cluster varieties. 
The triple $( \boldsymbol{\mathcal A}_{ \boldsymbol{\mathcal S}}, 
\boldsymbol{\mathcal X}_{ \boldsymbol{\mathcal S}},p)$ is the {\it `cluster ensemble'} associated to the initial seed $\boldsymbol{S}$. To simplify the writing and since it still makes sense (despite not being fully rigorous), one often replaces  the subscript $\boldsymbol{S}$ by $B$ in the notations, and even sometimes just removes it.

Actually, there are more structures associated to $( \boldsymbol{\mathcal A}_{B},  \boldsymbol{\mathcal X}_{B}, p)$ since, as explained in \cite[\S2]{FG}: 
\begin{enumerate}
\item[$-$] 
the map $p: \boldsymbol{\mathcal A}_{B}\rightarrow \boldsymbol{\mathcal X}_{D_4}$ corresponds to the quotient map under the $H_{\boldsymbol{\mathcal A}}$-action
of a certain algebraic torus $H_{\boldsymbol{\mathcal A}}$ acting on the 
$\boldsymbol{\mathcal A}$-cluster variety;
\sk 
\item[$-$] there is an `exact sequence' of cluster varieties/maps
\begin{equation}
\label{Eq:ClusterSequenceP}
 \boldsymbol{\mathcal A}_{B}
\stackrel{p}{\longrightarrow} \boldsymbol{\mathcal X}_{B}\stackrel{\lambda}{\longrightarrow}
H_{\boldsymbol{\mathcal X}}\rightarrow 1
\end{equation}
where  $H_{\boldsymbol{\mathcal X}}$ is a torus and  $\lambda$ is a map with a monomial expression in each $\boldsymbol{\mathcal X}$-cluster torus $\boldsymbol{{\mathcal X}{\bf T}}^t$. 
That \eqref{Eq:ClusterSequenceP} be `exact' means that one has 
$\boldsymbol{\mathcal U}_B={\rm Im}(p)={\lambda}^{-1}(1)$ as subvarieties of 
$ \boldsymbol{\mathcal X}_{B}$. The subvariety $\boldsymbol{\mathcal U}_B$ is called 
the {\it `secondary cluster variety'} by some authors and has been proved to be relevant regarding web geometry and the theory of functional of polylogarithms (eg. see  \cite[Theorem 0.7]{ClusterWebs}).
\end{enumerate}

We will consider the cluster variables as rational functions on the initial cluster tori. They 
enjoy several remarkable properties (separation formula, sign-coherence, positivity, Laurent phenomenon for the $\boldsymbol{\mathcal A}$-cluster variables) that we will not review here.  
 Given a finite set $\Sigma$ of cluster variables, one can consider the {\it `cluster web'}  
${\boldsymbol{\mathcal W}}_\Sigma$ formed by the foliations admitting the elements of $\Sigma$ as first  integrals.   A nice feature of the theory of cluster algebras is that it comes with several ways to get finite sets $\Sigma$ of cluster variables giving rise to webs carrying polylogarithmic ARs. 
The first way is given by the so called cluster algebras  of `{\it finite type}', which are by definition those admitting only a finite number of clusters. An early fundamental result of Fomin and Zelevinsky shows that the classification of such algebras is parallel to that of Dynkin diagrams. Thus given a Dynkin diagram $\Delta$ of rank $n\geq 2$, the set $\Sigma_\Delta$ of all the $\boldsymbol{\mathcal X}$-cluster variables of the corresponding cluster algebra is finite, which allows us to define the `{\it cluster web of Dynkin type $\Delta$}' as the web
$$
{\boldsymbol{\mathcal X\hspace{-0.05cm} \mathcal W}}_{\hspace{-0.05cm} \Delta}=
{\boldsymbol{\mathcal W}}_{\Sigma_\Delta}\, . 
$$
It is a (generalized) web of codimension 1 in $n$-variables.  

 \sk

More elaborated constructions of cluster webs can be considered, such as the {\it `secondary cluster web'} of Dynkin type $\Delta$, which by definition is the trace of ${\boldsymbol{\mathcal X\hspace{-0.05cm} \mathcal W}}_{\hspace{-0.05cm} \Delta}$ along the secondary cluster manifold $\boldsymbol{\mathcal U}\subset \boldsymbol{\mathcal X}$: 
$$
{\boldsymbol{\mathcal U\hspace{-0.05cm} \mathcal W}}_{\hspace{-0.05cm} \Delta}=
\Big( {\boldsymbol{\mathcal X\hspace{-0.05cm} \mathcal W}}_{\hspace{-0.05cm} \Delta}\Big)\big\lvert_{ 
\boldsymbol{\mathcal U}
}\, .
$$
Note that ${\boldsymbol{\mathcal U\hspace{-0.05cm} \mathcal W}}_{\hspace{-0.05cm} \Delta}$ is distinct from  ${\boldsymbol{\mathcal X\hspace{-0.05cm} \mathcal W}}_{\hspace{-0.05cm} \Delta}$ if and only 
if $\boldsymbol{\mathcal U}$ is a proper subvariety of $\boldsymbol{\mathcal X}$, that is if and only if 
the initial exchange matrix $B$ has rank strictly less that $n$, in which case one has 
$\dim\big( \boldsymbol{\mathcal U}\big)={\rm rk}(B)$.  
The relevance of the notion of cluster web is illustrated by several webs 
associated to classical functional identities satisfied by some polylogarithms (such as 
Spence-Kummer identity of the trilogarithm, or Kummer's one of the tetralogarithm)  which can be proved to be webs of cluster type (see \cite[\S5.2]{ClusterWebs}). 
 But for the purpose of this paper, we need to go further and define a more general notion of cluster web. \sk

Let $\Sigma$ be a finite set of $ \boldsymbol{\mathcal X}$-cluster variables defining a cluster web 
${\boldsymbol{\mathcal X\hspace{-0.05cm} \mathcal W}}_{\hspace{-0.05cm} \Sigma}$.  Assume that the map $p$ in \eqref{Eq:ClusterSequenceP} has rank $m<n$, or equivalently $H_{ \boldsymbol{\mathcal X}}$ has positive dimension $n-m$. Since $\lambda$ is dominant, it  has rank $n-m$ hence for any $\tau\in 
H_{ \boldsymbol{\mathcal X}}$, the fiber $ \boldsymbol{\mathcal X}_{B,\tau}=\lambda^{-1}(\tau)$ is 
a $m$-dimensional subvariety of  $ \boldsymbol{\mathcal X}_B$. 
One defines a more general kind of cluster web by considering the trace of  ${\boldsymbol{\mathcal X\hspace{-0.05cm} \mathcal W}}_{\hspace{-0.05cm} \Sigma}$ along $ \boldsymbol{\mathcal X}_{B,\tau}$: 
one sets
$$
{\boldsymbol{\mathcal X\hspace{-0.05cm} \mathcal W}}_{\hspace{-0.05cm} \Sigma,\tau }
= \Big( 
{\boldsymbol{\mathcal X\hspace{-0.05cm} \mathcal W}}_{\hspace{-0.05cm} \Sigma}
\Big)\big\lvert _{\lambda=\tau}\, . 
$$

In  \S\ref{SS:WdP4-as-a-cluster-web}, we will prove that any del Pezzo's web $\boldsymbol{\mathcal W}_{\hspace{-0.05cm} {\rm dP}_4 }$ is a cluster web of the form just above.  
\begin{center}
$\star$
\end{center} 

In the next two sections, we  make as explicit as possible the general material of the current section for the two del Pezzo webs 
$\boldsymbol{\mathcal W}_{{\bf dP}_5}$ and $\boldsymbol{\mathcal W}_{{\bf dP}_4}$. 

\section{\bf Bol's web}
\label{S:Bol-web}
There is no moduli for quintic del Pezzo surface: for ${\rm dP}_5$, we take  the total space 
of the blow-up $b: {\rm dP}_5 \longrightarrow \mathbf P^2$ at the four points 
$p_1=[1:0:0]$, $p_2=[0:1:0]$, $p_3=[0:0:1]$ and $p_4=[1:1:1]$. Instead of dealing with 
$\boldsymbol{\mathcal W}_{{\bf dP}_5}$ on ${{\bf dP}_5}$, we are going to work on 
$\mathbf P^2$ 
with  the push-forward web $b_*\big( \boldsymbol{\mathcal W}_{{\bf dP}_5}\big)$, which is entirely equivalent (cf. Proposition \ref{P:b*WdPr-on-P2}). This push-forward is classically known as Bol's web, denoted by $\boldsymbol{\mathcal B}$ which,  in the coodinates $x,y$ corresponding 
to the affine embedding $\mathbf C^2\rightarrow \mathbf P^2,\, (x,y)\mapsto [x:y:1]$, is 
 the web given by the arguments of  Rogers' dilogarithm in $\boldsymbol{\big(\mathcal Ab\big)}$: 
 $$
  \boldsymbol{\mathcal B}= 
 \boldsymbol{\mathcal W}\bigg( \, x\, , \, 
 y\, , \, 
 \frac{x}{y}\, , \, 
  \frac{y-1}{x-1}\, , \, 
   \frac{x(y-1)}{y(x-1)} \, 
   \bigg)\, . 
 $$
 
Bol's web is regular on $U=b({\rm dP}_5\setminus L_5)\subset \mathbf P^2$ which coincides with the complement in $\mathbf  C^2$ of the affine arrangement of 
  five lines  $\mathscr A\subset \mathbf C^2$ cut out by 
 $$
 xy(x-1)(y-1)(x-y)=0\, .
$$ 
The map  $(x,y)\longmapsto [0,1,x,y,\infty]$
where the latter expression stands for the projective equivalence class of the 
configuration  of 5 points $(0,1,x,y,\infty)$ on the projective line induces an affine 
biholomorphism  
$\varphi : U\stackrel{\sim}{\rightarrow} \mathcal M_{0,5}$ which extends to an isomorphism 
${\rm dP}_5\simeq  \overline{\mathcal M}_{0,5}$.  The web  
$\varphi_*(\boldsymbol{\mathcal B})$ on 
 $\mathcal M_{0,5}$ is the 5-web 
 denoted by $\boldsymbol{\mathcal W}_{ \mathcal M_{0,5} }$ 
 whose first integrals are 
 the five forgetful maps 
 $ \mathcal M_{0,5}\rightarrow  \mathcal M_{0,4}$. 
 \label{Modular-W-M05}
\begin{center}
\vspace{-0.2cm}
$\star$
\end{center} 
 
 Our goal below is to establish all the properties of this web stated in \S\ref{S:Web-WdP5-properties}.  Some of them are classical and well known, we will be quite succinct about them.  
 
 \sk 
 
 \subsection{Non linearizability}
For completeness, we recall the standard argument for proving the non linearizability of Bol's web (see \cite[\S6.1]{PP} for more details): 
in the coordinates $x,y$ we are working with, Bol's web is seen to be formed by  four pencils of lines  plus a pencil of conics. A classical result says that any (local) linearization of a linear planar 4-web is necessarily  induced by a projective transformation of $\mathbf P^2$. Since such a map cannot send any smooth conic onto a line, there is no chance that it could linearize the pencil of conics of $\boldsymbol{\mathcal B}$ which is therefore non linearizable  (and consequently non algebraizable). 

\subsection{Cluster web of type ${A\hspace{-0.225cm}A}_2$}
\label{SS:Clsuter-Web-type-A2}
To deal with the ARs of $\boldsymbol{\mathcal B}$, it is convenient to 
perform a change of coordinates  in order to get its cluster model. As an unordered 5-web, Bol's web     is obviously equivalent to 
     $$
  \boldsymbol{\mathcal W}\bigg(
     \, 
     \frac{1}{x-1}
     \, , \, 
   \frac{y-x}{x - 1}
      \, , \, 
    y-1
       \, , \, 
       \frac{y(x-1)}{y-x}
         \, , \,  
    \frac{x}{y-x}
      \,\bigg)
    $$
      and using the birational change of variables 
    $  \Psi:  (u_1,u_2)\mapsto (x,y)$  
      defined by   
            $$x = \frac{1 + u_1}{u_1} \qquad \mbox{ and } \qquad  y = \frac{1 + u_1 + u_2}{u_1}$$
       ({\it cf.}\,the `zig-zag map' of \cite[\S6.3.2]{ClusterWebs}),  
          one obtains that 
           $\boldsymbol{\mathcal B}$ is equivalent  to the web 
          defined by the following five rational first integrals 
         $ u_1$, $ u_2$, 
$(1+u_2)/u_1$, $(1+u_1+u_2)/(u_1u_2)$ and 
$ (1+u_1)/u_2$. 
This web is  the $ \mathcal X$-cluster web of type $A_2$  (see 
 \cite[\S3.2.1]{ClusterWebs}, especially Figure 10 there):  
  one has 
\begin{equation}
\label{Eq:Psi*(B)=XW-A2}
 \Psi^*\big(\boldsymbol{\mathcal B}\big)= 
   \boldsymbol{\mathcal W}\bigg( \, u_1\, , \,   u_2\, , \,
 \frac{1+u_2}{u_1}\, , \, 
  \frac{1+u_1+u_2}{u_1u_2}\, , \,
   \frac{1+u_1}{u_2} \, 
   \bigg)= \boldsymbol{
 \mathcal X \hspace{-0.03cm}
  \mathcal W}_{\rm A_2}\, . 
\end{equation}
 
\subsection{The abelian relations of 
$ {\mathcal X} \hspace{-0.27cm}{\mathcal X}
\hspace{-0.03cm}
{\mathcal W} \hspace{-0.46cm}{\mathcal W}_{{A\hspace{-0.2cm}A}_2}$}
\label{SS:ARs-WdP5}
 Dealing with the cluster model  of Bol's web above is quite convenient to write down its ARs in a nice form. The presentation below is essentially taken from  
 \cite[\S5.1.1]{ClusterWebs}. 

\subsubsection{}
Denoting by  $X_\ell$ for $\ell=1,\ldots,5$ the first integrals of 
$\boldsymbol{
 \mathcal X \hspace{-0.03cm}
  \mathcal W}_{\rm A_2}$ given in \eqref{Eq:Psi*(B)=XW-A2}, that is 
\begin{equation}
\label{Eq:lola}
  X_1=u_1\, , \qquad X_2=u_2\, , \qquad 
X_3= \frac{1+u_2}{u_1}\, , \qquad
X_4=  \frac{1+u_1+u_2}{u_1u_2}\, , \qquad
X_5 =   \frac{1+u_1}{u_2}   
  \end{equation}
   and setting $X_0=X_5=(1+u_1)/u_2$ as well as 
  $X_6=X_1=u_1$, one verifies that the identity  
  $$ 
  \frac{X_{\ell-1}X_{\ell+1}}{1+X_\ell}=1
  $$
  is satisfied for any $\ell=1,\ldots,5$ 
 (these are the identities defining the {\it `$Y$-system of type $A_2$'}, see \cite[\S3.3.1.3]{ClusterWebs}). For any such $\ell$, taking the logarithm of the previous identity gives
 \begin{equation}
 \label{Eq:AR-log-X-ell}
 {\rm Log}\big(\, X_{\ell-1}\,\big) 
  -{\rm Log}\big(\, 1+X_{\ell}\,\big) 
 +{\rm Log}\big(\, X_{\ell+1}\,\big) =0\, , 
\end{equation}
  a functional identity which corresponds to a logarithmic and combinatorial AR for 
  $\boldsymbol{
 \mathcal X \hspace{-0.03cm}
  \mathcal W}_{\rm A_2}$, that we will denote by $ARLog_\ell$. 
 The $\ell$-th component of  $ARLog_\ell$ is $ -{\rm Log}\big(\, 1+X_{\ell}\,\big) $ for each $\ell$. Since the $ {\rm Log}\big(\, 1+X_{\ell}\,\big) $'s (for $\ell=1,\ldots,5$) are linearly independent,  it follows that the $ARLog_\ell$'s form a free family of elements of 
 $\boldsymbol{AR}_{log}\Big( 
 \boldsymbol{
 \mathcal X \hspace{-0.03cm}
  \mathcal W}_{\rm A_2}
 \Big)$ which actually is a basis of this space. One has 
 $$
 \boldsymbol{AR}_{log}\Big( 
 \boldsymbol{
 \mathcal X \hspace{-0.03cm}
  \mathcal W}_{\rm A_2}
 \Big)=
 \boldsymbol{AR}_{C}\Big( 
 \boldsymbol{
 \mathcal X \hspace{-0.03cm}
  \mathcal W}_{\rm A_2}
 \Big)=\Big\langle
\,  ARLog_1
\,,\,\ldots\, ,\, 
ARLog_5\, 
 \Big\rangle\simeq \mathbf C^5\,.
 $$

 The {\it cluster Rogers' dilogarithm ${\mathrm R}$} is the function defined 
 by the integral formula 
$$
 {}^{}\hspace{0.4cm}
{\mathrm R}(x)=
\frac{1}{2} \bigintsss_{0}^x \left( \frac{{\rm Log}(1+v)}{v}-
 \frac{{\rm Log}(v)}{1+v}\right)\,  dv\,
$$ 
for any $x\geq 0$.  It is well known that it satisfies the 
functional identity 
\begin{equation}
\label{Eq:R-A2}
{\mathrm R}(u_1) + {\mathrm R}(u_2) + {\mathrm R}\left(\frac{1 + u_2}{u_1}\right) + 
{\mathrm R}\left(\frac{1 + u_1+u_2}{u_1u_2}\right) +{\mathrm R}\left(\frac{1 + u_1}{u_2}\right)=
\frac{{}^{}\hspace{0.1cm}\pi^2}{2}
\end{equation}
 for any $u_1,u_2>0$, which is nothing else but Abel's 5-term identity written in cluster form. 
 One denotes  by $\boldsymbol{ {\bf Ab}}$ the weight 2  polylogarithmic `cluster abelian 
  relation' associated to \eqref{Eq:R-A2}. 
 
 One thus has 
   \begin{equation}
  \label{Eq:AR-XWA2}
 \boldsymbol{AR}\Big( \boldsymbol{\mathcal X\mathcal W}_{\hspace{-0.07cm}A_2} \Big)= 
 \Big\langle
\,  ARLog_1\,,\,\ldots\, ,\, ARLog_5\, 
 \Big\rangle \oplus 
 \big\langle \, {\bf Ab}  \, \big\rangle\,   \, , 
 \end{equation}
 from which one recovers the well-known fact (due to Bol) that 
  $
 \boldsymbol{
 \mathcal X \hspace{-0.03cm}
  \mathcal W}_{\rm A_2}\simeq  \boldsymbol{
 \mathcal B}$ has maximal rank, with all its ARs being polylogarithmic of weight 1 or 2.

Another nice feature of choosing the cluster variables $X_\ell$ as first integrals for the $\boldsymbol{\mathcal X}$-cluster web of type $A_2$ 
 is that one can easily construct 
the dilogarithmic identity \eqref{Eq:R-A2} from the five logarithmic abelian relations $LogAR_{\ell}$'s. Indeed, considering the differential of 
$LogAR_{\ell}$, it follows that 
$$
\frac{dX_\ell}{1+X_\ell}=\frac{dX_{\ell-1}}{X_{\ell-1}}+\frac{dX_{\ell+1}}{X_{\ell+1}}
$$
for any $\ell$. Then, summing for $\ell $ in $ \mathbf Z/5\mathbf Z$, one gets 

\begin{align}
\label{Eq:RA2-from-LogARl}
-\sum_{\ell}  LogAR_{\ell}\,  \frac{dX_\ell}{X_\ell}= & \, 
\sum_{\ell }
\bigg( -{\rm Log}\big(X_{\ell-1}\big)+{\rm Log}\big(1+X_\ell\big)-{\rm Log}\big(X_{\ell+1}\big) \bigg)
\,
  \frac{dX_\ell}{X_\ell} \nonumber \\ 
  = & \,  
 \sum_{\ell}
 {\rm Log}\big(1+X_\ell\big)\,
  \frac{dX_\ell}{X_\ell} 
  -
  \sum_{\ell}
  {\rm Log}\big(X_{\ell}\big)\, \bigg( \frac{dX_{\ell-1}}{X_{\ell-1}}+\frac{dX_{\ell+1}}{X_{\ell+1}} 
  \bigg) 
 \nonumber   \\
   = & \,  \sum_{\ell}
\bigg( \frac{{\rm Log}\big(1+X_{\ell}\big)}{X_\ell}
- \frac{{\rm Log}\big(X_{\ell}\big)}{1+X_\ell}
\bigg) \, 
{dX_\ell}\\
  = & \, \sum_{\ell} 2 \,{\rm R}'\big(X_\ell\big) \,dX_\ell=
  2 \,d\bigg(   \sum_{\ell} {\rm R}\Big(X_\ell \Big) 
  \bigg)\,. \nonumber
\end{align}

Since all the $LogAR_{\ell}$'s vanish identically, the same holds true for 
 the total derivative of $ \sum_{\ell=1}^5 {\rm R}(u_\ell)$ hence this sum is  identically equal to a constant. We thus have recovered very symmetrically the $A_2$-cluster dilogarithmic identity \eqref{Eq:R-A2} from the logarithmic identities $LogAR_{\ell}$'s.

Using \eqref{Eq:lola} and  the identity  $
 2 \,d\big(   \sum_{\ell} {\rm R}(X_\ell) 
  \big)=
-\sum_{\ell}  LogAR_{\ell}\cdot {dX_\ell}/{X_\ell}$, one can compute easily the residues of the abelian relation ${\bf Ab}$ associated to   \eqref{Eq:R-A2}. 
Given $P$ one of the $F$-polynomials $u_1,u_2,1+u_1,1+u_2$ and $1+u_1+u_2$, one denotes by ${\rm Res}_P( {\bf Ab})$ the residue of ${\bf Ab}$ along the divisor cut out by $P=0$ in $\mathbf C^2$. Then easy computations show that, up to non zero multiplicative  constants,  one has: 
\begin{align*}
{\rm Res}_{u_1}\big(  {\bf Ab}\big)= &\, 
LogAR_{1}-LogAR_{3}-LogAR_{4}
 \\
 {\rm Res}_{u_2}\big(  {\bf Ab}\big)= &\, 
 LogAR_{2}-LogAR_{4}-LogAR_{5}
 \\
 {\rm Res}_{1+u_1}\big(  {\bf Ab}\big)= &\, 
 LogAR_{5}
 \\
 {\rm Res}_{1+u_2}\big(  {\bf Ab}\big)= &\, 
  LogAR_{3}
 \\
 \mbox{and }\quad 
 {\rm Res}_{1+u_1+u_2}\big(  {\bf Ab}\big)= &\, 
  LogAR_{4}\, .
\end{align*}
Clearly, the five residue abelian relations above  form another basis of the space 
$ \boldsymbol{AR}_{log}\Big( 
 \boldsymbol{
 \mathcal X \hspace{-0.03cm}
  \mathcal W}_{\rm A_2}
 \Big)$. 

\subsubsection{The symbolic abelian relations of 
$ {\mathcal X} \hspace{-0.27cm}{\mathcal X}
\hspace{-0.03cm}
{\mathcal W} \hspace{-0.46cm}{\mathcal W}_{{A\hspace{-0.2cm}A}_2}$}

It is interesting to discuss succinctly the structure of the space of symbolic ARs 
 of $\boldsymbol{
 \mathcal X \hspace{-0.03cm}
  \mathcal W}_{\rm A_2}$ (if only to prepare the reader for what is to come, since this is the approach we will use to describe the ARs of $\boldsymbol{\mathcal W}_{{\rm dP}_4}$ in the next section).

We set 
$$
F_1=u_1\, ,  \quad 
F_2=u_2\, ,  \quad
F_3=1+u_1\, ,  \quad 
F_4=1+u_2\qquad \mbox{  and  } \qquad F_5=1+u_1+u_2\, .
$$
Then the rational differential 1-forms 
\begin{equation}
\label{kappa-l}
\kappa_\ell=d{\rm Log}(F_\ell)=dF_\ell/F_\ell
\end{equation} 
for $ \ell=1,\ldots,5$ 
are such that their pull-backs under the blow-up map $b: X_4={\rm dP}_5\rightarrow \mathbf P^2$ form a basis of the space ${\bf H}={\bf H}_{X_4}$ of rational 1-forms on ${\rm dP}_5$   with logarithmic poles along the divisor of lines $L_4\subset X_4$  and holomorphic on $U_4=X_4\setminus L_4$. \sk

One sets $\Sigma=\{0,1,\infty\}\subset \mathbf P^1$. The logarithmic 1-forms $d{\rm Log}(z)=dz/z$ and $d{\rm Log}(z+1)=dz/(z+1)$ form a basis of 
${\bf H}^0\big( \mathbf P^1,\Omega^1_{ {\mathbf P}^1}(\,{\rm Log}(\Sigma)\big)\big)$.  
Hence for any $\ell=1,\ldots,5$, 
$$
\nu_{\ell,1}=X_\ell^*\left(\frac{dz}{z}\right)=\frac{dX_\ell}{X_\ell} \qquad \mbox{ and }\qquad \nu_{\ell,2}=X_\ell^*\left(\frac{dz}{z+1}\right)=
\frac{dX_\ell}{X_\ell+1} 
$$
form a basis of ${\bf H}_\ell=X_\ell^*\big( {\bf H}^0\big( \mathbf P^1,\Omega^1_{ {\mathbf P}^1}(\,{\rm Log}(\Sigma)\big)\big) \big)$  which is a subspace of the 5-dimensional vector space ${\bf H}$. Hence any $\nu_{\ell,s}$ ($s=1,2$) admits a unique expression as a linear combination in the $\kappa_i$'s. These decompositions are the following (with $\nu_\ell=\big(\nu_{\ell,1}, \nu_{\ell,2}\big)$ for any $\ell$): 
\begin{align}
\label{Eq:nu-i}
\nu_1= & \, \Big(\, \kappa_1\, , \, \kappa_3\,\Big) \nonumber\\
\nu_2= & \, \Big(\, \kappa_2\, , \, \kappa_4\,\Big)\nonumber\\
\nu_3= & \, \Big(\, \kappa_4-\kappa_1\, , \, \kappa_5-\kappa_1\,\Big)\\
\nu_4= & \, \Big(\, \kappa_5-\kappa_1-\kappa_2\, , \, \kappa_3+\kappa_4
-\kappa_1-\kappa_2
\,\Big)\nonumber\\
\mbox{and} \quad
\nu_5= & \, \Big(\, \kappa_3-\kappa_2\, , \, \kappa_5-\kappa_2\,\Big)\nonumber
\end{align}
We set $\nu_0=\nu_5$ and $\nu_6=\nu_1$. 
For any $\ell=1,\ldots,5$, the symbolic relation corresponding to  the identity 
\eqref{Eq:AR-log-X-ell} is 
\begin{equation}
\label{Eq:AR-log-X-ell-nu}
\nu_{\ell-1,1}-\nu_{\ell,2}+\nu_{\ell+1,1}=0
\end{equation}
and the weight 2 antisymmetric symbolic relation corresponding to the functional identity 
\eqref{Eq:R-A2} is 
\begin{equation}
\label{Eq:moulot}
\sum_{\ell=1}^5 \nu_{\ell,1} \wedge \nu_{\ell,2}=0\, , 
\end{equation}
an identity in $\wedge^2 {\bf H}$ which can be shown to be  satisfied by an easy direct computation.

\subsection{Birational symmetries and the Weyl group action}
\label{SS:Bir-Symmetries-W-action}
The two following maps  
$$
\Phi: \, (u_1,u_2)  \longmapsto \bigg( \,\frac{1+u_2}{u_1}\, , \,  \frac{1+u_1+u_2}{u_1u_2}\, \bigg) \qquad 
\mbox{and}\qquad 
\Psi: \, (u_1,u_2)  \longmapsto \bigg( \,\frac{1+u_1}{u_2}\, , \,  \frac{1+u_1+u_2}{u_1u_2}\, \bigg) \,, 
$$
  are birational symmetries of $\boldsymbol{\mathcal X\mathcal W}_{\!A_2}$ (which can be qualified as `cluster symmetries' since their components are cluster coordinates). 
 The former map has order 5 whereas the second is an involution. Their pull-backs  on 
 a certain affine chart $\mathbf C^2$ of 
 $\mathcal M_{0,5}$ under the map  $(x,y)\mapsto [\infty,0,1,x,y]\mapsto (u_1,u_2) = \big(\, 1/(x - 1),  (x -y)/(1-x)\,\big)$ ({\it cf.}\,the map ${\bf U}_2$ in \cite[\S6.3.2]{ClusterWebs}) correspond 
 on $\mathcal M_{0,5}$ 
to the cyclic shift $\phi : [p_1,\ldots,p_5]\longmapsto [p_5,p_1,\ldots, p_4]$ and  to the dihedral reflection 
$\psi : [p_1,\ldots,p_5]\longmapsto [p_2,p_1,p_5,p_4,p_3]$  respectively. It follows that $\Phi$ and $\Psi$ generate a group of cluster symmetries of 
$\boldsymbol{\mathcal X\mathcal W}_{\!A_2}$  isomorphic to the dihedral group of the pentagon (note that, instead of $\Psi$, one could have considered the simpler involution $\Gamma: (u_1,u_2)\mapsto (u_2,u_1)$ since one has $\langle \, \Phi\, , \,\Psi\,\rangle=\langle \,\Phi \, , \, \Gamma\,\rangle$   as groups of Cremona transformations).   \sk

 The involution 
 exchanging the fourth and the fifth points of any configuration in $\mathcal M_{0,5}$ is written  $(x,y)\mapsto (y,x)$ in the affine coordinates $x,y$ and 
as follows  in the coordinates $u_1,u_2$: 
 $$ J: (u_1,u_2)\longmapsto  \left( \, \frac{u_1}{1+u_2}\, , \, -\frac{u_2}{1+u_2}
\,  \right) \, . 
$$  

The 5-cycle $(12\cdots5)$ together with the transposition $(45)$ generate the whole permutation group $\mathfrak S_5$ which identifies naturally 
to ${\rm Aut}( \mathcal M_{0,5})$.  To this group corresponds, in the coordinates $u_1,u_2$, the one  
 generated by $\Phi$ and $J$ which is the group 
of birational symmetries of $\boldsymbol{\mathcal X\mathcal W}_{\!A_2}$:  one has natural identifications 
$$
\mathfrak S_5\simeq {\rm Aut}( \mathcal M_{0,5})\simeq \big\langle \, \Phi\, , \, J\, \big\rangle\, .
$$

It is straightforward to verify that the space ${\bf H}$ spanned by the $\kappa_\ell$'s is stable by pull-backs under the maps $\Phi$ and $J$ and that the two corresponding linear actions are characterized by the following relations: 
\begin{align}
\label{Eq:Phi-J-*-kappa}
\Phi^*(\kappa_1)=& \, \kappa_4-\kappa_1 &&  J^*(\kappa_1)=\kappa_1-\kappa_4  \nonumber \\
\Phi^*(\kappa_2)=& \, \kappa_5-\kappa_1-\kappa_2 &&  J^*(\kappa_2)=\kappa_2-\kappa_4 \nonumber\\
\Phi^*(\kappa_3)=& \, \kappa_5-\kappa_1-\kappa_2 &&  J^*(\kappa_3)=\kappa_5-\kappa_4 \\
\Phi^*(\kappa_4)=& \, \kappa_3+\kappa_4-\kappa_1-\kappa_2 &&  J^*(\kappa_4)=-\kappa_4\nonumber\\
\Phi^*(\kappa_5)=& \, \kappa_4+\kappa_5-\kappa_1-\kappa_2 &&  J^*(\kappa_5)=\kappa_3-\kappa_4\, .\nonumber
\end{align}
 
 The space ${\bf H}$ is easily seen to be an irreducible $\mathfrak S_5$-module of dimension 5. Moreover one immediately gets that the trace of the matrix of $J^*$ in the basis $(\kappa_\ell)_{\ell=1}^5$ is 1. From the character table of $\mathfrak S_5$, one gets  that as a 
 $\mathfrak S_5$-representation, ${\bf H}$ is isomorphic to the one with Young symbol $[32]$. 
 
 From the explicit formulas \eqref{Eq:nu-i}, \eqref{Eq:AR-log-X-ell-nu}, 
 and \eqref{Eq:Phi-J-*-kappa}, it is just a matter of computations in linear algebra 
to determine   the action of $\Phi^*$ and $J^*$ on the space ${\bf ARHLog}^1$ (isomorphic to $\boldsymbol{AR}_{log}(\boldsymbol{{\mathcal X}
\hspace{-0.03cm}
{\mathcal W} }_{A_2})$) which is freely spanned as a complex vector space by the formal ARs corresponding to the algebraic relations \eqref{Eq:AR-log-X-ell-nu} for $\ell=1,\ldots,5$. We obtain that, as a $\mathfrak S_5$-representation, ${\bf ARHLog}^1$ is irreducible and isomorphic to the irreducible representation with Young symbol $[221]$. 
 Finally, from the formulas \eqref{Eq:Phi-J-*-kappa} above, one also gets $J^*( \nu_{1,1}) = -J^*( \nu_{3,1})$ and 
$J^*( \nu_{1,2}) = J^*( \nu_{3,2})$ hence $J^*( \nu_{1,1}\wedge \nu_{1,2})= 
-  \nu_{3,1}\wedge \nu_{3,2}$.  It follows immediately from this that 
  the weight 2 antisymmetric AR corresponding to 
 \eqref{Eq:moulot} is a non-trivial 1-dimensional representation of $\mathfrak S_5\simeq 
 \big\langle \, \Phi\, , \, J\, \big\rangle$: we obtain that  
 $\big\langle {\bf Ab} \big\rangle$ is isomorphic to the signature representation of $\mathfrak S_5$. 
 
\begin{rem} 
The isomorphisms of $\mathfrak S_5$-representations 
$$
{\bf ARHLog}^1\simeq V^5_{[221]}\qquad \mbox{and} \qquad 
\big\langle {\bf Ab} \big\rangle \simeq {\bf sign}=V^1_{[1^5]}
$$
have been stated explicitly first by Damiano (see  \cite{Damiano} or the introduction of \cite{Pirio-W-M06}{\rm )}.
\end{rem}
 
\subsection{\bf Canonical algebraizations of Bol's web}
\label{S:Algebraization-WdP5}
For planar webs with maximal rank, being exceptional is precisely the opposite of being algebraizable in the classical terminology. Hence one has to explain the meaning of the heading of this subsection.
 
Here we take the standpoint that Bol's web is (naturally equivalent to) the web $\boldsymbol{\mathcal W}_{{\mathcal M}_{0,5}}$ which is defined by natural rational (hence algebraic) first integrals on the algebraic surface $\overline{\mathcal M}_{0,5}\simeq {\rm dP}_5$.  What we mean by the notion of a `{\it canonical algebraization of Bol's web}'  is a canonical way, given a germ of 5-web $\boldsymbol{\mathcal W}$ (at the origin of $\mathbf C^2$ say) to  first recognize whether  
 $\boldsymbol{\mathcal W}$ is equivalent to $\boldsymbol{\mathcal B}$, then, if it is indeed the case,  to construct in a canonical way a local biholomorphism $\psi_{\boldsymbol{\mathcal W}} :  \big( \mathbf C^2,0\big)\rightarrow {\mathcal M}_{0,5}$ such that the push-forward of $\boldsymbol{\mathcal W}$ by $ \psi_{\boldsymbol{\mathcal W}}$ coincides with 
 $\boldsymbol{\mathcal W}_{{\mathcal M}_{0,5}}\simeq \boldsymbol{\mathcal B}$ (possibly up to reindexing the foliations). 
 
 We describe below two constructions of such canonical maps. These two  constructions are a priori distinct but both are canonical, hence both coincide for Bol's web. The first construction is the most direct and elementary but seems to be new. The second approach is more in the classical spirit of web geometry, still  is new as well. It consists in constructing an algebraization map by means of (some of) the ARs of the web under consideration. 
\sk

In the rest of this subsection, $\boldsymbol{\mathcal W}$ stands for  a  fixed 5-web defined on an open domain $U \subset \mathbf C^2$.

 \subsubsection{Algebraization via the combinatorial abelian relations}
 \label{SS:algebraization-via-AR-C}
Let $\boldsymbol{AR}_C(\boldsymbol{\mathcal W})$ be the space of combinatorial ARs of 
$\boldsymbol{\mathcal W}$ (on $U$) which by definition is the subspace of 
$\boldsymbol{AR}(\boldsymbol{\mathcal W})$ spanned by the ARs carried by the 3-subwebs of $\boldsymbol{\mathcal W}$.  We define $\boldsymbol{ar}_C(\boldsymbol{\mathcal W})$ as the subspace of $\Omega^1(U)$ spanned by the components of all the elements of $\boldsymbol{AR}_C(\boldsymbol{\mathcal W})$. Note that $\boldsymbol{ar}_C(\boldsymbol{\mathcal W})$ is invariantly attached to $\boldsymbol{\mathcal W}$: for any biholomorphism $F: \widetilde U\rightarrow U$, one has $F^*\big( \boldsymbol{AR}_C(\boldsymbol{\mathcal W})\big)= \boldsymbol{AR}_C\big( F^*(\boldsymbol{\mathcal W})\big)$ hence
\begin{equation}
\label{Eq:FC}
F^*\Big( \boldsymbol{ar}_C(\boldsymbol{\mathcal W})\Big)= \boldsymbol{ar}_C\big( F^*(\boldsymbol{\mathcal W})\big)
 \, .
\end{equation}
  We can call the dimension of $ \boldsymbol{ar}_C(\boldsymbol{\mathcal W})$ 
the  `{\it combinatorial $\boldsymbol{a}$-rank}' of $\boldsymbol{\mathcal W}$. 
 Remark that  the dimension of $\boldsymbol{ar}_C(\boldsymbol{\mathcal W})$ is positive if and only if 
$\boldsymbol{AR}_C(\boldsymbol{\mathcal W})$ is non-trivial, which is often verified by the webs we consider in practice. In what follows, we set 
$$\mathfrak r=\dim \Big( \boldsymbol{ar}_C\big(\boldsymbol{\mathcal W}\big)\Big)\,.$$

\begin{lem}
If  $\mathfrak r$ is positive then necessarily  $\mathfrak r\geq \red{2}$ and the evaluation map ${\rm ev}_u : \boldsymbol{ar}_C(\boldsymbol{\mathcal W}) 
\longrightarrow \Omega^1_{U,u}\simeq \mathbf C^2$ is a surjective linear map of $\mathbf C$-vector spaces for any $u\in U$.
\end{lem} 
\begin{proof}
This follows immediately from the fact that the valuation of a non-trivial abelian relation with 3 terms is zero at any point. 
\end{proof}

We assume that $\mathfrak r>0$ from now on.  In this case, it follows from the preceding lemma that  
${\rm Ker}({\rm ev}_u)$ is a subspace of codimension $2$ of $\boldsymbol{ar}_C(\boldsymbol{\mathcal W})$ hence its annihilator, denoted by $An_u$, is a 2-plane in 
the dual space $\boldsymbol{ar}_C(\boldsymbol{\mathcal W})^\vee $. 
Composing the map $u   \longmapsto An_u$ from $U$ into the grassmannian 
$G_2\big( \boldsymbol{ar}_C(\boldsymbol{\mathcal W})^\vee
\big)$ with the Pl\"ucker embedding $\varrho$ 
  this grassmannian variety  into 
  the second wedge product of $ \boldsymbol{ar}_C(\boldsymbol{\mathcal W})^\vee$, we obtain a morphism 
\begin{align*}
{}^{} \qquad 
\Psi_{\boldsymbol{\mathcal W}} :  U & \longrightarrow 
 \mathbf P \Big( \wedge^2  \boldsymbol{ar}_C(\boldsymbol{\mathcal W})^\vee
\Big)\simeq \mathbf P^{ { \mathfrak r\choose 2}-1}
\\
 u & \longmapsto \wedge^2 An_u\, .
\end{align*}
We let the reader verify that $\Psi_{\boldsymbol{\mathcal W}}$ is invariantly attached to 
${\boldsymbol{\mathcal W}}$ as well in the sense that, for any biholomorphism $F$ as above, one has $$\Psi_{\boldsymbol{\mathcal W}}\circ F= \Psi_{F^*(\boldsymbol{\mathcal W})}$$ 
up to the natural identification between the target spaces induced by \eqref{Eq:FC}.

Given local coordinates $x,y$ and a basis $\eta_1,\ldots,\eta_r$ of $ \boldsymbol{ar}_C(\boldsymbol{\mathcal W})$, the map  $\Psi_{\boldsymbol{\mathcal W}}$  can be described explicitly as follows:  
for any $i=1,\ldots,r$,  let $\eta_i=\eta_i^1dx+\eta_i^2dy$ be the decomposition of 
$\eta_i$ in the basis $(dx,dy)$ and let us set $\eta^s=(\eta^s_1,\ldots,\eta_r^s)$ for $s=1,2$.  
Then up to the linear identification between $\boldsymbol{ar}_C(\boldsymbol{\mathcal W})$ and its dual induced by the choice of the basis $(\eta_i)_{i=1}^r$, one has that $\eta^1\wedge \eta^2\in \wedge^2 \mathbf C^{r}$ identifies with $\wedge^2 An_u$ from which it follows that 
$\Psi_{\boldsymbol{\mathcal W}}$ is the projectivization  of the map 
 with  components the $2\times 2$ minors of the $2\times r$ matrix whose two lines are $\eta^1$ and $\eta^2$: 
one has 
$$
\Psi_{\boldsymbol{\mathcal W}}=\Big[ \,\eta_i^1\eta_j^2-\eta_i^2\eta_j^1\, \Big]_{1\leq i<j\leq r} : 
U\longrightarrow \mathbf P^{ {r \choose 2}-1}\, . 
$$

When $\Psi_{\boldsymbol{\mathcal W}}$ is \'etale at some point, then 
$S_{\boldsymbol{\mathcal W}}={\rm Im}\big( \Psi_{\boldsymbol{\mathcal W}} \big)$ is a (a priori analytic)  surface geometrically attached to $\boldsymbol{\mathcal W}$ (up to a projective automorphism of the target projective space) hence the push forward web 
$ \big(\Psi_{\boldsymbol{\mathcal W}}\big)_*\big( \boldsymbol{\mathcal W} \big)$  is a canonical model of ${\boldsymbol{\mathcal W}}$.  

\begin{rem}
Note that when ${\boldsymbol{\mathcal W}}$ is a maximal rank 5-web, another canonical model of it has been previously considered, as the push-forward 
of ${\boldsymbol{\mathcal W}}$ by the associated Poincar\'e-Blaschke map $P_{\boldsymbol{\mathcal W}}$ ({\it cf.}\,\S5.2 and \S5.3 of \cite{PP} for more details).  When 
${\boldsymbol{\mathcal W}}$ is such that 
both push-forwards $ \big(\Psi_{\boldsymbol{\mathcal W}}\big)_*\big( \boldsymbol{\mathcal W} \big)$ and $ \big(P_{\boldsymbol{\mathcal W}}\big)_*\big( \boldsymbol{\mathcal W} \big)$
are well-defined, it would be interesting to compare these `two canonical models'.  Since the combinatorial ARs are the less transcendant ones among the ARs, one can expect the surface $S_{\boldsymbol{\mathcal W}}={\rm Im}\big( \Psi_{\boldsymbol{\mathcal W}} \big)$ to be less transcendantal/more algebraic (in a vague sense here but which it would be advisable to make more precise) than Poincar\'e-Blaschke surface $\Sigma_{\boldsymbol{\mathcal W}}={\rm Im}\big( P_{\boldsymbol{\mathcal W}} \big)$. 
\end{rem}
\begin{center}
$\star$
\end{center}

We now specialize the material introduced above in the specific case of Bol's web. For this aim, 
as a model for this web, 
it is more convenient to work with the web $\boldsymbol{\mathcal W}_{ \boldsymbol{\mathcal M}_{0,5} }$ induced by the five forgetting maps 
 on $\boldsymbol{\mathcal M}_{0,5}$. For this web, one has 
\begin{prop}
\label{P:Algeb-Bol-1}
\begin{enumerate}
\item[] ${}^{}$ \hspace{-1.3cm} {\rm 1.} The space  $\boldsymbol{ar}_C\big(\boldsymbol{\mathcal W}_{ \boldsymbol{\mathcal M}_{0,5} }\big)$ coincides 
with the space of holomorphic 1-forms on  ${\boldsymbol{\mathcal M}}_{0,5}$ with logarithmic poles along 
$\partial  {\boldsymbol{\mathcal M}}_{0,5}$, {\it i.e.} as subspaces of 
$\Omega^1( {\boldsymbol{\mathcal M}}_{0,5})$, one has 
$$
\boldsymbol{ar}_C\big(\boldsymbol{\mathcal W}_{ \boldsymbol{\mathcal M}_{0,5} }\big)
=
{\bf H}^0\left(  \, \overline{\boldsymbol{\mathcal M}}_{0,5}\, , \, 
\Omega^1_{ \overline{\boldsymbol{\mathcal M}}_{0,5}}
\Big( \, {\rm Log}\,\big( \partial  {\boldsymbol{\mathcal M}}_{0,5} \big) 
\,
\Big)\,
\right)\,.
$$
In particular $r=5$ and $\boldsymbol{ar}_C\big(\boldsymbol{\mathcal W}_{ \boldsymbol{\mathcal M}_{0,5} }\big)$ is well-defined on the whole  
$\boldsymbol{\mathcal M}_{0,5}$.
\mk
\item[{\rm 2.}] The map $\Psi_{\boldsymbol{\mathcal W}_{ \boldsymbol{\mathcal M}_{0,5}}}$ takes its values into a 5-dimensional linear subspace 
$\mathbf PW\simeq \mathbf P^5$ of $  \mathbf P\big(
\wedge^2 \boldsymbol{ar}_C\big(\boldsymbol{\mathcal W}_{ \boldsymbol{\mathcal M}_{0,5} }\big)^\vee
\big)\simeq \mathbf P^9$.  Moreover the corresponding map 
$ \boldsymbol{\mathcal M}_{0,5}
\longrightarrow \mathbf PW\simeq \mathbf P^5$
is the restriction to $ \boldsymbol{\mathcal M}_{0,5}$ of  the anticanonical embedding 
of $\overline{\boldsymbol{\mathcal M}}_{0,5}$ 
into $\mathbf P^5$.
\end{enumerate}
\end{prop}
\begin{proof}[Proof of Proposition \ref{P:Algeb-Bol-1}] 
In the case of $\boldsymbol{\mathcal X\mathcal W}_{\!A_2}$, let $U$ be an open domain containing the positive real orthant $(\mathbf R_{>0})^2$. 
Then it follows from \eqref{Eq:AR-XWA2} that  $\boldsymbol{ar}_C(\boldsymbol{\mathcal X\mathcal W}_{\!A_2})$  is the vector space spanned by  the  five 1-forms $\kappa_\ell$'s of 
\eqref{kappa-l} hence in particular one has $r=5$. According to \cite[Prop.\,6.21]{ClusterWebs}, the affine rational map $\mathbf C^2\dashrightarrow  \boldsymbol{\mathcal M}_{0,5}$, $(x,y)\longmapsto [\infty,0,1,x,y]$ induces an affine isomorphism between $\mathbf C^2\setminus A$ where $A$ is the arrangement of lines   cut out by $xy(x-1)(y-1)(x-y)=0$ and the complement of the cluster arrangement cut out by 
$u_1u_2(1+u_1)(1+u_2)(1+u_1+u_2)=0$ is the `cluster affine chart' 
$\mathbf C^2_{u_1,u_2}={\rm Spec}\big( \mathbf C[u_1,u_2]\big)$. 
 What corresponds to $\boldsymbol{ar}_C(\boldsymbol{\mathcal X\mathcal W}_{\!A_2})$ in the
in the coordinates $x,y$ is the vector space with basis 
\begin{equation}
\label{Eq=dxdx}
\frac{dx}{x}\,, \quad \qquad \frac{dy}{y}\,, \quad\qquad \frac{dx}{x-1}\,, \quad\qquad \frac{dy}{y-1}\qquad \quad
\mbox{and} \qquad \quad
\frac{dx-dy}{x-y}\, .
\end{equation}
  If $b: {\bf Bl}_{p_1,\ldots,p_5}(\mathbf P^2)\rightarrow \mathbf P^2$ stands the blow up at the points $p_1=[1,0,0]$, $p_2=[0,1,0]$, $p_3=[0,0,1]$ and $p_4=[1,1,1]$, then up to the identification $ {\bf Bl}_{p_1,\ldots,p_5}(\mathbf P^2)\simeq \overline{\boldsymbol{\mathcal M}}_{0,5}$, one easily verifies that the pull-backs under $b$ of the 1-forms \eqref{Eq=dxdx} form a basis of the space of 1-forms on ${\boldsymbol{\mathcal M}}_{0,5}$ with logarithmic poles along $\partial   \overline{\boldsymbol{\mathcal M}}_{0,5}$. 
  
Once the first point has been established, the second follows from \cite[\S2.4]{DFL}. \end{proof}

Because $\boldsymbol{ar}_C\big(\boldsymbol{\mathcal W}\big)$ 
and $\Psi_{\boldsymbol{\mathcal W}}$
 are attached to $\boldsymbol{\mathcal W}$ in an invariant way (see above), the previous proposition gives us immediately the following result which provides an effective tool to verify 
 whether a given 5-web is equivalent to Bol's web or not. 
 
 Recall that following C. Segre \cite{Segre} (see also \cite[\S1.4.4]{PP}), 
 any surface $S\subset \mathbf P^5$ with non-degenerate second-order osculation at a generic point either is included in the Veronese surface $v_2(\mathbf P^2)\subset \mathbf P^5$ or carries a 5-web $\boldsymbol{\mathcal W}_S$ which is projectively attached to $S$.  Moreover, in the last case (which is the generic one), given any  \'etale map $\varphi: U\rightarrow  \mathbf C^6$ inducing a 
 local parametrization  of $S$ when projectivized, there is an explicit formula for a global 
 symmetric form   $W_\varphi\in {\rm Sym}^5(T^*_U)$  
 defining the  pull-back web $\varphi^*\big( \boldsymbol{\mathcal W}_S \big)$ on $U$. 
 \begin{cor}
 \label{C:33}
A  5-web  $\boldsymbol{\mathcal W}$ on a domain 
$U\subset \mathbf C^2$ 
is equivalent to Bol's web if and only if  all the three following statements hold true: 
\begin{enumerate}
\item[{\rm 1.}] 
\vspace{-0.1cm}
one has $r=\dim \, \big( \boldsymbol{ar}_C(\boldsymbol{\mathcal W}_{ \boldsymbol{\mathcal M}_{0,5} })\big)=5$;\sk 
\item[{\rm 2.}]  the map $\Psi_{\boldsymbol{\mathcal W}} $ has values in a linear subspace $\mathbf PW\subset \mathbf P^9$ of dimension 5 and ${\rm Im}(\Psi_{\boldsymbol{\mathcal W}} ) =\Psi_{\boldsymbol{\mathcal W}} (U)$ is included in a quintic del Pezzo surface anticanonically embedded in  $\mathbf PW$;\sk
\item[{\rm 3.}] the initial web $\boldsymbol{\mathcal W}$ and the pull-back 
$\Psi_{\boldsymbol{\mathcal W}}^*\big( \boldsymbol{\mathcal W}_\Sigma\big)$ coincide.
\end{enumerate} 
 \end{cor}
 
This result has to be compared to the following one:
\begin{prop}[Bol's theorem \cite{Bol}]
Let  $\boldsymbol{\mathcal W}$ be a hexagonal $d$-web. Then
  either it is linearizable hence equivalent to $d$ pencils of lines,  or 
 $d=5$ and $\boldsymbol{\mathcal W}$ is equivalent to Bol's web.
\end{prop}

Since a hexagonal web $\boldsymbol{\mathcal W}$ has maximal rank, in order to verify if such a web is linearizable, it suffices to verify whether it is compatible with a projective connexion 
or not (see \cite[\S30]{BB} or \S6.1.5 in the modern reference \cite{PP}) which is equivalent to the fact that the slopes of the foliations of $\boldsymbol{\mathcal W}$ satisfy some differential algebraic identities (see \cite[Prop.\,6.1.10]{PP}).  Because hexagonality can be verified by means of formal computations (of differential algebras), Bol's theorem above provides an effective way to verify if a given 5-web is equivalent to Bol's web or not. However, the criterion evoked is non constructive in the sense that, when $\boldsymbol{\mathcal W}$ is equivalent to Bol's web, 
it does not indicate a natural (or even better, a canonical)  way to construct a map $\varphi$ from the definition domain of $\boldsymbol{\mathcal W}$  to $\boldsymbol{\mathcal M}_{0,5}$ such that 
$\boldsymbol{\mathcal W}=\varphi^*\big( \boldsymbol{\mathcal W}_{ \boldsymbol{\mathcal M}_{0,5}}\big)$.

Our Corollary \ref{C:33} for its part, is constructive and gives a canonical way of constructing such a map. However, it has its own disadvantage which is of relying on the explicit determination of the space $\boldsymbol{ar}_C(\boldsymbol{\mathcal W})$ which requires to determine all the combinatorial ARs of the web under scrutiny.  From a computational perspective, determining the AR of a hexagonal 3-web is a bit more involved than just verifying that this 3-web is indeed flat since it requires one additional step which consists in solving a homogeneous linear ODE of the first order. Solving such a differential equation can be done by performing one integration but this is a supplementary task which does not occur when using the criterion for characterizing Bol's web described in the preceding paragraph.
%
%

The new material introduced above has also the  interesting feature that it makes 
sense not only for 
hexagonal webs but more generally for webs with sufficiently many combinatorial ARs. 
Obviously, one has $\mathfrak r=\dim\big(\boldsymbol{ar}_C(\boldsymbol{\mathcal W})\big)
\leq \dim \big(\boldsymbol{AR}_C(\boldsymbol{\mathcal W})\big)\leq {\rm rk}(\boldsymbol{\mathcal W})$ and there are several examples of webs with maximal rank for which $\mathfrak r$ is significantly lower than the rank.  The map $\Psi_{\boldsymbol{\mathcal W}}$ could turn out to be an interesting tool to study webs of maximal rank and with sufficiently many combinatorial ARs. 
\begin{center}
\vspace{-0.4cm}$\star$
\end{center}

But it turns out that, given a  5-web $\boldsymbol{\mathcal W}$, there exists a canonical and more direct effective criterion for first ascertaining whether it is 
 equivalent to Bol's web and second, when it is the case, to build in a canonical way a local biholomorphism  $\varphi$ for which one has  $\boldsymbol{\mathcal W}=\varphi^*\big( \boldsymbol{\mathcal W}_{ \boldsymbol{\mathcal M}_{0,5}}\big)$. This is elementary but has not been remarked before and it is the purpose of the next subsection.

 \subsubsection{Algebraization via the canonical map $\Phi_{{\mathcal W}}$}
 \label{SSS:Via-Canonical-Map-WdP5}
%
%
%
We first introduce a general elementary construction (namely the `canonical map $\Phi_{\boldsymbol{\mathcal W}}$' associated to a given web $\boldsymbol{\mathcal W}$) which is the main ingredient which is considered below and will be used further on in \S\ref{SSS:Via-Canonical-Map-WdP4} when dealing with $\WdPq$.
\mk 

Given a holomorphic submersion $X$ on a domain of $\mathbf C^2$, 
 its `{\it slope function} $\zeta_X$' is  the meromorphic function 
defined by the formula   
$$\zeta_X=\partial_{u_1}(X)/\partial_{u_2}(X)\, , $$
 with the convention 
that $\zeta_X$ is constant and identically equal to $\infty\in \mathbf P^1$ in the case when 
$\partial_{u_2}(X)$ vanishes identically. 
\sk 

Now let $d$ be bigger than or equal to 3 and let $\boldsymbol{\mathcal W}$ be a $d$-web defined on an open domain $\Omega \subset \mathbf C^2$, with no singularity. Then there exist  holomorphic submersions $U_i: \Omega\rightarrow \mathbf C$ for $i=1,\ldots,d$, such that $\boldsymbol{\mathcal W}$ is formed by the foliations defined by the $U_i$'s.  Because $\boldsymbol{\mathcal W}$ has been assumed to have no singular points,  the slope function $\zeta_{U_i}$ can be seen as a holomorphic 
morphism $\zeta_{U_i} : \Omega\rightarrow \mathbf P^1$ for any $i$ and the $\zeta_{U_i}$'s take pairwise distinct values at any point of $\Omega$. 
Hence the map
\begin{equation}
\label{Eq:Phi-W}
\Phi_{\boldsymbol{\mathcal W}} := \big[\zeta_{U_i} \big]_{i=1}^d  :  \Omega\longrightarrow 
\boldsymbol{\mathcal M}_{0,d}
\end{equation}
is easily verified to be a well-defined  holomorphic morphism  which  is  independent of the choice of any local coordinates and is invariantly attached
to the web, {\it i.e.}\,for any local biholomorphism $F$ taking values into $\Omega$, as holomorphic maps from 
$F^{-1}(\Omega)$ to $\boldsymbol{\mathcal M}_{0,d}$, 
one has
$$
\Phi_{F*(\boldsymbol{\mathcal W})} = \Phi_{\boldsymbol{\mathcal W}} \circ F\, .
$$  

Now recall that $\boldsymbol{\mathcal M}_{0,d}$ naturally carries a ${ d \choose 4}$-web of codimension 1, noted by $\boldsymbol{\mathcal W}_{\hspace{-0.1cm}\boldsymbol{\mathcal M}_{0,d}}$, 
which is the one admitting as first integrals the forgetful maps
$f_I : \boldsymbol{\mathcal M}_{0,d} \rightarrow \boldsymbol{\mathcal M}_{0,4}\simeq \mathbf P^1\setminus \{0,1,\infty\}$, $[p_1,\ldots,p_d]\longmapsto [p_{i_1},\ldots,p_{i_4}]$
for all 4-tuples $I=(i_1,\ldots,i_4)$ such that $1\leq i_1<i_2<i_3<i_4\leq d$. 
\sk

When $\Phi_{\boldsymbol{\mathcal W}}$ has rank 2 at some point, the pull-back $\Phi_{\boldsymbol{\mathcal W}}^*\big( \boldsymbol{\mathcal W}_{\hspace{-0.1cm}\boldsymbol{\mathcal M}_{0,d}}\big)$ is a web on $\Omega$, possibly with singularities and possibly formed by strictly less than ${ d \choose 4}$ foliations, but which is invariantly attached to $\boldsymbol{\mathcal W}$. 
It is then natural to try 
to better understand the initial web $\boldsymbol{\mathcal W}$ via 
its canonical map $\Phi_{\boldsymbol{\mathcal W}}$, in terms of some  properties of 
$\boldsymbol{\mathcal W}_{\hspace{-0.1cm}\boldsymbol{\mathcal M}_{0,d}}$. 
\begin{center}
$\star$
\end{center}

We now look at the previous construction when specialized to Bol's web. 
To perform the computations,  it is more convenient to work with 
 $\boldsymbol{\mathcal X\mathcal W}_{\!A_2}$ and its cluster first integrals $X_\ell$ given in 
 \eqref{Eq:lola}.  The 5-tuple of slopes associated to the cluster first integrals of $  \boldsymbol{
 \mathcal X \hspace{-0.03cm}
  \mathcal W}_{\rm A_2}$ is 
   $$
 \zeta_  {\rm A_2}= 
\bigg( \, \infty \, , \,   0\, , \,
\frac{-( 1 + u_2)}{u_1}\, , \, 
\frac{u_2
(1 + u_2)}{u_1(1 + u_1)}
\, , \, 
 \frac{-u_2}{1 + u_1} \, 
   \bigg)\, . 
   $$   
   At any point $(u_1,u_2)\in (\mathbf R_{>0})^2$, the coordinates of  $\zeta_  {\rm A_2}$ are pairwise distinct points of $\mathbf P^1$  and  the map $\Phi_{\boldsymbol{
 \mathcal X \hspace{-0.03cm}
  \mathcal W}_{\rm A_2}}$  has rank 2. Thus  
   the same holds true at any point of a sufficiently small  open domain of $\mathbf C^2$ containing this positive quarter of plane. We fix such an open domain that we denote by $\Omega$.  
   The map $\Phi_{\boldsymbol{
 \mathcal X \hspace{-0.03cm}
  \mathcal W}_{\rm A_2}}: \Omega\rightarrow \boldsymbol{\mathcal M}_{0,5}$ being \'etale,  the 
  pull-back 
  of $ 
  \boldsymbol{\mathcal W}_{ \boldsymbol{\mathcal M}_{0,5}}$ under it, 
 that we will denote by  
 $ \boldsymbol{ \mathcal X \hspace{-0.03cm}
  \mathcal W}_{\rm A_2}'$, 
   is a genuine regular 5-web on $\Omega$. 
  
  It is not difficult to make $ \boldsymbol{
 \mathcal X \hspace{-0.03cm}
  \mathcal W}_{\rm A_2}'$  explicit in the coordinates $u_1,u_2$.  For $i=1,\ldots, 5$, one denotes by $ \zeta_  {\rm A_2}(i)$ the 4-tuple obtained from  $\zeta_  {\rm A_2}$ by removing its $i$-th coordinate.  As a (cluster) cross-ratio, we choose 
 the one defined by 
$$
\kappa(a_1,a_2,a_3,a_4)=-\frac{(a_1 - a_4)(a_2 - a_3)}{(a_1 - a_3)(a_2 - a_4)}\,
$$
for any 4-tuple $(a_1,\ldots,a_4)$  of pairwise distinct points of $\mathbf P^1$. 
The first integrals of $ \boldsymbol{
 \mathcal X \hspace{-0.03cm}
  \mathcal W}_{\rm A_2}'$ are the five rational maps 
  $\kappa\big(  \zeta_  {\rm A_2}(i)\big)$ for $i=1,\ldots,5$ and through elementary computations, 
  one obtains 
  that  as an ordered 5-web on $\Omega$, one has 
 $$
 \boldsymbol{
 \mathcal X \hspace{-0.03cm}
  \mathcal W}_{\rm A_2}'=
   \boldsymbol{\mathcal W}\bigg( \, u_1\, , \,   -(1+u_2)\, , \,
 \frac{1+u_2}{u_1}\, , \, 
 -\left( 1+\frac{1+u_1+u_2}{u_1u_2}\right)\, , \,
   \frac{1+u_1}{u_2} \, 
   \bigg)\, .
   $$ 
The map $\kappa\big(  \zeta_  {\rm A_2}(\ell)\big)$ coincides with $X_\ell$ for $\ell$ odd, and with $-(1+X_\ell)$ for $\ell$ even.  We thus have that $ \boldsymbol{
 \mathcal X \hspace{-0.03cm}
  \mathcal W}_{\rm A_2}$ and $ \boldsymbol{
 \mathcal X \hspace{-0.03cm}
  \mathcal W}_{\rm A_2}'$ are the same (moreover as ordered 5-webs). 
This observation immediately gives us the 
\begin{prop}
\label{P:Algeb-Bol-2}
A 5-web $\boldsymbol{\mathcal W}$ is equivalent to Bol's web if and only if 
$\boldsymbol{\mathcal W}=\Phi_{\boldsymbol{\mathcal W}}^*\big( \boldsymbol{\mathcal W}_{{\mathcal M}_{0,5}} \big)$.
\end{prop}
It should be noted that the condition that $\Phi_{\boldsymbol{\mathcal W}}^*\big( \boldsymbol{\mathcal W}_{{\mathcal M}_{0,5}} \big)$ is a 5-web requires that 
$\Phi_{\boldsymbol{\mathcal W}}$ is generically of rank 2, a condition that we assume to be implicitly satisfied here. \sk

Proposition \ref{P:Algeb-Bol-2} may look tautological at first sight, but actually it is not  and even better, an explicit and effective criterion for characterizing Bol's web can be extracted from it. 
Assuming that 
$\boldsymbol{\mathcal W}= \boldsymbol{\mathcal W}(U_1,\ldots,U_5)$  for some first integrals $U_1,\ldots,U_5$, one sets 
$$
X_i= \big({\partial U_i}/{\partial y} \big)\,\partial_x-
\big({\partial U_i}/{\partial x} \big)\,\partial_y
 $$ 
 for $ i=1,\ldots,5$ and for any $j,k$ such that $1\leq j<k\leq 5$, one denotes by 
 ${\mathcal J}_{jk}$ the jacobian determinant of the map $(U_j,U_k)$ in the coordinates $x,y$, {\it i.e.} 
$$
  {\mathcal J}_{jk}=
\begin{vmatrix}\,
{\partial U_j}/{\partial x} & {\partial U_j}/{\partial y}\,{}^{}\vspace{0.07cm}\\
\,{\partial U_k}/{\partial x} & {\partial U_k}/{\partial y}\,{}^{}
\end{vmatrix}\, .
$$

 Then the condition  
$\boldsymbol{\mathcal W}=\Phi_{\boldsymbol{\mathcal W}}^*\big( \boldsymbol{\mathcal W}_{{\mathcal M}_{0,5}} \big)$ in   Proposition \ref{P:Algeb-Bol-2} is equivalent to the fact that 
for any $i,j,k,l,m$ such that 
$\{i,j,k,l,m\}=\{1,\ldots,5\}$,  the following relation is identically satisfied
$$
X_i \left( \, \frac{ {\mathcal J}_{jm} \,{\mathcal J}_{kl}}{ {\mathcal J}_{jl}\, {\mathcal J}_{km}}
\, \right) =0\, . 
$$

Verifying that these relations hold true is straightforward and only requires to compute 
rational expressions in the $U_i$'s and their partial derivatives up to order 2. We thus get a new criterion for characterizing 5-webs equivalent to Bol's web, which, from our point of view, is better than the two considered in the preceding subsection: it requires less computations and furthermore for any 5-web 
$\boldsymbol{\mathcal W}\simeq \boldsymbol{\mathcal B}$, it furnishes the `algebraization map' essentially for free, by means of rational expressions in the first order partial derivatives of any first integrals of the web $\boldsymbol{\mathcal W}$ under scrutiny.

\subsection{The Gelfand-MacPherson construction}
 \label{SS:Gelfand-MacPherson--construction-WdP5}
 We now discuss the nice geometric construction of Bol's web as the quotient of an equivariant web given by Gelfand and MacPherson in their paper \cite{GM}. We first use the formalism they used (in terms of grassmannians, etc) then recast most things from another perspective (the key notion being that of `Cox variety') which is the one which will generalize and will be used to construct geometrically $\WdPq$ further in \S\ref{SS:WdP4-GM}.

\subsubsection{}
\label{SSS:GM-Weg-G2(V)}
 Let $n\in \mathbf N^*$ be arbitrary, set $N=n+3\geq 4$ 
and denote by 
 $G_2(V)$  the grassmannian variety of 2-planes in $V=\mathbf K^{N}$ 
where 
$\mathbf K$ stands for $\mathbf R$ or $\mathbf C$. Let  $(e_i)_{i=1}^N $ be the canonical basis of   $V$ and  $(x_i)_{i=1}^N$ be the corresponding linear coordinates.   Denoting by 
$$P: G_2(V) \hookrightarrow \mathbf P\big( \wedge^2 V\big): \xi\mapsto 
\big[ 
\Delta_{ij}(\xi)
\big]_{1\leq i<j\leq N}
$$ 
the corresponding Pl\"ucker embedding, we define 
$G_2^*(V)$ as the Zariski open subset of $G_2(V)$ formed by {\it `generic'} 2-planes, namely the 2-planes whose all Pl\"ucker coordinates do not vanish: 
 $$G_2^*(V)= P^{-1}\Big(  \wedge^2 V \setminus \big( \cup_{1\leq i<j\leq N}\{ \Delta_{ij}=0\} \big) \Big)\,.$$ 
 The special linear group ${\rm SL}_N(\mathbf K)$ acts on the grassmannian and its Cartan torus $H_{N}=H_{A_{n+2}}$ formed by diagonal matrices lets 
 $
G_2^*(V)$ invariant and acts freely on it in such a way that $
G_2^*(V)/H_{N}$ is  a smooth manifold with the quotient map $\chi_N=\chi_{H_{A_{n+2}}}: 
G_2^*(V)\longrightarrow G_2^*(V)/H_{N}$ being a $\mathbf K$-algebraic submersion. 
According to Gelfand and MacPherson, there is a nice geometric interpretation of the image 
$\chi_N(P)$ in the quotient of any element $P\in G_2^*(V)$: being generic, such a 2-plane intersects transversally each of the  coordinate hyperplanes $H_i=\{\,x_i=0\,\}$ hence the intersections $P\cap H_i$ are $N$ lines in $P$, which moreover are pairwise distinct. 
Setting $\ell_i(P)$ for the line $P\cap H_i$ ($i=1,\ldots,N$) viewed as an element of $\mathbf P(P)$, 
  and identifying the latter with a fixed projective line $\mathbf P^1$, one constructs 
  a well-defined map $ \mu_N : G_2^*(V)\longrightarrow {\rm Conf}_N(\mathbf P^1)=\boldsymbol{\mathcal M}_{0,N}$, $P\longmapsto [\ell_i(P)]_{i=1}^N$ which can be verified to be  not only $H_{N}$-invariant, but to provide a geometric model for the quotient map $\chi$: there is a natural isomorphism $
  G_2^*(V)/H_{N}\simeq\boldsymbol{\mathcal M}_{0,N}$ such that the following diagram commutes: 
  $$
  \xymatrix@R=0.4cm@C=1.5cm{ 
  {G}_{2}^{*}\big(V\big) 
\ar@/_1pc/@{->}[dr]_{\mu_N }  \ar@{->}[r]^{ \pi  \hspace{0.1cm}} &  {G}_{2}^{*}\big(V\big)/H_{N}  
\eq[d]
\\
& \boldsymbol{\mathcal M} _{0,N} \, . 
  }
$$
 
 For $i=1,\ldots,N$, we set $V_i=V/\langle e_i \rangle\simeq \mathbf K^{N-1}$ and we denote by 
 $\Psi_i: V\rightarrow V_i$ the corresponding linear projection. Since any generic 2-plane $P$ intersects $\langle e_i \rangle $ transversally, the image  $\Psi_i(P)$ is a 2-plane in $V_i$ which can be verified to be generic as well. We thus have a well-defined morphism 
 $ {G}_{2}^{*}\big(V\big)\rightarrow    {G}_{2}^{*}\big(V_i\big) $, $P\mapsto \Psi_i(P)$, again denoted by $\Psi_i$,  which is such that the following diagram also is commutative (where $\psi_i : \boldsymbol{\mathcal M}_{0,N}  \rightarrow \boldsymbol{\mathcal M}_{0,N-1}$ stands for the $i$-th forgetful map): 
\begin{align}
\label{Eq:quoquo}
\xymatrix@R=0.7cm@C=1.5cm{ 
  {G}_{2}^{*}\big(V\big) 
\ar@{->}[d]_{\mu_N }  \ar@{->}[r]^{ \Psi_i   \hspace{0.1cm}} &   {G}_{2}^{*}\big(V_i\big)   
\ar@{->}[d]^{\mu_{N-1} }
\\
\boldsymbol{\mathcal M}_{0,N}  \ar@{->}[r]^{ \psi_i   \hspace{0.1cm}} & 
\boldsymbol{\mathcal M}_{0,N-1} \, . 
  }
\end{align}

Each map $\Psi_i$ is a surjective $\mathbf K$-algebraic regular submersion inducing a codimension $k$ foliation on ${G}_{2}^{*}(V) $. Taking these foliations together gives us what we call the {\it `Gelfand-MacPherson web'}
$$
\boldsymbol{\mathcal W}_{ G_2(V)}^{GM}=\boldsymbol{\mathcal W}\Big( \, \Psi_1\, , \, \ldots, 
\Psi_{n+3}\, \Big)\, , 
$$
which is a $(n+3)$-web in a generalized sense on the grassmannian $G_2(V)$ (without singularities on  its generic part $G_2^*(V)$).  One easily checks that the foliations of $
\boldsymbol{\mathcal W}_{ G_2(V)}^{GM}$ are $H_N$-equivariant (actually the same can be said about the maps $\Psi_i$'s themselves, up to some technical details left to the reader). It follows that
the direct image/quotient $(\mu_N)_*\big( \boldsymbol{\mathcal W}_{ G_2(V)}^{GM}\big)=
\boldsymbol{\mathcal W}_{ G_2(V)}^{GM}/H_N$ exists as a web on $G_2^*(V)/H_N
\simeq  \boldsymbol{\mathcal M}_{0,n+3}$. Considering the diagrams \eqref{Eq:quoquo} for $i=1,\ldots,n+3$, it follows that
$$
{\boldsymbol{\mathcal W}_{ G_2(V)}^{GM}}_{\big/H_{n+2}}=\boldsymbol{\mathcal W}_{
\hspace{-0.05cm}
\boldsymbol{\mathcal M}_{0,n+3}}
$$
where the right hand side stands for Burau-Damiano's  curvilinear web on $\boldsymbol{\mathcal M}_{0,n+3}$, which is the  web by rational curves  induced by the 
forgetful maps 
$\psi_i : \boldsymbol{\mathcal M}_{0,n+3} \rightarrow  \boldsymbol{\mathcal M}_{0,n+2}$ for $i=1,\ldots,n+3$.\footnote{We refer the interested reader to our recent paper \cite{Pirio-W-M06} for much more about the webs $\boldsymbol{\mathcal W}_{ \hspace{-0.05cm} \boldsymbol{\mathcal M}_{0,n+3}}$, for $n\geq 2$.}

From the preceding discussion, we deduce that 
\begin{equation}
\label{Eq:GM-description-statement}
\begin{tabular}{l}
{\it Gelfand-MacPherson's web $\boldsymbol{\mathcal W}_{ G_2(V)}^{GM}$ is $H_{n+2}$-equivariant and its torus quotient}\\
{\it  naturally identifies with Burau-Damiano's curvilinear web 
$\boldsymbol{\mathcal W}_{
\hspace{-0.05cm}
 \boldsymbol{\mathcal M}_{0,n+3}}$ on $\boldsymbol{\mathcal M}_{0,n+3}$.}
\end{tabular}
\end{equation}

For $n=2$, 
we have $\boldsymbol{\mathcal W}_{
\hspace{-0.05cm}
 \boldsymbol{\mathcal M}_{0,5}}\simeq \boldsymbol{\mathcal B}$ hence the above statement applies and gives a description of Bol's web $\boldsymbol{\mathcal B}$ as an equivariant quotient by a torus action of a natural web defined by linear maps on the grassmannian of 2-planes in $\mathbf K^5$.  
 
 \begin{rem}
 The above geometric way to get Bol's web 
  is particularly interesting if we consider the major result obtained by Gelfand and MacPherson
 which is to construct geometrically (the real version of) Abel's dilogarithmic identity, by integrating along the fibers of the actions of the different tori involved 
   invariant representatives of some compatible characteristic classes on the grassmannians of 2-planes in $\mathbf R^5$ and in $\mathbf R^4_i=\mathbf R^5/\langle e_i\rangle$ for $i=1,\ldots,5$  (see \cite[\S0.3]{GM} or \S\ref{SS:HLog3-a-la-GM} further in this text{\rm )}. 
 \end{rem}

\subsubsection{}
\label{SSS:Cox-stuff}
In order to generalize the above geometric construction of Bol's web $\boldsymbol{\mathcal B}\simeq \boldsymbol{\mathcal W}_{ {\rm dP}_5}$ to $\boldsymbol{\mathcal W}_{ {\rm dP}_4}$, it is interesting to recast it from the bottom, that is in terms of del Pezzo's quitinc surface ${\rm dP}_5\simeq  \overline{\boldsymbol{\mathcal M}}_{0,5}$.   Our main references for the material below are 
\cite{Skorobogatov1} (see also 
 \cite[\S3]{SkorobogatovBook}), 
\cite{BatyrevPopov}, \cite{SerganovaSkorobogatov} and \cite{Derenthal}. 
We identify the quintic del Pezzo surface to the total space, denoted by $X_4$ below, of the blow-up of $\mathbf P^2$ at four points $p_1,\ldots,p_4$ in general position: 
${\rm dP}_5\simeq X_4= {\bf Bl}_{p_1+\cdots+p_4}(\mathbf P^2)$. As a $\mathbf Z$-basis for the Picard lattice ${\bf Pic}_{ \mathbf Z}(X_4)\simeq \mathbf Z^5$ 
of $X_4$, we take $\boldsymbol{\mathcal E}= (\boldsymbol{\ell}_i)_{i=0}^4$ where $
\boldsymbol{\ell}_0=\boldsymbol{h}$ stands for the pull-back of the class of a generic line in $\mathbf P^2$ and where $\boldsymbol{\ell}_i$ is the class of the exceptional divisor in $X_4$ associated to $p_i$ for $i=1,\ldots,4$.

For $\boldsymbol{\ell}\in\{\boldsymbol{h},\boldsymbol{\ell}_1,\ldots,\boldsymbol{\ell}_4\}$, let $\mathscr L_{\boldsymbol{\ell}}^\circ$ stand for the total space of the line bundle $\mathcal O_{X_4}(\boldsymbol{\ell})$ with the image of the zero section removed 
and let us set 
$${\mathcal T}_{\hspace{-0.07cm}X_4}=\mathscr L_{\boldsymbol{\ell}_0}^\circ\times_{X_4} 
 \cdots \times_{X_4} \mathscr L_{\boldsymbol{\ell}_4}^\circ\, .$$
   It is a torus bundle over $X_4$ on which naturally acts 
 the 
 {\it N\'eron-Severi torus} of $X_4$, defined by 
 $$T_{NS}(X_4)={\rm Hom}\big({\bf Pic}_{\mathbf Z}(X_4), \mathbf C^*\big)
\simeq  (\mathbf C^*)^{\boldsymbol{\mathcal E}}\, . 
 $$ 
It naturally acts on  ${\mathcal T}_{\hspace{-0.07cm}X_4}$,  and the associated quotient map identifies with the bundle map ${\mathcal T}_{\hspace{-0.07cm}X_4}\longrightarrow X_4$. 
 Taking the intersection  with the anticanonical class defines a linear form 
 $(-K,\cdot)$ on ${\bf Pic}_{\mathbf Z}(X_4)$ which gives rise to a 
 1-parameter subgroup $\tau_\kappa=\mathbf C^*(-K,\cdot)
\subset   T_{NS}(X_4)$.  Setting 
$ {\rm T}_{X_4}=T_{NS}(X_4)/\tau_\kappa \simeq (\mathbf C^*)^4$ we obtain that 
$\tau : {\mathscr T}_{\hspace{-0.02cm}X_4}={\mathcal T}_{\hspace{-0.07cm}X_4}/\tau_\kappa \rightarrow X_4 $ is 
a ${\rm T}_{X_4}$-torsor over $X_4$. Then we set 
$$
{\mathscr T}_{\hspace{-0.02cm}X_4}^*=\tau^{-1}\big(X_4^*\big) 
\subset {\mathscr T}_{\hspace{-0.02cm}X_4}
$$
where $X_4^*$ is the complement of the line divisor 
$L_4=\cup_{ \ell \in  \boldsymbol{\mathcal L}_4} \ell$ in $X_4$, that is  $X_4^*=X_4\setminus L_4$.

For  each line $\ell\in \mathcal L_4$,  we fix a nonzero section $x_\ell \in {\bf H}^0(X_4 , \mathcal O(\ell))$
  which is determined
up to multiplication by a nonzero scalar since ${\bf H}^0(X_4 , \mathcal O(\ell))$ has dimension 1. Since $x_\ell$ does not vanish on 
$ {\mathscr T}_{\hspace{-0.02cm}X_4}^* $ for any $\ell$, there exists a morphism 
\begin{align}
\label{Eq:embedding-F}
F :  {\mathscr T}_{\hspace{-0.02cm}X_4}^*  &  \longrightarrow  \mathbf P\big(  \mathbf C^{ \boldsymbol{\mathcal L}_4 }\big) \simeq \mathbf P^{l_4-1}=\mathbf P^9
 \\
 t & \longmapsto  \big[x_\ell(t)\big]_{ \ell \in  \boldsymbol{\mathcal L}_4 }
 \nonumber 
\end{align}
which can be proved to be an embedding. Note that $F$ is canonically defined up to 
post-composition by an element of the torus formed by 
the linear automorphisms of $\mathbf P\big(\mathbf C^{ \boldsymbol{\mathcal L}_4 }\big)$ 
with a diagonal representative  with respect to the basis $\boldsymbol{\mathcal L}_4$.  In order to describe the image of $F$ in $\mathbf P^9$, let $\big(X_\ell\big)_{\ell \in \mathcal L_4}$ be the homogeneous coordinates corresponding to the $x_\ell$'s considered above and  for any $\ell$, let $H_\ell$ stand for the coordinate hyperplane cut out by the equation $X_\ell=0$. 
Then the Zariski closure $\mathbf G_{2,5}$ of 
$ F\big( 
 {\mathscr T}_{\hspace{-0.02cm}X_4}^*
 \big)$ in 
$ \mathbf P\big(  \mathbf C^{ \mathcal L_4 }\big)$
is  isomorphic to the image of a  Pl\"ucker embedding  $P: G_2(\mathbf C^5)\hookrightarrow \mathbf P^9$ and moreover, one has 
$$
 {\mathscr T}_{\hspace{-0.02cm}X_4}^*
 \stackrel{\sim}{\longrightarrow}
 F\Big( 
   {\mathscr T}_{\hspace{-0.02cm}X_4}^*
 \Big)= \mathbf G_{2,5}^* =
 P\Big( 
 G_2^*\big(\mathbf C^5\big)
 \Big)
 =
 \mathbf G_{2,5}\setminus \Big(
 \cup_{\ell \in \boldsymbol{\mathcal L}_4} H_\ell\, 
 \Big)
 \subset \mathbf G_{2,5} \subset \mathbf P^9
 \, . 
$$
The grassmannian $\mathbf G_{2,5}$ is homogeneous under the action of a subroup $G$ of ${\rm PGL}
\big(  \mathbf C^{ \boldsymbol{\mathcal L}_4 }\big)$ isomorphic to  ${\rm PGL}_5(\mathbf C)$. Moreover, there exists an embedding ${\rm T}_{X_4}\hookrightarrow G$ 
whose image will be denoted by 
$\boldsymbol{T}_{X_4}$, 
 making of $F:  {\mathscr T}_{\hspace{-0.02cm}X_4}^* \longrightarrow 
 \mathbf G_{2,5}^*
$
 a $({\rm T}_{X_4}, \boldsymbol{T}_{X_4})$-equivariant isomorphism.  Because the action of 
$\boldsymbol{T}_{X_4}$ on $ \mathbf G_{2,5}^*$ is isomorphic to the one of $H_{4}$ on $G_2^*(\mathbf C^5)$ discussed above,  this gives an intrinsic way to obtain the geometric framework 
of Gelfand and MacPherson allowing to construct Bol's web geometrically as an equivariant quotient. 

Actually, not only the fibration $ G_{2}^*(\mathbf C^5)\rightarrow X_4^*$ can be recovered that way but its extension over the whole del Pezzo surface $X_4$ can be as well. The main ingredient for that 
is  the {\it Cox ring} of $X_4$ (with respect to $\boldsymbol{\mathcal E}$) which by definition is the commutative $\mathbf C$-algebra
 $$
{\rm  Cox}\big(X_4\big)
=\bigoplus_{ (m_0,m_1,\ldots,m_4)\in \mathbf Z^5
}
{\bf H}^0\Big( {X}_4, \mathcal O\big( m_0\,H+m_1 \,E_1+
 \cdots+ m_4 \,E_4 \big)\Big) 
 $$
 where $H$ is a fixed generic element of $\lvert \boldsymbol{h}\lvert$ and $E_i$ is the exceptional divisor associated to $p_i$ for $i=1,\ldots,4$, with the multiplication being induced by the pointwise multiplication of sections. 
 It is naturally $\mathbf Z_{\geq 0}$-graded (by taking the intersection  with the anticanonical class $-K_{X_4}$) and it is known (eg.\,see \cite{BatyrevPopov}) to be  generated by degree 1 elements which are 
 global nonzero  elements of ${\bf H}^0\Big( {X}_4, \mathcal O\big( \ell \big)\Big)$ for  $\ell$ ranging 
 in the set of lines included in $X_4$, with relations which all are homogeneous of degree 2 with respect to the considered grading. From this, it can be deduced that 
 $$
 {\bf P}(X_4)={\rm Proj}\Big( {\rm  Cox}\big(X_4\big)
\Big) 
 $$
 is a projective variety, which contains $ {\mathscr T}_{\hspace{-0.02cm}X_4}^*$ as a Zariski open subset and 
is such that \eqref{Eq:embedding-F} extends to a canonical embedding 
 $\overline{F} :    {\bf P}(X_4) \hookrightarrow \mathbf P\big(  \mathbf C^{ \mathcal L_4 }\big)$ whose image is the grassmannian ${\bf G}_{2,5}$. Moreover, the ${\rm T}_{X_4}$-action on $ {\mathscr T}_{\hspace{-0.02cm}X_4}^*$ extends to ${\bf P}(X_4)$ and the map $\overline{F}$ is $({\rm T}_{X_4}, \boldsymbol{T}_{X_4})$-equivariant, as is 
its restriction  to  $ {\mathscr T}_{\hspace{-0.02cm}X_4}^*$. 
 The set ${\bf G}_{2,5}^s$ of stable points for the $\boldsymbol{T}_{X_4}$-action coincides with that of semi-stable points (cf.\,\cite{Skorobogatov1}) and contains 
 $ {\mathscr T}_{\hspace{-0.02cm}X_4}^*$ as a Zariski open subset. One thus recovers the whole del Pezzo surface as the quotient of 
${\bf G}_{2,5}^s$ by $\boldsymbol{T}_{X_4}$: one has $X_4={\bf G}_{2,5}^s/\boldsymbol{T}_{X_4}$. 

From the discussion above, one deduces that 
\begin{equation}
\label{Eq:Cox-theoretic-description-statement}
\begin{tabular}{l}
{\it Gelfand-MacPherson's geometric description of Bol's web \eqref{Eq:GM-description-statement}
can be recovered}\\
{\it  in a intrinsic and almost canonical way, by means of 
the Cox ring theoretic mater-}\\
{\it ial discussed just above.}
\end{tabular}
\end{equation}

The point is that this Cox-ring theoretic material can be generalized to all the del Pezzo surfaces of degree $d\in \{2,\ldots,5\}$ which opens the door for a perspective \`a la Gelfand and MacPherson on  the corresponding del Pezzo's webs $\boldsymbol{\mathcal W}_{ {\rm dP}_d }$. We will discuss the case when $d=4$ further below in \S\ref{SS:WdP4-GM}.

\section{\bf The web 
${\mathcal W} \hspace{-0.46cm}{\mathcal W}_{ {\rm dP}_4
 \hspace{-0.4cm}
  {\rm dP}_4}$
}
\label{S:WdP4}

This section is the main core of this article. In it, we prove that the web of conics of a quartic del Pezzo surfaces satisfies all the properties listed in 
\S\ref{S:Web-WdP4-properties} above. 
%
%
\subsection{The web ${\mathcal W} \hspace{-0.46cm}{\mathcal W}_{ {\rm dP}_4
 \hspace{-0.4cm}
  {\rm dP}_4}$ and the identity ${\bf HLog}^3$ in explicit form}
 \label{SS:Explicit-WdP4}
 It is interesting to make the web $\boldsymbol{\mathcal W}_{ {\rm dP}_4}$
 and the hyperlogarithmic identity ${\bf HLog}^{3}$ as explicit as possible.\sk

Let ${\rm dP}_4$ be a fixed smooth del Pezzo quartic surface, that we see 
as the total space blow-up $ b : {\rm dP}_4\rightarrow \mathbf P^2$ at the following 
five points:  
\begin{equation}
\label{Eq:points-pi}
p_1=\big[\,1:0:0\,\big]\, , \hspace{0.25cm} p_2=\big[\,0:1:0\,\big]
\, , \hspace{0.25cm} 
p_3=\big[\,0:0:1\,\big]
\, , \hspace{0.25cm} 
p_4=\big[\,1:1:1\,\big]
\quad \mbox{and} \quad 
p_5=\big[\,\pi : \gamma : 1\,\big]
\end{equation}
for some fixed parameters $\pi,\gamma\in \mathbf C$ such that these five points are in general position. This latter condition  
is equivalent to the fact that 
\begin{equation}
\label{Eq:PG-pi-gamma}
\pi\gamma(\pi-1)(\gamma-1)(\pi-\gamma)\neq 0\, , 
\end{equation}
a condition that 
 we assume to be satisfied in what follows. 
 Let $x,y$ be the  affine coordinates corresponding to the affine embedding $\mu: (x,y)\mapsto [x:y:1]$. One verifies easily that the pull-back 
of $\boldsymbol{\mathcal W}_{ {\rm dP}_4}$
under 
the birational map 
$b^{-1}\circ  \mu$ is the 
web $\boldsymbol{\mathcal W}\big( U_1,\ldots,U_{10}\big)$
 defined  by the ten following rational functions: 
\begin{align}
\label{Eq:WdP4-Ui}
U_1= & \,  x && U_6=     \frac{
   (1 - x)\gamma + x + (\pi - 1)y - \pi}{ (x - 1)(y-\gamma)}
 \nonumber \\
U_2= & \,  \frac1y
&& U_7=
          \frac{ 
    (x - y)(y-\gamma)}{ y(\pi y - \gamma x - \pi + \gamma + x - y)}
   \nonumber  \\
U_3= & \, \frac{y}{x} && 
U_8=  \frac{    
     -x(x(\gamma - 1) + (1-y )\pi - \gamma + y)}{ (x - y)(x - \pi)}
\\
U_4= & \, 
\frac{x-y}{x-1} 
&&
U_9= \,
          \frac{  
      y(x - \pi)}{ x (y-\gamma)} 
 \nonumber\\
  U_5= & \, 
   \frac{\gamma(\pi-x ) }{\pi y - \gamma x}
&& 
U_{10}=  \frac{
       x(y - 1)}{y (x - 1)} 
        \nonumber
\end{align} 

The rational curves in $\mathbf P^2$  corresponding to the 11 lines in ${\rm dP}_4$ distinct from 
  the exceptional divisors $\ell_i=b^{-1}(p_i)$ ($i=1,\ldots,5$) 
are the lines at infinity plus the closures in $\mathbf P^2$ of the affine curves with equation $\mathscr L_i =0$, where the $\mathscr L_i's$ are the components of the following 10-tuple
  of polynomials in $x,y$:
\begin{align*}
 \mathscr L= \big( 
\mathscr L_i 
\big)_{i=1}^{10}
=
\bigg( \, x \, , &  \,  y \, , \,   y -\gamma  \, , \,    x -1\, , \,  x -\pi \, , \,   x -y \, , \,  y -1\, , \,   \\
&\, \gamma\,  \Big((x -y) \pi +x (y -1)\Big)  -\pi \, y   \big(x -1\big)
\, , \,  
\gamma \, \big(x -1\big)  -\pi\, (y -1)  +y -x \, , \,   \gamma  \, x -\pi  \, y\, 
\bigg) \, .
\end{align*}
We denote by $A_{\mathscr L}$ the union in $\mathbf C^2$ of the curves cut out by the equations $\mathscr L_i=0$ with $i=1,\ldots,10$. Then  
$\boldsymbol{\mathcal W}\big( U_1,\ldots,U_{10}\big)$ is a non singular web on the Zariski open set $ b( {\rm dP}_4\setminus L_4)= \mathbf C^2 \setminus A_{\mathscr L}$.

One considers the  associated logarithmic forms  (for $i=1,\ldots,10$): 
\begin{equation}
\label{Eq:hi}
  h_i=d\log(\mathscr L_i)={d\mathscr L_i}/{\mathscr L_i}\,.
  \end{equation} 
 Then for each $i$, the spectrum\footnote{By definition, the `spectrum' of  
a map $f: {\rm dP}_4\dashrightarrow \mathbf P^1$  
is the set of values $\lambda \in \mathbf P^1$ such that  $f^{-1}(\lambda)$ is not irreducible.} of the fibration in conics  $U_i\circ b : {\rm dP_4}\rightarrow \mathbf P^1$  is 
the support of the following normalized 4-tuple
 $${\mathfrak R}_i=\big( \, 0\, , \,  1 \, , \, {r}_i\, , \,  \infty \, \big)$$
 where for each $i$, $\mathfrak r_i$ is the $i$-th component of the following $10$-tuple of complex numbers: 
$$ 
{\mathfrak r}=
\big( {\mathfrak r}_i\big)_{i=1}^{10}
=
\left(\,
\pi 
 \, , \,
\frac{1}{\gamma}
 \, , \,
\frac{\gamma}{\pi}
 \, , \,
\frac{\pi-\gamma}{\pi -1}
 \, , \,
\frac{\gamma  (\pi -1)}{\pi-\gamma}
 \, , \,
\frac{\gamma -\pi}{\gamma}
 \, , \,
\frac{1}{1-\pi}
 \, , \,
1-\gamma
 \, , \,
\frac{\pi -1}{\gamma -1}
 \, , \,
\frac{\pi  (\gamma -1)}{\gamma  (\pi -1)} \, 
\right)\, .
$$
Remark that since \eqref{Eq:PG-pi-gamma} is assumed to hold true:
\begin{itemize}
\vspace{-0.2cm}
\item[$(i).$] all the $\mathscr L_i$'s are linearly independent as affine equations in $x$ and $y$; hence \sk
\item[$(ii).$]  the same holds true for the $h_i$'s (as logarithmic forms in the same variables); and  \sk
\item[$(iii).$] none of the $
r_i$'s coincides with an element of $\{0,1,\infty\}\subset \mathbf P^1$ hence each ${\mathfrak R}_i$ indeed is a 4-tuple of pairwise distinct elements of $\mathbf P^1$.
\end{itemize}

The $h_i$'s form a basis of a subspace denoted by ${\bf H}$, of the space of rational 1-forms on $\mathbf P^2$.  
For every $i$,  one sets $\eta_i=(\eta_{i,1}, \eta_{i,2}, \eta_{i,3})$ 
with $\eta_{i,s}=d{\rm Log}(U_i-\mathfrak R_{i,s})$ for $s=1,2,3$, that is 
$$
\eta_{i,1}=\frac{dU_i}{U_i}\, , 
\qquad 
\eta_{i,2}=\frac{dU_i}{U_i-1}
\qquad \mbox{ and } 
\qquad 
\eta_{i,3}=\frac{dU_i}{U_i-r_i}\, . 
$$

Any $\eta_{i,s}$ is an  element of ${\bf H}$ 
hence admits a unique expression as a linear combination in the 
$h_j$'s (see below for a conceptual explanation of this fact) and all of these decompositions  can be obtained in explicit form by means of straightforward computations. On gets the following formulas: 
 \begin{align}
\label{Eq:hh}
\eta_1=& \, \Big( h_{1} \, , \,h_{4} \, , \,h_{5}\Big)  \nonumber 
\\ 
\eta_2= & \, 
\Big(-h_{2} \, , \,h_{7}-h_{2} \, , \,h_{3}-h_{2}\Big) \nonumber 
\\
 \eta_3= & \, 
\Big(-h_{1}+h_{2} \, , \,-h_{1}+h_{6} \, , \,-h_{1}+h_{10}\Big)\nonumber 
\\  
\eta_4=  & \, 
\Big(-h_{4}+h_{6} \, , \,h_{7}-h_{4} \, , \,h_{9}-h_{4}\Big)\nonumber 
\\
 \eta_5= & \, 
\Big(-h_{10}+h_{5} \, , \,h_{3}-h_{10} \, , \,h_{9}-h_{10}\Big) 
\\  
\eta_6= & \, 
\Big( -h_{3}+h_{9}-h_{4} \, , \,h_{7}-h_{3}-h_{4}+h_{5} \, , \,-h_{3}-h_{4}+h_{8}\Big)\nonumber 
\\
 \eta_7= & \, 
\Big(h_{3}-h_{9}+h_{6}-h_{2} \, , \,h_{7}-h_{9}+h_{10}-h_{2} \, , \,-h_{9}-h_{2}+h_{8}\Big) \nonumber 
\\
 \eta_8= & \, 
\Big( h_{9}+h_{1}-h_{5}-h_{6} \, , \,h_{4}-h_{5}-h_{6}+h_{10} \, , \,-h_{5}-h_{6}+h_{8}\Big)\nonumber 
\\
 \eta_9= & \, 
\Big( 
-h_{3}-h_{1}+h_{5}+h_{2} \, , \,-h_{3}-h_{1}+h_{10} \, , \,-h_{3}-h_{1}+h_{8}\Big)
\nonumber  \\
 \mbox{and }\quad \eta_{10}= & \, 
\Big( 
h_{7}+h_{1}-h_{4}-h_{2} \, , \,-h_{4}+h_{6}-h_{2} \, , \,-h_{4}-h_{2}+h_{8}\, \Big)
\, . \nonumber 
\end{align}

For a triple $(a,b,c)$ of pairwise distinct points on $\mathbf C$ and given a base point $\zeta \in \mathbf C\setminus \{a,b,c\}$, we consider the  weight 3 hyperlogarithm $H_{a,b,c}^{\zeta}$ 
defined  by
$$ 
H_{a,b,c}^{\zeta}(z)= \int_{\zeta}^{z}   \Bigg(\int_{\zeta}^{u_3} \bigg(\int_{\zeta}^{u_2} \frac{du_1}{u_1-c}\bigg) \frac{du_2}{u_2-b}\Bigg)  \frac{du_3}{u_3-a} $$
for any $z$ sufficiently close to $\zeta$, and we denote by $AH_{a,b,c}^{\zeta}$ its antisymmetrization:
\begin{equation*}
AH_{a,b,c}^{\zeta}=\frac{1}{6} \, \bigg(\, H_{a,b,c}^{\zeta}-H_{a,c,b}^{\zeta}-H_{b,a,c}^{\zeta}+H_{b,c,a}^{\zeta}+H_{c,a,b}^{\zeta}-H_{c,b,a}^{\zeta}
\bigg) 
\, .
\end{equation*}

We now fix a base point $\xi\in \mathbf C^2\setminus A_{\mathcal L}$ and 
for $i=1,\ldots,10$, we set  $\xi_i=U_i(\xi)\in \mathbf C\setminus \{0,1,r_i\}$  and
 $$
 AH_i^3=AH^{\xi_i}_{0,1,r_i}\, .
 $$

 For any $i$, the symbol 
$\mathcal S_i=\mathcal S\big( AH_i^3(U_i)\big)$
of $AH_i^3\big(U_i\big)=U_i^*\big(AH_i^3\big)$ is the antisymmetric tensor 
$$ \mathcal S_i=\frac{dU_i}{U_i} \wedge   \frac{dU_i}{U_i-1}  
\wedge  
 \frac{dU_i}{U_i-r_i} \in \wedge ^3 {\bf H}\, .  $$ 
A basis of the space of weight 3 tensors $\wedge ^3 {\bf H}$ is given by the 
$h_i\wedge h_j\wedge h_k$'s for all triples $(i,j,k)$ such that $1\leq i< j< k\leq 10$.  From the formulas in \eqref{Eq:hh}, it is just an elementary computational matter to 
express each symbol 
$\mathcal S_i$ as a linear linear combination of the $h_i\wedge h_j\wedge h_k$'s. 
For instance, one has 
\begin{align*}
\mathcal S_1=& \, h_1\wedge h_4\wedge h_5\\
\mathcal S_2=& \, -h_2\wedge h_3\wedge h_7\\
\mathcal S_3=& \, h_1\wedge h_2\wedge h_{10}
-h_1\wedge h_2\wedge h_{6}
-h_1\wedge h_6\wedge h_{10}
+h_2\wedge h_6\wedge h_{10}\, ,\, \mbox{etc.}
\end{align*}

With the explicit expressions of the symbols $\mathcal S_i$'s at hand, another elementary computation gives us that 
$$ {}^{} \qquad \qquad 
\sum_{i=1}^{10} \mathcal S_i= \sum_{i=1}^{10}
\left( \,\frac{dU_i}{U_i} \wedge   \frac{dU_i}{U_i-1}  
\wedge  
 \frac{dU_i}{U_i-r_i}\,\right)
= 0
$$
in $\wedge ^3 {\bf H}$, from which one immediately deduces that the identity 
${\bf HLog}^3$ is written in explicit form 
\begin{equation}
\label{Eq:Eq-explicit-d=4}
\sum_{i=1}^{10} AH_i^3\big(U_i\big)=0\, , 
\end{equation}
a functional relation which holds true identically 
 on any sufficiently small neighbourhood of $\zeta$.
\mk 

Another thing which can be obtained quite easily from the explicit formulas 
 \eqref{Eq:hh} is the determination of the weight 2 antisymmetric ARs one can obtain from ${\bf HLog}^3$ by taking residues.  Indeed, for any $i=1,\ldots, 10$, 
 let us denote by $\mathcal R_{i}$ the
  weight 2 identity obtained by taking the residue of 
  ${\bf HLog}^3$ along the `line' $\ell_i$, {\it i.e.} 
  $$\mathcal R_{i}={\rm Res}_{\ell_i}\big( {\bf HLog}^3 \,\big)=
\Big(
{\rm Res}_{\ell_i}\big( \mathcal S_j\big)
\Big)_{j=1}^{10}
  \,.$$
Let $\chi_1,\ldots,\chi_{10}$ be the elements of the formal dual basis of the basis $(h_1,\ldots,h_{10})$ of ${\bf H}$: the  $\chi_i$'s are uniquely determined by the relations $h_j(\chi_i )=\delta_{ij}$ for all $i,j=1,\ldots,10$. 
For each $i$, one denotes by ${\bf H}_i$ the subspace of ${\bf H}$ with basis $dU_i/U_i$, $dU_i/(U_i-1)$ and $dU_i/(U_i-r_i)$. 
\begin{lem}
For any $i,j=1,\ldots,10$, the following hold true: 
\begin{itemize}
\item the residue ${\rm Res}_{\ell_j}\big( \mathcal S_i \big)$ is non-zero if and only if $\ell_j$ is contracted by $U_i$;\mk 
\item  when $\ell_j$ is contracted by $U_i$, then $z_{ij}=U_i(\ell_j)$ is one of the elements of 
the 4-tuple $\mathfrak R_i=(0,1,r_i,\infty)$, the $k$-th say.  Then 
${\rm Res}_{\ell_j}\big( \mathcal S_i \big)$ coincides with $(-1)^{k-1}/3$ times 
the pull-back under $U_i$ of the weight two symbol of $AH^2_{\mathfrak R_{i\hat{k}}}$ where 
$\mathfrak R_{i\hat{k}}=(\mathfrak R_{i,s}\big)_{i=1,i\neq k}^4 $. The explicit formula for ${\rm Res}_{\ell_j}\big( \mathcal S_i \big)$ according to the value of $z_{ij}$ is given in the following table: \mk \\
\scalebox{0.96}{
${}^{}$
\hspace{-1.6cm}
\vspace{-0.1cm}
\begin{tabular}{|c||c|c|c|c|}
\hline
\begin{tabular}{c}
\vspace{-0.4cm}\\
$\boldsymbol{z_{ij}}$ \vspace{0.1cm}\\
\end{tabular}
 & $\boldsymbol{0}$   &  $\boldsymbol{1}$ &  $\boldsymbol{r_i}$  &  $\boldsymbol{\infty}$ \\
\hline \hline  
\begin{tabular}{c}
\vspace{-0.2cm}\\
$\boldsymbol{3\,{\rm Res}_{\ell_j}\big( \mathcal S_i \big)}$\vspace{0.2cm}\\
\end{tabular}
 & $U_i^*\Big(\frac{dz}{z-1} \wedge \frac{dz}{z-r_i}\Big)$  &
$ - U_i^*\Big(\frac{dz}{z} \wedge \frac{dz}{z-r_i}\Big) $
  & $U_i^*\Big(\frac{dz}{z} \wedge \frac{dz}{z-1}\Big)$ &
  $- U_i^*\Big(\Big(\frac{dz}{z}-\frac{dz}{z-r_i}\Big) \wedge \Big( \frac{dz}{z-1}-\frac{dz}{z-r_i} \Big)\Big)$
\\ \hline
\end{tabular}}
\end{itemize}
\end{lem}

For $a,b\in \mathbf C$ and $c\in \mathbf P^1$, one considers the formal expression 
$$\boldsymbol{\mathfrak R}_{a,b,c}^{V}= 
\left( \frac{dV}{V-a}-\frac{dV}{V-c}\right)\wedge 
 \left( \frac{dV}{V-b}-\frac{dV}{V-c}\right)
$$
with the convention that $ \frac{dV}{V-c}=0 $ when $c=\infty$ (hence the elements of the second row in the table above may have been written $\boldsymbol{\mathfrak R}_{1,r_i,\infty}^{U_i}$, 
$ - \boldsymbol{\mathfrak R}_{0,r_i,\infty}^{U_i}$,  $  \boldsymbol{\mathfrak R}_{0,1,\infty}^{U_i}$  and 
  $- \boldsymbol{\mathfrak R}_{0,1,r_i}^{U_i}$, in this order).

With the preceding lemma at hand, it is straightforward to compute by hand the weight 2 abelian relations obtained by taking residues of ${\bf HLog}^3$  along the $\ell_i$'s.  One has: 
%
%
\begin{align*}
{\rm Res}_{\ell_1}\big({\bf HLog}^3\big)=\, & \bigg(\, 
 \boldsymbol{\mathfrak R}_{1,r_1,\infty}^{U_1}\, , \, 
0 
\, , \,  
-{ \boldsymbol{\mathfrak R}_{0,1,r_3}^{U_3}}
 0 
\, , \,
0 
\, , \,
0 
\, , \,
0 
\, , \,
0 
\, , \,
\boldsymbol{\mathfrak R}_{1,r_8,\infty}^{U_8}
\, , \, 
- \boldsymbol{\mathfrak R}_{0,1,r_{9}}^{U_{9}}
\, , \, 
\boldsymbol{\mathfrak R}_{1,r_{10},\infty}^{U_{10}}\, \bigg) \\
{\rm Res}_{\ell_2}\big({\bf HLog}^3\big)= \, & \bigg(\, 
0 
\, , \, 
 - \boldsymbol{\mathfrak R}_{0,1,r_2}^{U_2}
\, , \,  
  \boldsymbol{\mathfrak R}_{1,r_3,\infty}^{U_3}
\, , \,
0 
\, , \,
0 
\, , \,
0 
\, , \,
 -\boldsymbol{\mathfrak R}_{0,1,r_7}^{U_7}
\, , \,
0
\, , \, 
 \boldsymbol{\mathfrak R}_{1,r_9,\infty}^{U_9}
\, , \, 
- \boldsymbol{\mathfrak R}_{0,1,r_{10}}^{U_{10}} \, \bigg) \\
{\rm Res}_{\ell_3}\big({\bf HLog}^3\big)= \, & \bigg(\, 
0 
\, , \, 
\boldsymbol{\mathfrak R}_{0,1,\infty}^{U_{2}}
\, , \,  
 0 
\, , \,
0 
\, , \,
-\boldsymbol{\mathfrak R}_{0,r_5,\infty}^{U_{5}}
\, , \,
 -\boldsymbol{\mathfrak R}_{0,1,r_6}^{U_6}
\, , \,
\boldsymbol{\mathfrak R}_{1,r_7,\infty}^{U_{7}}
\, , \,
0
\, , \, 
 -\boldsymbol{\mathfrak R}_{0,1,r_9}^{U_9}
\, , \, 
0\, \bigg) \\
 {\rm Res}_{\ell_4}\big({\bf HLog}^3\big)= \, & \bigg(\, 
- \boldsymbol{\mathfrak R}_{0,r_1,\infty}^{U_1}
\, , \, 
0 
\, , \,  
 0 
\, , \,
- \boldsymbol{\mathfrak R}_{0,1,r_4}^{U_4}
\, , \,
0 
\, , \,
- \boldsymbol{\mathfrak R}_{0,1,r_6}^{U_6}
\, , \,
0 
\, , \,
 - \boldsymbol{\mathfrak R}_{0,r_8,\infty}^{U_8}
\, , \, 
0
\, , \, 
- \boldsymbol{\mathfrak R}_{0,1,r_{10}}^{U_10}
\, \bigg) \\
 {\rm Res}_{\ell_5}\big({\bf HLog}^3\big)= \, & \bigg(\, 
\boldsymbol{\mathfrak R}_{0,1,\infty}^{U_1}
\, , \, 
0 
\, , \,  
 0 
\, , \,
0 
\, , \,
\boldsymbol{\mathfrak R}_{1,r_5,\infty}^{U_5}
\, , \,
-\boldsymbol{\mathfrak R}_{0,r_6,\infty}^{U_6}
\, , \,
0 
\, , \,
-\boldsymbol{\mathfrak R}_{0,1,r_8}^{U_8}
\, , \, 
\boldsymbol{\mathfrak R}_{1,r_9,\infty}^{U_9}
\, , \, 
0\, \bigg) \\
 {\rm Res}_{\ell_6}\big({\bf HLog}^3\big)= \, & \bigg(\, 
0 
\, , \, 
0 
\, , \,  
 -\boldsymbol{\mathfrak R}_{0,r_3,\infty}^{U_3}
\, , \,
\boldsymbol{\mathfrak R}_{1,r_4,\infty}^{U_4}
\, , \,
0 
\, , \,
0 
\, , \,
\boldsymbol{\mathfrak R}_{1,r_7,\infty}^{U_7}
\, , \,
- \boldsymbol{\mathfrak R}_{0,1,r_8}^{U_8}
\, , \, 
0
\, , \, 
-\boldsymbol{\mathfrak R}_{0,r_{10},\infty}^{U_{10}} 
\, \bigg) \\
 {\rm Res}_{\ell_7}\big({\bf HLog}^3\big)= \, & \bigg(\, 
0 
\, , \, 
- \boldsymbol{\mathfrak R}_{0,r_2,\infty}^{U_2}
\, , \,  
 0 
\, , \,
- \boldsymbol{\mathfrak R}_{0,r_4,\infty}^{U_4}
\, , \,
0 
\, , \,
- \boldsymbol{\mathfrak R}_{0,r_6,\infty}^{U_6}
\, , \,
- \boldsymbol{\mathfrak R}_{0,r_7,\infty}^{U_7}
\, , \,
0
\, , \, 
0
\, , \, 
\boldsymbol{\mathfrak R}_{1,r_{10},\infty}^{U_{10}}
\, \bigg) \\
 {\rm Res}_{\ell_8}\big({\bf HLog}^3\big)= \, & \bigg(\, 
0 
\, , \, 
0 
\, , \,  
 0 
\, , \,
0 
\, , \,
0 
\, , \,
\boldsymbol{\mathfrak R}_{0,1,\infty}^{U_{6}} 
\, , \,
\boldsymbol{\mathfrak R}_{0,1,\infty}^{U_{7}} 
\, , \,
\boldsymbol{\mathfrak R}_{0,1,\infty}^{U_{8}} 
\, , \, 
\boldsymbol{\mathfrak R}_{0,1,\infty}^{U_{9}} 
\, , \, 
\boldsymbol{\mathfrak R}_{0,1,\infty}^{U_{10}} \, \bigg) \\
 {\rm Res}_{\ell_9}\big({\bf HLog}^3\big)= \, & \bigg(\, 
0 
\, , \, 
0 
\, , \,  
 0 
\, , \,
\boldsymbol{\mathfrak R}_{0,1,\infty}^{U_{4}}
\, , \,
\boldsymbol{\mathfrak R}_{0,1,\infty}^{U_{5}}
\, , \,
\boldsymbol{\mathfrak R}_{1,r_6,\infty}^{U_{6}}
\, , \,
-\boldsymbol{\mathfrak R}_{0,1,r_7}^{U_{7}}
\, , \,
\boldsymbol{\mathfrak R}_{1,r_8,\infty}^{U_{8}}
\, , \, 
0
\, , \, 
0\, \bigg) \\
 \mbox{ and } \quad {\rm Res}_{\ell_{10}}\big({\bf HLog}^3\big)= \, & \bigg(\, 
0 
\, , \, 
0 
\, , \,  
\boldsymbol{\mathfrak R}_{0,1,\infty}^{U_{3}}
\, , \,
0 
\, , \,
-\boldsymbol{\mathfrak R}_{0,1,r_5}^{U_{5}}
\, , \,
0 
\, , \,
-\boldsymbol{\mathfrak R}_{0,r_7,\infty}^{U_{7}}
\, , \,
-\boldsymbol{\mathfrak R}_{0,r_8,\infty}^{U_{8}}
\, , \, 
-\boldsymbol{\boldsymbol{\mathfrak R}}_{0,r_9,\infty}^{U_{9}}
\, , \, 
0\, \bigg) \, .
\end{align*}

For any $i=1,\ldots,10$, there exist exactly 3 residues ${\rm Res}_{\ell_j}\big( {\bf HLog}^3\big)$ the $i$-th components of which are non zero. Moreover, these components  form a basis of
$\wedge^2 {\bf H}_i$. For instance, for $i=1$,  ${\rm Res}_{\ell_j}\big( {\bf HLog}^3\big)$ has a non zero first component only for $j=1,4,5$, and these components are 
 $\mathfrak R_{1,r_1,\infty}^{U_1}$, $ -\mathfrak R_{0,r_1,\infty}^{U_1}$ and $ \mathfrak R_{0,1,\infty}^{U_1}$ respectively, that is are the pull-backs under $U_i$ of the three 2-wedge products
 $$
 \frac{du}{u-1}\wedge  \frac{du}{u-r_1}\, , \qquad 
  - \frac{du}{u}\wedge  \frac{du}{u-r_1} \qquad \mbox{ and } \qquad 
  \frac{du}{u}\wedge  \frac{du}{u-1}
 $$
which obviously form a basis of $\wedge^2 {\bf H}^0\big( \mathbf P^1,\Omega_{\mathbf P^1}^1\big( 
{\rm Log}\,\mathfrak R_1\big)\Big)$. From this, one immediately deduces that the ten residues ${\rm Res}_{\ell_i}\big( {\bf HLog}^3\big)$'s are linearly independent. Since 
${\rm Res}_{\ell_i}\big( {\bf HLog}^3\big) \in \boldsymbol{AR}^2_{asym}\big(  \boldsymbol{\mathcal W}_{{\rm dP}_4}\big)$ and because this space is of dimension 10 (see the next subsection), it follows that 
\begin{equation}
\label{Eq:lomo}
{\rm Res}\Big( {\bf HLog}^3 \Big)=
\Big\langle \,  
{\rm Res}_{\ell_i}\big({\bf HLog}^3\big)\, \big\lvert \, i=1,\ldots,10
\, \Big\rangle=
 \boldsymbol{AR}^2_{asym}\big(  \boldsymbol{\mathcal W}_{{\rm dP}_4}\big)\, . 
\end{equation}

 \subsection{\bf The symbolic abelian relations of ${\mathcal W} \hspace{-0.46cm}{\mathcal W}_{\hspace{-0.01cm}  {\rm dP}_4
 \hspace{-0.41cm}
  {\rm dP}_4}$}
 \label{SS:Symbolic-ARs-WdP4}
Here we apply the approach of \S\ref{SSS:Weyl-group-action-on-ARs} to the web under scrutiny: 
${\bf H}$ stands for the space of rational 1-forms, regular on $\mathbf C^2\setminus A$ and with logarithmic singularities along the components of $A$ spanned  by the 
$h_i=d{\rm Log}({\mathcal L}_i)$ for $i=1,\ldots,10$. These are linearly independent hence $\dim \, {\bf H}=10$. For any $i$, let  ${\bf H}_i$ be the subspace of ${\bf H}$ with basis the 1-form $\eta_{i,s}=d{\rm Log}(U_i-\mathfrak r_{i,s})$ for $s=1,2,3$. 

For any weight $w \geq 1$, one defines a vector space of {\it `symbolic weigth $w$ hyperlogarithmic abelian relations'}, denoted by  $  {\bf HLogAR}^w= {\bf HLogAR}^w(\boldsymbol{\mathcal W}_{ {\rm dP}_4})$, 
by requiring that 
the following sequence of vector spaces be exact: 
\begin{equation}
\label{Eq:ExactSequence}
0 \rightarrow  {\bf HLogAR}^w\longrightarrow \oplus_{i=1}^{10} {\bf H}_i^{\otimes w} \longrightarrow {\bf H}^{\otimes w}\, .
\end{equation}

The symmetric group $\mathfrak S_w$ naturally acts  linearly on ${\bf H}^{\otimes w}$ by permuting the components of the pure tensors: one has 
$\sigma\cdot k_{1}\otimes \cdots \otimes k_{w}= k_{\sigma(1)}\otimes \cdots \otimes k_{\sigma(w)}$ for any $\sigma\in \mathfrak S_w$ and any $k_1\otimes \cdots \otimes k_w\in 
 {\bf H}^{\otimes w}$ and this action is extended linearly to the whole space ${\bf H}^{\otimes w}$. Each of the inclusions ${\bf H}_i\hookrightarrow {\bf H}$ is obviously $\mathfrak S_w$-equivariant, hence  \eqref{Eq:ExactSequence} is a sequence
 of $\mathfrak S_w$-representations, with 
 the map ${\bf HLogAR}^w\longrightarrow \oplus_{i=1}^{10} {\bf H}_i^{\otimes w} $ being injective 
 by the very definition of ${\bf HLogAR}^w$ for any $w\geq 1$. In particular for $w=2$, one has a decomposition in direct sum
 $$
   {\bf HLogAR}^2=
  {\bf HLogAR}^2_{\rm sym}\oplus  {\bf HLogAR}^2_{\rm asym}
 $$
 where $ {\bf HLogAR}^2_{\rm sym}$ (resp.\,$ {\bf HLogAR}^2_{\rm asym}$)
 stands for the symbolic weight 2 hyperlogarithmic ARs which are symmetric (resp.\,antisymmetric) with respect to the action of $\mathfrak S_2\simeq \{ \pm 1\}$.  
 \sk
 
 Using the explicit decomposition of the $\eta_{i,s}$ in the basis formed by the $h_j$ with $j=1,\ldots,10$, it is just a matter of linear algebra computations to express 
 any tensor $\eta_{i,s_1}\otimes \cdots \otimes \eta_{i,s_w}\in {\bf H}_i^{\otimes w}$ as a linear combination in the $h_{i_1\ldots i_w}=h_{i_1}\otimes \cdots 
 \otimes h_{i_w} \in {\bf H}^{\otimes w}$ for all tuples $(i_1,\ldots,i_w)$ such that $1\leq i_1\leq \cdots \leq i_w\leq 10$.  By direct elementary computations, we obtain the following points: 
\begin{itemize}
\item the space $ {\bf HLogAR}^1$ has dimension $20$ (coherently with Lemma \ref{Lem:lolo}). 
This space admits a basis formed by combinatorial ARs hence  the subspace of combinatorial abelian relations 
$\boldsymbol{AR}_C(\boldsymbol{\mathcal W}_{ {\rm dP}_4})$ coincides with that of logarithmic ones  
$\boldsymbol{AR}_{log}(\boldsymbol{\mathcal W}_{ {\rm dP}_4})
\simeq {\bf HLogAR}^1$;
\sk 
\item the spaces $ {\bf HLogAR}^2_{\rm sym}$ and $ {\bf HLogAR}^2_{\rm asym}$ have dimension $5$ and 10 respectively, with $ {\bf HLogAR}^2_{\rm asym}$ spanned by the residues of 
${\bf HLog}^3$ (see \eqref{Eq:lomo} above);
\sk
\item the space $ {\bf HLogAR}^3$ is  spanned by 
${\bf HLog}^3=\big( \eta_{i,1}\wedge \eta_{i,2} \wedge \eta_{i,3} \big)_{i=1}^{10} $ 
hence has dimension 1;
\sk
\item we have ${\rm rk}(\boldsymbol{\mathcal W}_{ {\rm dP}_4})=
\sum_{k=1}^3 \dim\,  \big( {\bf HLogAR}^k\big)=
 20+(10+5)+1=36$ hence 
$\boldsymbol{\mathcal W}_{ {\rm dP}_4}$ has maximal rank with all its ARs being hyperlogarithmic of weight  less than or equal to 3.
\end{itemize}

As explained in \S\ref{SSS:Weyl-group-action-on-ARs}, all the subspaces of ARs considered in the first three points above are naturally acted upon by the Weyl group $W=W(D_5)$. It is interesting to describe more precisely the structures of these spaces as $W$-representations. It is the purpose of the next subsection. 
\subsubsection{\bf  Subspaces of ARs as $W$-modules}
\label{SSS:subspaces-of-ARs-as-W-modules}
We recall some notations and introduce new ones
 which we will use here, in relation with the representations 
of the Weyl group $W$ and in particular its actions on the set of lines included in the fixed del Pezzo quartic surface we are working with.
 \begin{itemize}
 \item we use the notation of \S\ref{SS:dP-surfaces-webs-conics}:  for instance,  $\boldsymbol{\mathcal K}$ stands for the set of conic classes in the Picard lattice of the considered del Pezzo quartic, etc. We set 
 $\boldsymbol{\mathfrak c}_1=\boldsymbol{h}-\boldsymbol{\ell}_1$: it is the class of the privileged conic fibration we will consider (which is the one such that $b_*\lvert \,\boldsymbol{\mathfrak c} _1\, \lvert$ is the linear system of lines on the projective plane  passing through the point $p_1$);
 \sk
 \item then $ \boldsymbol{\mathfrak c}_1^{red}$ denotes the set of non irreducible conics in the  pencil $\lvert\boldsymbol{\mathfrak c}_0\lvert$: a bit abusively, one can write
$$ \boldsymbol{\mathfrak c}_1^{red}= \big\{\, (\boldsymbol{h}-\boldsymbol{\ell}_1-\boldsymbol{\ell}_i)+(\boldsymbol{\ell}_i) \, \lvert \, i=2,\ldots,5\, \big\}\, , $$
with the convention here that  $ (\boldsymbol{h}-\boldsymbol{\ell}_1-\boldsymbol{\ell}_i)+(\boldsymbol{\ell}_i) $ stands for the non irreducible conic whose two irreducible components are the two lines $h-\ell_1-\ell_i$ and $\ell_i$, this for any $i=2,\ldots,5$
\sk
\item we set 
$W_{\boldsymbol{\mathfrak c}_1}={\rm Fix}_W(\boldsymbol{\mathfrak c}_1)=\big\langle s_2,\ldots,s_5\big\rangle$. 
This subgroup of $W=W(D_5)=\big\langle s_1,\ldots,s_5\big\rangle$ is naturally isomorphic to 
 $ W(D_4)$ and naturally acts (by permutations) on $\boldsymbol{\mathfrak c}_1^{red}$;
 \sk
\item we use the classical labelling of irreducible representation of $W=W(D_5)$ or $W_{\boldsymbol{\mathfrak c}_1}=W(D_4)$ by means of bipartitions of $5$ and $4$ respectively ({\it e.g.} see \cite[\S5.6]{GP}).
\sk
\item we will  denote by ${\bf 1}$ the trivial representation of the considered Weyl group. For instance, 
we have 
 ${\bf 1}= V_{[.5]}^1$  when dealing with $W\simeq W(D_5)$; 
 \sk
 \item we denote by 
$\mathbf Z^{\boldsymbol{\mathfrak c}_1^{red}}$ the free $\mathbf Z$-module spanned by the elements of 
${\boldsymbol{\mathfrak c}_1^{red}}$ and by ${\bf H}_{\boldsymbol{\mathfrak c}_1}$ the submodule of elements with vanishing trace. Since $W_{\boldsymbol{\mathfrak c}_1}$ acts by permutations on ${\boldsymbol{\mathfrak c}_1^{red}}$, 
$\mathbf Z^{\boldsymbol{\mathfrak c}_1^{red}}$ naturally is a $W_{\boldsymbol{\mathfrak c}_1}$-representation  of which 
 ${\bf H}_{\boldsymbol{\mathfrak c}_1}$ is a direct factor. More precisely,   one has 
 $\mathbf Z^{\boldsymbol{\mathfrak c}_1^{red}}={\bf 1}\oplus {\bf H}_{\boldsymbol{\mathfrak c}_1}$  as $W_{\boldsymbol{\mathfrak c}_1}$-modules ({\it cf.}\,Lemma \ref{L:Fc-module}).
\end{itemize} 
 
\paragraph{\bf Weight 1 abelian relations}
A description of the $W$-module structure of ${\bf HLogAR}^1$ will follow from the 
\begin{lem} 
\label{Lem:Hc1}
{\rm 1.} As $W_{{\mathfrak c}_1}$-representations, one has 
$\mathbf Z^{{\mathfrak c}_1^{red}}={\bf 1}\oplus {\bf H}_{\mathfrak c_1}$ with ${\bf H}_{\mathfrak c_1}\simeq V^3_{[ 2,2]^{\textcolor{black}{-}}}$.
\sk

\noindent {\rm 2.} As $W$-representations, 
one has $ \oplus_{ \mathfrak c \in \boldsymbol{\mathcal K} } {\bf H}_{\mathfrak c}\simeq {\rm Ind}_{W_{\mathfrak c_1}}^W
\big( {\bf H}_{\mathfrak c_1} \big)\simeq V^{10}_{[2,3]}\oplus V^{20}_{[2,21]}$.
\end{lem}
\begin{proof} This is obtained 
by explicit computations in GAP3 (detailed in the Appendix). 
\end{proof}

Since  $\oplus_{ \mathfrak c \in \boldsymbol{\mathcal K} } {\bf H}_{\mathfrak c}
\rightarrow {\bf H}_{{\rm dP}_4} $ is non zero and because the target is an irreducible $W$-representation, it follows that this map is surjective. Hence 
for $w=1$, \eqref{Eq:ExactSequence} is an exact sequence and because we have vertical isomorphisms as follows
\begin{equation}
\label{Eq:HLogAR1-W-structure}
\scalebox{1}{
  \xymatrix@R=0.4cm@C=1.1cm{ 
 0 
\ar@{->}[r]  &   {\bf HLogAR}^1  \ar@{->}[r]   
&  
 \oplus_{ \mathfrak c \in \boldsymbol{\mathcal K} } {\bf H}_{\mathfrak c} \eq[d] 
  \ar@{->}[r] & 
  {\bf H}_{{\rm dP}_4} \eq[d]  \ar@{->}[r] &  0 \\
 &   
&  
V^{20}_{[2,21]}\oplus V^{10}_{[2,3]}   \ar@{->}[r] & 
 V^{10}_{[2,3]}   \ar@{->}[r] &  0 \, , 
}}
\end{equation}
%
we deduce 
that as a $W$-representation, ${\bf HLogAR}^1 $ is irreducible and isomorphic to $V^{20}_{[2,21]}$

\paragraph{\bf Weight 2 abelian relations}
%
%
Aiming to decompose into $W$-irreducibles the space of weight 2 hyperlogarithmic ARs
   ${\bf HLogAR}^2$, we first determine those of the two  $W$-modules 
  $\oplus_{ \mathfrak c \in \boldsymbol{\mathcal K} } {\bf H}_{\mathfrak c}^{\otimes 2} $ and 
  $  {\bf H}_{{\rm dP}_4}^{\otimes 2}   $.   Since $M^{\otimes 2}=M^{\wedge 2}\oplus M^{\odot}$ for any $W$-module $M$    (with $M^2={\rm Asym}(M)=\wedge^2 M$ and $M^{\odot 2}={\rm Sym}(M)$), one can deal with symmetric and antisymmetric weight 2 tensors separately. 
   
 We deal with the case of antisymmetric ARs first.  
From straightforward  computations   with GAP3 ({\it cf.}\,the Appendix),  we easily obtain the following decompositions into $W$-irreducible representations: one has 
\begin{equation}
\label{Eq:AsymAR2}
\wedge^2 {\bf H}_{\mathfrak c_1}=  
\wedge^2 V^3_{[2.2]^-}
\simeq 
V_{[11]^-}^{3}
\quad 
\mbox{ and } 
\quad 
\wedge^2 {\bf H}_{ {\rm dP}_4}\simeq 
\wedge^2 V_{[2.3]}^{10}=
V_{[11.21]}^{20}\oplus 
V_{[11.3]}^{10}\oplus 
 V_{[1.31]}^{15} \, .
\end{equation}

On the other hand, remarking that 
 $\oplus_{ \mathfrak c \in \boldsymbol{\mathcal K} } \wedge^2 {\bf H}_{\mathfrak c}\simeq 
 {\rm Ind}_{W_{\mc_1}}^W\big( \wedge^2 {\bf H}_{\mathfrak c_1} \big) $ and computing with GAP3 again, one gets
\begin{equation}
\label{Eq:II-AsymAR2}
\oplus_{ \mathfrak c \in \boldsymbol{\mathcal K} } \wedge^2 {\bf H}_{\mathfrak c}\simeq {\rm Ind}_{W_{\mc_1}}^W\big( \wedge^2 {\bf H}_{\mathfrak c_1} \big) 
\simeq   {\rm Ind}_{W_{\mc_1}}^W\big( V^3_{[11]^-} \big)=
V_{[11.111]}^{10}\oplus V_{[11.21]}^{20}\, . 
\end{equation}
%

Putting \eqref{Eq:AsymAR2} and \eqref{Eq:II-AsymAR2} together into the following diagram of $W$-representations
$$
\scalebox{1}{
  \xymatrix@R=0.4cm@C=1.1cm{ 
 0 
\ar@{->}[r]  &   {\bf HLogAR}^2_{\rm asym}  \ar@{->}[r]   
&  
 \oplus_{ \mathfrak c \in \boldsymbol{\mathcal K} } \wedge^2 {\bf H}_{\mathfrak c} \eq[d] 
  \ar@{->}[r] & 
  \wedge^2 {\bf H}_{{\rm dP}_4} \eq[d]  &   \\
  & 
&  
V_{[11.111]}^{10}\oplus V_{[11.21]}^{20}  \ar@{->}[r] & 
V_{[11.21]}^{20}\oplus 
V_{[11.3]}^{10}\oplus 
V_{[1.31]}^{15}\, , 
  &  
}}
$$
we deduce that $ {\bf HLogAR}^2_{\rm asym}$ is an irreducible $W$-module isomorphic to $V_{[11.111]}^{10}$. 
\mk 

We proceed similarly regarding the symmetric weight 2 hyperlogarithmic ARs. By computations with GAP3 (see the Appendix), one gets 
\begin{align*}
\big({\bf H}_{\mathfrak c_1}\big)^{\odot 2} \simeq  &\,  V^3_{[2
.2]^-}\oplus V^2_{[.22]}\oplus {\bf 1}\\ 
\mbox{ and }\quad 
\big({\bf H}_{{\rm dP}_4}\big)^{\odot 2}\simeq &\, 
 V_{[2.21]} ^{20}  \oplus 
V_{[1.22]}^{10} \oplus 
V_{[2.3]}^{10}\oplus 
V_{[.32]}^{5}\oplus 
V_{[1.4]}^5\oplus 
V_{ [.41]}^4\oplus 
{\bf 1}\, . 
\end{align*}

On the other hand, since 
\begin{align} 
{\rm Ind}_{W_{{\mathfrak c}_1}}^{W}\Big( \,  V^3_{[ 2,2]^-} \Big)=& \, V_{[2.3]}^{10}\oplus V_{[2.21]}^{20}
\nonumber
\\
{\rm Ind}_{W_{{\mathfrak c}_1}}^{W}\Big(   \,V^2_{[ ,22]}\, \Big)=& \, V_{[.32]}^{5}\oplus V_{[.221]}^{5}\oplus V_{[1.22]}^{10}
\label{Eq:III} \\ 
{\rm Ind}_{W_{{\mathfrak c}_1}}^{W}\big(  {\bf 1}\big)= & \, V_{[1.4]}^5\oplus V_{[.14]}^4 \oplus V_{[.5]}^1 
\nonumber
\end{align}
(as it follows from Table \ref{Table:Ind-WD4-WD5}), 
one deduces that  $ {\bf HLogAR}^2_{\rm sym}\simeq V_{[.221]}^{5}$
as a $W$-module.

\paragraph{\bf The weight 3 abelian relation}
In \cite{Hlog} (see also \cite{CP}), we proved that ${\bf HLog}^3$ can be identified 
with an explicit element $(\omega_{ \mathfrak c})_{\mathfrak c \in \boldsymbol{\mathcal K} }
$ of $ \oplus_{\mathfrak c \in \boldsymbol{\mathcal K} } \wedge^3 {\bf H}_{ \mathfrak c }$ which is $W$-stable and spans a 1-dimensional subrepresentation 
$\big\langle (\omega_{ \mathfrak c})_{\mathfrak c \in \boldsymbol{\mathcal K} } \big\rangle$ 
isomorphic to the signature representation ${\bf sign}_{D_5}$ of $W=W(D_5)$.
\begin{center}
$\star$
\end{center}

We gather all the results of the above considerations in the following 
\begin{prop} 
\begin{enumerate}
\item[] ${}^{}$ \hspace{-1.4cm} {\rm 1.}\, The Weyl group $W(D_5)$ naturally acts on the spaces of hyperlogarithmic ARs 
${\bf HLogAR}^1\big( \boldsymbol{\mathcal W}_{{\rm dP}_4}\big)$, 
${\bf HLogAR}^2_{\rm sym}\big( \boldsymbol{\mathcal W}_{{\rm dP}_4}\big)$,  
${\bf HLogAR}^2_{\rm asym}\big( \boldsymbol{\mathcal W}_{{\rm dP}_4}\big)$ and 
${\bf HLogAR}^3\big( \boldsymbol{\mathcal W}_{{\rm dP}_4}\big)$. \mk
\item[] ${}^{}$ \hspace{-0.8cm} {\rm 2.}\,  As $W(D_5)$-modules, the following isomorphisms hold true:  \mk 
\begin{itemize}
\item[] {\rm [weight 1]} \quad${\bf HLogAR}^1\simeq V_{[2,21]}^{20}$;\mk 
\item[] {\rm [weight 2]} \quad ${\bf HLogAR}^2_{\rm asym} \simeq V_{[11,111]}^{10}$ \quad   and \quad 
${\bf HLogAR}^2_{\rm sym}\simeq V_{[.221]}^{5}$;\mk 
\item [] {\rm [weight 3]} \quad ${\bf HLogAR}^3=\big\langle \, {\bf HLog}^3\, \big\rangle
\simeq V_{[\cdot,1^5]}^{1}={\bf sign}_{D_5}$.\mk
\end{itemize}
In particular, all these subspaces of ARs are irreducible $W(D_5)$-modules.
\end{enumerate}
\end{prop}

 \subsection{\bf Subwebs and combinatorial characterization of ${\mathcal W} \hspace{-0.46cm}{\mathcal W}_{ {\rm dP}_4
 \hspace{-0.4cm}
  {\rm dP}_4}$}
 \label{SS:WdP4-Combinatorial-characterization}
The purpose of the present subsection is  
to give a description of the remarkable subwebs of 
$\boldsymbol{\mathcal W}_{ {\rm dP}_4}$ and to deduce from this a combinatorial characterization for this web similar to those stated about Bol's web and Burau's one $\boldsymbol{\mathcal W}_{ {\rm dP}_3}$
in \S\ref{SS:Combinatorial-characterization}. \sk

Because our considerations below are invariant up to (local analytic) equivalence, we are going to work with a simple and quite concrete model for $\boldsymbol{\mathcal W}_{ {\rm dP}_4}$, namely its birational model 
on the projective plane obtained as its direct image $b_*\big( \boldsymbol{\mathcal W}_{ {\rm dP}_4}\big) $ 
by the blow-up $b: {\rm dP}_4\rightarrow \mathbf P^2$ at five points in general position.  We denote by $p_1,\ldots,p_5$ these points and by $P$ their set: $P=\{p_1,\ldots,p_5\}$. Then for $i=1,\ldots, 5$, we set 

$-$ $\mathcal L_i$ for the pencil of lines of the plane passing through $p_i$;  and 

$-$ 
$\mathcal C_{\widehat \imath}$ for the pencil of conics passing through all the $p_k$'s except $p_i$.
\sk

Then we have 
$$
b_*\big( \boldsymbol{\mathcal W}_{ {\rm dP}_4}\big) =
\boldsymbol{\mathcal W}_P=
\boldsymbol{\mathcal W}\Big( 
\, \mathcal L_1\, , \, \ldots, \mathcal L_5\, , \, 
 \mathcal C_{\widehat 1} \, , \, \ldots\, , \,  \mathcal C_{\widehat 5}\, 
\Big)
$$
and this latter web is the planar model for  $\boldsymbol{\mathcal W}_{ {\rm dP}_4}$ we are going to work with in this subsection.

To describe in an efficient way some subwebs of $\boldsymbol{\mathcal W}_P$, it will be useful to use the following alternative notations for its foliations: for any $i=1,\ldots, 5$, we set
$$
\mathcal F_i^{+}=\mathcal L_i\qquad \mbox{and} 
\qquad \mathcal F_i^{-}= \mathcal C_{\widehat \imath} \, . 
$$
It will be also convenient to use the following notations: for $\underline{\epsilon}=(\epsilon_i)_{i=1}^5\in \{ \pm 1\}^5$, we will denote by $p(\underline{\epsilon})\in \{0,\ldots,5 \}$ the number of indices $i$ such that $\epsilon_i=+1$; and $I$ will stand for $\{ 1,\ldots,5 \}$. 
\begin{center}
$\star$
\end{center}

By direct (but lenghty) computations, it is not difficult to establish the following list of remarkable subwebs of 
$\boldsymbol{\mathcal W}_P$: 
\begin{itemize}
\item {\bf [Hexagonal 3-subwebs of $\boldsymbol{{\mathcal W}_{P}}$]}: 
the subwebs of  $\boldsymbol{{\mathcal W}}_{P}$ of this kind are exactly those of the form 
$$
\boldsymbol{\mathcal W}\Big(
\,  \mathcal F^\epsilon_i,
\mathcal F^\epsilon_j\, , \, 
\mathcal F^\epsilon_k\, \Big) \quad \mbox{ and } 
\quad 
\boldsymbol{\mathcal W}\Big(\, 
\mathcal F^\epsilon_i\, , \,
\mathcal F^\epsilon_j\, , \, 
\mathcal F^{-\epsilon}_k\, \Big)
$$
where  $i,j,k $  stand for pairwise distinct elements of $I$ and $\epsilon=\pm$.  
There are $2{ 5 \choose 3}=20$  subwebs of the first type and 
$2\big( { 5 \choose 2}\times 3)=60$ subwebs of the second kind. Thus there are as many 
associated abelian relations $LogAR^\epsilon_{ijk}$ and $LogAR^{\epsilon,-}_{ijk}$. By elementary linear algebra, one gets that there  are precisely 60 linear relations between these 80 ARs, hence it follows that the length 3  logarithmic ARs of $\boldsymbol{{\mathcal W}}_{P}$ span a subspace of $\boldsymbol{AR}^1(\boldsymbol{{\mathcal W}}_{P})$, of dimension 20. Since 
it is precisely the dimension of both
 $\boldsymbol{AR}^1(\boldsymbol{\mathcal W}_{ {\rm dP}_4})$ and  
$ \boldsymbol{AR}^1(\boldsymbol{\mathcal W}_{ P})$, these spaces coincide  hence  one has 
$$
\Big\langle\, LogAR^\epsilon_{ijk}\, , \, LogAR^{\epsilon,-}_{ijk}
\, \Big\rangle
=\boldsymbol{AR}_C\big(\boldsymbol{ \mathcal  W}_{P}\big)
={\bf HLogAR}^1\big(\boldsymbol{\mathcal  W}_{P}\big)\, .
$$
The description above of the hexagonal 3-subwebs of 
$\boldsymbol{\mathcal W}_P\simeq \boldsymbol{\mathcal W}_{ {\rm dP}_4}$
fully characterizes the map $r_{ \boldsymbol{\mathcal W}_P, 3} = 
r_{  \boldsymbol{\mathcal W}_{ {\rm dP}_4} , 3}
$. 
\bk
\item  {\bf [Hexagonal 4-subwebs of $\boldsymbol{{\mathcal W}_{P}}$]}: 
the subwebs of $\boldsymbol{{\mathcal W}_{P}}$ of 
this type are exactly those of the form 
$$
\boldsymbol{\mathcal W}\Big(\,\mathcal F_{i_1}^{\epsilon_1}\, , \, \ldots 
\, , \,
\mathcal F_{i_4}^{\epsilon_4}\, 
\Big)
$$
with $1\leq i_1<\ldots <i_4\leq 5$ and $\underline{\epsilon}=(\epsilon_i)_{i=1}^4\in \{ \pm 1\}^4$. 
\bk
\label{papage}
\item {\bf [Hexagonal 5-subwebs of $\boldsymbol{{\mathcal W}_{P}}$]}:   the subwebs of this type are exactly those of the form 
$$
\boldsymbol{\mathcal W}\Big(\,\mathcal F_1^{\epsilon_1}\, , \, \ldots 
\, , \,
\mathcal F_5^{\epsilon_5}\, 
\Big)
$$
with   $\underline{\epsilon}=(\epsilon_i)_{i=1}^5\in \{ \pm 1\}^5$. 
The ones which are hexagonal  and linearizable hence equivalent to five pencils of lines,  exactly are the ones with $p(\underline{\epsilon})$ even, all  the others (with  $p(\underline{\epsilon})$ odd) being equivalent to Bol's web.   \sk

From above, one immediately deduces the following fact which we will use later: for any hexagonal 5-subweb 
$
\boldsymbol{\mathcal W}\big(\,\mathcal F_i^{\epsilon_i}\, 
\big)_{i=1}^5
$, the complementary subweb $
\boldsymbol{\mathcal W}\big(\,\mathcal F_i^{-\epsilon_i}\,
\big)_{i=1}^5
$ is hexagonal too, but of the opposite type: if the former web is linearizable, then the latter is equivalent to Bol's web, and vice versa. 
\bk
\item Finally, it can be verified that $\boldsymbol{{\mathcal W}_{P}}$ does not admit any hexagonal $k$-subweb for $k\geq 6$. 
\end{itemize}

\newpage
In view of giving a combinatorial characterization of $\boldsymbol{\mathcal W}_{ {\rm dP}_4}$, let us introduce the following terminology: 
\begin{itemize}
\item 
given  a smooth quartic del Pezzo surface ${\rm dP}_4$, we denote by 
$\langle{\rm dP}_4  \rangle$ the set of all   unordered projective configurations 
$
[\{p_1,\ldots,p_5\}]\in {\rm Conf}_5(\mathbf P^2)/\mathfrak S_5$ such that ${\rm dP}_4\simeq {\bf Bl}_{p_1,\ldots,p_5}(\mathbf P^2)$. It is  
(classically?)  
 known that $\langle{\rm dP}_4  \rangle$ actually is a singleton (for a recent reference, see (2) of \cite[Theorem 2.1]{Hosoh} for instance).\footnote{We think it is worth mentioning that there is another way of obtaining $\langle{\rm dP}_4  \rangle$: the anticanonical image of ${\rm dP}_4$ in $\mathbf P^4$ is cut out by a pencil of quadrics $P_{{\rm dP}_4}$ among which exactly five are singular. The non-ordered projective configuration formed by these five singular quadrics as elements of 
$P_{{\rm dP}_4}\simeq \mathbf P^1$ 
coincides with $\langle{\rm dP}_4  \rangle$ (see \cite[\S5]{AM}).}
\sk
\item
given a 10-web $\boldsymbol{\mathcal W}$, one will say that  
$r_{ \boldsymbol{\mathcal W},3}$ is of del Pezzo's type if, possibly up to a relabelling of its foliations, one has 
$r_{\boldsymbol{\mathcal W},3}=r_{ \boldsymbol{\mathcal W}_{ {\rm dP}_4}, 3}$.
\end{itemize}
\begin{thm}
{\rm 1.} Let $\boldsymbol{\mathcal W}$ be a planar 10-web such that $r_{ \boldsymbol{\mathcal W}, 3}$ be 
of del Pezzo's type.
  Then there exists a quartic del Pezzo surface ${\rm dP}_4$ such that $\boldsymbol{\mathcal W} \simeq \boldsymbol{\mathcal W}_{ {\rm dP}_4}$. \sk 
  
  \noindent {\rm 2.}   Given another smooth quartic del Pezzo surface ${\rm dP}_4'$, 
the following assertions are equivalent: 
\begin{itemize}
\item[-] \vspace{-0.2cm}
  the two webs 
  $\boldsymbol{\mathcal W}_{ {\rm dP}_4}$ and $\boldsymbol{\mathcal W}_{ {\rm dP}_4'}$ are analytically equivalent (as unordered webs);
\item[-]  the two projective configurations  $\langle{\rm dP}_4  \rangle$ and $\langle{\rm dP}_4'  \rangle$ coincide: one has $\langle{\rm dP}_4  \rangle= \langle{\rm dP}_4'  \rangle$;
\item[-]  the two considered del Pezzo quartic surfaces are isomorphic: one has ${\rm dP}_4  \simeq {\rm dP}_4' $.
\end{itemize}
\vspace{-0.1cm}
Actually, the following stronger result holds true: any local analytic equivalence between the del Pezzo webs $\boldsymbol{\mathcal W}_{ {\rm dP}_4}$ and $\boldsymbol{\mathcal W}_{ {\rm dP}_4'}$ is the restriction of a global isomorphism between ${\rm dP}_4$ and ${\rm dP}_4'$. 
\end{thm}
\begin{proof} The proof is by a case analysis, our two main tools being 
first the fact that for any $d\geq 4$, a linearizable $d$-web admits only one linearization up to composition by projective transforms (see \cite[Cor.\,6.1.9]{PP}); and second, Bol's theorem giving a description of hexagonal webs 
(cf.\,page \pageref{Page:BolsTHM} above).

Let  $\boldsymbol{\mathcal W}=\big( {\mathcal F}_1,\ldots,{\mathcal F}_{10}\big)$ be as in the first point of the statement.  One can and will moreover assume that 
the two combinatorial functions $r_{ \boldsymbol{W},3}$ and $r_{ \boldsymbol{\mathcal W}_{ {\rm dP}_4}, 3}$ coincide. In particular, the 5-subweb $\boldsymbol{\mathcal H}_5=\big( {\mathcal F}_1,\ldots,{\mathcal F}_{5}\big)$ is hexagonal hence according to Bol's theorem, there are two possibilities: either $\boldsymbol{\mathcal H}_5$ is equivalent to a web formed by 5 pencils of lines or 
 $\boldsymbol{\mathcal H}_5$ is equivalent to Bol's web. 

Let us first assume that the second alternative occurs. Thus one can find and one fixes (a priori local) holomorphic coordinates $x$ and $y$ such that  $\boldsymbol{\mathcal H}_5=
\boldsymbol{\mathcal W}(U_1,\ldots,U_5)$ 
with $U_1= x$, $U_2 = y$, $U_3=x/y$, $U_4=(y-1)/(x-1)$ and $ U_5={x(y-1)}/{(y(x-1))}$.
One denotes by $V_1,\ldots,V_5$ five other first integrals such that $
\boldsymbol{\mathcal W}= \boldsymbol{\mathcal W}(U_1,\ldots,U_5,V_1,\ldots,V_5)$ and with the corresponding function 
$r_{\boldsymbol{\mathcal W},3}$ not only of del Pezzo's type but coinciding exactly with 
$r_{\boldsymbol{\mathcal W}_{ {\rm dP}_4} ,3}$. 
Considering the description of the hexagonal 5-subwebs of 
$\boldsymbol{\mathcal W}_{P}\simeq \boldsymbol{\mathcal W}_{ {\rm dP}_4}$ given above, it follows that the subweb $\boldsymbol{\mathcal W}(U_1,\ldots,U_4,V_5)= \boldsymbol{\mathcal W}(x ,  y  ,  {x}/{y}  , 
(y-1)/(x-1) ,V_5)$ is hexagonal too. Because $\boldsymbol{\mathcal W}(x ,  y  ,  {x}/{y}  , 
(y-1)/(x-1) $ is linear,  the foliation defined by $V_5$ is necessarily a pencil of lines  when considered in the  coordinates $x,y$.  Similarly, the subweb 
$\boldsymbol{\mathcal W}(U_1, U_2,U_3,U_5,V_4)= \boldsymbol{\mathcal W}(x ,  y    , 
x/y, {x(y-1)}/{y(x-1)} ,V_4)$ is hexagonal. In the coordinates $X,Y$ defined by $X=1/x$ and $Y=1/y$, its first four foliations become the pencils of lines whose some first integrals are the functions $X$, $Y$, $X/Y$ and $(X-1)/(Y-1)$.  Since the foliations $\mathcal F_{4}$ and $\mathcal F_{9}$ defined by $U_4$ and $V_4$ are distinct, the latter is a pencil of lines in the coordinates $X,Y$, from which it follows that $\mathcal F_{V_4}$ is a pencil of conics when considered in the initial coordinates $x,y$. Arguing similarly, one gets that 
in these coordinates, one has 
$$
\boldsymbol{\mathcal W}\bigg( \, x\, , \, y \, , \,  \frac{x}{y} \, , \,  
\frac{y-1}{x-1}
\, , \, 
\frac{x(y-1)}{y(x-1)}
\,  , \, 
V_1\,  , \, \ldots\,  , \,  V_5\, 
\bigg)
$$
with $\mathcal F_{V_k}$ being a pencil of conics for $k=1,\ldots,4$ and 
$\mathcal F_{V_5}$ a pencil of lines. 

From the description of hexagonal 3-subwebs of $\boldsymbol{\mathcal W}$ which corresponds to the one for $\boldsymbol{\mathcal W}_P$ given above, 
it follows that for any triple of indices $(i,j,k)$ such that $1\leq i<j<k\leq 5$, the 3-subwebs  $\boldsymbol{\mathcal W}(U_i,U_j,V_k)$, $\boldsymbol{\mathcal W}(U_i,V_j,V_k)$ and 
$\boldsymbol{\mathcal W}(V_i,V_j,V_k)$ are hexagonal or equivalently, their Blaschke-Dubourdieu's curvatures  vanish identically. These vanishings give several PDEs which have to be satisfied by the $V_i$'s which,  in turn,  give rise to many polynomial relations satisfied by the coefficients  of the pencils of conics/lines that the $V_i$'s are.  It is straightforward to solve the corresponding polynomial system (using a computer algebra system) to get that  the web under scrutiny is of the form 
$$\scalebox{1.2}{$
\boldsymbol{\mathcal W}
\bigg(\, x\,  , \, y\,  , \,  \frac{x}{y}\,  , \,   \frac{y -1}{x -1}\,  , \, 
\frac{x \left(y -1\right)}{y\left(x -1\right)}\,  , \, 
\frac{ \left(x -y \right)(y-\beta) }{\left(y -1\right) (\beta x -\alpha y )}
\,  , \, 
\frac{ \left(x -y \right) \left(x -\alpha \right)}{\left(x -1\right) (\beta x  -\alpha  y)}
\,  , \, 
\frac{\left(y -1\right) \left(x -\alpha \right) }{\left(x -1\right) \left(y -\beta  \right)}
\,  , \, 
\frac{y\left(x -\alpha  \right) }{x \left(y -\beta \right)}
\,  , \, 
 \frac{x -\alpha}{y-\beta }\, \bigg)
$} $$
for some complex parameters  $\alpha$ and  $\beta$.  In other terms, one has $\boldsymbol{\mathcal W}=\boldsymbol{\mathcal W}_P$ for  $P=\{p_1,\ldots,p_5\}$ with $p_1=[1:0:0], \ldots, p_4=[1:1:1]$ and $p_5=[\alpha: \beta:1]$. This proves the first point of the theorem when the initial hexagonal 5-subweb considered is assumed  not to be linearizable.
\sk

Let us now consider the case when $\boldsymbol{\mathcal W}$ does not admit any 5-subweb equivalent to Bol's web (otherwise the analysis just above would apply and give that $\boldsymbol{\mathcal W}=\boldsymbol{\mathcal W}_P$  for some set $P$ of 5 points on $\mathbf P^2$). Then  our initially considered hexagonal subweb $\boldsymbol{\mathcal H}_5$ as well as its complement $\boldsymbol{\mathcal H}_5^c$, which by definition is the 5-web formed by the foliations of $\boldsymbol{\mathcal W}$ which are not in 
$\boldsymbol{\mathcal H}_5$, both are equivalent to 5 pencils of lines.  Arguing as above, 
 one could endeavour to show
 that 
there exist affine coordinates 
with respect to which 
both $\boldsymbol{\mathcal H}_5$ and $\boldsymbol{\mathcal H}_5^c$ are formed by pencils of lines.  But this is not possible: it would imply that $\boldsymbol{\mathcal W}=
\boldsymbol{\mathcal H}_5\boxtimes \boldsymbol{\mathcal H}_5^c$ is hexagonal 
or equivalently  that one has $r_{ \boldsymbol{\mathcal W},3}\equiv 1$, contradicting 
the hypothesis $r_{ \boldsymbol{\mathcal W},3}= r_{\boldsymbol{\mathcal W}_{ {\rm dP}_4} , 3}$ since $r_{\boldsymbol{\mathcal W}_{ {\rm dP}_4} , 3}\not \equiv 1$.
%
%

Given a linearizable hexagonal 5-subweb $\boldsymbol{\mathcal H}_5$ of 
$\boldsymbol{\mathcal W}_{ {\rm dP}_4}$, it admits a linear model formed by the pencils of lines 
through five points $q_{\boldsymbol{\mathcal H}_5,1},\ldots,q_{\boldsymbol{\mathcal H}_5,5}$  in general position in $\mathbf P^2$. 
It is not difficult to verify that 
the unordered configuration 
$
\big\{
q_{\boldsymbol{\mathcal H}_5,i}\big\}_{i=1}^5\in {\rm Conf}_5(\mathbf P^2)/\mathfrak S_5$  coincides with $\langle{ {\rm dP}_4} \rangle$
which immediately implies the first part of 2. That any  any local analytic equivalence between  $\boldsymbol{\mathcal W}_{ {\rm dP}_4}$ and $\boldsymbol{\mathcal W}_{ {\rm dP}_4'}$ is the restriction of a global isomorphism between ${\rm dP}_4$ and ${\rm dP}_4'$ follows from the fact that any linearizable  5-web admits only one linearization up to composition with a projective transformation (details are left to the reader). This ends the proof of the theorem.
\end{proof}

 \subsection{\bf Algebraization of 
${\mathcal W} \hspace{-0.46cm}{\mathcal W}_{ {\rm dP}_4
 \hspace{-0.4cm}
  {\rm dP}_4}$}
 \label{SS:WdP4-algebraization}
In this subsection we explain how the algebraization results of 
\S\ref{S:Algebraization-WdP5} generalize to the web $\WdPq$.

 \subsubsection{Algebraization of $\WdPq$ via its combinatorial abelian relations}
 \label{SSS:Via-Combinatorial-ARs-WdP4}
We now specialize the material of \S\ref{SS:algebraization-via-AR-C} to the case of 
$\WdPq$. 

By direct computations, one can determine a basis of 
$\boldsymbol{AR}_C(\boldsymbol{\mathcal W}_{{\rm dP}_4})$ from which one can deduce that 
the logarithmic 1-forms  $h_i$'s defined in \eqref{Eq:hi} (with $i=1,\ldots,10$) form a basis of 
$\boldsymbol{ar}_C(\boldsymbol{\mathcal W}_{{\rm dP}_4})$. It is then straightforward, by 
direct computations again, to get the following 
 \begin{prop}
 \label{P:Caract-WdP4-via-AR-C}
A  10-web  $\boldsymbol{\mathcal W}$ defined on a domain of 
$\mathbf C^2$
is equivalent to a del Pezzo's web $\WdPq$
if and only if  all the three following statements hold true: 
\begin{enumerate}
\item[{\rm 1.}] 
\vspace{-0.1cm}
one has $r=\dim \, \big( \boldsymbol{ar}_C(\boldsymbol{\mathcal W}_{{\rm dP}_4})\big)=10$;\sk 
\item[{\rm 2.}]  the map $\Psi_{\boldsymbol{\mathcal W}} $ has values in a linear subspace $\mathbf PW\subset \mathbf P^{{10 \choose 2}-1}$ of dimension 24 and ${\rm Im}(\Psi_{\boldsymbol{\mathcal W}} ) =\Psi_{\boldsymbol{\mathcal W}} (U)$ is included in 
 the image of the triple anticanonical embedding of 
 a quartic del Pezzo surface ${\rm dP}_4$ in $\mathbf P^{24}$: one has 
 $\Psi_{\boldsymbol{\mathcal W}} (U)\subset {\rm Im}\big( \Phi_{\lvert -3K_{{\rm dP}_4}\lvert } \big)\subset \mathbf P^{24}$;
 \sk
\item[{\rm 3.}] the two push-forward webs 
$\big(\Psi_{\boldsymbol{\mathcal W}}\big)_*\big(\boldsymbol{\mathcal W}\big)
$ and $ \big(\Phi_{\lvert -3K_{{\rm dP}_4}\lvert}\big)_*\big(  \boldsymbol{\mathcal W}_{{\rm dP}_4} \big)$ coincide.
\end{enumerate} 
 \end{prop}
%
%
%
\begin{proof}[Proof\,(rough sketch)]
It suffices to verify that the three points of the Proposition indeed hold true when 
the 10-web under scrutiny is equivalent to $\boldsymbol{\mathcal W}_{{\rm dP}_4}$. 
For such a web, we take the web  $\boldsymbol{\mathcal W}$ on $\mathbf C^2\setminus A_{\mathscr L}$ defined by the 
10 rational functions in \eqref{Eq:WdP4-Ui}. Then it is just a matter of computation to verify that 
the 
logarithmic 1-forms $h_i$ defined in \eqref{Eq:hi} (with $i=1,\ldots,10$)  form a basis of 
$\boldsymbol{ar}_C(\boldsymbol{\mathcal W})$, from which one gets that the map $\Psi_{\boldsymbol{\mathcal W}}$ is rational and induced by the linear system $\mathfrak L
$ of plane curves of degree nine having a point of multiplicity 3 at each of the points $p_i$ (cf.\,\ref{Eq:points-pi}).  The Proposition follows from the fact that $\mathfrak L$  precisely is the push-forward of $\lvert -3K_{{\rm dP}_4}\lvert $
by the blow-up map $b: {\rm dP}_4\rightarrow \mathbf P^2$ at the $p_i$'s
\end{proof}


\subsubsection{Algebraization of $\WdPq$ via a canonical map $\Phi_{{\mathcal W}}$}
 \label{SSS:Via-Canonical-Map-WdP4}
In order to generalize Proposition \ref{P:Algeb-Bol-2} to $\WdPq$, the most direct attempt 
would be 
to 
consider the associated full canonical map $\Phi_{\WdPq}$ which takes values into the 
7-dimensional moduli space $\boldsymbol{\mathcal M}_{10}$.  But it can be verified that 
the pull-back of the 210-web $\boldsymbol{\mathcal W}_{\hspace{-0.1cm}\boldsymbol{\mathcal M}_{0,10}}$ under $\Phi_{\WdPq}$ is a 80-web, hence is a web formed by a relatively important number of foliations. Therefore, it is possibly not the most convenient web to consider in view of stating a criterion characterizing $\WdPq$ as a  10-web. 

Instead of dealing with $\Phi_{\WdPq}$, we are going to consider the canonical map 
$\Phi_{\boldsymbol{\mathcal H}_5}$ associated to a fixed (but otherwise arbitrary) hexagonal 5-subweb $\boldsymbol{\mathcal H}_5$ of $\WdPq$.   By elementary (but lengthy) computations, one verifies that  $ \boldsymbol{\mathcal W}_{{\rm dP}_4}$ is the 10-web obtained by taking the juxtaposition 
of $\boldsymbol{\mathcal H}_5$ with the pull-back of $ \boldsymbol{\mathcal W}_{\hspace{-0.1cm}\boldsymbol{\mathcal M}_{0,5}}$ under $\Phi_{\boldsymbol{\mathcal H}_5}$: 
 one has 
$$
 \boldsymbol{\mathcal W}_{{\rm dP}_4}= \boldsymbol{\mathcal H}_5 \boxtimes 
\Phi_{\boldsymbol{\mathcal H}_5}^*\big( \boldsymbol{\mathcal W}_{\hspace{-0.1cm}\boldsymbol{\mathcal M}_{0,5}}\big)\, .
%
%
%
$$

This immediately gives us the following result which has to be seen as an invariant criterion characterizing del Pezzo's web $\WdPq$: 
\begin{prop}
\label{P:Algeb-WdP4-2} 
Let $\boldsymbol{\mathcal W}$ be a planar 10-web. The three following properties are equivalent: 
\begin{enumerate}
\item[{\rm 1.}] 
\vspace{-0.1cm} the web 
$\boldsymbol{\mathcal W}$ is equivalent to a del Pezzo's web $ \boldsymbol{\mathcal W}_{{\rm dP}_4}$;
\sk 
\item[{\rm 2.}]
there exists a hexagonal 5-subweb $\boldsymbol{\mathcal H}_5\subset \boldsymbol{\mathcal W}$ such that 
\begin{equation}
\label{Eq:tokul}
\boldsymbol{\mathcal W}= \boldsymbol{\mathcal H}_5 \boxtimes 
\Phi_{\boldsymbol{\mathcal H}_5}^*\big( \boldsymbol{\mathcal W}_{\hspace{-0.1cm}\boldsymbol{\mathcal M}_{0,5}}\big)\, ; 
\end{equation}
\item[{\rm 3.}] $\boldsymbol{\mathcal W}$ admits a hexagonal 5-subweb and 
\eqref{Eq:tokul} holds true 
for any such subweb 
$\boldsymbol{\mathcal H}_5$.
\end{enumerate}
\end{prop}
\begin{proof}
As mentioned before, that 1.\,implies 3 
can be verified by straightforward computations. Since 3.\,$\Rightarrow$ 1.\,is obvious, it suffices to prove that 2.\,implies 1. But this follows from the following two facts, which can be established by elementary computations as well: (1)  a hexagonal 5-web $\boldsymbol H$ is such that $\boldsymbol{\mathcal H} \boxtimes 
\Phi_{\boldsymbol{\mathcal H}}^*\big( \boldsymbol{\mathcal W}_{\hspace{-0.1cm}\boldsymbol{\mathcal M}_{0,5}}\big)$ is a 10-web if and only if there exists a local biholomorphism 
$\varphi$ such that ${\bf H}=\varphi_*(\boldsymbol H)$ is the linear web formed by  five pencils of lines whose vertices $v_1,\ldots,v_5$ are five points in general position in $\mathbf P^2$; (2) in this case $\boldsymbol{{\bf H}} \boxtimes 
\Phi_{{\bf H}}^*\big( \boldsymbol{\mathcal W}_{\hspace{-0.1cm}\boldsymbol{\mathcal M}_{0,5}}\big)$ is the 10-web formed by ${\bf H}$ together with the five pencils of conics passing through all the $v_i$'s except one, that is $\boldsymbol{{\bf H}} \boxtimes 
\Phi_{{\bf H}}^*\big( \boldsymbol{\mathcal W}_{\hspace{-0.1cm}\boldsymbol{\mathcal M}_{0,5}}\big)$ is the direct image by the blow-up map of the del Pezzo's web on the del Pezzo quartic surface 
${\rm dP}_4={\bf Bl}_{v_1,\ldots,v_5}(\mathbf P^2)$. 
\end{proof}

As in the case of Bol's web ({\it cf.}\,the end of \S\ref{SSS:Via-Canonical-Map-WdP5}), an effective criterion can be deduced from the above proposition for
 a web $\boldsymbol{\mathcal W}(U_1,\ldots,U_d)$ to be equivalent to a del Pezzo web $\WdPq$ (the explicit statement of this criterion being a bit involved, it is left to the interested reader). 
\mk 

To end this section, we mention that several preliminary experiments led us to realize that one can obtain many webs carrying interesting ARs (in particular exceptional webs) under the form 
$$ \boldsymbol{\mathcal H}_k \boxtimes 
\Phi_{\boldsymbol{\mathcal H}_k}^*\big( \boldsymbol{\mathcal W}_{\hspace{-0.1cm}\boldsymbol{\mathcal M}_{0,k}}\big)
$$
where $\boldsymbol{\mathcal H}_k$ is ranging the space of planar hexagonal 
$k$-webs with $k\geq 4$.  See \S\ref{SS:H-Phi-H(WM0N)} at the end where some interesting cases are considered. We believe that there 
is much more to be understood regarding the notion of canonical map $\Phi_{\boldsymbol{W}}$ of a web  $\boldsymbol{W}$, especially when the latter is hexagonal. We hope to investigate this further in a future work.

 \subsection{\bf Geometric construction of 
${\mathcal W} \hspace{-0.46cm}{\mathcal W}_{ {\rm dP}_4
 \hspace{-0.4cm}
  {\rm dP}_4}$ 
  \textit{\textbf{\`a la}} Gelfand-MacPherson}
 \label{SS:WdP4-GM}
This subsection is devoted to establishing that 
there is a geometric construction of $\boldsymbol{\mathcal W}_{ {\rm dP}_4}$ which is similar to 
Gelfand-MacPherson's 
one  of Bol's web $\boldsymbol{\mathcal B}\simeq \boldsymbol{\mathcal W}_{ {\rm dP}_5}$ recalled above  in \S\ref{SS:Gelfand-MacPherson--construction-WdP5}.  
However the case of $\boldsymbol{\mathcal W}_{ {\rm dP}_4}$ is more involved than the latter one and we need to  introduce some material and recall some results about it first.

In the first subsection below, building on some work  by  Gelfand and some of its collaborators, 
for any simple Lie group $G$, we sketch a general notion of Gelfand-MacPherson web on any rational homogeneous space $G/P$ equivariantly embedded in the projectivization 
$\mathbf PV_\omega$ 
of an irreducible $G$-representation $V_\omega$.  Since we will only consider a specific case in what follows, we are very sketchy and do not give many details (nor proofs). 
\sk

In the next subsection, after having recalled some results about the embedding of the Cox variety of del Pezzo surfaces into certain homogeneous spaces, 
we explain how the web  $\WdPq$ can be obtained from 
Gelfand-MacPherson's web on the tenfold spinor variety $\mathbb S_{5}\subset \mathbf P^{15}$. 

 \subsubsection{Webs associated to moment polytopes of projective homogeneous spaces}
 \label{SSS:GM-webs}
 The material presented below is taken from \cite{GS}.  The case of type $A$ (which corresponds to that of usual grassmannian manifolds) has been considered before by Gelfand and MacPherson in  \cite{GM}.  We will also use freely some results of  
\cite{Atiyah} and  
\cite[\S3]{Vinberg}.  The presentation below is quite succinct, 
more details and proofs will be provided in the paper in preparation \cite{PirioGMWebs}. 

Let $D$ be a connected Dynkin diagram of rank $r$ and  $G=G(D)$  a  simple Lie group of type $D$. We fix a Cartan torus $H$ with Lie algebra $\mathfrak h$ and set of roots $\Pi\subset \mathfrak h^*_{\mathbf R}$ and choose a subset of positive roots $\Pi^+$ which in turn defines the associated subset of 
simple roots $\Phi=\{ \alpha_1,\ldots,\alpha_r\}\subset \Pi^+ $ which is naturally in bijection with the nodes of $D$. We denote by  $\omega_1,\ldots,\omega_r$ the fundamental weights of the considered root system.  We fix a standard parabolic subgroup $P\subset G$  and we set $I(P)$ for the subset of $\{1,\ldots,r\}$ whose elements are the indices $i$  such that $\alpha_i$ is also a root of $P$.  Then $\omega=\omega(P)=\sum_{i\in I(P)} \omega_i$ is a dominant weight and we denote by 
$V=V_\omega$ the $G$-module such that the representation $G\rightarrow {\rm GL}(V_\omega)$ is irreducible with highest weight $\omega$.  
One has ${\rm Stab}_G(e_\omega)=P$ 
for a highest weight vector $e_\omega\in V_\omega$, hence we get an 
 $G$-equivariant  embedding $\rho=\rho_\omega: G/P\hookrightarrow \mathbf PV_\omega$.  The image $X=\rho(G/P)$ is a $G$-homogeneous projective variety which is the projectivization of the $G$-orbit of the  highest-weight vector $e_\omega$.  Let  $\mathfrak W$ be the set of weights in $V$, their multiplicities being taken into account, and let $(e_w)_{w\in  \mathfrak W}$ be a weight basis of $V$. Then for 
   $x\in G/P$, there exists  
a $\mathfrak W$-tuple of complex numbers 
$(p^w(x))_{w\in \mathfrak W}\in \mathbf C^{\mathfrak W}$, unique up to a non-zero rescaling, such that $$\rho(x)=\bigg[\sum_{ w \in \mathfrak W} p^w(x)e_w \bigg]\in \mathbf PV$$ where $\big[\cdot \big] : V\setminus \{0\}\rightarrow \mathbf P V$ stands for the projectivization map. The 
$p^w(x)$'s for $w\in \mathfrak W$ are  called 
the  `{\it generalized Pl\"ucker coordinates}' of $x$. 
\sk

Let $W=N_G(H)/H$ be the associated Weyl group, which acts orthogonally on $\mathfrak h^*_{\mathbf R}$ leaving invariant the considered root lattice. 
The $W$-orbit $W\cdot \omega$ of the highest weight  
is included in $\mathfrak W$ and the points of $W\cdot \omega$ are the vertices of the so-called {\it weight}  or {\it moment polytope}
$$\Delta=\Delta_{D,\omega}\subset \mathfrak h^*_{\mathbf R}\simeq  \mathbf R^r\,.$$
 The elements of $\mathfrak W\setminus W \hspace{-0.04cm} \cdot \hspace{-0.04cm}\omega$ all lie inside $\Delta$.    In the two cases we are interested in in this paper, 
 the Weyl group $W$ acts transitively on $\mathfrak W$, that is the representation $V_\omega$ is `{\it minuscule}'. From now on, 
 essentially to simplify the exposition below, 
 we assume that $V_\omega$ is of this kind.  This implies in particular that the considered highest weight $\omega=\omega(P)$  is one of the fundamental weights $\omega_i$ hence corresponds to one of the nodes of the Dynkin diagram $D$. Equivalently, $P$ is one of the $r$ maximal standard parabolic subgroups of $G$.
 To exclude trivial situations, we also require that the dimension of $H$ is strictly less than the one of $G/P$, this in order that quotienting the latter space by the action of $H$ gives rise  to a positive dimensional quotient   (in some sense to be made precise). This amounts to leave aside the minuscule pairs of type $(A_n,\omega_1)$
 and $(A_n,\omega_n)$ for which one has $X=G/P\simeq \mathbf P^n$ which is such that $X/H$ is a point. \sk

By definition, the associated  {\it moment map} $\mu=\mu_X=\mu_{D,\omega}$ is the map 
\begin{align*}
\mu : X & \longrightarrow \mathfrak h^*_{\mathbf R}\,,\\
\hspace{0.15cm} 
x &  \longmapsto   \frac{\sum_{ w\in \mathfrak W} \,\lvert 
\,{p}^w(x)\,\lvert^2\,w}{  \sum_{ w\in \mathfrak W}\, \lvert 
\,{p}^w(x)\,\lvert^2}\, .
\end{align*}
It is a real-analytic map 
which 
satisfies nice convexity properties. 
Let   $H_{>0}\simeq (\mathbf R_{>0})^r$ stand for the connected component 
of the identity of the split real form $H(\mathbf R)$ of the Cartan torus $H$. 
Then  the following hold: 
\begin{itemize}
\item 
\vspace{-0.15cm}
one has $\mu(X)=\Delta$; 
\sk 
\item $X^*=\mu^{-1}(\mathring{\Delta})$ is a  Zariski-open subset of $X$ on which the moment map is regular (a real-analytic submersion of rank $r$ at any $x^*\in X^*$); 
\sk 
\item  for any $x^*\in X^*$, $\mu$ induces a real-analytic isomorphism between the real orbit 
$H_{>0}\!\cdot\! x^*$ and the interior $\mathring{\Delta}$ of the weight polytope, which extends to an isomorphism of real analytic varieties with corners  to the closures: $\mu : \overline{H_{>0}\!\cdot\! x^*}\stackrel{\sim}{\rightarrow} \Delta$;
\sk 
\item for $\sigma\in N_G(H)$, let $[\sigma]$ be the corresponding element in $N(H)/H\simeq W$. 
Then the vertex $[\sigma]\cdot \omega$ of $\Delta$ is the image by $\mu$ of the point of $X$ corresponding to the coset  $\sigma\cdot P$ viewed as a point of $G/P$
\sk 
\item the $H$-action on $X^*$ \textcolor{red}{is sufficiently nice for} $\boldsymbol{\mathcal Y}^*=X^* /H$ to be a smooth quasi-projective manifold with the quotient map $\chi_H: X^*\rightarrow \boldsymbol{\mathcal Y}^*$ being a $H$-torsor. Moreover, the quotient $\boldsymbol{\mathcal Y}^*$ is rational.\footnote{This has been explained to us by N. Perrin.}
\sk 
\end{itemize}
\begin{rem}
The quotient $\boldsymbol{\mathcal Y}^*=X^*/H$ is not the most natural one to be considered from the point of view of geometric invariant theory.  Indeed, in \cite{SerganovaSkorobogatov}, the authors show that in some cases (which include in particular the `minuscule cases' we are interested in), the set $X^{s}$ of 
$H$-stable  points of $X$ contains all the $x\in X$ for which there exists at most one weight $w$ such that  ${p}^w(x)=0$. Setting $X^{sf}=X^s\cap X^f$ where 
$X^f$ stands for the Zariski open subset of points $x\in X$ such that ${\rm Stab}_H(x)=Z(G)$, one has ${\rm codim}(X\setminus X^{sf})\geq 2$ and 
$X^*$ is strictly contained in $X^{sf}$.  The quotient $\boldsymbol{\mathcal Y}=X^{sf}/H$ 
is a smooth quasi projective variety 
 containing $\boldsymbol{\mathcal Y}^*$ as a Zariski open subset and enjoying nice properties\footnote{For instance: the Weyl group $W$ embeds into the automorphisms group  of  
 $\boldsymbol{\mathcal Y}$ (see \cite[Thm\,2.2]{Skorobogatov}) and the Picard lattice of $\boldsymbol{\mathcal Y}$ is naturally isomorphic to that of a del Pezzo surface obtained as the blow-up of $\mathbf P^2$ at $r$ points in general position, cf. \cite[p.\,418]{SerganovaSkorobogatov}.} which are a priori nicer than those satisfied by $\mathcal Y^*$. Depending on what one is aiming for, it might be more interesting to work with  $\boldsymbol{\mathcal Y}$ instead of 
 its dense open subset $\boldsymbol{\mathcal Y}^*$, but  it will not be the case in this text. 
\end{rem}

The main objects we are going to work with are the (closed) faces of the weight polytope $\Delta$, especially its {\it `facets'} that is its faces of codimension 1. We now recall some facts/results of \cite[\S3]{Vinberg} which will be relevant regarding our purpose. The Weyl group $W$ acts on the set $\mathfrak F(\Delta)$ of faces of $\Delta$. Given such a face $F$, let $W_F$ be the subgroup of ${\rm Stab}_W(F)$ generated by the transformations $w\in W$ leaving $F$ invariant and such that the mirror hyperplane ${\rm ker}(w)\subset \mathfrak h_{\mathbf R}^*$ passes through the center of mass $m_F$ of $F$ but does not contain the whole face (equivalently: $m_F$ belongs to ${\rm ker}(w)$ which is orthogonal to $F$).   
Then $W_F$ acts on the affine span   $\langle F\rangle$ with $m_F$ as unique fixed point. 
Making of  $\langle F\rangle$  a vector space by taking $m_F$ as origin, one obtains that 
$\Pi_F=\Pi\cap \langle F\rangle$ is a root system  of rank $\dim(F)$ and $W_F$ identifies as its Weyl group. Moreover, 
$W_F$ acts transitively on the set of vertices of $F$.

We will be essentially interested in the facets of the weight polytope $\Delta$. These faces are weight polytopes of rank $r-1$ whose types can be described as follows: let $C\subset \mathfrak h^*_{\mathbf R}$ be the dominant Weyl chamber associated to the chosen set of simple roots $\Phi$.  Given a facet $F$, there exists $w\in W$ such that $w\,F$ is a `{\it dominant facet}' (in Vinberg's terminology), that is a 1-codimensional face of $\Delta$ 
such that its intersection with $C$ has dimension $r-1$. Hence in order to describe the different facets of $\Delta$ one can restrict to study the dominant ones. 

 One proceeds as follows to get a complete list of the Dynkin type of the dominant facets of $\Delta$: let $D_\omega$ be the Dynkin diagram $D$ with the vertex corresponding to $\omega$ in black, all the other vertices being white.\footnote{Since the representation under consideration is minuscule, the corresponding dominant weight $\omega$ is fundamental hence naturally identifies to one of the nodes of the Dynkin diagram.}  Then the marked Dynkin diagrams associated to the dominant facets are those, noted by $D_{\omega,\omega'}$  obtained 
by removing an extremal vertex $\omega'\neq \omega$ from $D_\omega$ (see the figure below for the case relevant to describing $\WdPq$ \`a la Gelfand-MacPherson further on). 
Given an arbitrary facet $F$, let $w\in W$ such that $wF$ be dominant. 
One sets $D_F=D_{\omega,\omega'}$ where the latter 
is the rank $r-1$ Dynkin diagram with a blacked vertex associated to $wF$, according to the procedure described just above and one denotes by $\omega_F$ the marked vertex of $D_F$. One verifies that the pair $(D_F,\omega_F)$ is well-defined, that is does not depend 
 on the considered  element $w$ of the Weyl group such that $wF$ be dominant. 
 \sk
 \begin{figure}[h!]
\begin{center}
\scalebox{2.1}{
 \includegraphics{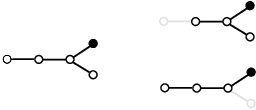}}
 \vspace{0.17cm}
\caption{The marked Dynkin diagram of type $(D_5,\omega_4)$ on the left and the two possible types of rank four 1-marked Dynkin diagrams for the facets on the right: of type $(D_4,\omega_3)$  (up) and 
$(A_4,\omega_4)$ (bottom).}
\label{Fig:Marked_Dynkin_Diagram_D5_omega4}
\end{center}
\end{figure}

We are now going to associate geometric objects (homogeneous varieties, rational maps, etc)  to the faces of $\Delta$.   Actually, since one is mainly interested in facets in this paper, we will restrict to this case below, even if most things could be stated in much more generality.

Given  a 1-codimensional face $F$ of $\Delta$,  always assumed to be topologically closed, we set/define: 
\begin{itemize}
\item  $V_F=\oplus_{w\in F\cap \mathfrak W} \,\mathbf C e_w$ and $V^F=\oplus_{w\in  \mathfrak W 
\setminus F}  \mathbf C e_w$ so that $V=V_F\oplus V^F$; 
\sk 
\item  $\Pi_F  : V=V_F\oplus V^F\rightarrow V_F$ is the linear projection 
with kernel $V^F$; 
\sk
\item  $\Gamma^F={\rm Stab}_\Gamma(V_F\oplus V^F)= \big\{\ \gamma \in \Gamma \ \big\lvert \  \gamma(V_F)= V_F \ \mbox{ and }\,  \gamma(V^F)= V^F \ \big\}$ for  $\Gamma $ one of the groups $G$, $P$ or $H$; note that since 
$H$ is the Cartan torus of $G$ and because $\mathfrak W$ is a weight basis, 
 one has $H^F=H$; 
\sk
\item $K^F=\big\{ \ g\in G^F \hspace{0.1cm} \lvert  \hspace{0.1cm}
\exists\ \lambda\in \mathbf C^*\hspace{0.1cm} \mbox{ s.t. } \hspace{0.1cm} 
   g \lvert_{V_F}=\lambda \,{\rm Id}_{V_F}
  \ \big\}$. 
  There is a sequence of inclusions of subgroups $K^F\subset P^F\subset G^F$;
  \sk
\item $\Gamma_F=\Gamma^F /\big( K^F \cap \Gamma^F\big)$ for $\Gamma$ 
being $G$, $P$ or $H$. More explicitly, one has 
$$G_F=G^F/K^F\, , \qquad H_F=H/\big(H\cap K^F\big)\qquad  \mbox{and}
\qquad  
P_F=P^F/K^F
\, .$$
\item  $X_F=X\cap \mathbf PV_F$
 and $X_F^*=\mu^{-1}\big(\mathring{F} \big)$ where $\mathring{F}$ stands for the relative interior of $F$. 
\end{itemize}

With these notations, one can state the
\begin{prop}
\label{Prop=XF=GF/PF}
\begin{enumerate}
\item[] ${}^{}$ \hspace{-1.1cm}{\rm 1.} The group $G_F$ is a simple complex Lie group of type $D_F$, 
$H_F$ is a Cartan subgroup and $P_F$ is the minimal standard parabolic subgroup of $G_F$ associated to the fundamental root $\omega_F$. 
\sk 
\item[2.] As a $G_F$-representation, $V_F$ is  irreducible and is the irrep of 
 highest weight  $\omega_F$. 
\sk
\item[3.] The variety $X_F$ is homogeneous under the natural action of $G^F$ and this action factorizes through the quotient map $G^F\rightarrow G_F$. Moreover, one has $X_F\simeq G^F/P^F=G_F/P_F$. 
\sk
\item[4.]
There is a natural affine isomorphism between the affine span of $F$ in $\mathfrak h_{\mathbf R}^*$ and 
the dual of the real Lie algebra of the torus $H_F$ such that the restriction of the moment map on $X_F$ identifies with that on $X_F$ associated to the triple $(G_F,P_F,V_F)$, {\it i.e.} one has $\mu\lvert_{X_F}=\mu_{D_F,\omega_F}$.  
In particular, $F$ identifies with the moment polytope $\Delta_{D_F,\omega_F}$ of $X_F$.
\end{enumerate}
\end{prop}

 Both the projectivization $\mathbf PV\dashrightarrow \mathbf P V_F$ of the  linear projection
 $\Pi_F$ and its restriction on  $X=G/P\subset \mathbf PV$ will be again denoted the same in what follows. 
\begin{prop}
\label{Prop:X->XF}
\begin{enumerate}
\item[] ${}^{}$ \hspace{-1.3cm}{\rm 1.}  The image of the projection $\Pi_F: X\dashrightarrow \mathbf PV_F$ coincides with $X_F$. 
 In short: 
 $$
X_F\stackrel{def}{=}X\cap \mathbf PV_F= \Pi_F( X)= G_F/P_F\, .
$$
Moreover, in terms of the moment map, one has $X_F=
\mu^{-1}\big({F}\big)= 
\overline{X_F^*}$.
\mk 
\item[2.] The induced dominant rational map $\Pi_F: X\dashrightarrow X_F$ is $(G^F,G_F)$-equivariant. 
\mk 
\item[3.] The map $ \Pi_F$ is defined at any point of $X^*$, one has $\Pi_F(X^*)=X_F^*$ and the induced surjective morphism $\Pi_F: X^*\rightarrow X_F^*$ is $(H,H_F)$-equivariant. 
\end{enumerate}
\end{prop}

We postpone the proofs 
of the  two above propositions to the paper to come \cite{PirioGMWebs}. Here we make only the following general commentaries: 
 most of the points of the proposition \ref{Prop=XF=GF/PF} follow quite directly from the results of \cite[\S3]{Vinberg}, the one requiring the most effort being the third point.
 
 As for Proposition \ref{Prop:X->XF}, if $F$ is a facet of the weight polytope, its affine span in $\mathfrak h_{\mathbf R}^*$ admits an equation of the form $\varphi_F=\alpha_F$ where $\varphi_F$ is a non zero linear form and $\alpha_F$ a certain real number.  Because the dual of $\mathfrak h_{\mathbf R}^*$ identifies with 
 $\mathfrak h_{\mathbf R}$, there exists $ \zeta_F\in  \mathfrak h_{\mathbf R}$ 
 such that $\varphi_F (\cdot)=(\zeta_F, \cdot)$ as linear forms on 
 $\mathfrak h_{\mathbf R}^*$.  Since  $ \zeta_F\neq 0$, 
 there is a one-parameter subgroup 
 $h_F: \mathbf C^*\rightarrow H, t\mapsto \exp(t \zeta_F)$
 associated to it.  Then most of Proposition \ref{Prop:X->XF} follows from Bia\l{}ynicki-Birula theory for the $\mathbf C^*$-action on $X$ provided by  $h_F$.  For instance, $X_F$ 
 is an irreducible component of the 
 subvariety $X^{h_F}$ of $X$ formed by its $h_F$-invariant points, 
  one has 
$\Pi_F(x)= \lim_{t\rightarrow 0}h_F(t )\cdot x$ for any $x\in X^*$, etc. 
 More details will be given in a forthcoming text. 
  

From the last point of the preceding proposition one immediately gets the 
\begin{cor}
\label{C:cocor}
For any facet $F$ of the moment polytope, there exists a well-defined surjective morphism $\pi_F:  \boldsymbol{\mathcal Y}^*\rightarrow 
 \boldsymbol{\mathcal Y}^*_F=X_F^*/H_F$ such that the following diagram commutes: 
$$\xymatrix@R=0.8cm@C=1.3cm{
X^*  \ar@{->}[d]_{\chi_H}  \ar@{->}[r]^{\Pi_F }   
& \ar@{->}[d]^{\chi_{H_F}}  X_F^* \\
\boldsymbol{\mathcal Y}^*  \ar@{->}[r]_{\pi_F }   
& \boldsymbol{\mathcal Y}^*_F}
$$
\end{cor}

\begin{rem}
The face maps $\Pi_F$ and $\pi_F$ considered above are specific cases of more general mappings 
associated with projective varieties acted upon by a torus already considered in the literature, {\it e.g.}\,in \cite{GoM} or in \cite{BrionProcesi}. 
\end{rem}

The map $\pi_F:  \boldsymbol{\mathcal Y}^*\rightarrow 
 \boldsymbol{\mathcal Y}^*_F$  is constant hence defines 
 a trivial foliation on $\boldsymbol{\mathcal Y}^*$ if and only if $ \boldsymbol{\mathcal Y}^*_F$ 
 reduces to a point.  Such a case is not interesting for our purpose and a facet of this kind will be said {\it `$\boldsymbol{\mathcal W}$-irrelevant'}. Accordingly, facets not of this type will be said to be {\it `$\boldsymbol{\mathcal W}$-relevant'}. 
\begin{exm}
\label{Ex:GM-Web-G2C5}
The case of the pair $(A_4,\omega_2)$ is considered in depth in \cite{GM} (see also 
\S\ref{SS:WdP4-GM} above). In this case,  one has $X=G/P=G_2(\mathbf C^5)$ and the image $\mu(X)\subset \mathbf R^5$ of the associated moment map is the  hypersimplex $\Delta_{5,2}$ formed by the 5-tuples 
of elements in $[0,1]$ summing up to 2:  one has $\Delta_{5,2}=\big\{\,  
(t_i)_{i=1}^5\in [0,1]^5\, \big\lvert \, 
\sum_{i=1}^5 t_i = 2\, \big\}$. This polytope has 10 facets, which are the intersections 
$\Delta_{5,2}(i,\epsilon)=\Delta_{5,2}\cap\{ t_i=\epsilon\}$ for $i=1,\ldots,5$ and $\epsilon=0,1$. 
For any $i$, the linear projection from $\mathbf R^5$ to $\mathbf R^4$ given by disregarding the $i$th coordinate induces an isomorphism from the facet $\Delta_{5,2}(i,1)$ onto a 3-simplex in $\mathbf R^4$, which reflects the fact that $X_{\Delta_{5,2}(i,1)}\simeq \mathbf P^3$. Moreover
the action of $H_{ \Delta_{5,2}(i,1)}\simeq \big(\mathbf C^*)^3$ on it identifies with the usual toric action hence it follows that $\boldsymbol{\mathcal Y}^*_{ \Delta_{5,2}(i,1) }$ is a point. 
This shows that the facets $\Delta_{5,2}(i,1)$ ($i=1,\ldots,5$) all are $\boldsymbol{\mathcal W}$-irrelevant. Each  facet $\Delta_{5,2}(i,0)$ is isomorphic to the hypersimplex 
$\Delta_{4,2}=\big\{\,  
(\tau_j)_{j=1}^4\in [0,1]^4\, \big\lvert \, 
\sum_{j=1}^4 \tau_j = 2\, \big\}$ which is the moment polytope of the 
${\rm PGL}_4(\mathbf C)$-homogeneous variety $X_{\Delta_{5,2}(i,0)}$ which is isomorphic to 
 the grassmannian $G_2(\mathbf C^4)$.  
One has $ \boldsymbol{\mathcal Y}^*=G_2(\mathbf C^5)^*/ H_4\simeq \boldsymbol{\mathcal M}_{0,5}$ and 
for each facet 
$\Delta_{5,2}(i,0)$, $ \boldsymbol{\mathcal Y}^*_{\Delta_{5,2}(i,0)}=G_2(\mathbf C^4)^*/  H_3\simeq \boldsymbol{\mathcal M}_{0,4}\simeq \mathbf P^1\setminus \{0,1,\infty\}$ 
and 
up to the above identifications, the map 
$\pi_{\Delta_{5,2}(i,0)} : \boldsymbol{\mathcal Y}^* \rightarrow  \boldsymbol{\mathcal Y}^*_{\Delta_{5,2}(i,0)}$ of Corollary \ref{C:cocor} 
corresponds to the $i$th forgetful map $ \boldsymbol{\mathcal M}_{0,5}\longrightarrow 
\boldsymbol{\mathcal M}_{0,4}$. In particular, all five facets $\Delta_{5,2}(i,0)$  are 
 $\boldsymbol{\mathcal W}$-relevant. 
\end{exm}

From the preceding example, it follows that Gelfand-MacPherson's web on $G_2(\mathbf C^5)$ 
(resp.\,its equivariant quotient by the $H_4$-action) previously  considered in  \S\ref{SSS:GM-Weg-G2(V)} is the web on $G_2(\mathbf C^5)^*$ (resp.\,on $ \boldsymbol{\mathcal Y}^*\simeq 
\boldsymbol{\mathcal M}_{0,5}$) whose first integrals exactly are the maps $\Pi_F : 
G_2(\mathbf C^5)^*\rightarrow G_2(\mathbf C^4)^*$ (resp.\,$\pi_F$) for $F$ ranging in the set of $ \boldsymbol{\mathcal W}$-relevant facets of the associated moment polytope.  
This suggests the following definition (in which we continue to use the notations introduced above and where we denote by $\mathfrak F_{\boldsymbol{\mathcal W}}( \Delta)$ the set of 
$\boldsymbol{\mathcal W}$-relevant facets of the considered moment polytope $\Delta$): 
\begin{defn}
Given a minuscule pair $(D,\omega)$ as above, the `Gelfand-MacPherson web'
 associated to it, denoted by 
$\boldsymbol{\mathcal W}^{GM}_{\hspace{-0.07cm} X}$,  is the `web' on $X^*$  induced by the facet maps $\Pi_F : 
X^*\rightarrow X_F^*$ for all $\boldsymbol{\mathcal W}$-relevant facets $F$: 
one has 
$$
\boldsymbol{\mathcal W}^{GM}_{\hspace{-0.07cm} X}=\boldsymbol{\mathcal W}\bigg(\, 
\Pi_F\, \, \big\lvert\  \, \forall \, F \in \mathfrak F_{\boldsymbol{\mathcal W}}( \Delta)\, 
\bigg) \, . 
$$

This web is equivariant under the action of $H$ hence descends to the quotient 
$\boldsymbol{\mathcal Y}^*$ and gives rise to a web denoted by $\boldsymbol{\mathcal W}^{GM}_{\hspace{-0.07cm} \boldsymbol{\mathcal Y}}$, which will be also called the `Gelfand-MacPherson web' associated to $(D,\omega)$ (but on $\boldsymbol{\mathcal Y}^*$): 
one has 
\begin{equation}
\label{Eq:W-GM/H}
\boldsymbol{\mathcal W}^{GM}_{\hspace{-0.07cm}\boldsymbol{\mathcal Y}}=
\big(\boldsymbol{\mathcal W}^{GM}_{\hspace{-0.07cm}X}\big)\big/_{ H}=
\boldsymbol{\mathcal W}\Big(\, 
\pi_F\, \, \big\lvert\  \, \forall \, F \in \mathfrak F_{\boldsymbol{\mathcal W}}( \Delta)\, 
\Big) 
 \, . 
\end{equation}
\end{defn}
A remark is in order about the use of the term `web' in this definition: it does not refer to the classical notion of web but to the more general one introduced in \cite{ClusterWebs} under the name of {\it `generalized web'}, which is just a collection of (possibly singular and/or of different codimensions) foliations on a given space, with the unique requirement that two distinct foliations intersect transversally. 
\begin{exm}
\label{Ex:GM-webs}
{\rm 1.} As it follows from Example {\rm \ref{Ex:GM-Web-G2C5}} above, for the pair $(A_4,\omega_2)$, 
one recovers Bol's web since in this case, one has $$
\boldsymbol{\mathcal W}^{GM}_{\hspace{-0.07cm}\boldsymbol{\mathcal Y}}=
\boldsymbol{\mathcal W}_{ \hspace{-0.07cm}\boldsymbol{\mathcal M}_{0,5}}\simeq 
 \boldsymbol{\mathcal W}_{ \hspace{-0.07cm} {\rm dP}_5}\simeq 
  \boldsymbol{\mathcal B}\, .$$
  
   \noindent{\rm 2.}  From \cite{GM}, one gets a simple general description of any Gelfand-MacPherson web associated to a pair $(A_n,\omega_k)$ with $k$ such that $1<k\leq n/2$ (which corresponds to $X=G/P=G_k\big(\mathbf C^{n+1}\big)${\rm )}.   
  Assume that both $n$ and $g$ are greater or equal to $3$ (this just in order to give a simple uniform description for the associated Gelfand-MacPherson web). Then $\boldsymbol{\mathcal Y}^*$ can be identified with a Zariski open subset of ${\rm Conf}_{n+k}\big(\mathbf P^{k-1}\big)$, that is the space of projective equivalence classes of $(n+1)$-tuples of points in general position in $\mathbf P^{k-1}$.    The facet maps all are $\boldsymbol{\mathcal W}$-relevant in this case and can be described as follows as rational maps defined on the birational model ${\rm Conf}_{n+k}\big(\mathbf P^{k-1}\big)$ of $\boldsymbol{\mathcal Y}^*$: there are 
 $2(n+k)$ facets hence as many facet maps, 
 the half of which are  
  the 
  forgetful maps $
   {\rm Conf}_{n+k}\big(\mathbf P^{k-1}\big)\rightarrow {\rm Conf}_{n+k-1}\big(\mathbf P^{k-1}\big)$, the $n+k$ others being the maps $
   {\rm Conf}_{n+k}\big(\mathbf P^{k-1}\big)\rightarrow {\rm Conf}_{n+k-1}\big(\mathbf P^{k-2}\big)$ induced by 
  the linear projection from a point of the configuration.  The forgetful maps 
and those induced by projections  from a point   
   define foliations of dimension $k-1$ and $n-k$ respectively.  Hence except when $n=2k-1$, the Gelfand-MacPherson web as defined above is a $2(n+k)$-web of mixed codimensions.
\end{exm}

  The case which we will prove to be the relevant one  for describing   del Pezzo's web $\boldsymbol{\mathcal W}_{     \hspace{-0.07cm} {\rm dP}_4}$
  \`a  la Gelfand-MacPherson  
   is the one associated to the pair $(D_5,\omega_4)$ or equivalently, to the spinor tenfold $\mathbb S_5$. In this case, $\boldsymbol{\mathcal Y}^*$ is 5-dimensional and 
  the Gelfand-MacPherson web  $\boldsymbol{\mathcal W}^{GM}_{\hspace{-0.07cm}\boldsymbol{\mathcal Y}}$ is a 10-web of codimension 2.  In the next subsection, we describe quite explicitly this web ({\it cf.}\,Proposition \ref{P:W-GM-Y5}) and then just after, we recall some results about the Cox variety of 
a del Pezzo surface   
and explain how one can deduce from them  a way to geometrically construct $\boldsymbol{\mathcal W}_{     \hspace{-0.07cm} {\rm dP}_4}$  from Gelfand-MacPherson's web 
$ \boldsymbol{\mathcal W}^{GM}_{\hspace{-0.01cm}\boldsymbol{\mathbb S}_{5}}$ (see Proposition \ref{Prop:WdP4-from-WGMS5}).

\subsubsection{\bf The GM-web of the spinor tenfold $\mathbb S_5 \hspace{-0.38cm} \mathbb S_5 $}
\label{SSS:GM-Web-S5}
Our goal here is to specialize the theory sketched above in 
\S\ref{SSS:GM-webs} in the case to be used for describing $\boldsymbol{\mathcal W}_{     \hspace{-0.07cm} {\rm dP}_4}$, namely the one associated to the pair $(D_5,\omega_4)$. 
In this case $G={\rm Spin}_{10}(\mathbf C)$ and $X=G/P$ is the spinor tenfold $\mathbb S_{5}$ 
and our ultimate goal here is to describe as explicitly as possible the associated Gelfand-MacPherson webs. 
 
We start by stating some general facts about spinor varieties, especially the spinor tenfold we are interested in. Everything here is classical hence we do not give any proof.  Among several recent  references on spinor varieties, we mention \cite{Manivel} and \cite{Corey}.

 \paragraph{\bf Notations and generalities.}
 All what is discussed in this preliminary paragraph actually holds true for all spinor varieties $\mathbb S_n$ with $n\geq 5$.  We start below by stating everything in the general case $n\geq 5$, with the exception of some details that we make explicit in the case $n=5$ which is the one relevant for our purpose.  

For any $n\geq 5$, we will use the following notations: 
\begin{itemize}
\item[$\bullet$]  one sets $i^*=i+n$ for $i=1,\ldots,n$;
\sk
\item[$\bullet$]  we set 
$I=\{1,\ldots,n\}$, $I^*=\{1^*,\ldots,n^*\}$ and $J=I\sqcup I^*$;
\sk
\item[$\bullet$] 
 setting $(i^*)^*=i$ for $i\in I$, the map  $j\mapsto j^*$  defines an 
 involution of $J$ exchanging $I$ and $I^*$; 
 \sk
\item[$\bullet$] we consider a complexe vector space  $E$ of dimension $2n$, with a 
basis $(e_j)_{j\in J}$ indexed by $J$. By means of this basis, we identify $E$ with 
$\mathbf C^{2n}$;
\sk
\item[$\bullet$]
for  $1\leq i_1<\ldots<i_k\leq n$, we write $e_{i_1\ldots i_k}$ for $e_{i_1}\wedge \cdots \wedge e_{i_k}$;
\sk
\item[$\bullet$]  we denote by $x_i,x_{i^*}$ for $i=1,\ldots,n$ the linear coordinates on $E$ dual to the $e_i,e_{i^*}$;  
\sk
\item[$\bullet$] we endow/equip $E$ with the quadratic form 
$
Q(x)=\sum_{i=1}^n x_i x_{i^*}
$;
\sk
\item[$\bullet$] a $n$-dimensional subspace $U\subset E$ is totally isotropic if $Q$ vanishes identically  on $U$; 
\sk 
\item[$\bullet$] we set 
$V=\oplus_{i=1}^n \mathbf C e_i$ and $W=\oplus_{i=1}^n \mathbf C e_{i^*}$. Both are totally isotropic $n$-planes. Moreover each of them is the dual of the other (where the duality is the one induced on $V$ by $Q$);
\sk
\item[$\bullet$] the orthogonal grassmannian $OG_n(\mathbf C^{2n})$ is the subset of $G_n(\mathbf C^{2n})$ formed by the totally isotropic $n$-dimensional subspaces of 
$\mathbf C^{2n}$;
\sk
\item[$\bullet$] 
${\rm SO}(\mathbf C^{2n},Q)$ (or just ${\rm SO}(\mathbf C^{2n})$ to simplify) denotes the algebraic subgroup of ${\rm SL}(\mathbf C^{2n})$ formed by the linear transformations of $E$  with determinant 1 and letting $Q$ invariant. The fundamental group of ${\rm SO}(\mathbf C^{2n})$ is $\mathbf Z/2\mathbf Z$ and its universal covering is the (complex) {\it spin group} ${\rm Spin}_{2n}={\rm Spin}\big(\mathbf C^{2n}\big)$; 
\sk
\item[$\bullet$] for $i=1,\ldots,5$, we set 
\begin{itemize}
\item[$-$]
$\pi_i: E\rightarrow E/\langle e_i\rangle$  and 
$\pi_{i^*}: E\rightarrow E/\langle e_{i^*}\rangle$; 
\item[$-$]
$H_i=\{ x_i=0\}=\big(\oplus_{j\neq i} \mathbf C e_j\big) \oplus W$ and 
$H_{i^*}=\{ x_{i^*}=0\}=V\oplus \big( \oplus_{j^*\neq i^*} \mathbf C e_{j^*}\big)$;
\end{itemize}
\sk  
\end{itemize}

 \paragraph{\bf Spinor varieties, half-spin representations and the Wick embeddings}
Let $\xi=[ A,B ]$ be a $n\times 2n$ matrix by blocks, with $A,B\in M_n(\mathbf C)$ such that ${\rm rk}(\xi)=2n$. Then one denotes by 
 $\langle \xi\rangle \in G_n(\mathbf C^{2n})$  the subspace of $\mathbf C^{2n}$ spanned by the $n$ row vectors of $\xi$. When $A$ (resp.\,$B$) is invertible, one  obviously has $\langle [ A,B ]\rangle
 =\langle [ \,{\rm Id}_n,A^{-1}B \,] \rangle$ (resp.\,$\langle [ A,B ]\rangle
 =\langle [\, B^{-1}A,{\rm Id}_n \,] \rangle$).  As is well known, $\langle \xi\rangle =\langle [ A,B ]\rangle $ is a totally isotropic subspace  hence belongs to the orthogonal grassmannian $
 OG_n(E)$ if and only if the matrix $A^tB$ is antisymmetric. 
 
 Another well-known fact is that the orthogonal grassmannian has two isomorphic disjoint connected components : 
$$
OG_n\big(\mathbf C^{2n}\big)=OG_n^+\big(\mathbf C^{2n}\big)\sqcup 
OG_n^-\big(\mathbf C^{2n}\big)\,.
$$
Let $\xi_1=[ A_1,B_1]$ and $\xi_2=[ A_2,B_2]$ be  two $n\times 2n$-matrices  defining two 
totally isotropic $n$-planes in $\mathbf C^{2n}$. These 
 two   $n$-planes belong to the same connected component $OG_n^{\pm}\big(\mathbf C^{2n}\big)$ of $OG_n\big(\mathbf C^{2n}\big)$
 if and only if  
 $\dim\big( \langle \xi_1\rangle+\langle \xi_2\rangle\big)={\rm rk}\big(
 \scalebox{0.7}{\Big[ \hspace{-0.2cm}\scalebox{0.9}{\begin{tabular}{l} $A_1$ $B_1$ \vspace{-0.1cm}\\ $A_2$ $B_2$ 
\end{tabular}}  \hspace{-0.4cm}
\Big]}\big)$ has the same parity as  $n$.  

There is a decomposition in direct sum $\wedge^n E=\wedge^n E^+\oplus \wedge^n E^-$ such that 
for  $\varepsilon=\pm $, one has $\mathbf P {\wedge^nE}^\varepsilon =\big\langle \,OG_n^\varepsilon\big(E\big)\,
\big\rangle$ and $OG_n^\varepsilon\big(E\big)=  G_n\big( E\big)
 \cap
\mathbf P {\wedge^nE}^\varepsilon
$. Both embeddings $OG_n^\varepsilon\big(E\big)\subset \mathbf P {\wedge^nE}^\varepsilon$ are 
(non canonically) isomorphic. Each is not minimal but is given by the second quadratic Veronese map of a minimal embedding of $OG_n^\pm(E)$ into the projectivization of an irreducible representation $ S^\pm=S_n^\pm$  of dimension $2^{n-1}$ of the spin group ${\rm Spin}_{2n}$, called the 
{\it `half-spin representation'}.
The corresponding variety in $\mathbf P S^\pm$ is denoted by $\mathbb S_n^\pm$ and is called the $n$-th spinor variety. It is the ${\rm Spin}_{2n}$-orbit of a highest weight vector in $S^\pm$  hence one has $ \mathbb S^\pm={\rm Spin}_{2n}/P$ for a certain parabolic subgroup $P$. 
The group ${\rm Spin}_{2n}$ is a simple complex Lie group of type $D_n$ and the parabolic subgroup $P$ just mentioned is the maximal one associated to the $(n-1)$-th node of the associated Dynkin diagram (in black in the figure below): 
 \begin{figure}[h!]
\begin{center}
\scalebox{2.1}{
 \includegraphics{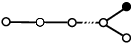}}
 \vspace{0.17cm}
\caption{The marked Dynkin diagram associated to the spinor variety  
$\mathbb S_n^+={\rm Spin}_{2n}/P_{n-1}$.}
\label{Fig:Marked_Dynkin_Diagram_Spinor-variety}
\end{center}
\end{figure}

\vspace{-0.3cm}
The two half-spin representations $S^\pm$ are abstractly indistinguishable from each other but for conveniency  we will use the notations  of \cite{Manivel} by identifying 
$S^\pm $ with $\wedge^{\rm \pm\varrho} V$ where
$\varrho=\varrho(n)\in \{ odd, even\}$ stands for the parity of $n$ (and $-\varrho$ for the opposite) 
and where 
$$\wedge^{even} V=  \oplus_{k=0}^{\lfloor n/2\rfloor} 
\wedge^{2k} V\qquad \mbox{ and } \qquad \wedge^{odd} V=  \oplus_{k=0}^{\lfloor n/2\rfloor} 
\wedge^{2k+1} V\, . $$ 

It is interesting to describe a classical affine parametrization of (a Zariski open subset of) $\mathbb  S^+$ in terms of antisymmetric matrices. We denote by ${\rm Asym}_n(\mathbf C)$ the vector space of antisymmetric matrices of size $n$.   Given a subset  $K\subset 
[ \hspace{-0.06cm}[\, n\,]\hspace{-0.06cm}]$
  and a matrix $M=(m_{ij})_{i,j=1}^n \in {\rm Asym}_n(\mathbf C)$, we set 
\begin{enumerate}
\vspace{-0.2cm}
\item[$-$] $\ell= \ell(K)={\rm Card}(K)\in \{0,\ldots,n\} $; 
\item[$-$]  $k_1,\ldots,k_\ell$ for the elements of $K$ labeled by increasing order: $1\leq k_1<\ldots<k_\ell\leq n$; 
\item[$-$]  $e_K=e_{k_1}\wedge \cdots \wedge e_{k_\ell} \in \wedge^\ell V$; 
\item[$-$]   ${}^c \hspace{-0.06cm}K= [ \hspace{-0.06cm}[\, n\,]\hspace{-0.06cm}] \setminus  K $; 
\item[$-$]  $M_K$ for the antisymmetric submatrix $(m_{k_u k_v})_{u,v=1}^\ell\in {\rm Asym}_\ell(\mathbf C)$;
\item[$-$] ${\rm Pf}(M)$ (resp.\,${\rm Pf}_K(M)$) for the Pfaffian of $M$ (resp.\,of $M_K$). 
\end{enumerate}

Then the {\it `Wick parametrization'} of $\mathbb S_n^+$ is the morphism defined as the projectivisation of the affine embedding 
\begin{align}
\label{Eq:Wick-Param}
W_n^+ : \, 
{\rm Asym}_n(\mathbf C) & \longmapsto 
S^+_n =
\wedge^{\varrho} V 
\\
M &\longmapsto \sum_{K} {\rm Pf}(M_K) e_{K^c}
\nonumber 
\end{align}
where in the sum $K$ ranges in the set of subsets of $[ \hspace{-0.06cm}[\, n\,]\hspace{-0.06cm}]$ of even cardinality.  
We will denote the Wick parametrization  ${\rm Asym}_n(\mathbf C) \rightarrow \mathbf P S_n^+$ by $W_n^+$ as well. 
It is known that 
$W_n^+$ embeds ${\rm Asym}_n(\mathbf C)\simeq \mathbf C^{n(n-1)/2}$ into $\mathbf PS_n^+$ 
and parametrizes a Zariski open subset of $\mathbb S_n^+$. One verifies that with the notation introduced above, the totally isotropic $n$-plane $W_n(A)$ associated to a given 
antisymmetric matrix $A\in {\rm Asym}_n(\mathbf C)$ is represented by the $n\times 2n$-matrix 
$\big[ {\rm Id}_n,A\big]$. 

Because this case is relevant for our purpose, let us make  the fifth Wick embedding  $W_5^+$ more explicit. We set $e_{klm}=e_k\wedge e_l\wedge e_m\in \wedge^3 V $ for $k,l,m$ such that $1\leq k<l<m\leq 5$,  
$e_{1\ldots5}=e_1\wedge e_2\wedge \ldots\wedge e_5\in \wedge^5 V\simeq \mathbf C$ and 
$\Sigma_5=\{ (i,j)\in \mathbf N^2\, \big\lvert \, 1\leq i<j \leq 5 \, \}$. 
For $x=(x_{ij})_{(i,j)\in \Sigma_5}\in \mathbf C^{ \Sigma_5}$, we denote by $A(x)$ the antisymmetric matrix whose $(i,j)$-th coefficient is $x_{ij}$ for any $\Sigma_5$. 
Then we have 
\begin{align}
\label{Eq:Wick}
\mathbf C^{ \Sigma_5}\simeq {\rm Asym}_5(\mathbf C) & \longrightarrow  \mathbf P\left( \wedge^{odd} V 
\right)\simeq \mathbf P^{15}
\\
(x_{ij})_{(i,j)\in \Sigma_5}
=x\mapsto A(x) \,  & \longmapsto  
e_{1\ldots 5}+\sum_{1\leq i<j\leq 5} x_{ij} \,e_{klm} + 
\sum_{i=1,\ldots,5}
{\rm Pf}\big( A_{\hat \imath}\big) \,e_i\, , 
\nonumber
\end{align}
where for $(i,j)\in \Sigma_5$, $(k,l,m)$ stands for the triple of increasing integers such that 
$\{i,j,k,l,m\}=\{1,\ldots,5\}$ and where for any $i$, $A_{\hat \imath}$ stands for the $4\times 4$ antisymmetric matrix obtained by removing 
the $i$th row and the $i$th column to $A(x)$.

\paragraph{\bf The Cartan torus and the weights.}
The Cartan torus $H$ we are dealing with is the rank $n$ torus formed of diagonal matrices 
\begin{equation}
\label{Eq:D-h}
D_h=
\scalebox{1}{\bigg[ \hspace{-0.2cm}\begin{tabular}{l} $h$ \, 0 \vspace{-0.1cm}\\ 0 \, $h^{-1}$ 
\end{tabular}  \hspace{-0.4cm}
\bigg]} \in {\rm SO}\big(\mathbf C^{2n}\big) \qquad \mbox{ with } 
\qquad 
h={\rm Diag}(h_1,\ldots,h_n)\in {\rm GL}_{2n}(\mathbf C)\, .
\end{equation}
The action of 
$D(h)$ on the $n$-plane associated to $\xi=[ A,B ]$ (with $A,B$ such that 
$ A^tB\in {\rm Asym}_n(\mathbf C)$ is given by 
$\langle\xi \rangle\cdot D_h=
\big\langle\, [ A,B]\cdot D_h \big \rangle $  with 
$
[ A,B]\cdot D_h=[ A h,B h^{-1}]
$. In the case when $A={\rm Id}_n$ (with $B$ antisymmetric), one has 
$$ 
\Big\langle \, \big[ \, {\rm Id}_n ,B\, \big] \, \Big\rangle \cdot D_h
=\Big\langle \, \big[ \,  {\rm Id}_n, h^{-1} B h^{-1}\, \big]\, \Big\rangle\, . 
$$

For any $i$, we set $\nu_i\in \mathfrak h^*_{\mathbf R}$ for the map associating 
$h_i$ to $D_h$ with $h={\rm Diag}(h_1,\ldots,h_n)$. Then the $\nu_i$'s form a basis of 
$\mathfrak h^*_{\mathbf R}$ allowing to identify it with $\mathbf R^n$.

From the property of the pfaffian that ${\rm Pf}(M C M^t)={\rm det}(M){\rm Pf}(C)$ 
for any antisymmetric matrix $C$ and any arbitrary square matrix $M$ (both of the same size), one deduces that for any $D_h\in H$,  any $A\in {\rm Asym}_n(\mathbf C)$ and any $K\subset [ \hspace{-0.06cm}[\, n\,]\hspace{-0.06cm}]$ of even cardinality, one has 
\begin{equation}
\label{Eq:Weight-e-Kc}
{\rm Pf}_K\big(h^{-1} Ah^{-1}\big)={\rm Pf}_K(A)/h^K
\end{equation}
 with $h^K=\prod_{l=1}^\ell h_{\kappa_l}$ if $K=\{\kappa_1,\ldots,\kappa_\ell\}$.  It follows that for any such subset $K$, the line $\mathbf C e_{K^c} \subset \wedge^{ \varrho}V$  is stable under the action of the Cartan torus, hence  $\{ e_{K^c} \}$
for $K$ ranging in the subsets of $[ \hspace{-0.06cm}[\, n\,]\hspace{-0.06cm}]$ with even cardinal 
not only forms a vector basis of $\wedge^\varrho V=S^+$, but actually is a weight basis of it as a ${\rm Spin}_{2n}$-representation.  For any such $K$, let $w(K)$ be the corresponding weight of $e_{K^c}$. It is an element of $\mathfrak h^*_{\mathbf R}=
\oplus_{i=1}^n \mathbf R \nu_i\simeq \mathbf R^n$ which can be determined explicitly in a very down to earth and elementary way from 
\eqref{Eq:Weight-e-Kc}. Combined with the fact that $e_{1\ldots n}=e_1\wedge \cdots \wedge e_n$ is known as being a vector with the highest weight which is $w(\emptyset)=
\frac12 \sum_{i=1}^n \gamma_i\in \mathfrak h^*_{\mathbf R}$, 
it follows that for any $K$ with  even cardinal, one has 
$$
w(K)=w(\emptyset)-\sum_{ k\in K} \gamma_k= \frac12 \sum_{l\in K^c} \gamma_l - 
 \frac12 \sum_{k\in K} \gamma_k\, .
$$

Up to the identification $ \mathfrak h^*_{\mathbf R}\simeq \mathbf R^n$ provided by the basis $(\gamma_i)_{i=1}^n$, we obtain the following explicit description of the set $W_n^+$ of  weights of the half-spin representation $S_n^+$: 
$$
\mathfrak W_n^+=\left\{ \, \frac{1}{2}(\epsilon_i)_{i=1}^n \in \big\{ \pm 1 \big\}^n \, \big\lvert \, 
\epsilon_1\cdots \epsilon_n=1
\,\right\}\, 
$$
and for any $w=(w_1,\ldots,w_n)$ in this set, an associated weight vector is given by 
$e_{w}=\wedge_{  \{k\, \lvert \, w_k>0\}} e_k$.  In this case, $(\frac12,\ldots,\frac12)$ is the highest weight and $e_{1\ldots n}$ is a highest weight vector.

Since it is the case relevant for our purpose in this text, the suitable subsets $K$  together with the associated weights $w(K)$, the weight vector $e_{K^c}$ as well as the corresponding coordinates 
$p_K(x)$  
in the Wick embedding (with $x\in \mathbf C^{\Sigma_5})$ are given in the following table (where the following conventions are used: 
the subsets $K$ are written as even tuples of increasing integers in $[ \hspace{-0.06cm}[\, 5\,]\hspace{-0.06cm}]$ and $i,j,k,l,m$ stand for the elements of $[ \hspace{-0.06cm}[\, 5\,]\hspace{-0.06cm}]$ with $i<j$ and $k<l<m$): 
\begin{table}[!h]
\begin{tabular}{|c||c|c|c|}
\hline
\, 
\begin{tabular}{c} \vspace{-0.35cm}\\
$\boldsymbol{K}$
\vspace{0.1cm}
\end{tabular}
  &   $\emptyset$ &  $(i,j)$  & $(j,k,l,m)$ \\ \hline 
\begin{tabular}{c} \vspace{-0.35cm}\\
 $\boldsymbol{e_{K^c}}$
 \vspace{0.1cm}
\end{tabular}
  & $e_{1\cdots 5}$ &  $e_{klm}$  & $e_i$  \\ \hline
\begin{tabular}{c} \vspace{-0.35cm}\\
$\boldsymbol{p_{K}}(x)$ 
\vspace{0.1cm}
\end{tabular}
& 1 &  $x_{ij}$  & ${\rm Pf}\big(A_{\hat \imath}\big)$  \\ \hline
\begin{tabular}{c} \vspace{-0.35cm}\\
$\boldsymbol{w(K)}$ 
\vspace{0.1cm}
\end{tabular}
 & $\frac{1}{2}(\nu_1+\cdots+\nu_5)$ &    
$\frac{1}{2}( \nu_k+\nu_l+\nu_m-\nu_i-\nu_j) $
 & 
$\frac{1}{2}(  \nu_i- \nu_j- \nu_k- \nu_l- \nu_m) $ 
   \\ \hline  
\end{tabular}
  \bk
\caption{The weight vectors, the associated Wick coordinates and the corresponding weights.}
\label{Table:tokolo}
\end{table}

\vspace{-0.2cm}
We will use a similar description of the weights and of the associated weight vectors for the other half-spin representation $S_n^+$: the set of weights is 
$$
\mathfrak W_n^-=\left\{ \, \frac{1}{2}\big(\epsilon_i\big)_{i=1}^n 
 \hspace{0.1cm}  \Big\lvert \hspace{0.06cm}  
 \begin{tabular}{l}
$\forall\, i\hspace{0.04 cm}  : \hspace{0.04 cm}  \epsilon_i \in 
\{ \pm 1 \}$ \\
$ \epsilon_1\cdots \epsilon_n=-1$ 
 \end{tabular}
\,\right\}\, 
$$
and for any $w=(w_1,\ldots,w_n)$ in this set, an associated weight vector is given by 
$e_{w}=\wedge_{  \{k\, \lvert \, w_k>0\}} e_k$.
In this case, $(\frac12,\ldots,\frac12,-\frac12)$ is the highest weight and $e_1\wedge \cdots \wedge  e_{n-1}$ is a highest weight vector.

\paragraph{\bf Isomorphisms.}
Here we describe some isomorphims between the half-spin representations under consideration which we will use further on. In what follows, $\varepsilon$ stands for $+$ or $-$.

For $i,j\in [ \hspace{-0.06cm}[\, n\,]\hspace{-0.06cm}]$ distinct, let $g_{ij}$ be the linear involution of $E$ defined, in the considered basis $(e_1,\ldots,e_n,e_{1^*},\ldots,e_{n^*})$ by exchanging $e_i$ and $e_j$ as well as  $e_{i^*}$ and 
$e_{j^*}$ and letting all the others $e_k$ and $e_{k^*}$ unchanged.  One has $g_{ij}\in {\rm SO}(E)$
hence interior conjugation by $g_{ij}$ gives rise to another representation that we will denote by $S_{n,ij}^\varepsilon$ to differentiate it from the initial $S_n^\varepsilon$.  
The  underlying vector spaces are the same, namely $\wedge^{\varepsilon\varrho} V$ and the isomorphism $\wedge^{\varepsilon\varrho} V \stackrel{\sim}{\longrightarrow}
\wedge^{\varepsilon\varrho} V$, again denoted by $g_{ij}$,  is the natural map $e_{L}\longmapsto e_{(ij)L}$ where 
$(ij)L$ stands for image of the subset $L\subset [ \hspace{-0.06cm}[\, n\,]\hspace{-0.06cm}]$ by the transposition $(ij)$.
If one denotes by $\mathbb S_{n,ij}^\varepsilon$ the projectivisation of the closed orbit in $\wedge^{\varepsilon\varrho} V $ when this space is viewed as the representation 
$S_{n,ij}^{\varepsilon}$, then the  restriction of the map $g_{ij}\in {\rm GL}(\wedge^{\varepsilon\varrho} V)$ gives rise to an isomorphism from $\mathbb S_n^\varepsilon$ onto $\mathbb S_{n,ij}^\varepsilon$. 
We remark that the line spanned by a highest weight vector in $S_n^\varepsilon$ is left invariant by $g_{ij}$ if and only if $\varepsilon=+$ of $\varepsilon=-$ and $n\not \in \{i,j\}$. Consequently in these cases, and in these cases only, one has $\mathbb S_{n,ij}^\varepsilon=\mathbb S_{n}^\varepsilon$  as subvarieties of $\mathbf P\big(\wedge^{\varepsilon\varrho} V\big)$.

The two half-spin representations $S_n^+$ and $S_n^-$ are isomorphic but not in a canonical way. 
We will need below to consider some  isomorphisms between them that we are going to describe now.

We fix $i\in [ \hspace{-0.06cm}[\, n\,]\hspace{-0.06cm}]$ and for any subset $L\subset 
 [ \hspace{-0.06cm}[\, n\,]\hspace{-0.06cm}]$, we set 
 $$
 L^i=\begin{cases}
 \, L\setminus \{i \} \hspace{0.3cm} \mbox{if }\, i\in L;\\
  \, L\cup \{i \} \hspace{0.2cm} \mbox{if }\, i\not \in L. 
 \end{cases}
 $$
The map $L\mapsto L^i$ induces a bijection from the set of subsets of  $[ \hspace{-0.06cm}[\, n\,]\hspace{-0.06cm}]$ whose cardinal has the same parity as $n$ onto the one of subsets with the opposite parity. Then associating $e_{L^i}$ to $e_L$ gives rise to a linear isomorphism
\begin{equation}
\label{Eq:Gamma-i}
\Gamma_i  : \wedge^{\pm\varrho} V \stackrel{\sim}{\longrightarrow} 
\wedge^{\mp\varrho} V
\end{equation}%
which can be shown to induce an  isomorphism of representations 
from $S_n^\pm $ onto $S_n^\mp$ The induced linear  isomorphism 
 of $\mathfrak h^*_{\mathbf R}$, denoted by $\overline{\Gamma}_i$, is characterized by 
 the relations $\overline{\Gamma}_i(\nu_j)=\nu_j$ if $j\neq i$ and $\overline{\Gamma}_i(\nu_i)=-\nu_i$. The corresponding bijection from $\mathfrak W_n^\pm$ onto $\mathfrak W_n^\mp$ simply involves replacing the $i$-th coordinate of the weight by its opposite, leaving all other coordinates unchanged.

\paragraph{\bf Wick parametrizations of $\mathbb S_n^-$.}
If the Wick parametrization \eqref{Eq:Wick-Param} is a priviledged birational para-  metrization of the spinor variety $\mathbb S_n^+$, we do not have anything similar at disposal regarding the other spinor manifold $\mathbb S_n^-$.  

However one gets several parametrizations of (a Zariski open subset of) $\mathbb S_n^-$ by composing $W_n^+$ with a projective isomorphism $G  $ from $S_n^+ $ onto $ S_n^-$. 
For such a $G$, one can for instance take a composition $\Gamma_n\circ g_{ij}$ for any distinct $i,j\in  [ \hspace{-0.06cm}[\, n\,]\hspace{-0.06cm}]$.  Any such map 
${\rm Asym}_n(\mathbf C)\rightarrow \mathbb S_n^-$ 
 will be called a `Wick parametrization' of $\mathbb S_n^-$.  This generalizes in a straightforward way to any spinor avatars $\mathbb S_{n,ij}^-$.

\paragraph{\bf The weight polytope and its facets.}
For $\varepsilon\in \{ +, -\}$, 
the set of weights of $S_{ {n}}^\varepsilon$ is the set $\mathfrak W_n^\varepsilon$ of $n$-tuples $\frac12 (\epsilon_1,\ldots,\epsilon_n)\in \mathfrak h_{} ^*\simeq \mathbf R^n$ with $\epsilon_i\in\{\pm 1\}$ for all indices $i \in  [ \hspace{-0.06cm}[\, n\,]\hspace{-0.06cm}]$  and such that 
$\epsilon_1\cdots\epsilon_n= \varepsilon 1$ ({\it i.e.}\,the number of $\epsilon_i$'s equal to -1 is even if $\varepsilon=+$ and is odd  when $\varepsilon=-$). 

It is known that the moment polytope $\Delta_n^\varepsilon$ of 
$\mathbb S_n^\varepsilon\subset \mathbf PS_{n}^\varepsilon$
coincides with the weight polytope: one has $\Delta_n^\varepsilon={\rm Conv}\big( \mathfrak W_n^\varepsilon \big)\subset \mathfrak h^*_{ \mathbf R}$. From the description of 
 $\mathfrak W_n^\varepsilon$ recalled just above, it follows that 
this polytope is a half-hypercube of dimension $n$ whose combinatorial properties are well-known. In particular, its facets are known: these are  
$2^n$  hypersimplices and  $2n$ are half-hypercubes. 
The former facets are easily seen to be of type $(A_{n-1}, \omega_{n-1})$ (see Figure \ref{Fig:Marked_Dynkin_Diagram_D5_omega4}) hence all are 
$\boldsymbol{\mathcal W}$-irrelevant. For this reason, we will not consider them further. On the other hand, the 
$2n$ half-hypercubic facets all are 
$\boldsymbol{\mathcal W}$-relevant and  are easy to describe: they come in pairs $(
\Delta_{n,i}^{\varepsilon,+}, \Delta_{n,i}^{\varepsilon,-})$ for $i=1,\ldots,n$, with 
\begin{align}
\label{Eq:Delta-n-i}
\Delta_{n,i}^{ \varepsilon,\pm}= & \, \Delta_{n}^\varepsilon \cap \big\{\,\lambda_i=\pm 1/2\, \big\}= {\rm Conv}( 
\mathfrak W_{n,i}^{\varepsilon,\pm}\big)
\end{align}
where 
${\mathfrak W}_{n,i}^{\varepsilon,\pm }$ stands for the subset of 
$\mathfrak W_{n}^{\varepsilon}$ formed by its elements whose $i$-th coordinates are  $\pm \frac{1}{2}$. Obviously, if $p_i: \mathbf R^n\rightarrow \mathbf R^{n-1}$ stands for the linear projection 
consisting in forgetting the $i$-th coordinate, one has 
$$
p_i\big(\mathfrak W_{n,i}^{ \varepsilon,\pm}\big)=\mathfrak W_{n-1}^{ \pm \varepsilon}
\qquad \mbox{ and  } \qquad 
p_i\big(\Delta_{n,i}^{ \varepsilon,\pm}\big)=\Delta_{n-1}^{ \pm \varepsilon}\, . 
$$

 \paragraph{\bf The face maps.}
 Our goal here is to describe the face maps associated to the $2n$ $\boldsymbol{\mathcal W}$-relevant
 hypercubic facets of $\Delta_n^+$. We fix $i\in  [ \hspace{-0.06cm}[\, n\,]\hspace{-0.06cm}]$ in what follows. 

For $\varepsilon=\pm$, we denote by $S_{n,i}^{+,\varepsilon}$ the subspace of $S_n^+$ spanned by the weight vectors with weight in  $\mathfrak W_{n,i}^{+, \varepsilon}$:  
$S_{n,i}^{+,+}$ (resp.\,$S_{n,i}^{+,-}$) admits as a basis  the set of of weight vectors $e_L$ for all 
$L\subset [ \hspace{-0.06cm}[\, n\,]\hspace{-0.06cm}]$ with $L^c=[ \hspace{-0.06cm}[\, n\,]\hspace{-0.06cm}]\setminus L$ of even cardinal and with $i\in L$ (resp.\,with $i\not \in L$): one has 
$$
S_{n,i}^{+,+}=\bigoplus_{  \substack{L\subset [ \hspace{-0.06cm}[\, n\,]\hspace{-0.06cm}], \, i\in L \\
\#L^c \,\scalebox{0.84}{\mbox{ even}}}} \mathbf C  e_L
\qquad \mbox{ and } 
\qquad
S_{n,i}^{+,-}=\bigoplus_{  \substack{L\subset [ \hspace{-0.06cm}[\, n\,]\hspace{-0.06cm}], \, i\not \in L \\
\#L^c \,\scalebox{0.84}{\mbox{ even}}}} \mathbf C  e_L\,. 
$$

The subgroup ${\rm Spin}_{2n,i}$ of 
${\rm Spin}_{2n}$ which lets  invariant
the  decomposition in direct sum 
\begin{equation}
\label{Eq:DecompDirectSum}
S^+_n=S^{+,+}_{n,i} \oplus S^{+,-}_{n,i}
\end{equation}
 is naturally isomorphic to ${\rm Spin}_{2(n-1)}$. Moreover,  up to the identification ${\rm Spin}_{2n,i}\simeq {\rm Spin}_{2(n-1)}$, $S^{+,+}_{n,i}$ and $S^{+,-}_{n,i}$ identify to the two half-spin representations $S_{n-1}^+$ and $S_{n-1}^-$ respectively: this follows easily from the well-known fact that for any $n\geq 5$, one has 
$
S_{ {n}}^+ \big\downarrow_{ {\rm Spin}_{2(n-1)}}=
S_{{n-1}}^+\oplus 
S_{ {n-1}}^-
$ ({\it cf.}\,Proposition 5.1 in \cite{Deligne} for instance).

For $\varepsilon\in \{ -,+\}$, one denotes by  $\Pi_{n,i}^{\varepsilon}$ the linear projection 
from $S^+_n$ onto $S^{+,\varepsilon }_{n,i}$ relatively to the decomposition in direct sum 
\eqref{Eq:DecompDirectSum}: 
\begin{align*}
\Pi_{n,i}^{\varepsilon} : S^+_n=S^{+,+}_{n,i} \oplus S^{+,-}_{n,i} & \longrightarrow
S^{+,\varepsilon }_{n,i}\,\\
 \,  \big( x^+,x^-\big) \, & \longmapsto \, x^{\varepsilon}\, . 
\end{align*}
We will use the same notation for the linear projection on the corresponding projective spaces $\mathbf P S^+_n\simeq \mathbf P^{2^{n-1}-1}$ and $\mathbf P S^{+,\varepsilon }_{n,i} \simeq \mathbf P^{2^{n-2}-1}$. 

In this setting,  Proposition \ref{Prop:X->XF} takes the following form: 
\begin{prop}
Up to the natural identification $S_{n,i}^{+,\varepsilon}\simeq S_{n-1}^\varepsilon$ mentioned above: 

$-$  the intersection $\mathbb S_{n,i}^\varepsilon\stackrel{{\rm def}}{=} \mathbb S_{n}^+ \cap \mathbf P S_{n,i}^{+,\varepsilon}$ identifies with $\mathbb S_{n-1}^\varepsilon$;

$-$ by restriction, the projection $\Pi_{n,i}^\varepsilon$ induces a dominant rational map 
$\Phi_{n,i}^\varepsilon : \mathbb S_{n}^+ \dashrightarrow \mathbb S_{n-1}^\varepsilon$. 
\end{prop}

It is easy to give an explicit expression for the map $\Phi_{n,i}^+ : \mathbb S_{n}^+ \dashrightarrow \mathbb S_{n-1}^+$ in terms of `Wick coordinates'. 
Let $\iota: [ \hspace{-0.06cm}[\, n\,]\hspace{-0.06cm}]\setminus \{i\} \stackrel{\sim}{\rightarrow}
[ \hspace{-0.06cm}[\, n-1\,]\hspace{-0.06cm}]$ stand for the bijection preserving the usual orders of the source and of the target.\footnote{More explicitly, for any $j \in [ \hspace{-0.06cm}[\, n\,]\hspace{-0.06cm}]\setminus \{i\}$, one has $\iota(j)=j$ if $j<i$  and $\iota(j)=j-1$ if $i<j$.}
To any such subset $K$, $\iota(K)$ is a subset of $[ \hspace{-0.06cm}[\, n-1\,]\hspace{-0.06cm}]\setminus \{i\}$ with even cardinal too. Hence the map which associates 
$e_{\iota(K)^c}$ to $e_{K^c}$ for any such $K$ gives rise to a linear isomorphism of vector spaces $\underline{\iota} : S_{n,i}^{+,+}\rightarrow  S_{n-1}^+$ whose projectivisation 
$\mathbf P S_{n,i}^{+,+}\rightarrow  \mathbf P S_{n-1}^+$ will be denoted the same. 
 On the other hand, the components of 
the composition $\Pi_{n,i}^\varepsilon\circ W_n^+: {\rm Asym}_n(\mathbf C) \rightarrow \mathbf P S_{n,i}^{+,+}$ are immediately seen to be the pfaffians ${\rm Pf}(A_K)$ for all subsets $K\subset 
[ \hspace{-0.06cm}[\, n\,]\hspace{-0.06cm}]\setminus \{i\}$ with even cardinal.  It is then easy to get the following result giving a nice expression in birational coordinates for 
$\Phi_{n,i}^+$: 
\begin{prop}
\label{Eq:Phi-n-i-+}
The following diagram of rational maps is commutative
$$
\xymatrix@R=1.2cm@C=1.6cm{
{\rm Asym}_n(\mathbf C) \,  \ar@{->}[d]_{\widetilde\Phi_{n,i}^+}
\ar@{^{(}->}[r]^{ {}^{} \hspace{0.6cm} W_n^+} &  \mathbb S_n^+  \, 
 \ar@{->}[d]^{\Phi_{n,i}^+} 
 \ar@{^{(}->}[r]& \mathbf P S_n^+ \ar@{-->}[r]^{ \Pi_{n,i}^{+}}&  \mathbf P S_{n,i}^ {+,+}
  \ar@{->}[d]^{\underline{\iota}}
 \\ 
 {\rm Asym}_{n-1}(\mathbf C)\,
 \ar@{^{(}->}[r]^{ {}^{} \hspace{0.6cm} W_{n-1}^+} &  \mathbb S_{n-1}^+  
  \ar@{^{(}->}[rr] & & \mathbf P S_{n-1}^ {+}\, , 
 }
$$
where $\widetilde\Phi_{n,i}^+ $ stands for the map associating to any antisymmetric matrix 
$A \in {\rm Asym}_n(\mathbf C) $ the submatrix $A_{\hat \imath}\in {\rm Asym}_{n-1}(\mathbf C) $ obtained by removing from it its $i$-th line and $i$-th column. 
\end{prop}

Even if we will not use this, it is interesting to give a more geometric description of the map $\Phi_{n,i}^+ : \mathbb S_{n}^+ \dashrightarrow \mathbb S_{n-1}^+$. To this aim, we need to identify $ \mathbb S_{n}^+$ with one of the components of the orthogonal grassmannian $OG_n(\mathbf C^{2n})$, which we are going to do as follows:  always assuming that 
$V=\oplus_{i=1}^n \mathbb Ce_i$ is in the component of $OG_n(\mathbf C^{2n})$ identified with 
$\mathbb S_n^+$, we see that the latter identifies with the set of totally isotrope subspaces $U\in 
OG_n(\mathbf C^{2n})$ such that the parity of the codimension
in $\mathbf C^{2n}$ 
 of the space spanned by $U$ and $V$  is the same as that of $n$: 
$$
\mathbb S_n^+\simeq \Big\{ \, U \in OG_n(\mathbf C^{2n})\hspace{0.1cm}   \big\lvert \hspace{0.1cm}   n+{\rm codim}\big( \langle U,V\rangle \big) \mbox{ is even }\, \Big\}\, .
$$ 

\begin{prop}
\label{Prop:Geom-Phini+}
 Up to the identifications above for $\mathbb S_n^+$ and $\mathbb S_{n-1}^+$, 
the face map $\Phi_{n,i}^+ : \mathbb S_n^+\dashrightarrow \mathbb S_{n-1}^+$ is the rational  application associating  to any generic  $\Pi\in \mathbb S_n^+$ the totally isotropic $(n-1)$-plane 
 $\pi_{i^*}(\Pi\cap H_i) = (\Pi\cap H_i)/_{\langle e_{i*}\rangle} $ of $E_{i,i^*}=H_i/{\langle e_{i*}\rangle} $. 
\end{prop}

We now turn to the description of the face maps associated to the facets $\Delta_{n,i}^{+,-}$ of 
$\Delta_{n}^{+}$.  
Rigorously speaking, such a map is a rational map 
$$ \Phi_{n,i}^- : 
\mathbb S_n^+\dashrightarrow \mathbb S_{n-1}^-\,.$$ 
The preceding geometric description of the maps 
$\Phi_{n,i}^+ : \mathbb S_n^+\dashrightarrow \mathbb S_{n-1}^+$ 
associated to the facets $\Delta_{n,i}^{+,+}$ 
 has the peculiarity of suggesting the corresponding statement for the other facets we are now considering.  
\begin{prop}
\label{Prop:Geom-Phini-}
The face map $\Phi_{n,i}^- : \mathbb S_n^+\dashrightarrow \mathbb S_{n-1}^-$ is the rational  application associating  to any generic  $\Pi\in \mathbb S_n^+$ the totally isotropic $(n-1)$-plane 
 $\pi_{i}(\Pi\cap H_{i^*}) = (\Pi\cap H_{i^*})/_{\langle e_{i}\rangle} $ of $E_{i^*,i}=H_{i^*}/{\langle e_{i}\rangle} $. 
\end{prop}

Although the geometric descriptions of the face maps $\Phi_{n,i}^{+}$
 and $\Phi_{n,i}^{-}$ given above are completely similar,  the description 
  of the  maps 
$\Phi_{n,i}^{-} $ in `Wick coordinates' is not as simple as the one for the maps $\Phi_{n,i}^{-}$ given in Proposition \ref{Eq:Phi-n-i-+}, the main reason behind this being that this requires an identification between 
$\mathbb S_{n-1}^-$ and $\mathbb S_{n-1}^+$ and there is no canonical such identification. 
Actually, given $i$, for any $j\in  [ \hspace{-0.06cm}[\, n-1\,]\hspace{-0.06cm}]\setminus \{ i\}$, 
 using the identification \eqref{Eq:Gamma-i}  (but with $V$ replaced by $ V_{\hat \imath}$),
we can construct a birational model $\widetilde \Phi_{n,i}^{-,j}  : {\rm Asym}_n(\mathbf C)\dashrightarrow 
{\rm Asym}_{n-1}(\mathbf C)$ of $\Phi_{n,i}^{-} $. However, since all these  birational models are obtained in the same way and to avoid writing down too  many formulas, 
we only treat below the case of $\widetilde \Phi_{n,n}^{-,n-1}$ (ie. $i=n$ and $j=n-1$) that we make entirely explicit in Wick coordinates. The corresponding description of any other $\Phi_{n,i}^{-,j}$, for any distinct 
$i$ and $j$,  will follow immediately from that (see Corollary \ref{Cor:Phi---n=5} below 
where all the $\Phi_{5,i}^{-,j}$ are explicitly given). 
\sk

We are going to use the following notations and conventions: 
\begin{itemize}
\item $x=(x_{ij})_{1\leq i<j\leq n}$ stands for an element of $\mathbf C^{\Sigma_n}$ 
and $A=A(x)$ denotes the antisymmetric matrix whose $(i,j)$-th coefficient is $x_{ij}$ for any $(i,j)\in \Sigma_n$; 
\item $\varrho$ stands for the parity of $n$; 
\item one sets $ V_{ \widehat{k}}= \oplus_{i=1, i\neq k }^{n}\mathbf C e_i$ for any $k=1,\ldots,n$;  
\end{itemize}

The  inclusion 
$[ \hspace{-0.06cm}[\, n-1\,]\hspace{-0.06cm}]  \subset 
[ \hspace{-0.06cm}[\, n\,]\hspace{-0.06cm}]$  gives rise to a canonical isomorphism $S_{n,n}-\simeq S_{n-1}^-$ that we will denote by $\Xi_n$.  Then for $j\in [ \hspace{-0.06cm}[\, n-1\,]\hspace{-0.06cm}] $, one gets a linear isomorphism $\Gamma_j\circ \theta_n: S_n^+\rightarrow S_{n-1}^+$. Setting 
$\Lambda_{n,j}= \Gamma_j\circ \theta_n\circ \Pi_{n,n}^- : S_n^+\rightarrow S_{n-1}^+$ and considering the composition $\Lambda_{n,j}\circ W_n^+: {\rm Asym}_n(\mathbf C)\rightarrow S_{n-1}^+$, we get an affine map which gives rise to an affine model for $\Phi_{n,n}^{-}\circ W_{n}^+ : {\rm Asym}_n(\mathbf C)\rightarrow \mathbb S_{n-1}^-$ but takes values into $ \mathbb S_{n-1}^+$.  
Explicit formulas of the maps involved in the previous considerations are made explicit in the following 
diagramm: \sk
$$
\xymatrix@R=0.1cm@C=0.2cm{
{\rm Asym}_n(\mathbf C)   \ar@{->}[r]^{W_n^+} & 
   S_n^+  =\wedge^{\varrho} V \ar@/^3pc/[rrrrrr]_{ \Lambda_{n,j}}
 \ar@{->}[rr]^{\hspace{0.85cm}\Pi_{n,n}^+}  & &  S_{n,n}^-  
  \ar@{->}[rr]^{\Xi_n}   & & S_{n-1}^-
  \ar@{->}[rr]^{\Gamma_{j} \hspace{0.3cm}}  && 
\wedge^{-\varrho} V_{ \widehat{\jmath}}  \, \simeq \, S_{n-1}^+
 \\
 A
  \ar@{|->}[r] &  \hspace{-0.3cm } 
    \sum\limits_{
    \scalebox{0.6}{ \begin{tabular}{c}
 $K\subset  [ \hspace{-0.06cm}[\, n\,]\hspace{-0.06cm}]$  \vspace{-0.1cm} \\
 $\#K = 0\, [2]$ 
 \end{tabular} }}  
  \hspace{-0.3cm } {\rm Pf}\big( A_K\big)\,e_{\,{}^c\hspace{-0.05cm}K}
    \ar@{|->}[rr]  &&  \hspace{-0.4cm } 
   \sum\limits_{ 
\scalebox{0.6}{ \begin{tabular}{c}
 $K\subset  [ \hspace{-0.06cm}[\, n\,]\hspace{-0.06cm}]$  \vspace{-0.07cm} \\
$\#K = 0\, [2]$   \vspace{-0.1cm} \\
  $n\in K$
 \end{tabular} }}     
 \hspace{-0.3cm }{\rm Pf}\big( A_K\big)\,e_{\,{}^c\hspace{-0.05cm}K}
  \ar@{|->}[rr]  &&  
   \hspace{-0.4cm } 
   \sum\limits_{ 
\scalebox{0.6}{ \begin{tabular}{c}
 $L\subset  [ \hspace{-0.06cm}[\, n-1\,]\hspace{-0.06cm}]$  \vspace{-0.07cm} \\
 $\#L = 1\, [2]$ 
 \end{tabular} }}     
 \hspace{-0.5cm }{\rm Pf}\big( A_{L\cup \{n\}}\big)\,e_{\,{}^c\hspace{-0.05cm}L}
   \ar@{|->}[rr]  &&  
   \hspace{-0.4cm } 
   \sum\limits_{ 
\scalebox{0.6}{ \begin{tabular}{c}
 $M\subset  [ \hspace{-0.06cm}[\, n-1\,]\hspace{-0.06cm}]$  \vspace{-0.07cm} \\
 $\#M = 0\, [2]$ 
 \end{tabular} }}     
 \hspace{-0.5cm }{\rm Pf}\big( A_{M^{j}\cup \{n\}}\big)\,e_{\,{}^c\hspace{-0.05cm}M}
    }
$$
We consider now that $j=n-1$ and we denote by $\Psi_n= \Lambda_{n,j}\circ W_n^+$ the composition of the linear maps above.  
Our goal is to find a map ${\rm Asym}_n(\mathbf C)\dashrightarrow 
{\rm Asym}_{n-1}(\mathbf C)$, $A\mapsto \tilde A$ such that for a generic antisymmetric matrix, 
$\Psi (A)$ and $W_{n-1}^+(\tilde A)$ coincide as elements of $S_{n-1}^+$, possibly up to multiplication by a non zero scalar.  To this end, we will proceed by analyse-synth\`ese.

Among the elements of the basis of $S_{n-1}^+$ we are dealing with, the one  of highest weight is $e_{
[ \hspace{-0.06cm}[\, n-1\,]\hspace{-0.06cm}]}=e_1\wedge \cdots \wedge e_{n-1}$ which is $e_{\,{}^c\hspace{-0.05cm}M}$ for 
$M$ the empty set. 
For $A=(x_{ij})_{i,j=1}^n\in {\rm Asym}_n(\mathbf C)$, the 
$e_{
[ \hspace{-0.06cm}[\, n-1\,]\hspace{-0.06cm}]}$-coefficient of $\Psi(A)$ is $x_{n-1,n}$ hence in view 
of our purpose, assuming that this coefficient of $A$ does not vanish, it is relevant to consider 
$ \Psi(A)/x_{n-1,n} $ which is the one likely to be written 
\begin{equation}
\label{Eq:zouk-zouk}
\frac{1}{x_{n-1,n}}\Psi(A) = 
W_{n-1}^+\big( \tilde A\big)
\end{equation}
 for a certain $\tilde A\in {\rm Asym}_{n-1}(\mathbf C)$.  Denoting by $\tilde x_{ij}$ the coefficients of $\tilde A$ with $i,j=1,\ldots,n-1$, we find that in the decomposition of $W_{n-1}^+(\tilde A)$
 in the weight basis: 
 \begin{itemize}
\item for $i<n-1$, $\tilde x_{i,n-1}$ is the coefficient of $e_{\,{}^c\hspace{-0.05cm}M}$ for $M=\{i,n-1\}$. In this case, one has  $M^{n-1}\cup \{n\}=\{ i,n\}$ hence 
the corresponding coefficient of $ \Psi(A)/x_{n-1,n} $ is ${\rm Pf}( A_{\{i,n\}})/x_{n-1,n} =
x_{i,n}/x_{n-1,n}$; 
 \mk 
 \item for $i,j$ with $1\leq i<j<n-1$,  $\tilde x_{i,j}$ is the coefficient of $e_{\,{}^c\hspace{-0.05cm}M}$ for $M=\{i,j\}$.  In this case, one has  $M^{n-1}\cup \{n\}=\{ i,j,n-1,n\}$ hence 
the corresponding coefficient of $ \Psi(A)/x_{n-1,n} $ is the quotient of the  pfaffian of the  $4\times 4$-antisymmetric matrix $A_{\{ i,j,n-1,n\}}$  by $x_{n-1,n}$. 
\end{itemize}

 This short analysis shows that given $A\in {\rm Asym}_n(\mathbf C)$ generic, if there exists  $\tilde A\in {\rm Asym}_{n-1}(\mathbf C)$ such that \eqref{Eq:zouk-zouk} holds true, then it is unique and its coefficients $\tilde x_{i,j}$ 
  are given by 
\begin{equation}
\label{Eq:formulas-x-tilde}
\tilde x_{k,n-1}=\frac{x_{k,n}}{x_{n-1,n}}\quad \mbox{for } k=1,\ldots,n-2 
\qquad \mbox{ and }\qquad 
\tilde x_{i,j}= \frac{ {\rm Pf}\big( A_{\{ i,j,n-1,n\}}
\big)
}{x_{n-1,n}}\quad  \mbox{for } (i,j)\in \Sigma_{n-2}\, . 
\end{equation}

Assume that $\tilde A$ indeed stands for the antisymmetric matrix whose coefficients are given by the above formulas.  That the equality \eqref{Eq:zouk-zouk} holds true in $S_{n-1}^+$ is equivalent to the fact that the following rational identities in the coefficients $x_{ij}$ of $A$ are satisfied: 
\begin{equation}
{}^{}\hspace{3cm}
\frac{{\rm Pf}\Big( A_{M^{n-1}\cup \{n\}}\Big)}{ x_{n-1,n} } =
{\rm Pf}\Big( {\tilde A}_M\Big)
\qquad 
\forall
\, M\subset [ \hspace{-0.06cm}[\, n-1\,]\hspace{-0.06cm}] \, \mbox{ with 
even cardinal}.
\end{equation}

These  identities are known to hold true: they are particular cases of a more general 
family of identities 
between `overlapping pfaffians'  established by Knuth (see \cite[Prop.\,2.1]{Okada}).  

\begin{prop}
\label{Eq:Phi-n-i--}
The following diagram of rational maps is commutative
$$
\xymatrix@R=1.2cm@C=1.6cm{
{\rm Asym}_n(\mathbf C)   \ar@{->}[d]_{\widetilde\Phi_{n,n-1}^-}
\ar@{^{(}->}[r]^{ {}^{} \hspace{0.6cm} W_n^+} &  \mathbb S_n^+  
\ar@{-->}[d]^{ \Phi_{n,n-1}^-}
  \ar@{^{(}->}[r] &  \mathbf P S_{n}^ {+}
   \ar@{-->}[d]^{ \Lambda_{n,n-1}}
 \\ 
 {\rm Asym}_{n-1}(\mathbf C)
 \ar@{^{(}->}[r]^{ {}^{} \hspace{0.6cm} W_{n-1}^+} &  \mathbb S_{n-1}^+  
  \ar@{^{(}->}[r] &  \mathbf P S_{n-1}^ {+}\, , 
 }
$$
where $\widetilde\Phi_{n,n}^- $ stands for the rational map associating to an antisymmetric matrix 
$A \in {\rm Asym}_n(\mathbf C) $ the  antisymmetric $n\times n$-matrix whose coefficients are given by   formulas  \eqref{Eq:formulas-x-tilde}. 
\end{prop}

It is only for convenience that this proposition has been stated for the specific case of the 
face map 
with respect to $\Delta_{n,n}^-$, when the considered linear projection 
$\mathbf P S_{n}^ {+} \dashrightarrow 
\mathbf P S_{n-1}^ {+}$ is induced by $\Lambda_n^{n-1}$.   
Up to an explicit bijection for any $j\in  [ \hspace{-0.06cm}[\, n-1\,]\hspace{-0.06cm}]$,  it is not difficult to give a similar formula in the case when the linear 
projection 
$\mathbf P S_{n}^ {+} \dashrightarrow  \mathbf P S_{n-1}^ {+}$ 
has been chosen to be the one induced by $\Lambda_{n}^j$.
And more generally, for any $i\in  [ \hspace{-0.06cm}[\, n-1\,]\hspace{-0.06cm}]$,  
one can give similar explicit formulas for several models 
${\widetilde{\Phi}}_{n,i}^{-,j} : {\rm Asym}_n(\mathbf C)\dashrightarrow 
{\rm Asym}_{n-1}(\mathbf C)$ for the face map with respect to the facet $\Delta_{n,i}^{-}$, one model for each 
$j\in  [ \hspace{-0.06cm}[\, n-1\,]\hspace{-0.06cm}]\setminus \{i\}$. The details and the (somewhat tedious!) task of making the ${\widetilde{\Phi}}_{n,i}^{-,j}$ explicit in the general case are left to the reader.
\sk 

However, as an example, let us consider the case when $n=5$ which moreover is the one relevant considering the purpose of this paper. 

From easy computations, one gets that  
$\widetilde \Phi_{5,5}^-= \widetilde \Phi_{5,5}^{-,4}: 
{\rm Asym}_5(\mathbf C)  \dashrightarrow {\rm Asym}_4(\mathbf C)$ is the rational map
$$
\big[ x_{i,j} \big]_{i,j=1}^5 
\longmapsto 
\left[\begin{array}{cccc}
0 & \frac{x_{1,2} x_{4,5}-x_{1,4} x_{2,5}+x_{1,5} x_{2,4}}{x_{4,5}} & \frac{x_{1,3} x_{4,5}-x_{1,4} x_{3,5}+x_{1,5} x_{3,4}}{x_{4,5}} & \frac{x_{1,5}}{x_{4,5}} 
\\
 -\frac{x_{1,2} x_{4,5}-x_{1,4} x_{2,5}+x_{1,5} x_{2,4}}{x_{4,5}} & 0 & \frac{x_{2,3} x_{4,5}-x_{2,4} x_{3,5}+x_{2,5} x_{3,4}}{x_{4,5}} & \frac{x_{2,5}}{x_{4,5}} 
\\
 -\frac{x_{1,3} x_{4,5}-x_{1,4} x_{3,5}+x_{1,5} x_{3,4}}{x_{4,5}} & -\frac{x_{2,3} x_{4,5}-x_{2,4} x_{3,5}+x_{2,5} x_{3,4}}{x_{4,5}} & 0 & \frac{x_{3,5}}{x_{4,5}} 
\\
 -\frac{x_{1,5}}{x_{4,5}} & -\frac{x_{2,5}}{x_{4,5}} & -\frac{x_{3,5}}{x_{4,5}} & 0 
\end{array}\right]\, .
$$
In order of describing the others face maps $\widetilde \Phi_{5,i}^{-,j}$ 
for $i,j\in [ \hspace{-0.06cm}[\, 5\,]\hspace{-0.06cm}]$ distinct : 
\begin{itemize}
\vspace{-0.15cm}
\item[$\bullet$] we denote by $x_{ij}$ the rational map 
 associating its $(i,j)$-th coefficient to a matrix; 
\item[$\bullet$] let $k_1,k_2$ and $k_3$ be such that 
$[ \hspace{-0.06cm}[\, 5\,]\hspace{-0.06cm}]\setminus \{i,j\} = \{ k_1,k_2, k_3 \}$ and 
$1\leq k_1<k_2<k_3\leq 5$; 
\item[$\bullet$]  we consider the rational map 
\begin{align*}
\Psi_{i,j} : {\rm Asym}_5(\mathbf C) & \, \dashrightarrow {\rm Asym}_{[ \hspace{-0.06cm}[\, 5\,]\hspace{-0.06cm}]\setminus \{i\}}(\mathbf C)\\
 (x_{ij})_{i,j \in 
[ \hspace{-0.06cm}[\, 5\,]\hspace{-0.06cm}]}  & 
\longmapsto (\tilde x_{u,v})_{u,v \in [ \hspace{-0.06cm}[\, 5\,]\hspace{-0.06cm}]\setminus \{i\}}
\end{align*}
whose  coefficients $ \tilde x_{u,v}$ are given by 
\begin{itemize}
\item[$-$] $\tilde x_{k,j}=x_{ki}/{x_{ij}}$ for $k\in \{ 1,\ldots,i-1\}$ with $k\neq j$;
\item[$-$]  $\tilde x_{k,j}=-x_{ik}/x_{ij}$ for $k \in \{ i+1,\ldots,5\}$ with $k\neq j$;   and 
\item[$-$]  $\tilde x_{k_a,k_b}={\rm Pf}\big( A_{ \{ k_a,k_b, i,j \} }\big) /x_{ij}$ for all $a,b$ such that $1\leq a <b\leq 3$. 
\end{itemize}
\end{itemize}
\begin{cor} 
\label{Cor:Phi---n=5}
For any $i,j \in [ \hspace{-0.06cm}[\, 5\,]\hspace{-0.06cm}]$ distinct, $\Psi_{i,j}$ is an expression in some Wick's coordinates for the face map $\Phi_{5,i}^- : \mathbb S_5^+\dashrightarrow \mathbb S_4$ associated to the facet $\Delta_{5,i}^-$ of the moment polytope $\Delta_{5}$. 
\end{cor}

 \paragraph{\bf The Gelfand-MacPherson web of the spinor variety.}
Writting $\mathbb S_n$ for $\mathbb S_n^+$, we can gather all the results above in the following statement which gives a geometric and explicit description of the Gelfand-MacPherson web of 
$\mathbb S_n$: 
\begin{prop}
\label{Prop:W-GM-Sn}
 \begin{itemize}
  \item[{\rm 1.}] As a web defined by rational first integrals on the $n$-th spinor variety, 
 one has  
  $$
  \boldsymbol{\mathcal W}_{ {}^{} \hspace{-0.05cm} \mathbb S_n}^{GM} = 
  \boldsymbol{\mathcal W}\Big( \, 
  \Phi_{n,i}^+\, , \, \Phi_{n,i}^- \hspace{0.15cm}  \big\lvert \hspace{0.15cm}   i=1,\ldots,n\, 
  \Big) \, . 
  $$
  It is a $2n$-web of dimension $n-1$ on  $\mathbb S_n$     (which is of dimension $n(n-1)/2$). 
  \mk
  \item[{\rm 2.}] Geometric descriptions of the face maps $\Phi_{n,i}^+$ and $\Phi_{n,i}^-$ are given above 
  in Proposition \ref{Prop:Geom-Phini+} and Proposition \ref{Prop:Geom-Phini-} respectively. 
  \mk
  \item[{\rm 3.}] 
 Up to the birational identifications ${\rm Asym}_m(\mathbf C)\simeq \mathbb S_m^+$ 
 provided by Wick's maps for $m=n,n-1$, 
 a corresponding birational model of  the face map $\Phi_{n,i}^+ : \mathbb S_n^+
 \dashrightarrow \mathbb S_{n-1}^+ $ is given by 
 \begin{align*}
 \widetilde \Phi_{n,i}^+  :  {\rm Asym}_n(\mathbf C) & \dashrightarrow {\rm Asym}_{n-1}(\mathbf C)\\
 A & \longrightarrow A_{\hat \imath} \, .
 \end{align*}
As for the $n$ other face maps $\Phi_{n,i}^- : \mathbb S_n^+
 \dashrightarrow \mathbb S_{n-1}^- $,  for each $j\in [ \hspace{-0.06cm}[\, n\,]\hspace{-0.06cm}]\setminus \{i\}$ and up to some explicit identifications, it admits a birational model  
  \begin{align*}
 \widetilde \Phi_{n,i}^{-,j}  :  {\rm Asym}_n(\mathbf C) & \dashrightarrow {\rm Asym}_{n-1}(\mathbf C)\\
 A & \longrightarrow \widetilde A_{ij} 
 \end{align*}
 where up to sign, the non trivial coefficients of $\widetilde A_{ij}$ are of the form $ x_{k,i}/x_{i,j}$  or ${\rm Pf}( A_{ \{k,l,i,j\}})/x_{i,j}$ for  $k,l\in 
 [ \hspace{-0.06cm}[\, n\,]\hspace{-0.06cm}]\setminus \{i,j\}$.
\end{itemize}
\end{prop}

 \paragraph{\bf Birational models for the quotients of $\mathbb S_5$ and $\mathbb S_4$ by the Cartan tori.}
For any $n\geq 4$, we recall that $\mathbb S_n^*$ stands for the complement in $\mathbb S_n \subset \mathbf P S_n^+$ of the union of the hyperplanes of coordinates (relatively to the weight basis $ \big\{ \, e_{{}^c \hspace{-0.05cm}L}\, \lvert \hspace{0.1cm} L\subset  [ \hspace{-0.06cm}[\, n\,]\hspace{-0.06cm}]\, \mbox{ even }
 \big\}$ of $S_n^+=\wedge^{\varrho} V$). 
 The stabilizer of  any $\xi \in \mathbb S_n^*$ in $H$  is $\{ \pm {\rm Id}_E \}\simeq \mathbf Z/2\mathbf Z$ thus $H'=H/\langle  \pm {\rm Id}_E \rangle$ acts freely on $\mathbb S_n^*$ hence  the quotient 
  $\boldsymbol{\mathcal Y}_n^*=\mathbb S_n^*/H'$ exists as a smooth quasi-projective variety. 
 However, even in the cases when $n=5$ and $n=4$ which are the two cases relevant for describing $\boldsymbol{\mathcal W}_{ \hspace{-0.1cm}{\rm dP}_4 }$ in terms of a Gelfand-MacPherson web, describing precisely $\boldsymbol{\mathcal Y}_n^*$ is a bit delicate. 
 Since it will suffice for what we are aiming for, we will instead consider and work with simple birational models of $\boldsymbol{\mathcal Y}_5^*$
  and $\boldsymbol{\mathcal Y}_4^*$ that we are going to describe now.  Note that everything below can be easily generalized to $n$ arbitrary but we will not elaborate on this as it is of no interest for our purpose.
 
We deal in detail with the case when $n=5$ and are more succinct about the case when $n=4$, since it is treated in a very similar way.  Let  $A_5(x)=(x_{i,j}\big)_{i,j=1}^{5}\in {\rm Asym}_5(\mathbf C)$ be a 
$5\times 5$  antisymmetric matrix  such that any of its non diagonal coefficients $x_{i,j}$ (with $i\neq j$) does not vanish.  One sets $H(x)=
\scalebox{0.75}{\bigg[ \hspace{-0.07cm}\begin{tabular}{l} $h(x)$ \, 0 \vspace{-0.1cm}\\ 0 \, $h(x)^{-1}$ 
\end{tabular}  \hspace{-0.4cm}
\bigg]} $ where $h(x)$ stands for the diagonal $5\times 5$ matrix ${\rm Diag}(\sqrt{h_1},\ldots,\sqrt{h_{5}})$ whose diagonal coefficients are given by 
\begin{align}
\label{Eq:sqrt-hi}
\sqrt{h_{1}}= &\, \frac{x_{3,4} x_{1,5} x_{1,2}}{x_{2,3} x_{4,5}}
 \quad 
&&\sqrt{h_{2}}=\frac{x_{1,2} x_{2,3} x_{4,5}}{x_{3,4} x_{1,5}}
 \quad 
&&\sqrt{h_{3}}=\frac{x_{2,3} x_{3,4} x_{1,5}}{x_{1,2} x_{4,5}}
 \quad \\ 
\sqrt{h_{4}}=& \, \frac{x_{3,4} x_{1,2} x_{4,5}}{x_{2,3} x_{1,5}}
&&\sqrt{h_{5}}=\frac{x_{1,5} x_{2,3} x_{4,5}}{x_{3,4} x_{1,2}} \, . 
\nonumber 
\end{align}

Actually, the presence of square roots in this definition induces a certain ambiguity with the consequence that  $H(x)$ is not well-defined as an element of the diagonal torus in ${\rm GL}(E)$. However, this ambiguity can be resolved arguing as follows: 
the set  of diagonal matrices $D_h$ for $h\in (\mathbf C^*)^5$ described above (in \eqref{Eq:D-h}) is the image of the Cartan torus $H_{D_5}$ of ${\rm Spin}_{10}$ in ${\rm GL}(E)$. The map $(\mathbf C^*)^5\rightarrow \mathbf C^*$, $(h_s)_{s=1}^5\mapsto h_i$, denoted abusively by $h_i$ from now on, is a character on $H$ which admits a square root when pulled-back on $H_{D_5}$, which will be denoted by $\sqrt{h_i}$. 
 With this notation, the coordinate ring of $H_{D_5}$ is $\mathbf C\big[ H_{D_5}\big]=\mathbf C\big[ h_1^{\pm 1/2},\ldots,
 h_5^{\pm 1/2} \big]$ and now the formulas \eqref{Eq:sqrt-hi} unambiguously define a rational map 
 ${\bf H}: {\rm Asym}_5(\mathbf C)\dashrightarrow H_{D_5}$.  For  $x$ as above, ${\bf H}$ is defined at $x$ and 
given an arbitrary matrix $A(y)=(y_{i,j})_{i,j=1}^5 \in {\rm Asym}_5(\mathbf C)$, one has 
$$ A(y) \circ {\bf H}(x)=
h(x)^{-1}\cdot \Big[\, {\rm Id}_5, A(y)\,\Big]\cdot H(x)=
\left[\begin{array}{ccccc}
0 & \frac{y_{1,2}}{\sqrt{h_{1}}\, \sqrt{h_{2}}} & \frac{y_{1,3}}{\sqrt{h_{1}}\, \sqrt{h_{3}}} & \frac{y_{1,4}}{\sqrt{h_{1}}\, \sqrt{h_{4}}} & \frac{y_{1,5}}{\sqrt{h_{1}}\, \sqrt{h_{5}}} 
\\
 -\frac{y_{1,2}}{\sqrt{h_{1}}\, \sqrt{h_{2}}} & 0 & \frac{y_{2,3}}{\sqrt{h_{2}}\, \sqrt{h_{3}}} & \frac{y_{2,4}}{\sqrt{h_{2}}\, \sqrt{h_{4}}} & \frac{y_{2,5}}{\sqrt{h_{2}}\, \sqrt{h_{5}}} 
\\
 -\frac{y_{1,3}}{\sqrt{h_{1}}\, \sqrt{h_{3}}} & -\frac{y_{2,3}}{\sqrt{h_{2}}\, \sqrt{h_{3}}} & 0 & \frac{y_{3,4}}{\sqrt{h_{3}}\, \sqrt{h_{4}}} & \frac{y_{3,5}}{\sqrt{h_{3}}\, \sqrt{h_{5}}} 
\\
 -\frac{y_{1,4}}{\sqrt{h_{1}}\, \sqrt{h_{4}}} & -\frac{y_{2,4}}{\sqrt{h_{2}}\, \sqrt{h_{4}}} & -\frac{y_{3,4}}{\sqrt{h_{3}}\, \sqrt{h_{4}}} & 0 & \frac{y_{4,5}}{\sqrt{h_{4}}\, \sqrt{h_{5}}} 
\\
 -\frac{y_{1,5}}{\sqrt{h_{1}}\, \sqrt{h_{5}}} & -\frac{y_{2,5}}{\sqrt{h_{2}}\, \sqrt{h_{5}}} & -\frac{y_{3,5}}{\sqrt{h_{3}}\, \sqrt{h_{5}}} & -\frac{y_{4,5}}{\sqrt{h_{4}}\, \sqrt{h_{5}}} & 0 
\end{array}\right]
$$
where $\circ$ stands here for the action of ${\bf H}(x)$ on the point of $\mathbb S_5$ corresponding to $A(y)$ and where the dots $\cdot$ in the expression after the first equal sign refer to  matricial products. 
Taking $y=x$ in the above identity, an elementary computation gives us that one has $A(x) \circ {\bf H}(x)=A'(x)$
with 
$$
A'(x)=\left[\begin{array}{ccccc}
0 & 1 & \frac{x_{4,5}\, x_{1,3}}{x_{1,5} \, x_{3,4}} & \frac{x_{2,3} \, x_{1,4}}{x_{3,4} \,x_{1,2}} & 1 
\vspace{0.2cm}\\
 -1 & 0 & 1 & \frac{x_{1,5}\, x_{2,4}}{x_{1,2}\, x_{4,5}} & \frac{x_{3,4} \, x_{2,5}}{x_{2,3} \, x_{4,5}} 
\vspace{0.2cm} \\
 -\frac{x_{4,5} \,x_{1,3}}{x_{1,5} \, x_{3,4}} & -1 & 0 & 1 & \frac{x_{1,2} \, x_{3,5}}{x_{1,5} \, x_{2,3}} 
\vspace{0.2cm}\\
 -\frac{x_{2,3} \, x_{1,4}}{x_{3,4} \, x_{1,2}} & -\frac{x_{1,5} \, x_{2,4}}{x_{1,2} \, x_{4,5}} & -1 & 0 & 1 
\vspace{0.2cm} \\
 -1 & -\frac{x_{3,4} \, x_{2,5}}{x_{2,3} \, x_{4,5}} & -\frac{x_{1,2} \, x_{3,5}}{x_{1,5} \, x_{2,3}} & -1 & 0 
\end{array}\right] \in {\rm Asym}_5(\mathbf C)\,. 
$$
Conversely, one verifies that $A(x)=A'(x) \circ {\bf H}(x)^{-1}$. Then denoting by $u_{a,b}$ the coefficients of 
$A'(x)$ over the diagonal which are distinct from 1 (ie.\,$u_{1,3}=x_{4,5} x_{1,3}/(x_{1,5} x_{3,4})$, 
$u_{2,3}=x_{2,3} \, x_{1,4}/(x_{3,4} \,x_{1,2})$ etc), 
 using $\sqrt{h_i}$ ($i=1,\ldots,5$) for the quantities defined in \eqref{Eq:sqrt-hi} and 
  considering all these quantities as rational functions in the $x_{ij}$'s, 
we get a rational map 
\begin{align} 
\label{Eq:Map-u-h}
{\rm Asym}_5(\mathbf C)& \, \dashrightarrow \mathbf C^5_u\times H_{D_5} 
\\ A(x) & 
 \longmapsto   
 \left(\, 
\big( \, u_{1,3}\, , \, \ldots, u_{3,5}
\,  \big)
\, , \, \Big( \sqrt{h_1},\ldots,\sqrt{h_5}\, 
\Big)
\right)  
\nonumber 
\end{align}
which is easily shown to provide 
  a birational model for the quotient of 
$\mathbb S_5$ by $H_{D_5}$: 
\begin{prop}
\label{Prop:Birat-Model-P5}
 The map \eqref{Eq:Map-u-h} is birational. The corresponding isomorphism of fields of rational functions $\mathbf C( \mathbb S_5) \simeq \mathbf C\big( {\rm Asym}_5(\mathbf C)\big)
\simeq 
\mathbf C(u)\otimes_{ \mathbf C } \mathbf C(H_{D_5})$
 is such that 
$$ \mathbf C\big( \mathbb S_5\big)^{H_{D_5}}\simeq \mathbf C\big( \mathbf C^5_u\times H_{D_5}\big)^{ H_{D_5}}
\simeq \mathbf C\big( u\big)\, .$$ 
In particular,  the rational map 
 $\mathcal P_5 : 
 {\rm Asym}_5 (\mathbf C) \dashrightarrow  \mathbf C^5$ 
given by 
\begin{align*}
 \scalebox{0.79}{$\begin{bmatrix}
0 & x_{1,2} & x_{1,3}& x_{1,4}  &  x_{1,5}\\
 -x_{1,2}&  0& x_{2,3} & x_{2,4} & x_{2,5}\\
 -x_{1,3}&  -x_{2,3} & 0& x_{3,4}  & x_{3,5} \\
-x_{1,4} & -x_{2,4}  & -x_{3,4}& 0 &  x_{4,5}\\
-x_{1,5} & -x_{2,5}  & -x_{3,5}&  -x_{4,5}&  0
 \end{bmatrix} $} \longmapsto 
\left( \, \frac{x_{4,5}\, x_{1,3}}{x_{1,5} \,x_{3,4}}
\, , \, 
\frac{x_{1,5} \,x_{2,4}}{x_{1,2} \, x_{4,5}}
\, , \, 
\frac{x_{1,2} \,x_{3,5}}{x_{1,5}\, x_{2,3}}
\, , \, 
\frac{x_{2,3} \, x_{1,4}}{x_{3,4} \,x_{1,2}}
\, , \, 
\frac{x_{3,4} \,x_{2,5}}{x_{2,3} 
\, x_{4,5}} \right)
\end{align*}
is a  birational model  in Wick's coordinates for the quotient map $\chi_{5} : {\mathbb S}_5\rightarrow \boldsymbol{\mathcal Y}_5 ={\mathbb S}_5/ H_{D_5}$.
\end{prop}

This proposition can be generalized for $n\geq 4$ arbitrary. However, this requires
 two similar but formally distinct treatments according to the parity of $n$. We will not elaborate further on this but  only give the corresponding statement in the case when $n=4$ which we are going to use in the next paragraph.

\begin{prop}
\label{Lem:Birat-Model-P4}
 The rational map
 $\mathcal P_4 : 
 {\rm Asym}_4 (\mathbf C) \dashrightarrow  \mathbf C^2$
given by 
\begin{align*}
 \scalebox{0.95}{$\begin{bmatrix}
0 & x_{1,2} & x_{1,3}& x_{1,4} \\
 -x_{1,2}&  0& x_{2,3} & x_{2,4}\\
 -x_{1,3}&  -x_{2,3} & 0& x_{3,4} \\
-x_{1,4} & -x_{2,4}  & -x_{3,4}& 0
\end{bmatrix}$} & \longmapsto \left( 
\, 
\frac{x_{1,4}\,x_{2,3}}{x_{1,2}\,x_{3,4}}
\, , \,  \frac{x_{1,3}\,x_{2,4}}{x_{1,2}\,x_{3,4}}
\, 
\right) 
\end{align*}
is a  birational model for the quotient map  $\chi_{4} : {\mathbb S}_4\dashrightarrow 
\boldsymbol{\mathcal Y}_4 =  {\mathbb S}_4/ H_{D_4}$.
\end{prop}

%
 \paragraph{\bf Birational models for the equivariant quotient of the Gelfand-MacPherson web $\boldsymbol{\mathcal W}_{ {}^{} \hspace{-0.05cm} \mathbb S_5}^{\boldsymbol{GM}}$.}
Our aim here is to give an explicit birational model for the $H_{D_5}$-equivariant quotient 
of Gelfand-MacPherson web as defined in \eqref{Eq:W-GM/H}.

For $n=4,5$, we denote by $\boldsymbol{Y}_n$ the affine complex space which is the target space of the map $\mathcal P_n$ of the two preceding propositions (thus $\boldsymbol{Y}_5=\mathbf C^5$ and $\boldsymbol{Y}_4=\mathbf C^2$).
From  the two previous propositions and Corollary \ref{C:cocor}, it follows that 
there exist birational maps 
$\omega_n : \boldsymbol{Y}_n \stackrel{\sim}{\dashrightarrow} \boldsymbol{\mathcal Y}_n$ for $n=5,4$;  and   dominant rational maps 
$\varphi_{5,i}^\varepsilon : 
\boldsymbol{\mathcal Y}_5 \dashrightarrow 
\boldsymbol{\mathcal Y}_4$ and 
$\psi_i^\varepsilon : \mathbf C^5= \boldsymbol{Y}_5\dashrightarrow \boldsymbol{Y}_4=\mathbf C^2$ for $i=1,\ldots,5$ and $\varepsilon=\pm$, 
 making  
the following diagram  commutative:  
\begin{equation}
\label{Diag:gogogo}
\begin{tabular}{c}
\xymatrix@R=1.3cm@C=1.3cm{
{\rm Asym}_5\big(\mathbf C\big)   \ar@{->}[d]^{ {\mathcal P}_5} \ar@{->}[r]^{ {}^{}\hspace{0.4cm}W_5}  &
\mathbb S_5 
 \ar@{->}[rr]^{ \Phi_{5,i}^\varepsilon}   
\ar@{->}[d]^{\pi_5}
&  &\mathbb S_4  \ar@{->}[d]^{\pi_4}  \ar@{<-}[r]^{W_4}  & {\rm Asym}_4\big(\mathbf C\big)  \ar@{->}[d]^{ {\mathcal P}_4} \\
\ar@/^2pc/@{.>}[u]^{ \scalebox{1.2}{$\textcolor{darkgreen}{\sigma}$}}
\mathbf C^5
\ar@/_2pc/[rrrr]^{ \psi_{5,i}^\varepsilon}
  \ar@{->}[r]^{\omega_5}& \, \,  \boldsymbol{\mathcal Y}_5
 \ar@{->}[rr]^{ \varphi_{5,i}^\varepsilon}  
  & &  \boldsymbol{\mathcal Y}_4\ar@{<-}[r]^{\omega_4}   &  \mathbf C^2}
\end{tabular}\vspace{0.3cm}
\end{equation}
(Remarks: ({\it i}) This is a diagram of rational maps hence  all the arrows in it should be dashed according to our typographic convention. They are not just for aesthetic reasons; ({\it ii}) The green map 
$\sigma$ with the dotted arrow on the left is defined a few paragraphs below, cf. \eqref{Eq:Sigma}). 

Since $\mathcal P_5$ and $\mathcal P_4$ are birational models for the torus quotients of 
$\mathbb S_5$ and $\mathbb S_4$ respectively, it follows that the pull-back under $\omega_5$ of the equivariant quotient of $\boldsymbol{\mathcal W}_{ {}^{} \hspace{-0.05cm} \mathbb S_5}^{{GM}}$ by $H_{D_5}$ is the web on $\mathbf C^5$ defined by the 
rational maps $\psi_{5,i}^{\varepsilon}$: one has 
$$
\omega_5^*\bigg( 
\left( \boldsymbol{\mathcal W}_{ {}^{} \hspace{-0.05cm} \mathbb S_5}^{{GM}}
\right) \big/_{ H_{D_5}} \bigg)=
\boldsymbol{\mathcal W}\Big( \hspace{0.1cm}  \psi_{5,i}^{\varepsilon} \hspace{0.15cm} \big\lvert 
\hspace{0.15cm} 
i=1,\ldots,5
 \, ,  \,   \varepsilon =\pm \hspace{0.1cm} \Big)\, . 
$$

Because  we have explicit rational expressions for the maps $\Phi_{5,i}^\varepsilon$ ( in Wick's coordinates) as well as for $\mathcal P_5$ and $\mathcal P_4$, one can easily get expressions of the same kind for the maps $\psi_{5,i}^{\varepsilon}$'s hence give an explicit birational model for $\boldsymbol{\mathcal W}_{ {}^{} \hspace{-0.05cm} \boldsymbol{\mathcal Y}_5}^{{GM}}= 
\boldsymbol{\mathcal W}_{ {}^{} \hspace{-0.05cm} \mathbb S_5}^{{GM}}
/  H_{D_5}$. 
\mk 

Let us start by making explicit
the maps $\psi_i^+=\psi_{5,i}^{+}$'s (for which there is no ambiguity). 
The rational map $\sigma : \mathbf C^5 \dashrightarrow {\rm Asym}_5(\mathbf C)$ given by
\begin{equation}
\label{Eq:Sigma}
y=
\big(\,  y_{1,3} \, , \,  y_{1,4}
\, , \,  y_{2,4} 
 \, , \,
 y_{2,5} 
 \, , \,
 y_{3,5} 
 \big) \mapsto 
\left[\begin{array}{ccccc}
0 & 1 & y_{1,3} & y_{1,4} & 1 
\\
 -1 & 0 & 1 & y_{2,4} & y_{2,5} 
\\
 -y_{1,3} & -1 & 0 & 1 & y_{3,5} 
\\
 -y_{1,4} & -y_{2,4} & -1 & 0 & 1 
\\
 -1 & -y_{2,5} & -y_{3,5} & -1 & 0 
\end{array}\right]
\end{equation}
is obviously a rational section of $\mathcal P_5$.  It follows that for $i=1,\ldots,5$, one has  
 $$
 \psi_{i}^+=  \mathcal P_4\circ {W_4}^{-1} \circ \Phi_{5,i}^+   \circ W_5  \circ  \sigma : 
  \mathbf C^5 \dashrightarrow \mathbf C^2 \, . 
  $$ 
 Since $\widetilde \Phi_{i}^+= {W_4}^{-1} \circ \Phi_{5,i}^+   \circ W_5$ is known (see 
 Proposition \ref{Prop:W-GM-Sn}.3),  it is just a matter of elementary computations to get the following explicit and quite simple formulas for the  $\psi_{5,i}^+$'s: 
\begin{align}
\label{Eq:Maps-psi-i}
\psi_{1}^+(y)=   &\,  \big(\, y_{2,5} \, , \, y_{2,4} \,y_{3,5}\, \big)\, , 
&&\psi_{2}^+(y)=  \left(\, \frac{1}{y_{1,3}}\, ,\,  \frac{y_{1,4} \,y_{3,5}}{y_{1,3}}\, \right)
\, , \quad  \psi_{3}^+(y)= \left(\, y_{2,4}\, , \, y_{1,4} \, y_{2,5}\, \right) \\
\psi_{4}^+(y)= & \,   \left(\, \frac{1}{y_{3,5}}\, , \,  \frac{y_{1,3}\, y_{2,5}}{y_{3,5}}\, \right)
&&
\psi_{5}^+(y)=  \big(\, y_{1, 4} \, , \,  y_{1, 3}\, y_{2, 4} \, \big)\, . 
\nonumber
\end{align}
We now turn to the  face maps $\psi_{i}^-=\psi_{5,i}^-$ for $i=1,\ldots,5$.  As explained above, for $i\in [ \hspace{-0.06cm}[\, 5\,]\hspace{-0.06cm}]$ given, there is no natural choice for an expression in Wick's coordinates for the 
map $\Phi_{5,i}^-$  and one has to choose and work with one of the maps $\widetilde \Phi_{5,i}^{-,j}$ with $j\neq i$.  But once such a map has been chosen, possibly in an arbitrary way, one can compute an explicit rational expression for a model
of $\psi_{i}^-$ with respect to some Wick's coordinates, which we will denote by 
$\psi_{i,j}^{-}$. 

For example, 
the map $\widetilde \Phi_{5,1}^{- , 5}$ is the rational map 
which associates to a generic antisymmetric matrix $(x_{ij})_{i,j=1}^5 \in {\rm Asym}_5(\mathbf C)$ 
the following one 
$$
\left[\begin{array}{cccc}
0 & \frac{x_{1,2} x_{3,5}-x_{1,3} x_{2,5}+x_{1,5} x_{2,3}}{x_{1,5}} & \frac{x_{1,2} x_{4,5}-x_{1,4} x_{2,5}+x_{1,5} x_{2,4}}{x_{1,5}} & \frac{x_{1,2}}{x_{1,5}} 
\\
 -\frac{x_{1,2} x_{3,5}-x_{1,3} x_{2,5}+x_{1,5} x_{2,3}}{x_{1,5}} & 0 & \frac{x_{1,3} x_{4,5}-x_{1,4} x_{3,5}+x_{1,5} x_{3,4}}{x_{1,5}} & \frac{x_{1,3}}{x_{1,5}} 
\\
 -\frac{x_{1,2} x_{4,5}-x_{1,4} x_{2,5}+x_{1,5} x_{2,4}}{x_{1,5}} & -\frac{x_{1,3} x_{4,5}-x_{1,4} x_{3,5}+x_{1,5} x_{3,4}}{x_{1,5}} & 0 & \frac{x_{1,4}}{x_{1,5}} 
\\
 -\frac{x_{1,2}}{x_{1,5}} & -\frac{x_{1,3}}{x_{1,5}} & -\frac{x_{1,4}}{x_{1,5}} & 0 
\end{array}\right]
\in {\rm Asym}_4(\mathbf C)\, .
$$
From this,  one easily deduces that the rational map $\psi_{1,5}^{-} : \mathbf C^5 \dashrightarrow \mathbf C^2$ is given by
\begin{equation}
\label{Eq:Psi1*5}
\psi_{1,5}^{-} (y)= 
\left(\, 
\frac{y_{1,4} y_{3,5}-y_{1,3}-1}{y_{1,4} \left(y_{1,3} y_{2,5}-y_{3,5}-1\right)}
\, , \, 
\frac{ y_{1,3}\left(y_{1,4} y_{2,5}-y_{2,4}-1\right)}{y_{1,4} \left(y_{1,3} y_{2,5}-y_{3,5}-1\right)}
\, 
\right)
\end{equation}
for  $y= \big(\,  y_{1,3} \, , \,  y_{1,4}
\, , \,  y_{2,4} 
 \, , \,
 y_{2,5} 
 \, , \,
 y_{3,5} 
 \big) \in \mathbf C^5$.  Similar explicit formulas can be given for any of the $\psi_{i,j}^-$'s: for instance,  straightforward but a bit lenghty (hence not reproduced here) computations give us that 
%
%
\begin{align}
\label{Eq:Psii*j}
\psi_{2,5}^-(y)=  &\,  \left(\, 
\frac{y_{2,4} y_{3,5}-y_{2,5}-1}{y_{2,4} \left(y_{1,3} y_{2,5}-y_{3,5}-1\right)}
\, , \, 
\frac{y_{1,4} y_{2,5}-y_{2,4}-1}{y_{2,4} \left(y_{1,3} y_{2,5}-y_{3,5}-1\right)}
\, 
\right) \nonumber  
\\
\psi_{3,5}^-(y)=  &\,  
\left(\, 
\frac{ y_{1,3} \left(y_{2,4} y_{3,5}-y_{2,5}-1\right)}{y_{1,3} y_{2,5}-y_{3,5}-1}
\, , \, 
 \frac{y_{1,4} y_{3,5}-y_{1,3}-1}{y_{1,3} y_{2,5}-y_{3,5}-1}\, \right)
\\
\psi_{4,5}^-(y)=  &\,  \left(\, 
 \frac{y_{1,4}  \left(y_{2,4} y_{3,5}-y_{2,5}-1\right)}{y_{1,4} y_{2,5}-y_{2,4}-1}
\, , \, 
\frac{y_{2,4} \left(y_{1,4} y_{3,5}-y_{1,3}-1\right) }{y_{1,4} y_{2,5}-y_{2,4}-1}
\, 
\right) \nonumber 
\\
\mbox{and } \quad \psi_{5,4}^-(y)=  &\, \left(\, 
\frac{y_{2,4} y_{3,5}-y_{2,5}-1}{y_{3,5} \left(y_{1,4} y_{2,5}-y_{2,4}-1\right)}
\, , \, 
\frac{y_{2,5} \left(y_{1,4} y_{3,5}-y_{1,3}-1\right) }{y_{3,5} \left(y_{1,4} y_{2,5}-y_{2,4}-1\right)}
\, 
\right) \, . \nonumber 
\end{align}
%
%
%
%
 
 To simplify the writing, we set $\psi_i^-=\psi_{i,5}^-$ for $i=1,\ldots,4$ and $\psi_5^-=\psi_{5,4}^-$. \mk 
 
\begin{prop}
\label{P:W-GM-Y5}
In the system of rational coordinates $\big(y_{1,3} \, , \,  y_{1,4}
\, , \,  y_{2,4} 
 \, , \,
 y_{2,5} 
 \, , \,
 y_{3,5} \big)$ on $\mathbf C^5\simeq \boldsymbol{\mathcal Y}_5$, Gelfand-MacPherson's web $\boldsymbol{\mathcal W}_{ {}^{} \hspace{-0.05cm} \boldsymbol{\mathcal Y}_5}^{{GM}}$ is defined by the ten rational first integrals 
$\psi_i^{\pm}$ for $i=1,\ldots,5$.
\end{prop} 

 \paragraph{\bf Serganova-Skorobogatov's embedding and a description of $\WdPq$ \`a la Gelfand-MacPherson.}
Our goal in this paragraph is to explain how to construct $\WdPqq$ from $\boldsymbol{\mathcal W}_{ {}^{} \hspace{-0.05cm} \mathbb S_5}^{{GM}}$. For this sake, we need to recall first a result (due to Popov in its Diplomarbeit \cite{Popov}, see also \cite{SerganovaSkorobogatov,Derenthal})
 about a natural embedding of the Cox variety of a given del Pezzo quartic surface ${\rm dP}_4$ into the spinor tenfold.

The construction sketched in \eqref{SSS:Cox-stuff} for the del Pezzo quintic surface 
can be generalized to any del Pezzo surface ${\rm dP}_d$ for $d\in \{2,\ldots,6\}$.
Given a del Pezzo surface of this degree, there exists  a blow-up 
$b : X_r={\bf Bl}_{p_1,\ldots,p_r}(\mathbf P^2) \rightarrow \mathbf P^2$ 
of  $r=9-d$ points $p_1,\ldots,p_r$  in general position in $\mathbf P^2$ such that 
the considered del Pezzo surface identifies with $X_r$. Such a description of ${\rm dP}_d$ is not unique\footnote{The set of descriptions of ${\rm dP}_d$ as such a blow-up is in one-to-one correspondence with the set of sets of $r$ pairwise disjoint lines on ${\rm dP}_d$.  
The Weyl group $W_r$ acts transitively on the latter set which allows to compute its cardinal explicitly (cf.\cite{Lee}).}  but we fix one and denote the surface by $X_r$ from now on to indicate this.  We will denote $X_r^*$ the affine Zariski-open subset of $X_r$ which is the complement of the union of all lines contained in $X_r$. 

For any  $i=1,\ldots,r$,  we set $E_i=b^{-1}(p_i)$ and we denote by $\boldsymbol{e}_i$ the class of this exceptional divisor in the Picard lattice ${\bf Pic}(X_r)$ of $X_r$. 
As mentioned above, ${\bf Pic}(X_r)$ is freely spanned over $\mathbf Z$ by the $\boldsymbol{e}_i$'s ($i=1,\ldots,n$) and  by 
the class, denoted by $\boldsymbol{h}$,  of the preimage $H=b^{-1}(L)$ of a line $L\subset \mathbf P^2\setminus 
\{p_1,\ldots,p_r\}$. By definition, the {\it `Cox ring'} associated to the divisors $H,E_1,\ldots,E_r$ whose classes are generators of the Picard lattice is 
$${\rm Cox}(X_r)=
\hspace{-0.2cm}
\bigoplus_{ m_0,\ldots,m_r
\in \mathbf Z} {\bf H}^0\bigg( X_r, \mathcal O_{X_r}\Big(m_0 H+\sum_{i=1}^r m_i E_i\Big) \bigg)
\, , $$
the ring structure being defined by the usual multiplication of sections. 

One defines the {\it `affine Cox variety'}  $\mathbf A(X_r) $ as the spectrum of the Cox ring:  
$$
\mathbf A(X_r)={\bf Spec}\Big( {\rm Cox}(X_r) \Big) \, .$$ 
 The Cox ring is naturally graded by the $\mathbf Z_\geq 0$-monoid of effective classes. From the latter, we get a  $\mathbf Z_{\geq 0}$-grading defined by $(-K,\cdot)$ which allows us to define the 
{\it `Projective Cox variety'} 
$$\mathbf P(X_r)={\bf Proj}\Big( {\rm Cox}(X_r) \Big) \,.$$
 The {\it `N\'eron-Severi torus'} of $X_r$, which by definition  is the torus 
 $$T_{NS}={\rm Hom}_{ \mathbf Z } \Big( {\bf Pic}(X_r),\mathbf C^*\Big)\simeq \big( \mathbf C^*\big)^{\{\boldsymbol{h},\boldsymbol{e}_1,\ldots,\boldsymbol{e}_r\}}\, ,$$ 
 naturally acts on the Cox ring by 
$t\cdot \sigma_D=t_{\boldsymbol{h}}^{n_0}t_{\boldsymbol{e}_1}^{n_1}\ldots t_{\boldsymbol{e}_r}^{n_r}\sigma_D$ for any 
$t=(t_{\boldsymbol{h}},t_{\boldsymbol{e}_1},\ldots,t_{\boldsymbol{e}_r})\in T_{NS}$ and any $\sigma_D\in {\bf H}^0\big( X_r, \mathcal O_{X_r}\big(n_0 H-\sum_{i=1}^r n_i E_i\big)\big) $ with $D=n_0\boldsymbol{h}-\sum_{i=1}^r n_i \boldsymbol{e}_i\in {\rm Pic}_{\mathbf Z}(X_r)$.  Let $\tau_K$ be  the 1-parameter subgroup of $T_{NS}$ spanned by the $\mathbf Z$-linear form $(-K,\cdot) : 
{\bf Pic}_{\mathbf Z}(X_r)\rightarrow \mathbf Z$. 
Then quotienting by $\tau_K$ gives rise to a surjective quotient map 
$\mathbf A(X_r) \setminus \{ 0\}\rightarrow  \mathbf P(X_r)$ from which one gets that the quotient torus ${\rm T}_r=T_{NS}/\tau_k\simeq \big(\mathbf C^*\big)^r$ 
naturally acts on $\mathbf P(X_r)$.

Here are some basic and fundamental results about the (affine and projective) Cox varieties of  $X_r$ and the torus actions on them considered above: 
\begin{itemize}
\item the Cox ring ${\rm Cox}(X_r)$ is generated by the (classes of the) lines contained in $X_r$. Denoting 
by $\sigma^\ell$ a nonzero section in ${\bf H}^0\big( X_r,\mathcal O_{X_r}(\ell)\big)$ for any $\ell\in \boldsymbol{\mathcal L}_r$,  there is an isomorphism ${\rm Cox}(X_r)\simeq \mathbf C\big[ \sigma_\ell\, \lvert \, 
\ell\in \boldsymbol{\mathcal L}_r \big]/J_r$ for a certain ideal $J_r$ which is spanned by elements of $(-K)$-degree 2;
\sk 
\item  for any divisor $Z$ on $X_r$, we denote $ \mathscr L_Z^\circ$ the total space of the line bundle associated to the sheaf $\mathcal O_{X_r}(Z)$ with the zero section removed.  Then the fiber product over $X_r$  of the $\mathbf C^*$-bundles $ \mathscr L_Z^\circ$ for $Z\in \{ H,E_1,\ldots,E_r\}$, denoted by $\mathcal T_{\hspace{-0.06cm}X_r}$,  embeds as an open-Zariski subset
into ${\bf A}(X_r)$: one has
\begin{equation}
\label{Eq:T-Xr-Torsor}
\mathcal T_{\hspace{-0.06cm}X_r}= \mathscr L_H^\circ \times_{X_r}
 \mathscr L_{E_1}^\circ \times_{X_r}\cdots  \times_{X_r}  \mathscr L_{E_r}^\circ
\subset  {\bf A}(X_r)\, .
\end{equation}
The N\'eron-Severi torus  naturally acts on $\mathcal T_{\hspace{-0.06cm}X_r}$ making of the projection 
$\nu : \mathcal T_{\hspace{-0.06cm}X_r}\rightarrow X_r$ a $T_{NS}$-bundle in a natural way.  We  denote by $\mathcal T_{\hspace{-0.06cm}X_r}^*$ the preimage of $X_r^*$ by $\nu$: 
$\mathcal T_{\hspace{-0.06cm}X_r}^*=\nu^{-1}\big( X_r^*\big)$. 
\sk
\item  the set of lines $\boldsymbol{\mathcal L}$ identifies with the set of weights of a minuscule representation 
of the Lie group $G_r$ of type $E_r$. Identifying $G_r$ with its image in ${\rm GL}
\big(\mathbf C^{\boldsymbol{\mathcal L} }\big)$, we get by restriction a group monomorphism from 
the Cartan torus $H_r$ of $G_r$ into the torus of  matrices which are diagonal with respect to the natural basis (namely $\boldsymbol{\mathcal L}$) of $\mathbf C^{\boldsymbol{\mathcal L} }$.  The stabilizer $P_r$ of the line spanned by the highest weight $\ell_{h}$ is a maximal parabolic subgroup of $G_r$ and the 
orbit 
$$
\boldsymbol{\mathcal G}_r=G_r\cdot \langle \ell_h \rangle \subset \mathbf P \big( \mathbf C^{\boldsymbol{\mathcal L} } \big)
$$ is closed and identifies with the projective homogeneous space $G_r/P_r$.   Finally, let $T_r$ be the subtorus of ${\rm GL}
\big(\mathbf C^{\boldsymbol{\mathcal L} }\big)$ spanned by $H_r$ and the 1-dimensional subtorus formed by scalar matrices;
\sk
\item seeing each global section $\sigma_\ell \in {\bf H}^0\big(X_r,\mathcal O_{X_r}(\ell)\big)$ as a non constant regular rational function on ${\bf A}(X_r)$, one constructs an embedding 
$$\widehat \varrho _r=\big( \sigma_\ell \big)_{\ell \in \boldsymbol{\mathcal L}_r }: {\bf A}(X_r) \hookrightarrow \mathbf C^{\boldsymbol{\mathcal L}_r}$$ which is 
$({T}_{NS},T_r \big)$-equivariant and which, after projectivization, induces a  
$H_r$-equivariant embedding  $\varrho_r: 
  {\bf P}(X_r) \hookrightarrow \mathbf P\big( \mathbf C^{\boldsymbol{\mathcal L}_r}\big)$ such that the following  diagram commutes:
\begin{equation*}
  \xymatrix@R=0.8cm@C=0.2cm{ 
\mathbf A(X_r) \setminus \{ 0\} \,  
\ar@{^{(}->}[rrrr]^{\widehat\varrho_r  \hspace{-0.4cm} }
\ar@{->}[d]
   &  &&&   \mathbf C^{\boldsymbol{\mathcal L} }\setminus \{ 0\} \,
  \ar@{->}[d]
   &
 \\
\mathbf P(X_r) \ar@{^{(}->}[rrrr]^{\varrho_r  \hspace{-0.4cm} }
  &  &&&  \mathbf P \big( \mathbf C^{\boldsymbol{\mathcal L} } \big)\, ,
  }
\end{equation*}
where the (surjective) vertical maps are given by quotienting by $\tau_K$ and $ \mathbf C^* {\rm Id}_{ \mathbf C^{\boldsymbol{\mathcal L}_r}}$ respectively. We denote by $\boldsymbol{\mathcal P}_r$ the image of $\varrho_r $: one has 
$\boldsymbol{\mathcal P}_r=\varrho\big( \mathbf P(X_r)  \big) \subset  \mathbf P \big( \mathbf C^{\boldsymbol{\mathcal L} } \big)$. Its is an irreducible closed subvariety of $ \mathbf P \big( \mathbf C^{\boldsymbol{\mathcal L} } \big)$, of dimension $r+2$;
\sk
\item an important result in this area  is that for suitable choices of the sections $\sigma_\ell$ for $\ell$ ranging in  $\boldsymbol{\mathcal L}_r$, one can assume that $\varrho_r$ lands into $\boldsymbol{\mathcal G}_r\simeq G_r/P_r$: one has 
\begin{equation}
\label{Eq:P(Xr)-hookrightarrow-PCLr}
\mathbf P(X_r) \hookrightarrow \boldsymbol{\mathcal P}_r\subset \boldsymbol{\mathcal G}_r 
\subset \mathbf P \big( \mathbf C^{\boldsymbol{\mathcal L} } \big)\,;
\end{equation}
\item let $\boldsymbol{\mathcal G}_r^{sf}$ be the subset of ${\boldsymbol{\mathcal G}}_r$
 formed by its  points  whose stabilizer in $H_r$ coincides with the center of $G_r$, {\it i.e.}  
 ${\boldsymbol{\mathcal G}}_r^
{sf}=\big\{\, x \in {\boldsymbol{\mathcal G}}_r=G_r/P_r \, \lvert \, {\rm Stab}_{H_r}(x)=Z(G_r)\,\big\}$. Of course $\boldsymbol{\mathcal G}_r^{sf}$ is $H_r$-invariant and 
from general and classical results of geometric invariant theory, it follows that 
$\boldsymbol{\mathcal Y}_r=\boldsymbol{\mathcal G}_r^{sf}/H_r$ is a smooth quasiprojective variety and that the quotient mapping $\chi : \boldsymbol{\mathcal G}_r^{sf}\rightarrow \boldsymbol{\mathcal Y}_r$  is a geometric quotient of $\boldsymbol{\mathcal G}_r^{sf}$ by $H_r$;
\sk
\item
for $\ell \in \boldsymbol{\mathcal L} $, one sets $\mathcal H_\ell$ for the hyperplane in 
$ \mathbf P \big( \mathbf C^{\boldsymbol{\mathcal L} } \big)$ defined  by the vanishing of the $\ell$-th coordinate and we denote by $\mathcal H_{\boldsymbol{\mathcal L} }$ 
 the union of these hyperplanes: $\mathcal H_{\boldsymbol{\mathcal L} }
=\cup_{\ell \in {\boldsymbol{\mathcal L} }} \mathcal H_{\ell }$.  Then
for any subset $\Gamma\subset \mathbf P \big( \mathbf C^{\boldsymbol{\mathcal L} } \big)$ (resp.\,$Y\subset \boldsymbol{\mathcal Y}_r$), 
one sets $\Gamma^*=\Gamma\setminus \mathcal H_{\boldsymbol{\mathcal L} }$ (resp.\,
$Y^*=\chi\big( \chi^{-1}(Y)^* \big)$\big);
\sk
\item 
denote by $\mathbf P(X_r)^\circ$ the image of the $T_{NS}$-torsor  $\mathcal T_{\hspace{-0.06cm}X_r}$ ({\it cf.}\ \eqref{Eq:T-Xr-Torsor}) by the projection ${\bf A}(X_r)\setminus \{0\}\rightarrow \mathbf P(X_r)$. 
Then the map $\varrho_r$ embeds $\mathbf P(X_r)^\circ$ 
 into $\boldsymbol{\mathcal G}_r^{sf}$ and 
 there exists an embedding $f_{S\hspace{-0.05cm}S}: X_r \hookrightarrow \boldsymbol{\mathcal Y}_r$, named after Serganova and Skorobogatov,  which makes the following diagram commutative 
 (where  we set $\boldsymbol{\mathcal P}_{ \hspace{-0.06cm}r}^\circ=\varrho_r\big( \mathbf P(X_r)^\circ \big)$\big): 
\begin{equation}
\label{Eq:Embedding-SS}
  \xymatrix@R=1cm@C=0.4cm{ 
\mathbf P(X_r)^\circ  \,  
\ar@{^{(}->}[rrrr]^{\varrho_r \hspace{-0.4cm} }
\ar@{->}[d]
   &  &&&  \boldsymbol{\mathcal P}_{ \hspace{-0.06cm}r}^\circ  \hspace{-0.5cm}  &
\subset \hspace{-0.0cm} \ar@{->}[d]^\gamma   \boldsymbol{\mathcal G}_r^{sf}  &
\hspace{-0.5cm}
 \subset \mathbf P \big( \mathbf C^{\boldsymbol{\mathcal L} } \big)
 \\
 X_r   \ar@{->}[rrrrr]^{f_{S\hspace{-0.05cm}S}  \hspace{0.4cm} }
  &  &&& & \boldsymbol{\mathcal Y}_r \, .
  }
\end{equation}
Moreover, identifying $X_r$ with its image 
by $f_{S\hspace{-0.05cm}S}$ 
in 
$\boldsymbol{\mathcal Y}_r $, the intersection 
$X_r \cap  \boldsymbol{\mathcal Y}_r^*$ coincides with 
the complement $X_r^*$ of the union of all the lines contained in it: coherently with the notations, one has $X_r^*=X_r \cap  \boldsymbol{\mathcal Y}_r^*$;
\sk 
\item the initial del Pezzo surface $X_r$ can be recovered from its projective Cox variety 
$\mathbf P(X_r)\simeq 
\boldsymbol{\mathcal P}_r$ 
as the GIT quotient of the latter by the torus ${H}_r$: one has $X_r=\mathbf P(X_r) /\hspace{-0.1cm}/H_r$;
\sk
\item  let $N_r=N_{G_r}(H_r)$ be the normalizer of the torus $H_r$ in $G_r$.  The homogeneous space 
$\boldsymbol{\mathcal G}_r$ and its open subset $\boldsymbol{\mathcal G}_r^{sf}$ both can be seen to be $N_r$-invariant hence one deduces an action by automorphisms of the Weyl group $W_r=N_r/H_r$ on the quotient $\boldsymbol{\mathcal Y}_r=\boldsymbol{\mathcal G}_r^{sf}/H_r$. On the other hand, $W_r$ acts by permutations on the set of lines $\boldsymbol{\mathcal L}_r$  hence gives rise to a linear action of $W_r$ in $\mathbf C^{\boldsymbol{\mathcal L}_r}$ which leaves $\boldsymbol{\mathcal G}_r^{sf}$ invariant. The map $\gamma: \boldsymbol{\mathcal G}_r^{sf}\rightarrow \boldsymbol{\mathcal Y}_r$ can be proved to be $W_r$-equivariant with respect to theses two actions;  
\sk
\item  for $w\in W_r$, let $\psi_w$ be the automorphism of $\boldsymbol{\mathcal Y}_r$ given by the group monomorphism $W\hookrightarrow {\rm Aut}(\boldsymbol{\mathcal Y}_r) $ mentioned above. Any such $\psi_w$ lets $  \boldsymbol{\mathcal Y}_r^*$ invariant. 
As a consequence, it follows that the map 
$f_{S\hspace{-0.05cm}S}: X_r\hookrightarrow \boldsymbol{\mathcal Y}_r $ is not unique: 
another such map can be obtained by post-composing $f_{S\hspace{-0.05cm}S}$ with an automorphism 
$\psi_w$,  for any element $
w$ 
 of the Weyl group. 
\end{itemize}


\begin{rem}

\noindent {\rm 1.} For $r=4$, all the material above can be found in \cite{Skorobogatov1}. Actually, many things simplify in this case, since one has 
$\boldsymbol{\mathcal P}_r=\boldsymbol{\mathcal G}_r=G_3(\mathbf C^5)$.

\noindent {\rm 2.} In the case when $r=5$, which is the one we are interested in in this paper, the first occurrence in the literature of a map as in \eqref{Eq:P(Xr)-hookrightarrow-PCLr} 
we are aware of is the one at the top of  \cite[p.\,35]{Popov}. 
\end{rem}

We set $\boldsymbol{\mathcal L}'=\boldsymbol{\mathcal L} \setminus \{ E_1,\ldots , E_r\}$. 
For $\ell\in {\boldsymbol{\mathcal L} }'$, 
its direct image $\overline{\ell}=b(\ell)$ by $b$ is an irreducible rational curve in $\mathbf P^2$, of degree $\delta_{\ell}\in \{1,2,3\}$.  We fix one of the lines $\ell_\infty$ by viewing $\overline{\ell}_\infty$ at infinity in the projective plane, with respect to the choice of some fixed affine coordinates $x,y$ on $\mathbf C^2=\mathbf P^2\setminus \overline{\ell}_\infty$. 
For any $\ell \in {\boldsymbol{\mathcal L} }\setminus \{ E_1,\ldots,E_r, \ell_\infty\,\}$, let $F_{\ell}\in \mathbf C[x,y]$ be an irreducible polynomial of degree 
$\delta_\ell$ such that $\overline{\ell}$ is the closure of the affine curve  in $\mathbf C^2$ 
cut out by  $F_\ell(x,y)=0$. We set also $F_{E_i}=1$ for any $i=1,\ldots,r$ and $F_{\ell_{\infty}}=1$.  We denote by   $U_{ \hspace{-0.06cm} \boldsymbol{\mathcal L} }$ the 
open subset of $\mathbf C^2$ whose  points are the pairs $(x,y)$ such that $F_\ell(x,y)$ 
does not vanish for any line $\ell$: in other terms, one has 
\begin{equation}
\label{Eq:U-L}
U_{ \hspace{-0.06cm} \boldsymbol{\mathcal L} }=\mathbf P^2\setminus \left( \,
\cup_{\ell  \in \boldsymbol{\mathcal L}'}  \overline{\ell}\, 
\right)= b\big( X_r^*\big)\simeq X_r^*\, . 
\end{equation}
\begin{prop}
\label{P:AMC}
Up to renormalizing each $F_\ell$ by a suitable non zero multiplicative constant, the rational morphism  
\begin{align}
\label{Eq:Lou}
 F_{ \hspace{-0.06cm} \boldsymbol{\mathcal L} }=\big[ F_\ell \big]_{\ell \in \boldsymbol{\mathcal L} } \, : \, 
 U_{ \hspace{-0.06cm} \boldsymbol{\mathcal L} }  & 
 \longrightarrow \mathbf P \big( \mathbf C^{\boldsymbol{\mathcal L} } \big)\\
  (x,y) & \longmapsto 
\Big[F_\ell(x,y) \Big]_{  \ell\in \boldsymbol{\mathcal L} }
\nonumber
\end{align} 
can be made taking its values into $ \boldsymbol{\mathcal G}_r^{sf}  \subset \mathbf P \big( \mathbf C^{\boldsymbol{\mathcal L} } \big) $ and such that  the following diagram commutes : 
\begin{equation}
\label{Eq:kokoko}
  \xymatrix@R=0.7cm@C=0.3cm{ 
  &  &  &&& & \boldsymbol{\mathcal G}_r^{sf} \ar@{->}[d]^\gamma    \incl[r] & \mathbf P \big( \mathbf C^{\boldsymbol{\mathcal L} } \big) &
 \\
  \ar@/^1pc/[rrrrrru]^{ F_{ \hspace{-0.06cm} \boldsymbol{\mathcal L} }} 
 \mathbf C^2\supset  U_{ \hspace{-0.06cm} \, \boldsymbol{\mathcal L} }
\eq[r] & 
  X_r^*   \hspace{0.1cm} 
  \incl[r]    &  X_r  
  \ar@{->}[rrrr]^{f_{S\hspace{-0.05cm}S}  \hspace{0.1cm}}
  &  &&& \boldsymbol{\mathcal Y}_r \, .   }
  \end{equation} 
\end{prop}
\begin{proof}
We first recall a well-known fact: let $\mu : \mathscr L\rightarrow X_r$ be the total space of a line bundle $L$ over $X_r$. Given a non zero section $\sigma\in{\bf H}^0(X_r,L)$ with associated divisor $D_\sigma=\{ \sigma=0\} \subset X_r$, one gets a regular morphism $\tilde \sigma^{-1} : \mu^{-1}(X_r\setminus D) \rightarrow \mathbf C$, which associates to $\xi\in \mathscr L_x=\mu^{-1}(x)$ the quotient $\xi/\sigma(x)$. It extends to a global rational function   $\tilde \sigma^{-1} : \mathscr L\dashrightarrow \mathbf P^1$ which is such that the rational function $\tilde \sigma ^{-1} \circ \sigma:  X\dashrightarrow \mathbf P^1$ is constant equal to 1. 

We fix  non zero sections $\sigma_{E_i}\in { \bf H}^0\big( X_r , \mathcal O_{X_r}(E_i)\big)$ ($i=1,\ldots,5$) and $\sigma_\infty \in {\bf H}^0(x_r,\mathcal O_{X_r}(H)\big)$, the latter section being such that $b(\{
\sigma_\infty=0 \})=\overline{\ell}_{\infty}$
(in other terms, if $i$ and $j$ are 
the two indices such that $p_i$ and $p_j$ lie on $\overline{\ell}_\infty$, one has  $
\sigma_\infty=\sigma_{\ell_\infty}\cdot \sigma_{E_i}\cdot \sigma_{E_j}$ with $ \sigma_{\ell_\infty}\in { \bf H}^0\big( X_r , \mathcal O_{X_r}(\ell_\infty)\big)\setminus \{0\}$\big).  Then $(\sigma_\infty,\sigma_{E_1},\ldots,\sigma_{E_r})$ gives a global section   of the $T_{NS}$-bundle $\mathcal T^*_{ \hspace{-0.06cm} X_r}\rightarrow X_r^*$  defined as the restriction of $\mathcal T_{ \hspace{-0.06cm} X_r}\rightarrow X_r$
over $X_r^*$. Viewing 
$\mathcal T^*_{ \hspace{-0.06cm} X_r}$ as a Zariski open subset of 
${\bf A}(X_r)$, 
one can consider $\widehat \varrho_r\circ (\sigma_\infty,\sigma_{E_1},\ldots,\sigma_{E_r}) \circ b^{-1} :  U_{\boldsymbol{\mathcal L} }\simeq X_r^* \longrightarrow \big( \mathbf C^*)^{\boldsymbol{\mathcal L}}$. Obviously, the projectivization of this map is the map \eqref{Eq:Lou} and the commutativity of  \eqref{Eq:kokoko}  follows immediately from that of \eqref{Eq:Embedding-SS}. 
\end{proof}
\mk 

\begin{rem} 
The fact that $\varrho_r$ can be assumed to land into $G_r/P_r \subset \mathbf P \big( \mathbf C^{\boldsymbol{\mathcal L} } \big) $ and the existence 
of an embedding $f_{S\hspace{-0.05cm}S}:  X_r  \hookrightarrow \boldsymbol{\mathcal Y}_r $ such that 
the diagram \eqref{Eq:Embedding-SS} commutes is proved by a recurrence on $r$ 
 in \cite{SerganovaSkorobogatov}.  We believe that these results could be obtained in a more direct and explicit way by considering the map \eqref{Eq:Lou}.

The idea of considering  the mapping \eqref{Eq:Lou} has been suggested to us by A.-M. Castravet. 

\end{rem}
In the case under consideration, since the Wick map 
\eqref{Eq:Wick} 
is birational, an explicit way to construct a rational map $F_{ \hspace{-0.06cm} \boldsymbol{\mathcal L}}$ landing into the spinor tenfold 
and 
making \eqref{Eq:kokoko} commutative is by giving its factorization $\widetilde F_{ \hspace{-0.06cm} \boldsymbol{\mathcal L} }:  \mathbf C^2 \dashrightarrow 
{\rm Asym}_5(\mathbf C)$ by $W_5^+$ (ie. one has $W_5^+\circ \widetilde F_{ \hspace{-0.06cm} \boldsymbol{\mathcal L} }= F_{ \hspace{-0.06cm} \boldsymbol{\mathcal L} }$ as rational maps). 
 
%
%
%
%


In order to relate our notations with the standard ones of \cite[Planche IV]{Bourbaki}, we introduce the 
basis $(\boldsymbol{f}_i)_{i=1}^5$ of the root space $R_5=(-K)^\perp$ defined by the following relations
$$ 
\boldsymbol{f}_{i}-\boldsymbol{f}_{i+1}=\rho_i=\boldsymbol{e}_{i}-\boldsymbol{e}_{i+1} \quad \mbox{ for } i=1,\ldots,4 \qquad 
\mbox{ and } \qquad \boldsymbol{f}_{4}+\boldsymbol{f}_{5} = \rho_5=
\boldsymbol{h} -\boldsymbol{e}_{1}-\boldsymbol{e}_{2}-\boldsymbol{e}_{3}\, . 
$$
Explicitly, one has: 
\begin{align*}
\boldsymbol{f}_{1} = & \,  \frac{1}{2}\Big(
\boldsymbol{e}_{1}-\boldsymbol{e}_{2}-\boldsymbol{e}_{3}-\boldsymbol{e}_{4}-\boldsymbol{e}_{5}+\boldsymbol{h}\Big) && 
 \boldsymbol{f}_{2} =  \frac{1}{2}\Big(-\boldsymbol{e}_{1}+
\boldsymbol{e}_{2}-\boldsymbol{e}_{3}-\boldsymbol{e}_{4}-\boldsymbol{e}_{5}+\boldsymbol{h}\Big) \\
 \boldsymbol{f}_{3} =& \,   \frac{1}{2}\Big(
-\boldsymbol{e}_{1}-\boldsymbol{e}_{2}+\boldsymbol{e}_{3}-\boldsymbol{e}_{4}-\boldsymbol{e}_{5}+\boldsymbol{h}\Big)
 && \boldsymbol{f}_{4} =  \frac{1}{2}\Big(
 -\boldsymbol{e}_{1}-\boldsymbol{e}_{2}-\boldsymbol{e}_{3}+\boldsymbol{e}_{4}-\boldsymbol{e}_{5}+\boldsymbol{h}
\Big)
\\ 
\mbox{and }\quad 
\boldsymbol{f}_{5} = & \, \frac{1}{2}\Big(
-\boldsymbol{e}_{1}-\boldsymbol{e}_{2}-\boldsymbol{e}_{3}
-\boldsymbol{e}_{4}+\boldsymbol{e}_{5}+\boldsymbol{h}\Big)\, .
\end{align*}

The weight $w_\ell$ of a line $\ell\in \boldsymbol{\mathcal L}$ is given by its orthogonal projection onto 
the root space $R_5$, perpendicularly to the canonical class $K=-3\boldsymbol{h}+\sum_{i=1}^5 \boldsymbol{e}_i$.  
Straightforward computations give us the following formulas for the $w_\ell$'s expressed in the basis $(\boldsymbol{f}_i)_{i=1}^5$ of $R_5$ defined just above: 
\begin{align*}
w_{\boldsymbol{e}_1}= & \,   \left({\frac{1}{2}}, {\frac{\hspace{-0.05cm}-1}{\hspace{0.1cm}2}}, {\frac{\hspace{-0.05cm}-1}{\hspace{0.1cm}2}}
, {\frac{\hspace{-0.05cm}-1}{\hspace{0.1cm}2}}, {\frac{\hspace{-0.05cm}-1}{\hspace{0.1cm}2}}\right) && 
 w_{\boldsymbol{h}-\boldsymbol{e}_{1}-\boldsymbol{e}_{5}}=\left({\frac{\hspace{-0.05cm}-1}{\hspace{0.1cm}2}}, {
\frac{1}{2}}, {\frac{1}{2}}, {\frac{1}{2}}, {\frac{\hspace{-0.05cm}-1}{\hspace{0.1cm}2}}\right)
\\
w_{\boldsymbol{e}_2}= & \, \left({\frac{\hspace{-0.05cm}-1}{\hspace{0.1cm}2}}, {\frac{1}{2}}, {\frac{\hspace{-0.05cm}-1}{\hspace{0.1cm}2}}, {\frac{\hspace{-0.05cm}-1}{\hspace{0.1cm}2}}, {\frac{\hspace{-0.05cm}-1}{\hspace{0.1cm}2}} \right)
&& w_{\boldsymbol{h}-\boldsymbol{e}_{2}-\boldsymbol{e}_{3}}=\left({\frac{1}{2}}, {\frac{\hspace{-0.05cm}-1}{\hspace{0.1cm}2}},
{\frac{\hspace{-0.05cm}-1}{\hspace{0.1cm}2}}, {\frac{1}{2}}, {\frac{1}{2}}\right)
\\ 
w_{\boldsymbol{e}_3}= & \,  \left({\frac{\hspace{-0.05cm}-1}{\hspace{0.1cm}2}}, {\frac{\hspace{-0.05cm}-1}{\hspace{0.1cm}2}}
, {\frac{1}{2}}, {\frac{\hspace{-0.05cm}-1}{\hspace{0.1cm}2}}, {\frac{\hspace{-0.05cm}-1}{\hspace{0.1cm}2}}\right)
&& w_{\boldsymbol{h}-\boldsymbol{e}_{2}-\boldsymbol{e}_{4}}= \left({\frac{1}{2}}, {\frac{\hspace{-0.05cm}-1}{\hspace{0.1cm}2}}, {\frac{1}{2}}, {\frac{\hspace{-0.05cm}-1}{\hspace{0.1cm}2}}, {\frac{1}{2}}\right) 
\\ 
w_{\boldsymbol{e}_4}= & \, \left({\frac{\hspace{-0.05cm}-1}{\hspace{0.1cm}2}}, {\frac{\hspace{-0.05cm}-1}{\hspace{0.1cm}2}}, {\frac{\hspace{-0.05cm}-1}{\hspace{0.1cm}2}}, {\frac{1}
{2}}, {\frac{\hspace{-0.05cm}-1}{\hspace{0.1cm}2}}\right)
&& w_{\boldsymbol{h}-\boldsymbol{e}_{2}-\boldsymbol{e}_{5}}= \left({\frac{1}{2}}, {\frac{\hspace{-0.05cm}-1}{\hspace{0.1cm}2}}, {\frac{1}{2}}, {\frac{1}{2}}, {\frac{\hspace{-0.05cm}-1}{\hspace{0.1cm}2}} \right)  
\\ 
w_{\boldsymbol{e}_5}= & \, \left({\frac{\hspace{-0.05cm}-1}{\hspace{0.1cm}2}}
, {\frac{\hspace{-0.05cm}-1}{\hspace{0.1cm}2}}, {\frac{\hspace{-0.05cm}-1}{\hspace{0.1cm}2}}, {\frac{\hspace{-0.05cm}-1}{\hspace{0.1cm}2}}, {\frac{1}{2}}\right)
&& w_{\boldsymbol{h}-\boldsymbol{e}_{3}-\boldsymbol{e}_{4}}= \left({\frac{1}{2}}, {
\frac{1}{2}}, {\frac{\hspace{-0.05cm}-1}{\hspace{0.1cm}2}}, {\frac{\hspace{-0.05cm}-1}{\hspace{0.1cm}2}}, {\frac{1}{2}}\right) 
\\ 
w_{\boldsymbol{h}-\boldsymbol{e}_{1}-\boldsymbol{e}_{2}}= & \, \left({\frac{\hspace{-0.05cm}-1}{\hspace{0.1cm}2}}, {\frac{\hspace{-0.05cm}-1}{\hspace{0.1cm}2}}, {
\frac{1}{2}}, {\frac{1}{2}}, {\frac{1}{2}}\right)
&& w_{\boldsymbol{h}-\boldsymbol{e}_{3}-\boldsymbol{e}_{5}}= \left({\frac{1}{2}}, {\frac{1}{2}}, {\frac{\hspace{-0.05cm}-1}{\hspace{0.1cm}2}}, {\frac{1}{2}}, {\frac{\hspace{-0.05cm}-1}{\hspace{0.1cm}2}}\right)
\\ 
w_{\boldsymbol{h}-\boldsymbol{e}_{1}-\boldsymbol{e}_{3}}= & \, \left({\frac{\hspace{-0.05cm}-1}{\hspace{0.1cm}2}}, {\frac{1}{2}}, {\frac{\hspace{-0.05cm}-1}{\hspace{0.1cm}2}}, {
\frac{1}{2}}, {\frac{1}{2}}\right) 
&& w_{\boldsymbol{h}-\boldsymbol{e}_{4}-\boldsymbol{e}_{5}}= \left({\frac{1}{2}}, {\frac{1}{2}}, {\frac{1}{2}}, {\frac{\hspace{-0.05cm}-1}{\hspace{0.1cm}2}}, {\frac{\hspace{-0.05cm}-1}{\hspace{0.1cm}2}}\right)
\\ 
w_{\boldsymbol{h}-\boldsymbol{e}_{1}-\boldsymbol{e}_{4} }= & \, \left({\frac{\hspace{-0.05cm}-1}{\hspace{0.1cm}2}}, {\frac{1}{2}}, {\frac{1}{2}}, {\frac{\hspace{-0.05cm}-1}{\hspace{0.1cm}2}}, {\frac{
1}{2}}\right) 
&&\hspace{-0.25cm} w_{2 \boldsymbol{h}-\sum_{i=1}^5 \boldsymbol{e}_{i}}=\left({\frac{1}{2}}, {\frac{1}{2}}
, {\frac{1}{2}}, {\frac{1}{2}}, {\frac{1}{2}}\right)\, .
\end{align*}
%
%
%

 We assume that the considered del Pezzo quartic surface $X_5$ is the total space of the blow-up of the projective plane at the  five points in general position $p_1,\ldots,p_5$ which are defined as the projectivization 
 of the following five points of $\mathbf C^3\setminus \{0\}$: 
$$
\widehat p_1=\big(1,0,0\big)\, , \quad 
\widehat p_2=\big(0,1,0\big)\, , \quad
\widehat p_3=\big(0,0,1\big)\, , \quad
\widehat p_4=\big(1,1,1\big)\quad  \mbox{ and } 
\quad 
\widehat p_5=\big(a,b,1\big)\,.
$$
We set $ \widehat p_{xy}=(x,y,1)\in \mathbf C[x,y]^3$ and for any $i,j$ such that $1\leq i<j\leq 5$, we introduce the following equation of the line through $p_i$ and $p_j$ in the affine coordinates $x,y$: 
$$
P_{ij}=\det \left(\,  \widehat p_i,\widehat p_j,\widehat p_{xy}\, \right)=0
$$
(hence one has $P_{12}=1$, $P_{13}=-y$, $P_{14}=1-y$, etc.). 
As an affine equation for the conic passing through the five points $p_i$'s, we take 
$$
\mathcal C_{ab}=
\left(a-b  \right) x y +b\left(1-a \right) x +a\left(b -1 \right) y\, . $$ 

From the comparison between the weights of the lines $w_\ell$'s given above and those associated to the Wick coordinates as indicated in Table \ref{Table:tokolo}, it follows easily that the map  $\widetilde F_{ \hspace{-0.06cm} \boldsymbol{\mathcal L} }: \mathbf C^2\dashrightarrow {\rm Asym}_5(\mathbf C)$
discussed above is necessarily of the form  $(x,y)\dashrightarrow  \big(  \lambda_{ij}\,P_{ij}\big)_{i,j=1}^5$, where the $\lambda_{ij}$'s are suitable scalar constants 
 symmetric in $i$ and $j$. Let us consider the rational map ${\bf F}: \mathbf C^2\dashrightarrow {\rm Asym}_5(\mathbf C)$ defined by 
$$
{\bf F}(x,y)=
\frac{1}{ \mathcal C_{ab} }
\scalebox{0.9}{$
\left[\begin{array}{ccccc}
0 & -b -a  & -\left(a +1\right) y  & \left(1-y \right) a  & y -b
\\
 b +a  & 0 & -x \left(1+b \right) & \left(1-x \right) b  & x-a  
\\
 \left(a +1\right) y  & x \left(1+b \right) & 0 & y -x  & bx-a y  
\\
 \left(y -1\right) a  & \left(x-1 \right) b  & x-y  & 0 & 
 (b-1)x + (1-a)y - b + a
\\
 b -y  & a -x  & a y -b x  &  (1-b)x + (a-1)y + b - a  & 0 
\end{array}\right]
$}\, .
$$

For any $i=1,\ldots,5$, we denote by 
 ${\bf F}(x,y)_{ \hat \imath}$ the $4\times 4$  antisymmetric matrix obtained from  ${\bf F}(x,y)$ by removing its $i$-th line and its $i$-th column. 
 By direct computations, we get the 
\begin{lem} 
For $i=1,\ldots,5$, one has:
$$
{\rm Pf}\Big( {\bf F}(x,y)_{ \hat \imath}\Big)=\frac{1}{\mathcal C_{ab}}\, .
$$
\end{lem}
From above, it comes that the relations ${\rm Pf}\big( {\bf F}(x,y)_{ \hat \imath}\big)=1/\mathcal C_{ab}$ for $i=1,\ldots,5$ are precisely those ensuring that the following holds true: 
\begin{cor}
The map 
 $W_5\circ {\bf F}: \mathbf C^2 \dashrightarrow \mathbb S_5\subset \mathbf P\big( \mathbf C^{ \boldsymbol{\mathcal L}}
\big)$ is of the form \eqref{Eq:Lou}. Consequently, $\mathcal P_5\circ {\bf F}: \mathbf C^2 \dashrightarrow \mathbf C^5$ is a birational model for 
Serganova-Skorobogatov's embedding 
$f_{S\hspace{-0.05cm}S}: X_r \hookrightarrow \boldsymbol{\mathcal Y}_r$.
\end{cor}

Up to some conventions distinct from ours\footnote{The main difference between the convention of \cite{SturmfelsVelasco} and ours is that, in the case under scrutiny, 
Sturmfels and Velasco take $\wedge^{\rm even} V$ for $S^+$, and not  $\wedge^{\rm odd} V$ as we do in this text.}, 
an explicit formula for a map $ \mathcal A_{S\hspace{-0.05cm}V} : \mathbf C^2\dashrightarrow {\rm Asym}_5(\mathbf C)$ such that $W_5\circ \mathcal A_{S\hspace{-0.05cm}V} $ be of the form \eqref{Eq:Lou} 
can be found in \cite[\S4]{SturmfelsVelasco} (hence 
the subscript `SV'  which is for `Sturmfels-Velasco').
This map  is given by 
$$
\mathcal A_{S\hspace{-0.05cm}V}(x,y)=\begin{bmatrix}
0 &  \frac{(y-x)(s_2-s_1)}{(b-a)}  & 
\frac{(1-x)(s_3-s_1)}{(1-a)}  &  -\frac{x s_1}{a} &  1 \vspace{0.1cm}\\
-\frac{(y-x)(s_2-s_1)}{(b-a)} & 0 & \frac{(1-y)(s_3-s_2)}{(1-b)} 
&  -\frac{y s_2}{b}&  1\vspace{0.1cm} \\
 - \frac{(1-x)(s_3-s_1)}{(1-a)}& -\frac{(1-y)(s_3-s_2)}{(1-b)} 
& 0&-{s_3}&  1\vspace{0.1cm} \\
\frac{x s_1}{a}  & \frac{y s_2}{b}& {s_3}& 0&  1  \vspace{0.1cm} \\
 -1& -1 & -1& -1&  0  \\
\end{bmatrix}
$$
where 
the $s_i$'s stand for some auxiliary generic (but fixed) scalar parameters. 

Compared to our map ${\bf F}$ 
 and  with regard to  the variables $x$ and $y$, 
 Sturmfels-Velasco's map $\mathcal A_{S\hspace{-0.05cm}V} $ has two nice features: first it is a polynomial map while the coefficients of ${\bf F}(x,y)$ are genuine rational functions; secondly, the formula above for $\mathcal A_{S\hspace{-0.05cm}V}$ is just the specialization in the case under consideration of a more general formula (related to a natural embedding in the $n$-th spinor variety $\mathbb S_n$, of the Cox variety of the blow-up of $\mathbf P^n$ in $n+3$ points in general position, this for any $n\geq 2$). 
 Since it leads to simpler formulas, we prefer to work with Sturmfels-Velasco map 
$\mathcal A_{S\hspace{-0.05cm}V}$ below, although it involves some auxiliary generic scalar parameters (the $s_k$'s for $k=1,2,3$).

It follows from Proposition \ref{Prop:Birat-Model-P5} and Proposition \ref{P:AMC} that 
the map $\mathcal P_5\circ \mathcal A_{S\hspace{-0.05cm}V} : \mathbf C^2 \dashrightarrow \mathbf C^5$ is a birational model for 
$ \chi_5  \circ W_5^+\circ \mathcal A_{S\hspace{-0.05cm}V} =f_{S\hspace{-0.05cm}S} $. 
A straightforward computation gives us that 
$\mathcal P_5\circ \mathcal A_{S\hspace{-0.05cm}V}$ 
is given explicitly by 
$$ (x,y) \longmapsto   
\left( \, 
\frac{(x-1) \,{s}_{13}}{(a-1) {s_3}}
\, , \, 
\frac{y \, (a-b){s_2} }{ (x-y) \,b{s}_{12}}
\, , \, 
\frac{(x-y)  \,(b-1){s}_{12} }{(y-1) \,(a-b) {s}_{23}}
\, , \, 
\frac{x\, (y-1) \, (a-b)s_1 s_{23}}{(x-y) \,a(b-1)  {s_3} {s_{12}}}
\, , \, 
\frac{(b-1){s_3}}{(y-1){s_{23}}} \, \right)\
$$
where $s_{ij}=s_i-s_j$ for $i,j=1,\ldots,3$.   

%
%
%
%

For any $i=1,\ldots,5$, let $\mathcal F_i^+$ be the foliation on $\mathbf C^2$ with first integral 
$$\psi_i^+ \circ  \mathcal P_5\circ \mathcal A_{SV}: \mathbf C^2 \dashrightarrow  \mathbf C^2$$ 
({\it cf.}\,\eqref{Eq:Maps-psi-i} for the maps $\psi_i^+$).  A priori $\mathcal F_i^+$ might be a trivial foliation (that is a foliation by points, which  would occur exactly when $\psi_i^+ \circ  \mathcal P_5\circ \mathcal A_{SV}$ is generically of rank 2) but it turns out that this is precisely not the case. Indeed,  all the maps 
$\psi_i^+ \circ  \mathcal P_5\circ \mathcal A_{SV}$ are generically of rank 1, all the 
$\mathcal F_i^+$'s are foliations by rational curves and direct easy computations give that these are the following:  
$$
\mathcal F_1^+=\mathcal F_y\, , \qquad 
\mathcal F_2^+=\mathcal F_x\, , \qquad 
\mathcal F_3^+=\mathcal F_{\hspace{-0.03cm}\scalebox{1}{${\frac{x}{y}}$}}\, , \qquad 
\mathcal F_4^+=\mathcal F_{\hspace{-0.03cm}\scalebox{0.9}{${\frac{y-1}{x-1}}$}}\qquad 
\mbox{and} \qquad 
\mathcal F_5^+=\mathcal F_{\hspace{-0.03cm}\scalebox{0.9}{${\frac{x(y-1)}{y(x-1)}}$}}\, .
$$

For any $i=1,\ldots,5$, let $\mathcal F_i^-$ be the foliation on $\mathbf C^2$ 
whose first integral is any one of the maps $\psi_{i}^{-,j} \circ  \mathcal P_5\circ \mathcal A_{SV}: \mathbf C^2 \dashrightarrow  \mathbf C^2$ with $j\neq i$ ({\it cf.}\,the formulas \eqref{Eq:Psi1*5} and  \eqref{Eq:Psii*j}).  As for the $\mathcal F_i^+$'s, a priori $\mathcal F_i^-$ might be a trivial foliation but, as before, it turns out that this is precisely not the case. Indeed,  the 
$\mathcal F_i^-$'s are foliations by rational curves and direct easy computations give that these are the following five:  
\begin{align*}
\mathcal F_1^-=& \, \mathcal F_{\hspace{-0.1cm}\scalebox{1}{$\frac{(x-y) (a-x)}{x ((y-1) a +(1-x) b -y +x)}$}}
&& \mathcal F_2^-=\,\mathcal F_{\hspace{-0.1cm}\scalebox{1}{$\frac{(x-y ) (b -y)}{(a y -b x -a +b +x -y) y}$}}
&& \mathcal F_3^-=\,\mathcal F_{\hspace{-0.1cm}\scalebox{1}{$\frac{(x-1) (b -y)}{(-1+y) a +(1-x) b -y +x}$}}\vspace{0.15cm} \\
& &&
\hspace{-0.7cm}
\mathcal F_4^-= \,\mathcal F_{\hspace{-0.1cm}\scalebox{1}{$\frac{(b -y) x}{a y -b x}$}}
\quad 
&& 
\hspace{-0.7cm}
\mathcal F_5^-=\,\mathcal F_{\hspace{-0.1cm}\scalebox{1}{$\frac{b -y}{a y -b \,x}$}}\, .
\end{align*}

In more geometric terms, recalling that for any 
$i=1,\ldots,5$, 
one denotes 
by 
$\mathcal L_{{p_i}}$ the pencil of lines through $p_i$ and by $\mathcal C_{\widehat{p_i}}$ the pencil  of conics through the $p_j$'s for $j \in \{1,\ldots,5\}\setminus \{i\}$, one  has
\begin{align*}
\mathcal F_1^+=& \, \mathcal L_{{p_1}}
&\mathcal F_2^+
=\mathcal L_{{p_2}}
&& \hspace{0.1cm}
\mathcal F_3^+
=\mathcal L_{{p_3}}
&&
\mathcal F_4^+
=\mathcal L_{{p_4}}
&&\mathcal F_5^+=\mathcal C_{\widehat{p_5}}
\textcolor{white}{\, .}
\\
\mathcal F_1^-=& \, \mathcal C_{\widehat{p_1}} 
&
\mathcal F_2^-=\mathcal C_{\widehat{p_2}} 
&& \mathcal F_3^-=\mathcal C_{\widehat{p_3}} 
&& \mathcal F_4^-=\mathcal C_{\widehat{p_4}} 
&&  \mathcal F_5^-=\mathcal L_{{p_5}} \, .
\end{align*}
%
These foliations are affine birational models  of the ten  fibrations in conics on 
the quartic del Pezzo surface ${\rm dP}_4$  under consideration. We thus get the 
\begin{prop}
\label{Prop:WdP4-from-WGMS5}
As ordered 5-webs, one has 
\begin{align*}
\boldsymbol{\mathcal W}\left( \, 
 \mathcal L_{{p_1}}\, , \,
\mathcal L_{{p_2}}\, , \,
\mathcal L_{{p_3}}\, , \,
\mathcal L_{{p_4}}\, , \,
\mathcal C_{\widehat{p_5}}\,  \right)&\, = 
f_{S\hspace{-0.05cm}S}^* \Big( \boldsymbol{\mathcal W}\big( 
\, \psi_1^+,\ldots\, , \, \psi_5^+\,
\big)\Big)= 
\mathcal A_{S\hspace{-0.05cm}V}^* \bigg(
\boldsymbol{\mathcal W}\Big( \, \Phi_1^+,\ldots,\Phi_5^+ \, \Big)
\bigg)  
 \\
\mbox{and }\quad 
\boldsymbol{\mathcal W}\left( \, 
\mathcal C_{\widehat{p_1}} \, , \, 
\mathcal C_{\widehat{p_2}} \, , \, 
\mathcal C_{\widehat{p_3}} \, , \, 
\mathcal C_{\widehat{p_4}} \, , \, 
\mathcal L_{{p_5}} \, \right)&\, 
= 
f_{S\hspace{-0.05cm}S}^* \Big( \boldsymbol{\mathcal W}\big( 
\, \psi_1^-,\ldots\, , \, \psi_5^-\,
\big)\Big)= 
\mathcal A_{S\hspace{-0.05cm}V}^* \bigg(
\boldsymbol{\mathcal W}\Big( \, \Phi_{i}^-,\ldots,\Phi_{5}^- \, \Big)
\bigg)  \, .
\end{align*}
Consequently, del Pezzo's web $\boldsymbol{\mathcal W}_{
 \hspace{-0.03cm}
  {\rm dP}_4}$ is the pull-back, under Serganova-Skorobogatov  embedding $f_{S\hspace{-0.05cm}S}: {\rm dP}_4^*  \hookrightarrow \boldsymbol{\mathcal Y}_5^* $, of the quotient by the action of the Cartan torus $H_{D_5}$ of ${\rm Spin}_{10}$, of Gelfand-MacPherson web  $\boldsymbol{\mathcal W}^{GM}_{ \mathbb S_5}$ on the tenfold spinor variety  ${ \mathbb S_5}$. In mathematical terms, one has
$$ 
\boldsymbol{\mathcal W}_{  \hspace{-0.03cm} {\rm dP}_4 }= 
f_{S\hspace{-0.05cm}S}^*
\Big( {{\boldsymbol{\mathcal W}}^{GM}_{ \boldsymbol{\mathcal Y}_5}}\Big)
=
\mathcal A_{S\hspace{-0.05cm}V}^* \Big(
{\boldsymbol{\mathcal W}}^{GM}_{ \mathbb S_5}
\Big)\,  .
$$
\end{prop}
We believe that this description {\it \`a la Gelfand-MacPherson} of $\boldsymbol{\mathcal W}_{
 \hspace{-0.03cm}
  {\rm dP}_4}$ is quite interesting and deserves further investigations. 
We make a few comments on this  in the subsections 
\S\ref{SS:HLog3-a-la-GM} and 
 \S\ref{Eq:GM-Webs-WGM-r}  further.

 \subsection{\bf The web ${\mathcal W} \hspace{-0.46cm}{\mathcal W}_{ {\rm dP}_4
 \hspace{-0.4cm}
  {\rm dP}_4}$ is modular}
 \label{SS:WdP4-as-a-modular-web}
In this subsection, we show that a del Pezzo $\boldsymbol{\mathcal W}_{ {\rm dP}_4}$ can be obtained in a natural way from a web defined by modular rational maps  between modular spaces. 
The general set-up is that of \cite[\S1.2.7.2]{ClusterWebs} (see also \S1.2.7.3 therein) 
and our  arguments below rely on explicit computations which we will not detail.

Let ${\rm Conf}_6(\mathbf P^2)$ be the space of configurations of six points in general position on $\mathbf P^2$: if $Z\subset \big(\mathbf P^2\big)^6$ stands for the algebraic subset  of 6-tuples of points  not in linear general position, one has
$$
{\rm Conf}_6\big(\mathbf P^2\big)= \Big( \big(\mathbf P^2\big)^6- Z \Big)\hspace{0.03cm}{\big/ \hspace{0.03cm}{\rm PGL}_3(\mathbf C)}\, . 
$$

To any pair $(i,J)$ where $i$ is an element of $ \{1,\ldots,6\}$   and $J$ a subset of $ \{1,\ldots,6\}\setminus \{i\}$ of cardinality 4 is associated a rational morphism
$$
\pi_{i,J} : \, 
{\rm Conf}_6(\mathbf P^2) \longrightarrow \mathcal M_{0,4},\, [p_1,\ldots,p_6]\longmapsto \big[ \pi_i(p_j)\big]_{j\in J}
$$ 
where $\pi_i: \mathbf P^2\dashrightarrow \mathbf P^1$ is the linear projection from the $i$th point of the configuration. There are $6\times { 5 \choose 4}=30$ such maps  and any two of these define two distinct foliations of codimension 1 on ${\rm Conf}_6(\mathbf P^2)$. We thus get a natural (modular) 30-web by hypersurfaces\footnote{The term `web' refers here to a more general notion than the one classically referred to. Indeed,  the foliations of $\boldsymbol{\mathcal W}_{ {\rm Conf}_6(\mathbf P^2) }$ do not satisfy the `general position property' usually required in web geometry but the weaker property that two of its foliations intersect transversely (note that this is somehow tautological here  since this only means that these two foliations of codimension 1 are distinct which is indeed the case!). For a discussion about the more general notion of `web' considered here, see \cite[\S1.1.1]{ClusterWebs}.} on this space: 
$$
\boldsymbol{\mathcal W}_{ {\rm Conf}_6(\mathbf P^2) }=
\boldsymbol{\mathcal W}\Big(  \, \pi_{i,J}\, \big\lvert \, i=1,\ldots,6,\, J\subset \{1,\ldots,6\}\setminus \{i\},\, \lvert J\lvert =4\, \Big)\, . 
$$

Given a quartic del Pezzo surface ${\rm dP}_4$, let $q_1,\ldots,q_5$ be 5 points in general position in $\mathbf P^2$ such that ${\rm dP}_4$ identifies with the total space of the blow-up at the $q_i$'s. Denoting by $\beta: {\rm dP}_4\rightarrow \mathbf P^2$ the corresponding morphism, we get a well-defined rational map by setting
\begin{equation}
\label{Eq:Map-B}
B : {\rm dP}_4 \dashrightarrow {\rm Conf}_6(\mathbf P^2) , \, x\longmapsto \big[ q_1,\ldots,q_5, \beta(x)\big]
\, . 
\end{equation}

\begin{prop}
The web $\boldsymbol{\mathcal W}_{ {\rm dP}_4}$ coincides with the pull-back of 
$\boldsymbol{\mathcal W}_{ {\rm Conf}_6(\mathbf P^2) }$ under the map $B$:
$$
\boldsymbol{\mathcal W}_{ {\rm dP}_4}=B^*\Big( 
\boldsymbol{\mathcal W}_{ {\rm Conf}_6(\mathbf P^2) }
\Big)\,.
$$
\end{prop}
%
%

Note that the map $B$ in \eqref{Eq:Map-B} is not canonical since there are several non equivalent ways to describe the considered del Pezzo quartic surface ${\rm dP}_4$ as the total space of a blow-up of $\mathbf P^2$ at five points in general position. However the pull-back of $\boldsymbol{\mathcal W}_{ {\rm Conf}_6(\mathbf P^2) }$ under any such map is independent of it and is $\boldsymbol{\mathcal W}_{ {\rm dP}_4}$.

 \subsection{\bf The web ${\mathcal W} \hspace{-0.46cm}{\mathcal W}_{ {\rm dP}_4
 \hspace{-0.4cm}  {\rm dP}_4}$ as a cluster web}
 \label{SS:WdP4-as-a-cluster-web}
In this subsection, we show that a del Pezzo web $\boldsymbol{\mathcal W}_{ {\rm dP}_4}$ admits a birational model which is a cluster web (of type $D_4$). Our  arguments below  rely on explicit computations that we omit.

 \subsection{\bf The cluster algebra setting.}  
As the initial exchange matrix of type $D_4$ that we will consider, we take 
$$
B_{D_4}=\begin{bmatrix}
0 & 1 & 0 &0 \\
-1 & 0 & -1 & -1\\
0 & 1 & 0 &0 \\
0 & 1 & 0 &0
\end{bmatrix}\,.
$$

The $\mathcal X$-cluster web of type $D_4$ associated to $B_{D_4}$ is the web denoted by 
$\boldsymbol{\mathcal X\hspace{-0.05cm}\mathcal W}_{D_4}$, 
 whose first integrals are the $\mathcal X$-cluster variables obtained from the initial seed 
$$\boldsymbol{\mathcal S}_0=\big((a_i)_{i=1}^4,   (x_i)_{i=1}^4, B_{D_4} \big)\,.$$
 It is a 52-web in four variables  which has been studied in \cite{ClusterWebs} (see \S7.3.1.2 and more specifically \S7.3.2.1 therein). Since the cluster variables can be constructed explicitly  (using a computer algebra system), this web can be made entirely explicit but there is no point to give explicit expressions for its 52 cluster first integrals here. 

Let $\boldsymbol{\mathcal W}_{ {\rm dP}_4}$ be a given fixed del Pezzo's web. The purpose of this section is to explain how one can obtain a cluster description of it from $\boldsymbol{\mathcal X\hspace{-0.05cm}\mathcal W}_{D_4}$ by specializing   a general construction the broad outlines of which were described in \cite[\S8.4.2]{ClusterWebs} to the case under consideration.

To this aim, we have to deal with the associated `cluster ensemble' introduced by Fock and Goncharov in \cite{FG}. In the case we are interested in, it is the triple $( \boldsymbol{\mathcal A}_{D_4},  \boldsymbol{\mathcal X}_{D_4}, p)$ where $\boldsymbol{\mathcal A}_{D_4}$ and $  \boldsymbol{\mathcal X}_{D_4}$ stand respectively for the $\boldsymbol{\mathcal A}$- and 
$\boldsymbol{\mathcal X}$-cluster variety constructed from $\boldsymbol{\mathcal S}_0$ and 
$p: \boldsymbol{\mathcal A}_{D_4}\rightarrow \boldsymbol{\mathcal X}_{D_4}$ is the cluster map

Actually, there are more structures associated to $( \boldsymbol{\mathcal A}_{D_4},  \boldsymbol{\mathcal X}_{D_4}, p)$ since, as explained in \cite[\S2]{FG}: 
\begin{enumerate}
\item[$-$] 
the map $p: \boldsymbol{\mathcal A}_{D_4}\rightarrow \boldsymbol{\mathcal X}_{D_4}$ corresponds to the quotient map under the $H_{\boldsymbol{\mathcal A}}$-action
of a certain algebraic torus $H_{\boldsymbol{\mathcal A}}$ acting on the 
$\boldsymbol{\mathcal A}$-cluster variety;
\sk 
\item[$-$] there is an`exact sequence' of cluster varieties/maps
\begin{equation}
\label{Eq:ClusterSequence}
 \boldsymbol{\mathcal A}_{D_4}
\stackrel{p}{\longrightarrow} \boldsymbol{\mathcal X}_{D_4}\stackrel{\lambda}{\longrightarrow}
H_{\boldsymbol{\mathcal X}}\rightarrow 1
\end{equation}
where  $H_{\boldsymbol{\mathcal X}}$ is a torus and  $\lambda$ is a map with a monomial expression in each $\boldsymbol{\mathcal X}$-cluster chart. 
That \eqref{Eq:ClusterSequence} be `exact' means that one has 
$\boldsymbol{\mathcal U}={\rm Im}(p)={\lambda}^{-1}(1)$ as subvarieties of 
$ \boldsymbol{\mathcal X}_{D_4}$.
\end{enumerate}

Since $B_{D_4}$ is a $4\times 4$ matrix of rank 2,  the map  $p$ has rank 2 as well hence  
both  $H_{\boldsymbol{\mathcal A}}$ and $H_{\boldsymbol{\mathcal X}}$ are tori of dimension 2. 
In the initial cluster coordinates $a_1,\ldots,a_4$ and $x_1,\ldots,x_4$, the two cluster maps
$p$ and $\lambda$ 
are written :  
\begin{align*}
p : \big(a_s\big)_{s=1}^4 & \longmapsto \Big( \, a_2^{-1}\, , \,
a_1a_3a_4 \, , \,
a_2^{-1} 
\, , \,a_2^{-1}  \, 
\Big)
\qquad \mbox{ and } \qquad 
\lambda: 
(x_i)_{i=1}^4  \longmapsto \bigg( \, 
\frac{x_1}{x_3}
\, ,\, 
\frac{x_1}{x_4}\, 
 \bigg)\, .
\end{align*}

For $\tau \in \big(\mathbf C^*\big)^2$, one defines 
$\boldsymbol{\mathcal X\hspace{-0.05cm}\mathcal W}_{D_4,\tau}$ as the restriction of $\boldsymbol{\mathcal X\hspace{-0.05cm}\mathcal W}_{D_4}$ along $ \boldsymbol{\mathcal X}_\tau=\lambda^{-1}(\tau)\subset \boldsymbol{\mathcal X}_{D_4}$:  
$$
\boldsymbol{\mathcal X\hspace{-0.05cm}\mathcal W}_{D_4,\tau}=\big( 
\boldsymbol{\mathcal X\hspace{-0.05cm}\mathcal W}_{D_4}
\big)\lvert_{\lambda=\tau}\, . 
$$ 

Here are some facts/remarks which can be made about the cluster webs $\boldsymbol{\mathcal X\hspace{-0.05cm}\mathcal W}_{D_4,\tau}$: 
\begin{itemize}
\item   For $\tau \in H_{\boldsymbol{\mathcal X}}$ generic,  $\boldsymbol{\mathcal X\hspace{-0.05cm}\mathcal W}_{D_4,\tau}$ is a 38-web (see below for explicit expressions for the cluster first integrals of this web).
\sk
\item For $\tau=(a,b)$ with $(a-1)(b-1)=0$ but $(a,b)\neq (1,1)$, 
 $\boldsymbol{\mathcal X\hspace{-0.05cm}\mathcal W}_{D_4,\tau}$ is a 28-web.
\mk
\item The case when $\tau={\boldsymbol{1}}=(1,1)$ is particularly interesting. Indeed,  
$\boldsymbol{\mathcal X\hspace{-0.05cm}\mathcal W}_{D_4,{\bf 1}}$ is a 18-web which is equivalent to the web associated to Kummer's identity of the tetralogarithm 
(see \cite[\S5.2]{ClusterWebs}).  
In particular, $\boldsymbol{\mathcal X\hspace{-0.05cm}\mathcal W}_{D_4,{\bf 1}}$ carries polylogarithmic ARs of weight four.
\end{itemize}
\mk 

In view of giving a cluster description of $\boldsymbol{\mathcal W}_{ {\rm dP}_4}$ (up to equivalence), we are going to consider a certain 10-subweb of the 38-web $\boldsymbol{\mathcal X\hspace{-0.05cm}\mathcal W}_{D_4,\tau}$ for $\tau=(a,b)\in H_{\boldsymbol{\mathcal X}}$ generic (that is such that 
 $a$ and $b$ are such that $ab(a-1)(b-1)\neq 0$), which 
will be denoted by $\boldsymbol{\mathcal X\hspace{-0.05cm}\mathcal W}_{D_4,\tau}^{10}$. 
\sk 

For $\tau=(a,b)$ as above, the first integrals for $\boldsymbol{\mathcal X\hspace{-0.05cm}\mathcal W}_{D_4,\tau}$ are given by considering the pull-backs of the 52 $\boldsymbol{\mathcal X}$-cluster variables obtained from the seed $\boldsymbol{\mathcal S}_0$  under the  affine 
parametrization  $(u,v)\mapsto (x_1,\ldots,x_4)=(u^{-1},v, u^{-1},u^{-1})$ 
of $\boldsymbol{\mathcal U}$.\footnote{We use $u^{-1}$ instead of $u$ in this parametrization in order to get `nicer' formulas for the cluster first integrals of $\boldsymbol{\mathcal X\hspace{-0.05cm}\mathcal W}_{D_4,\tau}$.} There are some redundencies among these pull-backs from which one can extract the following  list of explicit expressions for the 38 first integrals of the web $\boldsymbol{\mathcal X\hspace{-0.05cm}\mathcal W}_{D_4,\tau}$:
\label{XWD4-tau}
\begin{align*}
& \, \textcolor{red}{\frac{1}{u} }
\, , \, 
 \frac{1}{v},\frac{a u +1}{v},\frac{u +1}{v}
 \, , \, \frac{b u +1}{v},
 \textcolor{red}{ \frac{v +1}{a \,u} } \, , \, \frac{(u +1) (a u +1)}{v}\, , \, \frac{(u +1) (b u +1)}{v},  \\
& \, \frac{(b u +1) (a u +1)}{v}\, , \, \frac{v +1+u}{u v}\, , \,  \textcolor{red}{\frac{v +1+u}{b \,u (u +1)}}\, , \, 
 \textcolor{red}{\frac{b u +v +1}{u (b u +1)}},\frac{b u +v +1}{b\,  uv},
 \textcolor{red}{\frac{a u +v +1}{u (a u +1)}}\, , \, \\ 
&\, 
\frac{a u +v +1}{a\,uv}\, , \,   \textcolor{red}{\frac{a b u^{3}+(a (b +1)+b) u^{2}+u (a +b +1)+v +1}{u v}}\, , \, 
\textcolor{red}{\frac{b u^{2}+b u +u +v +1}{u v}}\, , \, 
\\ 
&\,
\textcolor{red}{\frac{a u^{2}+a u +u +v +1}{u v}}\, , \, \textcolor{red}{\frac{a b u^{2}+(a +b) u +v +1}{b\,uv}}\, , \, \frac{(u +1) (b u +1) (a u +1)}{v}\, , \, 
\\ 
&\,
\textcolor{red}{\frac{a b u^{3}+(a (b +1)+b) u^{2}+(v +1) (a +b +1) u +(v +1)^{2}}{a u^{2} v}}\, , \, \frac{(v +1+u) (b u +v +1)}{(b u^{2}+b u +u +v +1) a u}\, , \, 
\\ &\,
\frac{(v +1+u) (a u +v +1)}{(a u^{2}+a u +u +v +1) b u}\, , \, 
\frac{(b u +v +1) (a u +v +1)}{u (a b u^{2}+(a +b) u +v +1)},\frac{b u^{2}+b u +u +v +1}{a u (u +1) (b u +1)}\, , \, \\ 
&\, \frac{a u^{2}+a u +u +v +1}{b u (u +1) (a u +1)}\, , \, 
\frac{a b u^{2}+(a +b) u +v +1}{u (b u +1) (a u +1)}\, , \, 
\frac{(b u^{2}+b u +u +v +1) (a u^{2}+a u +u +v +1)}{u v (v +1+u)}\, , \, 
\\ &\,
\frac{(b u^{2}+u (b +1)+v +1) (a b u^{2}+(a +b) u +v +1)}{v b u (b u +v +1)}\, , \, 
\frac{(a u^{2}+(a +1) u +v +1) (a b u^{2}+(a +b) u +v +1)}{v a u (a u +v +1)}\, , \, 
\\ &\, \frac{(v +1+u) (b u +v +1)}{u^{2} v b}\, , \, 
\frac{(b u^{2}+u (b +1)+v +1) (a u^{2}+(a +1) u +v +1) (a b u^{2}+(a +b) u +v +1)}{a u^{3} b v^{2}}
\\ &\, \frac{(v +1+u) (a u +v +1)}{u^{2} v a}\, , \, 
\frac{(b u +v +1) (a u +v +1)}{v b u^{2} a}\, , \, \frac{(b u^{2}+b u +u +v +1) (a u^{2}+a u +u +v +1)}{v b u^{2} a (u +1)}\, , \, 
\\ &\,
\frac{(v +1+u) (b u +v +1) (a u +v +1)}{u^{3} v b a}
\, , \, 
\frac{(b u^{2}+u (b +1)+v +1) (a b u^{2}+(a +b) u +v +1)}{u^{2} v a (b u +1)}\, , \, 
\\ &\,
\frac{(a u^{2}+(a +1) u +v +1) (a b u^{2}+(a +b) u +v +1)}{u^{2} v b (a u +1)} \, .
\end{align*}

The ten cluster variables in red are those which are the first integrals of the subweb we are interested in, 
 which is defined as follows 
as an ordered 10-web: 
\begin{align*}
\boldsymbol{\mathcal X\hspace{-0.05cm}\mathcal W}_{D_4,\tau}^{10}=
\boldsymbol{\mathcal W}
\bigg( \,\,  & \frac{1+v}{a u}\, , \,\frac{1+u+v}{b u (1+u)}\, , \,\frac{1+b u +v}{u (1+b u)} 
\, , \,
\frac{1}{u}\, , \,
\frac{1+a u +v}{u (1+a u)}\, , \,
\\ &\,
\frac{a b u^{3}+(a (b +1)+b) u^{2}+u (a +b +1)+v +1}{u v}
\, , \, 
\frac{a b u^{2}+(a +b) u +v +1}{b   \,uv}\, ,\,\\
&\,
\frac{a u^{2}+a u +u +v +1}{u v}\, , \,
 \frac{a b u^{3}+(a (b +1)+b) u^{2}+(v +1) (a +b +1) u +(v +1)^{2}}{a \,u^{2} v}
 \\ & \hfill \frac{b u^{2}+b u +u +v +1}{u v}
 \,\,
\bigg)\,.
\end{align*}



We consider the affine birational change of variables  
$\Phi: \mathbf C^2\dashrightarrow \mathbf C^2$, $(x,y)\mapsto (u,v)$ given by 
\begin{equation}
\label{Eq:uv}
u=\frac{(b-1) x - b y + 1}{b\,(1-y)}
\qquad \mbox { and } \qquad 
v= (b-1)\frac{b x - (b-1) y - x y}{b\, (1-y)^2}\, .
\end{equation}

By straightforward computations, one obtains that as a 10-tuple of rational functions, one has $\Phi^*\big(\boldsymbol{\mathcal X\hspace{-0.05cm}\mathcal W}_{D_4,\tau}^{10}\big) 
=\big( V_1,\ldots,V_{10}\,\big) $ where the $V_i$'s are the following elements of $\mathbf Q[a,b](x,y)$: 
\begin{align*}
V_1 = &\, \frac{b -y}{a \left(y -1\right)}      && V_6 = \frac{a \left(b -1\right) x^{2}+\big(
\left(y -1\right) b 
-\left(y +1\right) \left(b -1\right) a  \big) x +a \left(b -1\right)\,y}{\left(b -y \right) x -y \left(b -1\right)}
\\
V_2 = &\,  \frac{1}{x -1}     && V_7 = \frac{\left(\left(1-a \right) y +a x \right) \left(y -1\right)}{\left(-b -x +1\right) y +b x}
\\
V_3 = &\, \frac{y}{x -y}      && V_8 =\frac{\left(a \left(x -1\right)+b \right) \left(y -1\right)}{\left(-b -x +1\right) y +b x} 
\\
V_4 = &\,  \frac{b\,\left(y -1\right) }{b\,\left(x -y \right)  -x +1}     && V_ 9=
\frac{\left(\left(a x +b -1\right) y -a x \right) b}{a \big(\left(-b -x +1\right) y +b x \big)} 
\\
V_5 = &\,  \frac{b\,\big((a-1) y +b-a \big)}{a\,\big(
(b-1)x-by+1\big)  +b\,\left(y -1\right) }     && 
V_{10} = \frac{x b \left(y -1\right)}{\left(-b -x +1\right) y +b x}\,\, .
\end{align*}

Then one sets 
\begin{equation}
\label{Eq:ab-pi-gamma}
\big( \, \pi \, , \, \gamma\,\big)= \left(\,
\frac{a - b}{a}\, , \,   \frac{a-b}{a-1}\,\right)
\end{equation}
(which is equivalent to $
\big( \, a \, , \, b \,\big)=
 \big(\,{\gamma}/({\gamma-\pi})\, , \,  {(1-\pi)\gamma}/{(\gamma-\pi) }$\,\big) and 
%
\begin{align*}
\label{BB}
q_1 =  \big[1:  0 : 0  \big]\, , \hspace{0.2cm}
q_2 =  \big[ 0 : 1 : 0 \big]\, , \hspace{0.2cm}
q_3= \,  \big[  1 : 0 : 0  \big] \, , \hspace{0.2cm}
q_4=   \big[ 1 :  -1: 1 \big]
\hspace{0.3cm} \mbox{and} \hspace{0.3cm} 
q_5 = \left[ \pi :  \gamma  :   1  \right] \,.
\end{align*}

From  \eqref{Eq:ab-pi-gamma}, one has $ab(a-1)(b-1)\neq 0$ if and only if $\pi\gamma(\pi-1)(\gamma-1)(\pi-\gamma)\neq 0$. Since the former condition has been assumed, the same holds for the second which geometrically means that  the five points $q_i$  are in general position in $\mathbf P^2$.  One easily checks that for $i=1,\ldots,5$, the rational function $V_i$ is a primitive first integral for the pencil of lines through $q_i$, whereas  
$V_{5+i}$ is a primitive first integral for the pencil of conics through all the $q_k$'s except $q_i$. 
Denoting here by $\beta_{\pi,\gamma}: {\rm dP}_4(\pi,\gamma) ={\bf Bl}_{q_1+\cdots+q_5}(\mathbf P^2\big) \rightarrow \mathbf P^2$ the blow-up of the plane at the $q_i$'s,one obtains the 
\begin{prop} 
For any  parameters $\pi,\gamma\in \mathbf C$ such that $\pi\gamma(\pi-1)(\gamma-1)(\pi-\gamma)\neq 0$,  one has
$$
\boldsymbol{\mathcal W}_{ {\rm dP}_4(\pi,\gamma) } =  
\big( \Phi\circ \beta_{\pi,\gamma}\big)^*\Big( 
\boldsymbol{\mathcal X\hspace{-0.05cm}\mathcal W}_{D_4,\tau}
\Big)
$$
where $\tau=(a,b)$ is related to $(\pi,\gamma)$ via the relation \eqref{Eq:ab-pi-gamma}. 
In particular, $\boldsymbol{\mathcal W}_{ {\rm dP}_4(\pi,\gamma) } $ is cluster. 
\end{prop}
%
%
%
%
%

This proposition is established by means of simple arguments or calculations that are explicit but do not say much about either the conceptual context underlying the proof or the numerous calculations we have performed to obtain it.  We think it is worth saying a few words about the process that led us to establish that any del Pezzo web associated with a quartic del Pezzo surface is a cluster web.  
There were three main points to be investigated: 
\begin{enumerate}
\item[1.]
\vspace{-0.15cm}
 first, given a del Pezzo quartic ${ {\rm dP}_4} $, we had the intuition to look for 
$\boldsymbol{\mathcal W}_{ {\rm dP}_4}$ as a subweb of a certain cluster 38-web 
$\boldsymbol{\mathcal X\hspace{-0.05cm}\mathcal W}_{D_4,\tau}$;
\item[2.] then given a web  $\boldsymbol{\mathcal X\hspace{-0.05cm}\mathcal W}_{D_4,\tau}$, 
 the task was to find one of its 10-subwebs likely to be equivalent to $\boldsymbol{\mathcal W}_{ {\rm dP}_4}$ ;
\item[3.] finally, once such a 10-subweb $\boldsymbol{\mathcal X\hspace{-0.05cm}\mathcal W}_{D_4,\tau}^{10}$  was identified, the problem was to find an analytic equivalence of the latter with $\boldsymbol{\mathcal W}_{ {\rm dP}_4}$.
\end{enumerate}
We dealt with these three points as follows: 
\begin{itemize}
\item[$-$] regarding the first, we recall that  Bol's web $\boldsymbol{\mathcal W}_{ {\rm dP}_5}$,  which lives on ${\rm dP}_5$,  is also (equivalent to) the $\boldsymbol{\mathcal X}$-cluster web of type $A_2$ hence is naturally defined on the cluster variety $\boldsymbol{\mathcal X}_{A_2}$. 
The latter essentially\footnote{`Essentially'  
means here `up to a birational map inducing an isomorphism in codimension 1'.}
is the Looijenga interior ${\rm dP}_5\setminus P$ where $P$ is an anticanonical pentagon in ${\rm dP}_5$.\footnote{An `anticanonical pentagon' is an element of the anticanonical linear system formed by 5 lines cyclically labeled with the $i$-th intersecting only lines with labels $i-1$ and $i+1$ in two distinct points.}
This suggests to consider the case of the Looijenga interior $\boldsymbol{U}={\rm dP}_4\setminus P_4$ where $P_4$ now stands for an anticanonical square in $\lvert -K_{{\rm dP}_4} \lvert$. Using \cite[\S2.4]{Mandel} and especially Theorem 2.13 therein, we come to the conclusion that $ \boldsymbol{U}$ can be described as a fiber of the map $\lambda$ for the cluster ensemble of type $D_4$; \sk 
\item[$-$] having identified the cluster algebra set-up which might be the right one, it is straightforward to get the cluster variables defining $\boldsymbol{\mathcal X\hspace{-0.05cm}\mathcal W}_{D_4,\tau}$. The problem we are facing at this step is that  the latter web  is a 38-web and it is not at all clear, assuming that it is possible (what we proved to be the case), to distinguish 10 of the 38 cluster variables on 
$\boldsymbol{\mathcal X}_{D_4,\tau}$ which could define a web equivalent to 
$\boldsymbol{\mathcal W}_{ {\rm dP}_4}$. We succeeded in this task by brute force computations of some arithmetic and combinatorial invariants of the 10-subwebs of $\boldsymbol{\mathcal X\hspace{-0.05cm}\mathcal W}_{D_4,\tau}$. Heavy and lengthy computations led us to consider more closely the web 
$\boldsymbol{\mathcal X\hspace{-0.05cm}\mathcal W}^{10}_{D_4,\tau}$
defined by the 10 red cluster variables of page \pageref{XWD4-tau}, which share  the same invariants with $\boldsymbol{\mathcal W}_{ {\rm dP}_4}$. 
\sk
\item[$-$] Finally, we deduce the explicit expression \eqref{Eq:uv} for the components of 
an analytic equivalence between  $\boldsymbol{\mathcal X\hspace{-0.05cm}\mathcal W}^{10}_{D_4,\tau}$ and $\boldsymbol{\mathcal W}_{ {\rm dP}_4}$ by means of a 
 more thorough investigation/comparison of the invariants of these two webs. 
\end{itemize}

\section{\bf Questions and perspectives}
\label{S:What-next}
In this section we discuss some questions and perspectives we find interesting considering the results obtained before in this text together with others that we briefly recall.

\subsection{\bf Del Pezzo's web on singular del Pezzo's quartic surfaces.}
\label{S:Singul}

The singular quartic surfaces of $\mathbf P^4$ which can be obtained by 
blowing-up up five points on $\mathbf P^2$ no longer assumed to be in general position have been classified by ancient geometers and all are explicitly known ({\it e.g.}\,see \cite{Segre}, \cite{Timms}, 
\cite[p.\,286]{HP}, etc. For modern references, see \cite[\S4]{CorayTsfasman} or 
\cite[\S3.4]{DerenthalSingular}).  


Since the del Pezzo webs 
 of the smooth quartic del Pezzo surfaces all satisfy so many nice features, it is natural to wonder about the properties of the webs formed by the fibrations in conics on the 
singular quartic del Pezzo surfaces. By means of direct computations, we have studied the webs by conics on the del Pezzo quartic surfaces with isolated singular points, mainly with regard to their rank and their abelian relations. We have found that way interesting webs, many of maximal rank,  that we are going to describe succinctly below. 

We consider only quartic surfaces in  $\mathbf P^4 $ with only finitely many singular points.  There are 15 possibilities  for the type of singularities for such a surface ({\it cf.}\,\cite[Prop.\,5.6]{CorayTsfasman} or 
\cite[Table 6]{DerenthalSingular}), and for each of them there are specific numbers for the lines and for the conic fibrations on the corresponding singular quartic surfaces. Since we are interested in exceptional webs, one has to restrict oneself to singular quartic surfaces $S\subset \mathbf P^4$ carrying at least 5 pencils of conics, so that the most singular types of such surfaces can be left aside and we just have to deal with the type of singularities, denoted by  $\Sigma$,  being one of the following ones: $A_1$, $2A_1$, $A_2$, $3A_1$, $A_1A_2$, $A_3$.\footnote{Here the notation $3A_1$ (resp.\,$A_1A_2$) refers to quartic surfaces whose 
singular set is formed by three singular points of type $A_1$ (resp.\,by two singular points, one of type $A_1$, the other of type $A_2$), etc.}

We use the following notations: 
\begin{itemize}
\item $\Sigma$ denotes an element of $\{ \,  \emptyset \, , \,A_1\, , \, 2A_1\, , \, A_2\, , \, 3A_1\, , \, A_1A_2\, , \, A_3\,  \}$;
\sk
\item $S=S(\Sigma)$ stands for a del Pezzo quartic surface in $\mathbf P^4$ with  $\Sigma$ as type of singularities;
\sk
\item  any singular del Pezzo quartic $S\hspace{-0.01cm}(\Sigma)$ admits a desingularization $\nu : \widetilde S\hspace{-0.01cm}(\Sigma)\rightarrow S(\Sigma)$ 
the composition of which with the inclusion $S\hspace{-0.01cm}(\Sigma)\subset \mathbf P^4$ is given by 
the anticanonical linear system $\lvert -K_{\widetilde S\hspace{-0.01cm}(\Sigma)}\lvert$. 
Moreover,  as  for smooth del Pezzo quartic surfaces, 
$\widetilde S(\Sigma)$ is obtained as the blow-up  $\beta_\Sigma : \widetilde S(\Sigma)\rightarrow \mathbf P^2$ in five points, but these points are no longer in general position (they even can be infinitely near).  The rational map $\phi_{ \lvert -K_{\widetilde S(\Sigma)}\lvert }\circ 
\beta_\Sigma^{-1} :\mathbf P^2\dashrightarrow \mathbf P^4$ is given by a linear system of cubic curves whose base locus scheme will be denoted by $B_\Sigma$; 
\sk
\item 
\vspace{-0.45cm}
any quartic  del Pezzo
 $S(\Sigma)\subset \mathbf P^4$ is cut out by two linearly independent quadratic 
equations, each being obtained from 
a symmetric $5 \times  5$ matrix. Denoting by 
$A$ and $B$ these matrices, one can always assume that $A$ is regular and 
define the {\it Segre symbol} $\mathcal S_\Sigma$ of $S$, as an uplet of integers encoding the decomposition in Jordan blocs of $A^{-1}B\in {\rm Sym}_5(\mathbf C)$;
\sk
\item ${\mathcal \mu_\Sigma}$ is the dimension of the moduli space of del Pezzo quartic surfaces 
$S\hspace{-0.01cm}(\Sigma)$;
\sk
\item ${\ell_\Sigma} $ denotes the number of lines contained in $S\hspace{-0.01cm}(\Sigma)$;
\sk
\item ${\kappa_\Sigma} $ stands for the number of pencils of conics on $S\hspace{-0.01cm}(\Sigma)$;
\sk
\item ${{\mathcal W}_{ S\hspace{-0.01cm}(\Sigma)}}$ denotes the ${\kappa_\Sigma}$-web defined by all conic fibrations on
 $S\hspace{-0.01cm}(\Sigma)$.
\sk
\end{itemize}

Depending on the type of singularities $\Sigma$, most of the elements of the list above are given in 
 Table \ref{Table:Singular-dP4} below. In the last column, we give an explicit expression for the birational model of ${{\mathcal W}_{ S\hspace{-0.01cm}(\Sigma)}}$ obtained by taking the  pull-back of this web under the map $\phi_{ \lvert -K_{\widetilde S(\Sigma)}\lvert }\circ 
\beta_\Sigma^{-1}$.  The base loci $B_\Sigma$ are pictured using the convention that genuine points on the base $\mathbf P^2$ are in black, whereas the infinitely near points are pictured white.

\begin{table}[!h]
{
\begin{tabular}{|c|c|c|c|c|c|l|}
\hline
\begin{tabular}{c} \vspace{-0.3cm}
\\
$\boldsymbol{\Sigma}$ 
\vspace{0.05cm}
\end{tabular}
 &  
$\boldsymbol{B_\Sigma}$
&
$\boldsymbol{\mathcal S_\Sigma}$
&
$\boldsymbol{\mathcal \mu_\Sigma}$
& 
$
\boldsymbol{\ell_\Sigma}
$ 
& 
$\boldsymbol{\kappa_\Sigma}$
&
\hspace{0.6cm}
  {\bf Birational model for the}
  $\boldsymbol{\kappa_\Sigma}${\bf-web} 
   $\boldsymbol{{\mathcal W}_{ S(\Sigma)}}$
  \\ \hline \hline
\begin{tabular}{c} \vspace{-0.5cm}
\\
$\emptyset$ 
\vspace{0.02cm}
\end{tabular}
  &
 \begin{tabular}{c} \vspace{-0.2cm}
\\ 
    \scalebox{0.12}{
 \includegraphics{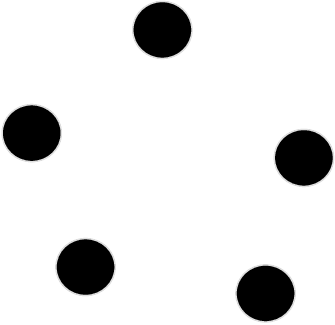}} 
 \vspace{0.02cm}
\end{tabular}
  & $\big[11111\big]$ & 2 &  16  & 10 & 
  \hspace{0.15cm}\scalebox{1}{$\boldsymbol{{\mathcal W}_{ \hspace{-0.05cm} {\rm dP}_4}}
  =\boldsymbol{\mathcal W}
\Big(x, y, \frac{x}{y},  \frac{y -1}{x -1},  \, 
\ldots \,  , \, \frac{\left(y -1\right) \left(x -\alpha \right) }{\left(x -1\right) \left(y -\beta  \right)}
 \,  , \, 
\frac{y\left(x -\alpha  \right) }{x \left(y -\beta \right)}
 \,  , \,
 \frac{x -\alpha}{y-\beta }\, \Big)$ }
 \\
  \hline
\begin{tabular}{c} \vspace{-0.35cm}
\\
$A_1$
\vspace{0.03cm}
\end{tabular}
  &   \begin{tabular}{c} \vspace{-0.2cm}
\\ 
    \scalebox{0.12}{
 \includegraphics{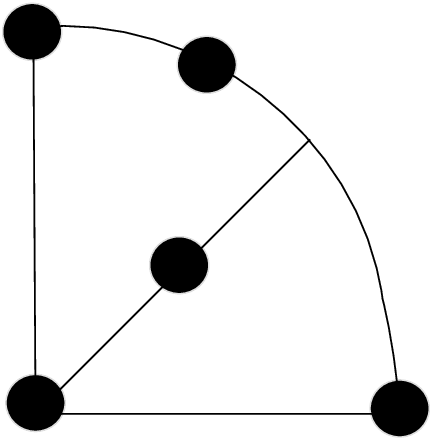}} 
 \vspace{0.02cm}
\end{tabular}  & $\big[2111\big]$ & 1 &  12  &  8 & 
\hspace{0.15cm}\scalebox{1}{$\boldsymbol{\mathcal W}\Big(\, x,y,  x-a y , \frac{x}{y},
 \frac{y-1}{x-1}, \frac{x(y-1)}{y(x-1)} \,  , \, 
\frac{\left(x-a y  \right) \left(y-1 \right)}{y \left(x-a y +a -1\right)}
 \,  , \,
\frac{\left(x -1\right) \left(x-a y  \right)}{x \left(x-a y +a  -1\right)}\, 
\Big)$}
 \\
  \hline 
     \multirow{2}{*}{$2A_1$}
 &   \begin{tabular}{c} \vspace{-0.2cm}
\\ 
    \scalebox{0.12}{
 \includegraphics{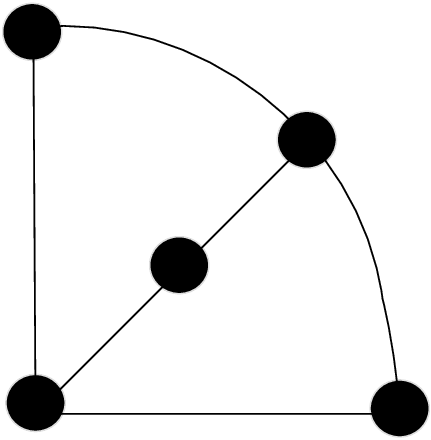}} 
 \vspace{0.02cm}
\end{tabular} & $\big[221\big]$ & 0 &  9  &  6 & 
\hspace{0.15cm} \scalebox{1}{$\boldsymbol{\mathcal W}\Big(x,y,  x-y , \frac{x}{y},
 \frac{y-1}{x-1}, \frac{x(y-1)}{y(x-1)}
\Big)$}
\vspace{0.0cm} \\
 &   \begin{tabular}{c} \vspace{-0.2cm}
\\ 
    \scalebox{0.12}{
 \includegraphics{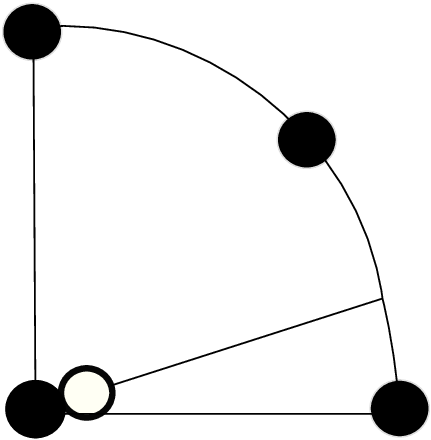}} 
 \vspace{0.02cm}
\end{tabular} & $\big[(11)111\big]$ & 1 &  8  &  7 & 
\hspace{0.15cm}
\scalebox{1}{$\boldsymbol{\mathcal W}\Big(x,y,  x-y , \frac{x}{y},
\frac{x \left(x -y \right)}{x-a y }, 
\frac{y \left(x -y \right)}{x-a y}, \frac{x y}{x-a y}
\Big)$}
 \\
  \hline 
\begin{tabular}{c} \vspace{-0.3cm}
\\
$3A_1$
\vspace{0.04cm}
\end{tabular}  
&  
 \begin{tabular}{c} \vspace{-0.2cm}
\\ 
    \scalebox{0.12}{
 \includegraphics{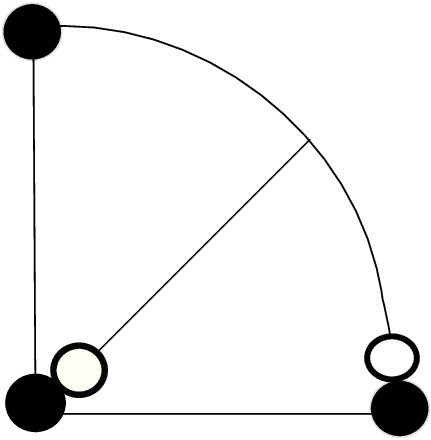}} 
 \vspace{0.02cm}
\end{tabular}
  & $\big[(11)21\big]$ & 0 &  6  &  5 & 
\hspace{0.15cm} $\boldsymbol{\mathcal W}\Big( x, y, \frac{x}{y}, \frac{x y}{x +y}, \frac{y^{2}}{x +y}\Big)$ 
 \\
  \hline 
\begin{tabular}{c} \vspace{-0.35cm}
\\
$A_2$
\vspace{0.02cm}
\end{tabular}  &  
\begin{tabular}{c} \vspace{-0.2cm}
\\ 
    \scalebox{0.12}{
 \includegraphics{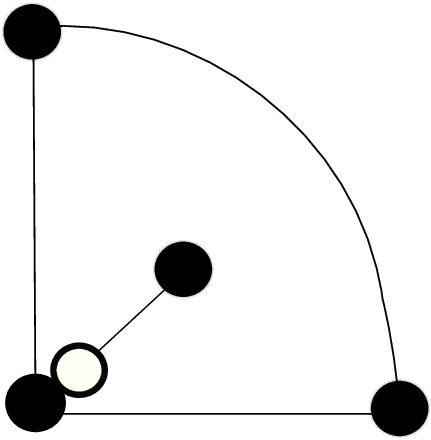}} 
 \vspace{0.02cm}
\end{tabular}
  & $\big[311\big]$ & 0 &  8  &  6& 
\hspace{0.15cm}\scalebox{1}{$\boldsymbol{\mathcal W}\Big(x,y, \frac{x}{y},
 \frac{y-1}{x-1}, \frac{x(y-1)}{y(x-1)},
  \frac{x-y}{xy}
\Big)$} 
 \\
  \hline 
\begin{tabular}{c} \vspace{-0.35cm}
\\
$A_1A_2$
\vspace{0.02cm}
\end{tabular}  & 
 \begin{tabular}{c} \vspace{-0.2cm}
\\ 
    \scalebox{0.12}{
 \includegraphics{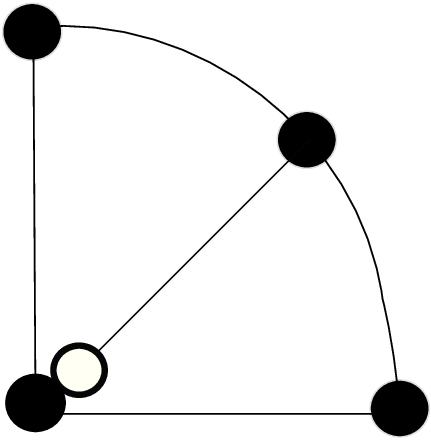}} 
 \vspace{0.02cm}
\end{tabular}
   & $\big[32\big]$ & 0 &  6  &  4 & 
\hspace{0.15cm}
\scalebox{1}{$\boldsymbol{\mathcal W}\Big(x,y,  x-y , \frac{x}{y}, \frac{x y}{x- y}
\Big)$} 
 \\
  \hline 
     $A_3$
 &  \begin{tabular}{c} \vspace{-0.2cm}
\\ 
    \scalebox{0.12}{
 \includegraphics{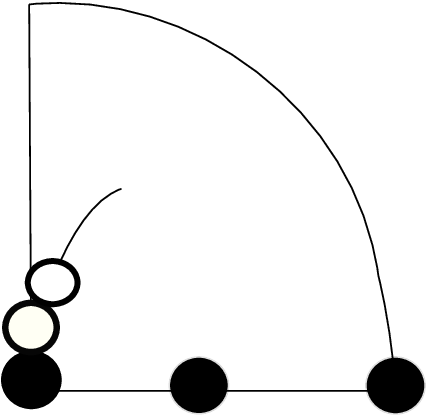}} 
 \vspace{0.02cm}
\end{tabular} & $\big[(21)11\big]$ & 0 &  4  &  5 & 
\hspace{0.15cm}
\scalebox{1}{$\boldsymbol{\mathcal W}\Big(
y, \frac{x}{y}, \frac{1+x}{y}, \frac{y^{2}+x}{x y}, 
\frac{x^{2}+y^{2}+x}{x y}
\Big)$} 
 \\
  \hline 
\end{tabular}\bk\sk
}
\caption{Webs by conics on quartic del Pezzo surfaces with finitely many singular points, for all types of singularities $\Sigma$ such that ${\kappa_\Sigma} \geq 5$. (Base points on the base $\mathbf P^2$ are in black, whereas those in white are infinitely near the other points they are touching).}
\label{Table:Singular-dP4}
\end{table}

Using several effective tools, 
we have studied all the webs explicited in this table, especially with regard  
to their abelian relations and their rank. From this, we get the
\begin{prop}
The del Pezzo's webs in Table \ref{Table:Singular-dP4} all have maximal rank hence are exceptional.
\end{prop}
Generally speaking, the webs by conics on del Pezzo quartic surfaces with finitely many singularities 
all are maximal rank and this seems to be verified as well for many cubic surfaces of $\mathbf P^3$ with the same type of singularities (however, recall that $\boldsymbol{\mathcal W}_{ \hspace{-0.03cm} {\rm dP}_3}$ does not have maximal rank for a generic smooth 
del Pezzo cubic $ {\rm dP}_3$). 
Given that being of maximum rank (and exceptional when non-linearizable) is a rather exceptional phenomenon for a web, we believe that the fact that many del Pezzo webs on smooth or singular del Pezzo surfaces are of maximum rank cannot be a mere coincidence. This leads us to ask the following question 
\begin{question}
Does it exist (and if yes, which is it) a general conceptual reason explaining why so many del Pezzo's webs (that is, webs formed by considering all the conic fibrations) of del Pezzo surfaces 
without or with  only finitely many singular points  
are of maximal rank? 
\end{question}


\subsection{\bf Behavior of the ${\mathcal W} \hspace{-0.5cm}{\mathcal W}_{{\bf dP}_d}$'s in families and degeneracies}
\label{S:degen}
For $d\in \{1,2,3,4\}$ there is a $2(5-d)$-dimensional algebraic family 
$d{\mathcal P}_d$
of degree $d$ del Pezzo surfaces  giving rise to a family of webs 
${\mathscr W}_{ \hspace{-0.05cm}{\rm dP}_d} \rightarrow d{\mathcal P}_d$ over it.  
First, it is interesting to ask how the webs $\boldsymbol{\mathcal W}_{{\rm dP}_d}$'s behave when ${\rm dP}_d$ varies in $d{\mathcal P}_d$.

Another kind of natural questions is  how the 
$\boldsymbol{\mathcal W}_{{\rm dP}_d}$'s may degenerate: there are modular compactifications 
$\overline{d{\mathcal P}}_d$ of the 
 moduli space $d{\mathcal P}_d$
 and it is natural to wonder first whether ${\mathscr W}_{{\rm dP}_d} \rightarrow d{\mathcal P}_d$ can be extended over $\overline{d{\mathcal P}}_d$, then if it is the case, what are the  properties (in particular those concerning its  ARs) of a web $\boldsymbol{\mathcal W}_{{\rm dP}'_d}$ on a degenerate del Pezzo surface 
 ${\rm dP}'_d\in \overline{d{\mathcal P}}_d \setminus d{\mathcal P}$. Since the 
 singular del Pezzo surfaces considered in \S\ref{S:Singul} 
are such degenerations, questions asked in the current subsection are generalizations of those raised in the above subsection.

 In view of considering a very concrete example,  let us consider the webs on 
$\mathbf P^2$  
 defined by the 10 rational first integrals 
\eqref{Eq:WdP4-Ui},  denoted by 
$\boldsymbol{\mathcal W}_{{\rm dP}_4(\gamma,\pi)}$. Assuming that \eqref{Eq:PG-pi-gamma} is satisfied, each such web is a planar model of the 10-web by conics on a smooth del Pezzo quartic surface
which will be denoted by ${\rm dP}_4(\gamma,\pi)$.  Let us consider the naive specialization $\boldsymbol{\mathcal W}_{{\rm dP}_4,-1}
=\lim_{\pi,\gamma\rightarrow -1}\boldsymbol{\mathcal W}_{{\rm dP}_4(\gamma,\pi)}$
of the webs $\boldsymbol{\mathcal W}_{{\rm dP}_4(\gamma,\pi)}$'s: it is the web whose first integrals are the 
limits of all the $U_i$'s 
 in \eqref{Eq:WdP4-Ui}  when both parameters $\pi$ and $\gamma$ go to -1.
It is no longer a 10-web but the following 8-web 
$$\boldsymbol{\mathcal W}_{{\rm dP}_4,-1}
 =\boldsymbol{\mathcal W}\left(\, 
x\, , \,  
\frac{1}{y} \, , \,
 \frac{y}{x} \, , \,
  \frac{x - y}{x - 1}\, , \,
   \frac{x + 1}{x - y} \, , \,
    \frac{x - y}{(x - 1)(y + 1)} \, , \,
    \frac{y(x + 1)}{x(y + 1)} \, , \,
     \frac{x(y - 1)}{y(x - 1)}
  \,\right)\, . 
$$
One  verifies that this web has maximal rank 21 with all its ARs being generalized hyperlogarithms of weight 1 or 2.\footnote{Here, the adjective phrase `{\it generalized hyperlogarithmic'} refers to the notion of `{\it generalized iterated integrals'} considered in \cite[\S1.5.3]{ClusterWebs}.} 
   Denote by ${\bf Hlog}^3_{\gamma,\pi}$ the weight 3 hyperlogarithmic  functional identity carried by  $\boldsymbol{\mathcal W}_{{\rm dP}_4(\gamma,\pi)}$ for
$\gamma$ and $\pi$  generic (that is, satisfying \eqref{Eq:PG-pi-gamma}). From the description of the ARs of $\boldsymbol{\mathcal W}_{{\rm dP}_4,-1}$ just mentioned, it comes that ${\bf Hlog}^3_{\gamma,\pi}$ somehow disappears under the specialization $(\gamma,\pi)\rightarrow (-1,-1)$. It would be interesting to understand this phenomenon better and more generally to see how  the ARs of the generic 
web $\boldsymbol{\mathcal W}_{{\rm dP}_4(\gamma,\pi)}$ and those of the specialization $\boldsymbol{\mathcal W}_{{\rm dP}_4,-1}$ are related.

\subsection{\bf Global single-valued version of ${\bf HLog}^{3}$}
\label{S:single-valued}
An interesting well-known feature of Abel's identity $\boldsymbol{\big(\mathcal Ab\big)}$
is that there is a global single-valued but real analytic version of it. More precisely, 
let $D$ stand for the {\it `Bloch-Wigner dilogarithm'}, which is  defined 
for any $z\in \mathbf P^1\setminus\{0,1,\infty\}$ 
by   
\begin{equation}
\label{Eq:Bloch-Wigner}
D(z)={\rm Im}\Big( {\bf L}{\rm i}_2(z) \Big)+{\rm Arg}\big(1-z\big)\cdot {\rm Log}\,\lvert \, z\, \lvert\, , 
\end{equation}
where ${\rm Arg} : \mathbf C^*\rightarrow  ]-\pi, \pi]$ denotes   
the main branch of the complex argument.
Bloch-Wigner's dilogarithm $D$ is real-analytic and extends continuously to the whole Riemann sphere. Denoting again by $D$ this extension, one has $D(0)=D(1)=D(\infty)=0$  and the most remarkable of its features is that it satisfies the following global version of the identity $\boldsymbol{\big(\mathcal Ab\big)}$: one has 
 $$
D(x)-D(y)-D\bigg(\frac{x}{y}\bigg)-D\left(\frac{1-y}{1-x}\right)
+D\left(\frac{x(1-y)}{y(1-x)}\right)=0\,
$$
for any $x,y\in \mathbf P^1$ such that none of the five arguments of $D$ in this identity be indeterminate. \sk 

The hyperlogarithmic identity ${\bf HLog}^{3}$ is holomorphic, because the functions $AH^{3}_i$'s involved in it are holomorphic.  On the other hand this functional identity 
 is only locally satisfied, because the $AH^{3}$'s extend to global but multivalued functions on the whole projective line.  
Remark that  the function $AH^{2}$ involved in ${\bf HLog}^{2}\simeq \boldsymbol{\mathcal Ab}$ is a holomorphic version of Rogers dilogarithm $R$ which admits $D$ as a single-valued global cousin.  The point is that it has been proved in
 \cite[\S7]{Brown2004} (see also \cite[\S2.5]{VanhoveZerbini2018} or 
 \cite[\S3]{CDG2021}) that the most general hyperlogarithm $HL$ admits a single-valued global version $HL^{\rm sing}$.  Specializing this to the case of the weight 3 antisymmetric hyperlogaritms $AH^{3}_i$ involved in 
 ${\bf HLog}^{3}$ gives 
 global single-valued functions ${AH}^{3,sing}_i$ for $i=1,\ldots, 10$, regarding which the following questions immediately arise: 
\begin{question}
\begin{itemize}
\item[] ${}^{}$ \hspace{-1.1cm}{\rm 1.}  {\it Can one give an explicit formula for the ${AH}^{3,sing}_i$'s?
 }
\sk 
\item[] ${}^{}$ \hspace{-0.8cm} {\rm 2.} {\it Do the ${AH}^{3,sing}_i$'s satisfy the global univalued version of ${\bf HLog}^{3}$, {\it i.e.}\,using the same nota-
  \\
${}^{}$ \hspace{-0.6cm} tions
 than in  the first page of this text, do we have identically  
$\sum_{i=1}^{10} 
 {AH}^{3,sing}_i\big( U_i\big)=0$ 
on the \\
${}^{}$ \hspace{-0.6cm} whole  considered del Pezzo quartic surface ${\rm dP}_4$?}
\sk 
\end{itemize}
\end{question}
We believe that  the answer to 2. is positive, possibly up to considering suitable 
 single-valued versions of the antisymmetric hyperlogarithms $AH_i^{3}$. An approach for answering this question might be to use Theorem 1.1 of  \cite[\S3]{CDG2021} but this would require first to build a `motivic lift' of the identity ${\bf HLog}^{3}$.

\subsection{\bf Motivic lift}
In \cite{CDG2021}, the authors discuss motivic avatars of hyperlogarithms, some single-valued (still motivic) versions of these, classical realisations of them, as well as a motivic approach of the functional identities satisfied by the hyperlogarithms.  
Thanks to its algebraic nature, 
we expect that ${\bf HLog}^{3}$  indeed can  be lifted into the `motivic world'. 
 It would be interesting to investigate this more rigourously.

\subsection{\bf Supersymmetric version of ${\bf HLog}^3$}
Several authors have undertaken to build a theory of cluster superalgebras. 
In the preprint \cite{GMSV}, the authors  explore physical applications of this notion. In particular, in the fourth section, they discuss a `super cluster polylogarithm identity' which can be seen as a supersymmetric version of Abel's identity 
$\boldsymbol{\mathcal Ab}\simeq  {\bf HLog}^{2} $. 
Because ${\bf HLog}^{3} $ is cluster as well (as shown above in \S\ref{SS:WdP4-as-a-cluster-web}), it is natural to ask the 
\begin{question}
Does it exist a cluster supersymmetric version of ${\bf HLog}^3$?
\end{question}

\subsection{\bf Non-archimedean version  version of ${\bf HLog}^3$}
Analytic $p$-adic versions of classical polylogarithms have been constructed by Coleman in \cite{Coleman} where he moreover proved that the $p$-adic version of Rogers dilogarithm  satisfies a (or `the'?) $p$-adic analog of the 5 terms identity ({\it cf.}\,\cite{Coleman} or \cite{Wojtkowiak}. Since Coleman's work (where another approach was used), a theory of non-archimedean iterated integrals has become available (see \cite{Berkovich}) which led us to think that the antisymmetric hyperlogarithms $AH^3_i$ involved in  ${\bf HLog}^3$ should have non-archimedean avatars. If this is so, it is then natural to ask the  
\begin{question}
Does it exist a non-archimedean version of ${\bf HLog}^3$?
\end{question}

\subsection{\bf Quantum version of ${\bf HLog}^3$}
\label{SS:Quantum-Version}
The identity $\boldsymbol{\mathcal Ab}\simeq {\bf HLog}^2$ can be seen as a classical limit of a quantum pentagon dilogarithmic  identity 
$\boldsymbol{q\mathcal Ab}$ 
(see \cite[\S8.6.1]{ClusterWebs} and the references given there). Does this generalize to ${\bf HLog}^3$? More precisely: 
\begin{question}
Does it exist a quantum version ${\bf qHLog}^3$ of the identity ${\bf HLog}^3$ such that the latter could be seen as a classical limit of the former?
\end{question}

\subsection{\bf Scattering diagram interpretation of ${\bf HLog}^3$}
There is now a conceptual interpretation of why the 5-terms dilogarithmic identity  (or even its quantum version $\boldsymbol{q\mathcal Ab}$) holds true in terms of a certain `scattering diagram'. Indeed, that $\big(\boldsymbol{\mathcal Ab}\big)$ (or 
$\boldsymbol{q\mathcal Ab}$) 
holds 
reflects the fact that a certain (quantum) `scattering diagram' is consistent (see the preprints \cite{Nakanishi1,Nakanishi2} by Nakanishi, or his forthcoming book `Cluster Algebras and Scattering Diagrams'\footnote{This book is available  on arXiv (in three parts).}).

On another hand, there is a scattering diagram $S\hspace{-0.06cm}C_{ {\rm dP}_4}$ associated to any del Pezzo surface ${\rm dP}_4$ which has been made entirely explicit by Arg\"uz \cite{Arguz}. Considering this and what has been said in the preceding paragraph, it is natural to wonder whether something similar does occur for ${\bf HLog}^3$ with respect to $S\hspace{-0.06cm}C_{ {\rm dP}_4}$: 
\begin{question}
Can the identity ${\bf HLog}^3$ be more conceptually interpreted in terms of some (algebraic, combinatorial or topological) properties of the scattering diagram 
$S\hspace{-0.06cm}C_{ {\rm dP}_4}$?
\end{question}

\subsection{Construction of ${\bf HLog}^{3}$ \`a la Gelfand-MacPherson}
\label{SS:HLog3-a-la-GM}
One of the most exciting results obtained by Gelfand and MacPherson in \cite{GM} is the geometric construction, henceforth providing a conceptual explanation for it, which they give of Abel's identity $\boldsymbol{\big(\mathcal Ab\big)}$ of the dilogarithm. 
In concise terms, in a real setting: 
\begin{itemize}
\item first they identify, for any $n\geq 1$,  the quotient of a certain Zariski open subset $G_2^*(\mathbf R^{n+3})$ formed of `generic' 2-planes in $\mathbf R^{n+3}$
by the action of the positive part $H_{>0}=H(\mathbf R_{>0})\simeq (\mathbf R_{>0})^{n+m-1}$ of the Cartan torus $H$ of ${\rm SL}_{m+n}(\mathbf R)$, with the real moduli space $\mathcal M_{0,n+3}({\mathbf R})$ of  projective configurations 
of $n+3$ points on the real projective line;
\sk
\item then they define a general notion of {\it `polylogarithmic differential form'} on 
 $\mathcal M_{0,n+3}({\mathbf R})$, which is any form $\omega(P)_{n+3}$ obtained by integrating   along the fibers of the 
$H_{>0}$-action on $G_2^*(\mathbf R^{n+3})$,  a ${\rm SO}_{n+3}(\mathbf R)$-invariant closed differential form 
$\Omega(P)$ whose cohomology class coincides with a previously given 
stable characteristic class $P\in {\bf H}^*\big( 
G_2(\mathbf R^{n+3}), \mathbf R\big)$;
\sk
\item via an isomorphism induced by the moment map, the (closure in the grassmannian of the) 
$H_{>0}$-orbit of any generic 2-plane can be seen to have 
the combinatorial structure of the hypersimplex $\Delta_2^{n+3}=\{\, (t_i)_{i=1}^{n+3}\in [0,1]^{n+3}\, \big\lvert \, \sum_i t_i=2\, \}$. This polytope has $2(n+3)$-facets (1-codimensional faces),  which are the intersections $\Delta_{2,i,\nu}^{n+3}=
\Delta_2^{n+3}\cap \{ \, t_i=\nu\, \}$ for $i=1,\ldots,n+3$ and $\nu\in \{0,1\}$. Up 
to isomorphisms induced by natural linear inclusions $
\mathbf R^{n+2}\hookrightarrow \mathbf R^{n+3}$), the $\Delta_{2,i,0}^{n+3}$'s are $(n+1)$-simplices while the $\Delta_{2,i,1}^{n+3}$'s
are hypersimplices $\Delta_2^{n+2}$;
\sk
\vspace{-0.4cm}
\item for any $i$, let $\mathbf R^{n+2}_i$ be the coordinate hyperplane 
in $ \mathbf R^{n+3}$
defined by the vanishing of the $i$-th coordinate. 
Then $G_2(\mathbf R^{n+2}_i)$ is naturally contained in $G_2(\mathbf R^{n+3})$ and is $H_{>0}$-stable.  Moreover, there 
exists a  $H_{>0}$-equivariant rational map  $F_{i}: G_2(\mathbf R^{n+3})\dashrightarrow G_2(\mathbf R^{n+2}_i)$ the $H_{>0}$-reduction of which is the $i$-th forgetful map $f_i : \mathcal M_{0,n+3}({\mathbf R})\rightarrow \mathcal M_{0,n+2}({\mathbf R})$;
\sk
%
\item when $P \in {\bf H}^4\big( G_2(\mathbf R^{n+2}), \mathbf R\big)$ stands for the first Pontryagin class of the rank 2 tautological bundle on $G_2(\mathbf R^{n+2})$,  Gelfand and MacPherson prove that the following facts hold true: 
\begin{enumerate}
\item[$-$] the polylogarithmic differential forms $\omega(P)_{n+3}$  for $n=1,2$ 
satisfy  
\begin{equation}
\label{Eq:dOmegaP5}
d\omega(P)_{5}=\sum_{i=1}^5 f_i^*\big( 
\omega(P)_{4}
\big)\,, 
\end{equation}
 a differential relation
 in $\Omega^1\big(  \mathcal M_{0,5}({\mathbf R})\big)$ which follows 
from Stokes' theorem for integration along the fibers of the $H_{>0}$-actions on 
$G_2(\mathbf R^{5})$ and $G_2(\mathbf R^{4})$;
\sk
\item[$-$] the function $\omega(P)_{5}$ vanishes identically 
whereas up to 
the identification $\mathcal M_{0,4}({\mathbf R})\simeq \mathbf R\setminus \{0,1\}$ provided by the cross-ratio, the differential 1-form $\omega(P)_{4}$ coincides with the total derivative of Rogers dilogarithm $R$: one has
\begin{equation}
\label{Eq:OmegaP5-OmegaP4}
\omega(P)_{5}= 0 \qquad 
\mbox{ and } 
\qquad 
\omega(P)_{4}= dR\, ;
\end{equation}
\end{enumerate}
\item from \eqref{Eq:dOmegaP5} and \eqref{Eq:OmegaP5-OmegaP4} together, it follows immediately that  $
0= \sum_{i=1}^5 f_i^*\big( 
dR \big)
$, 
which is nothing else but the differential form of Abel's identity $\boldsymbol{\big(\mathcal Ab\big)}$.\mk 
\end{itemize}

Because $\boldsymbol{\mathcal W}_{ {\rm dP}_4}$ can be obtained from 
Gelfand-MacPherson web 
$ \boldsymbol{\mathcal W}^{GM}_{ {\mathcal Y}_5}
=\big(\boldsymbol{\mathcal W}^{GM}_{ {\mathbb S}_5}\big)\big/ H_5$ and because of the many similarities between ${\bf HLog}^2\simeq \boldsymbol{\mathcal Ab}$ and 
${\bf HLog}^3$, we are naturally led to ask the following 
\begin{question}
\label{Question:tolo}
Can the functional identity ${\bf HLog}^3$ be obtained \`a la Gelfand-MacPherson by means of an invariant differential form $\Omega(P)$ on a real form ${\bf S}_5$ of the spinor tenfold $\mathbb S_5$ which represents a certain characteristic class $P\in 
{\bf H}^*\big( {\bf S}_5,\mathbf R\big)${\rm ?} 
\end{question}

We believe that the proper answer to this question is `no' but this is because it is not the relevant question to ask, as we explain in the next subsection.

\subsection{Study of the Gelfand-MacPherson web 
$ {\mathcal W} \hspace{-0.46cm}{\mathcal W}^{G\hspace{-0.18cm} G\hspace{-0.02cm}   M\hspace{-0.24cm}M}_{
\hspace{-0.06cm}
 {\mathcal Y}_5\hspace{-0.36cm} {\mathcal Y}_5
}$ and of its abelian relations}
\label{Eq:GM-Webs-WGM-r}

The so many similarities between the webs $\boldsymbol{\mathcal W}_{ \hspace{-0.05cm}{\rm dP}_5}$ and $\boldsymbol{\mathcal W}_{ {\rm dP}_4}$ may let us think that the del Pezzo's webs $\boldsymbol{\mathcal W}_{ {\rm dP}_d}$ are the most natural generalizations of $\boldsymbol{\mathcal W}_{ \hspace{-0.05cm}{\rm dP}_5}\simeq \boldsymbol{\mathcal B}$ but this expectation is too naive, as it appears as soon as one considers the next web 
$\boldsymbol{\mathcal W}_{ {\rm dP}_3}$ since,  by a direct computation, one verifies that the Blaschke-Dubourdieux's curvature of this web is nonzero, which implies that it does not have maximal rank contrarily to $\boldsymbol{\mathcal W}_{ \hspace{-0.05cm}{\rm dP}_5}$ and $\boldsymbol{\mathcal W}_{ {\rm dP}_4}$. Another point which can be opposed to the claim that the $\boldsymbol{\mathcal W}_{ {\rm dP}_d}$'s are the most natural generalization of $\boldsymbol{\mathcal W}_{ \hspace{-0.05cm}{\rm dP}_5}$ is that starting from $d=4$ included, there is actually a positive dimensional moduli of such webs, which is not the case for 
 the previous two webs, namely $\boldsymbol{\mathcal W}_{ \hspace{-0.05cm}{\rm dP}_5}$
and also $\boldsymbol{\mathcal W}_{ \hspace{-0.05cm}{\rm dP}_6}=
 \boldsymbol{\mathcal W}(\,x\, , \, y\, , \, x/y\big)$ which are unique. \sk

Since the  $\boldsymbol{\mathcal W}_{  \hspace{-0.05cm} {\rm dP}_4}$'s can all be obtained from 
Gelfand-MacPherson web 
 $ \boldsymbol{\mathcal W}^{GM}_{
\boldsymbol{\mathcal Y}_5
}$, it is not unreasonable to ask whether this latter web cannot be seen as a  more natural generalization of $\boldsymbol{\mathcal W}_{  \hspace{-0.05cm}{\rm dP}_5}\simeq 
 \boldsymbol{\mathcal W}^{GM}_{ \hspace{-0.05cm}
\boldsymbol{\mathcal Y}_4}$ than the $\boldsymbol{\mathcal W}_{  \hspace{-0.05cm}{\rm dP}_4}$'s.  Preliminary results indicate that it is indeed the case.   The web $ \boldsymbol{\mathcal W}^{GM}_{ \hspace{-0.05cm}
\boldsymbol{\mathcal Y}_5}$ is a well-defined 10-web on the 5-dimensional rational variety $\boldsymbol{\mathcal Y}_5$, defined by 10 rational first integrals 
$\psi_i^{\epsilon}: \boldsymbol{\mathcal Y}_5\dashrightarrow \mathbb S_4^\times/H_{D_4}\simeq \mathbf P^2$ with $i=1,\ldots,5$ and $\epsilon=\pm $.  What seems to be the relevant abelian relations to be considered for this web are the 2-ARs, namely 
the tuples of 2-forms $\big( \eta_{i}^\epsilon\big)_{i,\epsilon}$ such that 
$\sum_{i,\epsilon} \big(\psi_i^{\epsilon}\big)^* \big( 
\eta_{i}^\epsilon \big)=0$  (possibly just locally) on $\boldsymbol{\mathcal Y}_5$.
They form a vector space which we denote by $\boldsymbol{AR}^{2}\big( 
 \boldsymbol{\mathcal W}^{GM}_{ \hspace{-0.05cm}
\boldsymbol{\mathcal Y}_5}\big)$. 
Following the same approach as the one described in \cite[\S1.3.4.]{ClusterWebs}, 
one can define the `virtual 2-rank' $\rho^{(2)}(\boldsymbol{\mathcal W})$ of any subweb $\boldsymbol{\mathcal W}$ of $ \boldsymbol{\mathcal W}^{GM}_{ \hspace{-0.05cm}
\boldsymbol{\mathcal Y}_5}$. If $\psi$ stands for one of the first integrals  
$\psi_i^\epsilon$ of $ \boldsymbol{\mathcal W}^{GM}_{ \hspace{-0.05cm}
\boldsymbol{\mathcal Y}_5}$, we denote by 
\begin{itemize}
\item 
$\mathbf C(\psi)$ the subalgebra of rational functions on $\boldsymbol{\mathcal Y}_5$ formed by compositions $f\circ \psi$ with $f \in \mathbf C\big(\mathbb S_4/H_{D_4}\big)\simeq \mathbf C( \mathbf P^2)$;\sk
\item ${\rm Log}\mathbf C(\psi)$ the family of multivalued functions on $\boldsymbol{\mathcal Y}_5$ of the form  ${\rm Log}(f\circ \psi)$ with $f \in \mathbf C\big(\mathbb S_4/H_{D_4}\big)\simeq \mathbf C( \mathbf P^2)$;
\sk
\item $d{\rm Log}\mathbf C(\psi)$ the space of $\psi$-logarithmic differential 1-forms, that is of rational 1-forms on 
$\boldsymbol{\mathcal Y}_5$ of the form  $d{\rm Log}(f\circ \psi)=
d(f\circ \psi)/(f\circ \psi)$ with $f \in \mathbf C\big(\mathbb S_4/H_{D_4}\big)\simeq \mathbf C( \mathbf P^2)$.
\end{itemize}

By means of explicit computations in Maple, we have established the 
\begin{thm} 
\label{THM:WdP5-WGMS5-similarities}
%
%
\begin{enumerate}
\item[]  \hspace{-1.2cm} {\rm 1.} One has  $\rho^{(2)}\Big(
\boldsymbol{\mathcal W}^{GM}_{ \hspace{-0.05cm}
\boldsymbol{\mathcal Y}_5}
\Big)= 11$ 
and $\rho^{(2)}\big(
\boldsymbol{\mathcal W}
\big)\leq 1$  for every 5-subweb  $\boldsymbol{\mathcal W}$ of $\boldsymbol{\mathcal W}^{GM}_{ \hspace{-0.05cm}
\boldsymbol{\mathcal Y}_5}$. 
\sk
\sk
\item[2.] Among all the 
 5-subwebs of $\boldsymbol{\mathcal W}^{GM}_{ \hspace{-0.05cm}
\boldsymbol{\mathcal Y}_5}$, exactly 16 have virtual 2-rank equal to 1. These are the subwebs $\boldsymbol{\mathcal W}^{\underline{\epsilon}}=\boldsymbol{\mathcal W}\big(\, \psi_1^{\epsilon_1}\, , \, \ldots,  
\psi_5^{\epsilon_5}\,\big)$ for the sixteen 5-tuples $\underline{\epsilon}=(\epsilon_i)_{i=1}^5\in \{\pm 1\}$
such that $p(\underline{\epsilon})=\#\,\{ i
 \, \lvert \, \epsilon_i=1\,\}$ is odd.\footnote{This has to be compared with the description of Bol's subwebs of 
$\boldsymbol{\mathcal W}_{  \hspace{-0.05cm} {\rm dP}_4}$ given page \pageref{papage}.} Each such subweb $\boldsymbol{\mathcal W}^{\underline{\epsilon}}$ actually  has maximal rank 1, with 
$\boldsymbol{AR}^{2}\big( \boldsymbol{\mathcal W}^{\underline{\epsilon}}\big)$
spanned by 
a
 2-AR 
${\bf LogAR}^{\underline{\epsilon}}$ which is complete, irreducible and logarithmic, in the sense  that the $\psi_i^{\epsilon_i}$-th component of 
 ${\bf LogAR}^{\underline{\epsilon}}$
belongs to $d{\rm Log}\mathbf C(\psi_i^{\epsilon_i})$  for every 
$i$.
\sk
\item[3.]  The ${\bf LogAR}^{\underline{\epsilon}}$'s for all odd 5-tuples 
$\underline{\epsilon}$'s span 
a vector space denoted by 
$\boldsymbol{AR}^{2}_C\big( 
 \boldsymbol{\mathcal W}^{GM}_{ \hspace{-0.05cm}
\boldsymbol{\mathcal Y}_5}\big)$ and called the space of `combinatorial ARs' of 
$\boldsymbol{\mathcal W}^{GM}_{ \hspace{-0.05cm}
\boldsymbol{\mathcal Y}_5}$.  Moreover, one has 
$$
{\rm rk}^{(2)}_C
\Big( 
 \boldsymbol{\mathcal W}^{GM}_{ \hspace{-0.05cm}
\boldsymbol{\mathcal Y}_5}\Big)=\dim\, 
\boldsymbol{AR}^{2}_C\Big( 
 \boldsymbol{\mathcal W}^{GM}_{ \hspace{-0.05cm}
\boldsymbol{\mathcal Y}_5}\Big)
=\rho^{(2)}\Big(
\boldsymbol{\mathcal W}^{GM}_{ \hspace{-0.05cm}
\boldsymbol{\mathcal Y}_5}
\Big)-1=10
\,.
$$
\item[4.] There exists an AR of $\boldsymbol{\mathcal W}^{GM}_{ \hspace{-0.05cm}
\boldsymbol{\mathcal Y}_5}$, denoted by ${\bf DilogAR}^{(2)}$, which is 
complete, irreducible and 
 whose components are dilogarithmic, in the sense that  
  for any 
 first integral $\psi_i^\epsilon$ of $\boldsymbol{\mathcal W}^{GM}_{ \hspace{-0.05cm}
\boldsymbol{\mathcal Y}_5}$, 
 the $\psi_i^\epsilon$-th component of 
 ${\bf DilogAR}^{(2)}$
  belongs to 
 ${\rm Log}\mathbf C(\psi_i^\epsilon)\, d{\rm Log}\mathbf C(\psi_i^\epsilon)$. Moreover, one has 
\begin{equation}
\label{Eq:Decomp}
\boldsymbol{AR}^{2}\Big( 
 \boldsymbol{\mathcal W}^{GM}_{ \hspace{-0.05cm}
\boldsymbol{\mathcal Y}_5}\Big)=
\boldsymbol{AR}^{2}_C\Big( 
 \boldsymbol{\mathcal W}^{GM}_{ \hspace{-0.05cm}
\boldsymbol{\mathcal Y}_5}\Big)\oplus 
\Big\langle\,
{\bf DilogAR}^{(2)} 
\, 
\Big\rangle
\end{equation}
which implies $
{\rm rk}^{(2)}
\Big( 
 \boldsymbol{\mathcal W}^{GM}_{ \hspace{-0.05cm}
\boldsymbol{\mathcal Y}_5}\Big)
=\rho^{(2)}\Big(
\boldsymbol{\mathcal W}^{GM}_{ \hspace{-0.05cm}
\boldsymbol{\mathcal Y}_5}
\Big)=11
$: the web $\boldsymbol{\mathcal W}^{GM}_{ \hspace{-0.05cm}
\boldsymbol{\mathcal Y}_5}$  has maximal 2-rank.\footnote{It would be more rigorous to  state this as {\it `the 2-rank of $\boldsymbol{\mathcal W}^{GM}_{ \hspace{-0.05cm}
\boldsymbol{\mathcal Y}_5}$ is AMP'}, using the terminology introduced in \cite[\S1.3.5]{ClusterWebs}.}
\sk 
\item[5.] One can associate a divisor $\boldsymbol{\mathcal D}_w$ of 
$\boldsymbol{\mathcal Y}_5$ to each weight $w$ of the minuscule half-spin representation $S_5^+$. Then  the 2-ARs of $\boldsymbol{\mathcal W}^{GM}_{ \hspace{-0.05cm}
\boldsymbol{\mathcal Y}_5}$ are regular on 
 the complement in $\boldsymbol{\mathcal Y}_5$ of the  union of all the $\boldsymbol{\mathcal D}_w$'s which coincides with $\boldsymbol{\mathcal Y}_5^* $. Moreover, the sixteen ARs ${\bf LogAR}^{\underline{\epsilon}}$'s are exactly the logarithmic ARs obtained by considering the residues 
of ${\bf DilogAR}^{(2)}$ along the $\boldsymbol{\mathcal D}_w$'s: for any 
odd 5-tuple $\underline{\epsilon}$, there exists a uniquely defined weight $w(\underline{\epsilon})$ such that,  up to multiplication by a non-zero constant, one has ${\rm Res}_{ \boldsymbol{\mathcal D}_{w(\underline{\epsilon})} }\big( 
{\bf DilogAR}^{(2)}
\big) = {\bf LogAR}^{\underline{\epsilon}}
$. 
\sk 
\item[6.] The action of the Weyl group $W(D_5)$ 
by automorphisms on $\boldsymbol{\mathcal Y}_5$ ({\it cf.}\,\cite[Theorem 2.2]{Skorobogatov}{\rm )} gives rise to a linear 
$W(D_5)$-action  on $
\boldsymbol{AR}^{2}\big( 
 \boldsymbol{\mathcal W}^{GM}_{ \hspace{-0.05cm}
\boldsymbol{\mathcal Y}_5}\big)$ which lets invariant the decomposition in direct sum \eqref{Eq:Decomp}. This decomposition actually is the one in $W(D_5)$-irreducibles: the irrep associated to the 1-dimensional component 
$\big\langle\,
{\bf DilogAR}^{(2)} 
\, 
\big\rangle$ is the signature representation whereas $\boldsymbol{AR}^{2}_C\Big( 
 \boldsymbol{\mathcal W}^{GM}_{ \hspace{-0.05cm}
\boldsymbol{\mathcal Y}_5}\Big)$ is the $W(D_5)$-irreducible module $V^{10}_{[11,111]}$.\footnote{Seeing it as  a decomposition in $W(D_5)$-irreducibles,  
\eqref{Eq:Decomp},
must be compared with some results given in \S\ref{SSS:subspaces-of-ARs-as-W-modules}: in some sense which could be made precise, 
$\boldsymbol{AR}^{2}_C\big( 
 \boldsymbol{\mathcal W}^{GM}_{ \hspace{-0.05cm}
\boldsymbol{\mathcal Y}_5}\big)$ corresponds to ${\bf HLogAR}^2_{\rm asym}$ 
and ${\bf DilogAR}^{(2)}$ to ${\bf HLog}^3$.}
\sk
\item[7.] For any del Pezzo surface ${\rm dP}_4$, using Serganova-Skorobogatov embedding $f_{S\hspace{-0.05cm} S}: {\rm dP}_4 \hookrightarrow \boldsymbol{\mathcal Y}_5$ (cf.\,\eqref{Eq:Embedding-SS} above), 
the weight 3 hyperlogarithmic abelian relation ${\bf HLog}^3$ 
of $\boldsymbol{\mathcal W}_{  \hspace{-0.05cm} {\rm dP}_4}$ 
(resp.\,the sixteen elements of ${\bf HLogAR}^2_{\rm asym}$
 equivalent to ${\bf HLog}^2$)  
can be obtained in a natural way from 
${\bf DilogAR}^{(2)} $ (resp.\,from the sixteen logarithmic 2-abelian relations ${\bf LogAR}^{\underline{\epsilon}}\in \boldsymbol{AR}^{2}_C\big( 
 \boldsymbol{\mathcal W}^{GM}_{ \hspace{-0.05cm}
\boldsymbol{\mathcal Y}_5}\big)${\rm )} by means of one single integration. 
\end{enumerate}
\end{thm}

The similarities between the statements above with the corresponding ones  for Bol's web $\boldsymbol{\mathcal B}\simeq
\boldsymbol{\mathcal W}_{  \hspace{-0.05cm} {\rm dP}_5}$ in \S\ref{S:Web-WdP5-properties}   are even more striking than those between the latter statements and those in 
\S\ref{S:Web-WdP4-properties} about $\boldsymbol{\mathcal W}_{  \hspace{-0.05cm} {\rm dP}_4}$.  For this reason, and also because the theorem above can be generalized to all the Gelfand-MacPherson webs 
$ \boldsymbol{\mathcal W}^{GM}_{ \hspace{-0.05cm}
\boldsymbol{\mathcal Y}_r}$ for $r=4,\ldots,8$ (see the next subsection below), 
we believe that these latter webs are those which must really be considered as the most direct/fundamental generalizations of $\boldsymbol{\mathcal B}\simeq
\boldsymbol{\mathcal W}_{  \hspace{-0.05cm} {\rm dP}_5}\simeq 
 \boldsymbol{\mathcal W}^{GM}_{ \hspace{-0.05cm}
\boldsymbol{\mathcal Y}_4}$, and not really the del Pezzo's webs $\boldsymbol{\mathcal W}_{  \hspace{-0.05cm} {\rm dP}_{9-r}}$ which actually are 2-dimensional slices of the Gelfand-MacPherson webs.  All these considerations 
make it natural to ask the following question, which we believe is the one to be considered instead of {\bf Question \ref{Question:tolo}}:
\begin{question}
\label{Question:DilogAR-a-la-GM}
Can the dilogarithmic 2-abelian relation ${\bf DilogAR}^{(2)}$ 
of the web $ \boldsymbol{\mathcal W}^{GM}_{ \hspace{-0.05cm}
\boldsymbol{\mathcal Y}_5}$
be obtained following the geometric approach of  Gelfand-MacPherson,  
by integrating along the orbits of the Cartan torus $H_{D_5}$, 
 of an invariant differential form $\Omega(P)$ a real form ${\bf S}_5$ of the spinor tenfold $\mathbb S_5$ which represents a certain characteristic class $P\in 
{\bf H}^*\big( {\bf S}_5,\mathbf R\big)${\rm ?} 
\end{question}

We plan to dedicate a future paper to the theorem and the question above (see \cite{PirioGMWebs}).

\subsection{\bf Generalizations of the questions above to the  identities ${\bf HLog}^{r-2}$ for $r=4,\ldots,8$}
The questions asked for ${\bf HLog}^{3}$ until subsection \S\ref{SS:Quantum-Version} (included) actually generalize straightforwardly to the identities ${\bf HLog}^{r-2}$ for any $r=4,\ldots,8$. 
\sk

It is not clear for the moment whether the webs $\boldsymbol{\mathcal W}_{{\rm dP}_d}$ are of cluster type for $d\in \{3,2,1\}$ but since they are Fano surfaces, there exists in each case a  $d$-gone with rational components $D\subset {\rm dP}_d$ 
 such that $({\rm dP}_d,D)$ be a log Calabi-Yau pair ({\it i.e.}\,$K_{ {\rm dP}_d}+D$ is trivial).

In a work in progress, we have established that: 
\begin{itemize}
\item  all the  del Pezzo's webs $\boldsymbol{\mathcal W}_{ {\rm dP}_d}$'s can be obtained from 
Gelfand-MacPherson web 
 $ \boldsymbol{\mathcal W}^{GM}_{
{\mathcal Y}_{9-d}
}$ or, in other terms:  there are versions of Proposition \ref{Prop:WdP4-from-WGMS5} for
$\boldsymbol{\mathcal W}_{ \hspace{-0.05cm}{\rm dP}_d}$ for any $d\in \{4,\ldots,8\}$; 
\sk
\item Theorem \ref{THM:WdP5-WGMS5-similarities} generalizes to any Gelfand-MacPherson web  $ \boldsymbol{\mathcal W}^{GM}_{
{\mathcal Y}_{9-d}
}$ for $d=4,\ldots,7$.\footnote{We believe that there exists a version of Theorem \ref{THM:WdP5-WGMS5-similarities} for 
the 2160-web $ \boldsymbol{\mathcal W}^{GM}_{
{\mathcal Y}_{8}
}$ but this is not proved yet.} 
\sk
\end{itemize}
These preliminary but already interesting results make  it  possible to ask a suitable version of Question \ref{Question:DilogAR-a-la-GM} for any  Gelfand-MacPherson web  $ \boldsymbol{\mathcal W}^{GM}_{
\boldsymbol{\mathcal Y}_{9-d}
}$ for $d=4,\ldots,7$ (and possibly for $d=8$ as well). We will study all this more in depth  in our coming work \cite{PirioGMWebs}.

\subsection{Playing with canonical maps}
\label{SS:H-Phi-H(WM0N)}
In addition of its `naturalness', the notion of canonical map of a planar web introduced before in this text (see \S\ref{SSS:Via-Canonical-Map-WdP5}) seems to be relevant with respect to the question of constructing planar webs carrying many abelian relations.  Since hexagonal webs are of maximal rank, it is a natural and rather easy task to look at some new webs which can be constructed from hexagonal webs and their canonical maps. We discuss this for some interesting examples below. 
\sk 

For $k\geq 4$, let $\boldsymbol{\mathcal H}$ be a hexagonal planar $k$-web:  
there exist $k$ pairwise distinct points $p_1,\ldots,p_k$ in $\mathbf P^2$, not necessarily in general position, such that 
$\boldsymbol{\mathcal H}=\boldsymbol{\mathcal W}\big( \mathcal L_{p_1}\, , \, \ldots
\, , \,  \mathcal L_{p_k}\,\big)$, where $\mathcal L_{p_i}$ stands for the pencil of lines through $p_i$ for any $i=1,\ldots,k$.  We define the {\it `canonical extension'} of 
$\boldsymbol{\mathcal H}$ as the planar web denoted by $\overline{\boldsymbol{\mathcal H}}^{can}$, 
obtained by juxtaposing to it the preimage of $\boldsymbol{\mathcal W}_{\hspace{-0.1cm}\boldsymbol{\mathcal M}_{0,k}}$ under its canonical map
$\Phi_{\boldsymbol{\mathcal H}}: \mathbf P^2\dashrightarrow \boldsymbol{\mathcal M}_{0,k}$:  in mathematical symbols, one has
$$\overline{\boldsymbol{\mathcal H}}^{can}=\boldsymbol{\mathcal H}\boxtimes 
\Phi_{\boldsymbol{\mathcal H}}^*\big( \boldsymbol{\mathcal W}_{\hspace{-0.1cm}\boldsymbol{\mathcal M}_{0,k}}\big)\, . $$

The first two examples to have in mind are those with the $p_i$'s in general position for  $k=4$ and $k=5$: the corresponding canonical extensions are (models of) the two del Pezzo's webs studied in this text, namely $\boldsymbol{\mathcal W}_{  \hspace{-0.05cm}{\rm dP}_5}$  and 
$\boldsymbol{\mathcal W}_{  \hspace{-0.05cm}{\rm dP}_4}$ respectively. Since these two webs carry interesting hyperlogarithmic ARs, it is natural to ask about the general case: \sk

\noindent{\bf Questions:} {\it given $k$ points $p_1,\ldots,p_k\in \mathbf P^2$, pairwise distinct but  not necessarily in general position, what can be said about the canonical extension of 
the hexagonal web $\boldsymbol{\mathcal H}=\boldsymbol{\mathcal W}\big( \mathcal L_{p_1}\, , \, \ldots
\, , \,  \mathcal L_{p_k}\,\big)$, in particular regarding its abelian relations ?}\mk 

If our main interest lies in the ARs of $\overline{\boldsymbol{\mathcal H}}^{can}$, even determining its basic invariants already is  interesting. For instance, even determining the degree $\overline{k}^{can}$ of $\overline{\boldsymbol{\mathcal H}}^{can}$ does not seem obvious: 
one expects that $\overline{k}^{can}$ only depends on the combinatorial type 
  of the projective configuration $[p_i]_{i=1}^k \in \boldsymbol{\mathcal M}_{0,k}$ but what might be a closed formula for $\overline{k}^{can}$ is not clear (at least to the author of these lines). 
\sk 

\vspace{-0.4cm}
Here is a sample of a few explicit examples we have considered for some degenerated configurations of $k$ points in the plane for $k=5,6$. One checks that $\Phi_{\boldsymbol{\mathcal H}}$ is constant (hence  ${\boldsymbol{\mathcal H}}
=\overline{\boldsymbol{\mathcal H}}^{can}$) when all the vertices $p_i$'s are aligned 
hence we will not consider this case further. 
\\

  {\bf { $\big[$Five points$\big]$}}: there are four possible combinatorial types for the 
  $p_i$'s, we consider each of\\
  ${}^{}$\hspace{0.3cm}  them separately: 
\begin{itemize}
\item when four of the $p_i$'s are on a same line but not the fifth point, then again 
${\boldsymbol{\mathcal H}}
=\overline{\boldsymbol{\mathcal H}}^{can}$ hence there is not much interesting to add.
\mk 
\item  When there are two distinct sets of 3 aligned points among the $p_i$'s, 
then a normal form for the associated hexagonal 5-web is 
 $\boldsymbol{\mathcal H}=\boldsymbol{\mathcal W}
\big(\, 
x\, , \,  y\, , \,  x - y
\, , \,  \frac{x}{y}\, , \, 
\frac{y-1}{x-1 }\, 
\big)$.  From elementary computations, we get that 
$\overline{\boldsymbol{\mathcal H}}^{can}={\boldsymbol{\mathcal H}}\boxtimes \mathcal F_{\frac{x(y-1)}{y(x-1)}}  = 
\boldsymbol{\mathcal W}
\Big(
x\, , \,  y\, , \,  x - y
\, , \,  \frac{x}{y}\, , \, 
\frac{y-1}{x-1 }\, , 
\, \frac{x(y-1)}{y(x-1) }\, 
\Big)
$
 and that this 6-web has maximal rank (hence is exceptional), with the following invariants: 
$$ {}^{}\qquad 
 {\rm Hex}_\bullet\Big(\,\overline{\boldsymbol{\mathcal H}}^{can}\,\Big)=(16,9,2,0) 
\qquad \mbox{ and }\qquad 
{\rm Flat}_\bullet\Big(\,\overline{\boldsymbol{\mathcal H}}^{can}\,\Big)=(16,11,2,1)\,.$$
\item  If exactly three of the $p_i$'s are aligned, then 
one can assume that 
$\boldsymbol{\mathcal H}=\boldsymbol{\mathcal W}
\Big(\, 
x\, , \,  y\, , \,  x - y
\, , \,  \frac{x}{y}\, , \, 
\frac{y-b}{x-a }\, 
\Big)$ with $ab(a-b)\neq 0$.  
Elementary computations give us that 
$$\overline{\boldsymbol{\mathcal H}}^{can}={\boldsymbol{\mathcal H}}\boxtimes
\boldsymbol{\mathcal W}\bigg( \, 
\frac{
 y(y-x + a  - b)}{(y - b)(x - y)}
\, , \, 
\frac{x(y-x+a-b)}{(x-a)(x-y)}
\, , \,  
\frac{x(y-b)}{y(x-a)} 
\,\bigg)
$$
and that this 8-web is of maximal rank (hence is exceptional) and has the following numerical invariants: 
$$ {}^{}\qquad {\rm Hex}_\bullet\Big(\,\overline{\boldsymbol{\mathcal H}}^{can}\,\Big)=(
38, 28, 8, 0, 0, 0)
\quad \mbox{ and }\quad 
{\rm Flat}_\bullet\Big(\,\overline{\boldsymbol{\mathcal H}}^{can}\,\Big)=(
38, 34, 8, 4, 0, 1)\,.$$
\item  Finally, when the $p_i$'s are in general position, then one can take 
$\boldsymbol{\mathcal H}=\boldsymbol{\mathcal W}
\Big(\, 
x\, , \,  y\, , \,   \frac{x}{y}
\, , \,  \frac{y-1}{x-1}\, , \, 
\frac{y-b}{x-a }\, 
\Big)$ with $ab(a-1)(b-1)(a-b)\neq 0$. Hence one has $\overline{\boldsymbol{\mathcal H}}=\boldsymbol{\mathcal W}_{  \hspace{-0.05cm}{\rm dP}_4}$ where ${\rm dP}_4$ is the del Pezzo quartic surface obtained as the blow-up of $\mathbf P^2$ at the $p_i$'s. 
Moreover, one has 
$$ {}^{}\qquad {\rm Hex}_\bullet\Big(\,\boldsymbol{\mathcal W}_{  \hspace{-0.05cm}{\rm dP}_4}\,\Big)=(80, 80, 32, 0, 0, 0, 0, 0
)
\quad \mbox{ and }\quad 
{\rm Flat}_\bullet\Big(\,\boldsymbol{\mathcal W}_{  \hspace{-0.05cm}{\rm dP}_4}\,\Big)=(80, 90, 32, 10, 0, 5, 0, 1
)\,.$$
\sk
\end{itemize}

\vspace{-0.3cm}
  {\bf { $\big[$Six points$\big]$}}: there are eight possible combinatorial types for the 
  $p_i$'s (see \cite[p.\,141]{GS}). Several  \\
  ${}^{}$\hspace{0.3cm} 
computations lead us to believe that in essentially any case, the canonical extension is not flat (hence not of maximal rank), except when the combinatorial type of the $p_i$'s is of a specific type.
\begin{itemize}
\item  Assuming that the $p_i$'s are the vertices of the non Fano arrangement 
({\it eg.}\,see Figure 2 in \,\cite{Hlog}), one can assume that 
$\boldsymbol{\mathcal H}=\boldsymbol{\mathcal W}\big(\, 
x\, , \, y\, , \,  \frac{x}{y}\, , \,  \frac{y-1}{x-1}\, , \,   \frac{y-1}{x}\, , \,  \frac{y}{x-1}  
\,\big)$ in which case one gets
$$
\overline{\boldsymbol{\mathcal H}}^{can}=
{\boldsymbol{\mathcal H}}\boxtimes 
\boldsymbol{\mathcal W}\bigg( \, 
\frac{x(y - 1)}{y(x - 1)}
\, , \, 
\frac{(x - 1)(y - 1)}{xy}
\, , \,  
\frac{y(y - 1)}{x(x - 1)}\,\bigg)\, .
$$
It can be verified ({\it cf.}\,\cite[\S4.3]{Hlog}) that the 9-web $\overline{\boldsymbol{\mathcal H}}^{can}$ is a model of the well-known trilogarithmic Spence-Kummer's web  
$\boldsymbol{\mathcal W}_{\hspace{-0.05cm}{\mathcal S}{\mathcal K}}$ 
which is exceptional and carries two linearly independant trilogarithmic abelian relations (see \cite[\S2.2.3.1]{ClusterWebs} for more details). 
\sk 
\end{itemize}

The consideration of the preceding examples indicates (in our opinion) that the study 
of the configurations of points in the plane giving hexagonal webs ${\boldsymbol{\mathcal H}}$
such that $\overline{\boldsymbol{\mathcal H}}^{can}$ be of maximal rank is interesting and deserves a further study.

\newpage

\section*{\bf Appendix: some representation-theoretic computations in GAP}
\label{Appendix}
In this appendix we briefly explain the computations we used to establish some representation-theoretic facts 
stated above in the text.
These computations were 
mainly performed using the software GAP3 \cite{GAP3} which has the advantage of being available 
online.\footnote{One can run GAP3 
by typing the command {\ttfamily  /ext/bin/gap3} in a \href{https://cocalc.com/features/terminal}{CoCalc terminal window}.}

\subsection*{The set-up}
First, to define the Weyl group we will work with 
and return the corresponding Dynkin diagram within GAP, we run the following commands: 
\begin{lstlisting}[language=GAP]
   gap> WD5:=CoxeterGroup("D",5);
   gap> PrintDiagram( WD5 );
\end{lstlisting}
The last command returns the following diagram : 
\begin{equation}
\label{Eq:Diag-Gap3-WD5}
  \xymatrix@R=0.1cm@C=0.4cm{ 
  1\ar@{-}[rd]   & &   &
 \\
  & 3  \ar@{-}[r] & 4   \ar@{-}[r]  & 5\, . \\
   2\ar@{-}[ru]   & &   &
  }
\end{equation}

To define and get the corresponding characters table, we type
\begin{lstlisting}[language=GAP]
   gap> TableWD5:=CharTable(WD5);
   gap> Display(TableWD5);
\end{lstlisting}
From the outcome of the last command, one can extract the  character table \ref{Table:CharTableWD5} (in which the dot $\cdot$ is used instead of $0$ to make its reading easier). 

\begin{table}[!h]
\resizebox{5.8in}{1.5in}
{
\begin{tabular}{|l||c|c|c|c|c|c|c|c|c|c|c|c|c|c|c|c|c|c|}
\hline
         & \scalebox{0.8}{$\boldsymbol{\big(1^5.\big)}$} &  \scalebox{0.8}{$\boldsymbol{(1^3.1^2)}$}  & 
          \scalebox{0.8}{$\boldsymbol{\big(1.1^4\big)}$ } & 
       \scalebox{0.8}{   $\boldsymbol{\big(21^3.\big)}$ } & 
       \scalebox{0.8}{   $\boldsymbol{\big(1^2.21\big)}$} & 
       \scalebox{0.8}{   $\boldsymbol{\big(21.1^2\big)}$} & 
       \scalebox{0.8}{   $\boldsymbol{\big(.21^3\big)}$ } & 
      \scalebox{0.8}{    $\boldsymbol{\big(221.\big)}$}  &
      \scalebox{0.8}{    $ \boldsymbol{\big(1.22\big)}$ } &  
     \scalebox{0.8}{     $\boldsymbol{\big(2.21\big)}$} &  
     \scalebox{0.8}{     $\boldsymbol{\big(311.\big)}$ }  & 
     \scalebox{0.8}{     $\boldsymbol{\big(1.31\big)}$} & 
      \scalebox{0.8}{    $\boldsymbol{\big(3.11\big)}$}  & 
     \scalebox{0.8}{     $\boldsymbol{\big(32.\big)}$}  & 
    \scalebox{0.8}{      $\boldsymbol{\big(.32\big)}$} & 
    \scalebox{0.8}{      $\boldsymbol{\big(41.\big)}$} & 
    \scalebox{0.8}{      $\boldsymbol{\big(.41\big)}$} &   
    \scalebox{0.8}{      $ \boldsymbol{\big(5.\big)}$}
  \\
      \hline \hline 
$\boldsymbol{[1^2.1^3]}$     &   10   &  -2   &   2  &   -4   &   2   &   .  &   -2  &    2  &   -2&  .    &  1 &    -1  &    1  &  -1  &   1  &  .  &  .   &   . \\
 \hline
$\boldsymbol{[1.1^4]}$    &     5    &  1  &   -3  &   -3   &  -1   &   1   &   3   &   1   &   1&   -1    &  2  &    .     &-2   &  .  &   .  & -1 &   1   &   .   \\ \hline
$\boldsymbol{[.1^5 ]}$   &     1  &    1  &    1   &  -1  &   -1  &   -1  &   -1   &   1  &    1 & 1 &     1   &   1   &   1    &-1   & -1  & -1  & -1   &   1 \\ \hline
$\boldsymbol{[1^3.2]}$    &     10 &    -2   &   2   &  -2   &   .    &  2 &    -4  &   -2  &    2  &  .  &    1  &   -1  &    1   &  1  &  -1  &  .  &  .  &    .
\\ \hline
$\boldsymbol{[1^2.21]}$    &     20   &  -4    &  4   &  -2    &  2  &   -2    &  2   &   .   &   .  &  .   &  -1   &   1  &   -1  &   1  &  -1  &  .  &  .  &    .
\\ \hline
$\boldsymbol{[1.21^2]}$    &     15  &    3  &   -9    & -3    & -1   &   1    &  3  &   -1   &  -1  & 1  &    .   &   .   &   .   &  . &    .  &  1 &  -1  &    .
\\ \hline
$\boldsymbol{\big[.21^3\big]}$    &      4  &    4   &   4  &   -2  &   -2  &   -2  &   -2 &     .   &   .   &   .   &   1    &  1      & 1   &  1   &  1  &  . &   .    & -1 
\\ \hline
$\boldsymbol{\big[1.2^2\big]}$      &    10  &    2   &  -6   &   .   &   .     & .  &    .   &   2    &  2   &   -2 &    -2   &   . &     2   &  . &    . &   . &   . &     .
\\ \hline
$\boldsymbol{\big[2.21\big]}$      &    20  &   -4  &    4   &   2  &   -2  &    2   &  -2 &     .  &    .  &  .   &  -1   &   1   &  -1  &  -1    & 1  &  .  &  .   &   . \\
 \hline
$\boldsymbol{\big[.2^21\big]}$    &       5  &    5    &  5  &   -1  &   -1   &  -1  &   -1   &   1   &   1 &    1  &   -1  &   -1   &  -1  &  -1   & -1  &  1  &  1  &    . \\
 \hline
$\boldsymbol{\big[1^2.3\big]}$   &      10  &   -2  &    2   &   2   &   .  &   -2  &    4  &   -2  &    2&    .  &    1   &  -1   &   1 &   -1  &   1  &  . &   .  &    .
\\
 \hline
$\boldsymbol{\big[1.31\big]}$     &     15   &   3   &  -9   &   3   &   1  &   -1  &   -3   &  -1  &   -1&1   &   .   &   .    &  .    & .   &  . &  -1 &   1  &    . 
\\
 \hline
$\boldsymbol{\big[.31^2\big]}$      &     6  &    6  &    6   &   .    &  .    &  .   &   .  &   -2   &  -2&    -2   &   .   &   .    &  .  &   .  &   .   & . &   .  &    1
\\
 \hline
$\boldsymbol{\big[2.3\big]}$       &    10  &   -2  &    2   &   4   &  -2  &    .  &    2  &    2   &  -2&  .   &   1  &   -1   &   1   &  1 &   -1   & .  &  .   &   .
\\
 \hline
$\boldsymbol{\big[.32\big]}$    &        5   &   5  &    5   &   1   &   1    &  1   &   1   &   1  &    1&      1   &  -1  &   -1 &    -1  &   1  &   1 &  -1  & -1  &    .
\\ \hline
$\boldsymbol{\big[1.4\big]}$     &       5   &   1  &   -3    &  3   &   1  &   -1   &  -3   &   1  &    1&  -1   &   2   &   .     &-2  &   .  &   .  &  1 &  -1  &    .
\\ \hline
$\boldsymbol{\big[.41\big]}$   &         4  &    4  &    4    &  2    &  2   &   2  &    2  &    .    &  .&  .    &  1  &    1   &   1  &  -1  &  -1  &  .  &  .  &   -1
\\ \hline
$\boldsymbol{\big[.5\big]}$     &        1   &   1   &   1    &  1   &   1     & 1     & 1     & 1   &   1&
 1   &   1  &    1  &    1 &    1     &1   & 1  &  1    &  1
 \\ \hline 
\end{tabular}}
\bk
\caption{Character table of the Weyl group  of type $D_5$.}
\label{Table:CharTableWD5}
\end{table}

A few explanations are in order about this table. As is well-known ({\it e.g.}\,see \cite[\S5.6]{GP}), for any odd integer $n\geq 4$,\footnote{The case when $n$ is even is a bit more involved but fully understood as well.} the characters as well as the conjugacy classes of the Weyl group of type $D_n$ can be labelled by means of the `bipartitions  of $n$', that is pairs $(\lambda,\lambda')$ of partitions such that 
$\lvert \lambda\lvert+\lvert \lambda'\lvert=n$. 
In Table \ref{Table:CharTableWD5},  we use  
$[\lambda.\lambda']$ for labelling the line associated to the corresponding character  while 
$(\lambda.\lambda')$ is used to label the column associated to the corresponding conjugacy class. When $\lambda $ is the empty partition, we just write $[.\lambda']$  and $(.\lambda')$ respectively and similarly when $\lambda'=\emptyset$. 
\sk

\vspace{-0.4cm}
The roots of the root space of type $D_5$ associated to a del Pezzo quartic surface we considered in \S\ref{SSS} are the $\boldsymbol{\rho}_i$'s for $i=1,\ldots,5$, defined as follows: one has 
$\boldsymbol{\rho}_i=\boldsymbol{e}_i-\boldsymbol{e}_{i+1}$ for $i=1,\ldots,4$ and $\boldsymbol{\rho}_5=\boldsymbol{h}-\boldsymbol{e}_1-\boldsymbol{e}_2-\boldsymbol{e}_3$. 
One deduces that the associated Dynkin diagram is the following 
\begin{equation}
\label{Eq:Diag-dP-WD5}
  \xymatrix@R=0.1cm@C=0.4cm{ 
     & &   & \boldsymbol{\rho}_4
 \\
  \boldsymbol{\rho}_1 \ar@{-}[r] & \boldsymbol{\rho}_2  \ar@{-}[r] & \boldsymbol{\rho}_3  \ar@{-}[ru]  \ar@{-}[rd] &   \\
    & &   & \boldsymbol{\rho}_5
  }
\end{equation}


Let $\sigma_1,\ldots,\sigma_5$ be the generators of {\ttfamily  WD5} as encoded in GAP3, with $\sigma_i$ corresponding to the vertex labeled by $i$ in \eqref{Eq:Diag-Gap3-WD5}. We recall that 
the generators of $W=\langle s_1,\ldots,s_5\rangle$ are $s_1,\ldots,s_5$ where 
$s_i=s_{\rho_i}$ for any $i$ (cf.\,\S\ref{SSS}). 
Comparing \eqref{Eq:Diag-dP-WD5} with
\eqref{Eq:Diag-Gap3-WD5}, it comes that the isomorphism between {\ttfamily  WD5} and $W$ that we are considering is induced by the following relations between the $\sigma_i$'s and the $s_j$'s: 
\begin{equation}
\label{Eq:sigma-i-s-i}
\sigma_1=s_4\, , \quad 
\sigma_2=s_5\, , \quad 
\sigma_3=s_3\, , \quad 
\sigma_4=s_2\, , \quad
\sigma_5=s_1\, . 
\end{equation}

For $\boldsymbol{\mathfrak c}_1=\boldsymbol{h}-\boldsymbol{e}_1\in 
=\boldsymbol{\mathcal K}$, we label as follows the elements of 
$\boldsymbol{\mathfrak c}_{1}^{red}$: 
$$ 
\boldsymbol{c}_{1,2}= \boldsymbol{e}_2+ \big(\boldsymbol{h} -\boldsymbol{ e}_1 - \boldsymbol{e}_2\big)
\quad 
\boldsymbol{c}_{1,3}=
 \boldsymbol{e}_3 +\big(\boldsymbol{h} -\boldsymbol{ e}_1 - \boldsymbol{e}_3\big)
 \quad 
\boldsymbol{c}_{1,4}=
 \boldsymbol{e}_4+ \big(\boldsymbol{h} -\boldsymbol{ e}_1 - \boldsymbol{e}_4\big)
  \quad 
\boldsymbol{c}_{1,5}=
   \boldsymbol{e}_5+\big(\boldsymbol{h} -\boldsymbol{ e}_1 - \boldsymbol{e}_5\big) 
$$
(recall our convention: here $\boldsymbol{e}_i+ \big(\boldsymbol{h} - \boldsymbol{e}_1 - \boldsymbol{e}_i\big)$ stands for the non irreducible conic on the (fixed) del Pezzo quartic surface we are working with, whose two irreducible components are the two lines $\boldsymbol{e}_i$ and $\boldsymbol{h}-\boldsymbol{e}_1-\boldsymbol{e}_i$, this for any $i=2,\ldots,5$. 

The stabilizer $W_{\boldsymbol{\mathfrak c}_1}$ of $\boldsymbol{\mathfrak c}_1$ in $W$ is the subgroup spanned by the $s_j$'s for $j=2,\ldots,5$. It is a Weyl group of type $D_4$: 
$$
W_{\boldsymbol{\mathfrak c}_1}=\big\langle  s_2,\ldots,s_5 \big\rangle \simeq W(D_4)\, .
$$
From \eqref{Eq:sigma-i-s-i}, it comes that  the corresponding subgroup   in GAP can be defined  by means of the following command
\begin{lstlisting}[language=GAP]
   gap> Wc1 := ReflectionSubgroup( WD5, [1,2,3,4] );
\end{lstlisting}
and we get the characters table of $W_{\boldsymbol{\mathfrak c}_1}$ (see Table \ref{Table:Wc1} below) 
using the GAP command
\begin{lstlisting}[language=GAP]
   gap> TableWc1:=CharTable(Wc1); Display(TableWc1);
\end{lstlisting}

\begin{table}[!h]
\resizebox{5.8in}{1.5in}
{
\begin{tabular}{|c||c|c|c|c|c|c|c|c|c|c|c|c|c|}
\hline 
 & $\boldsymbol{1111.}$ & $\boldsymbol{11.11}$ & $\boldsymbol{.1111}$ & $\boldsymbol{211.}$ & $\boldsymbol{1.21}$ & $\boldsymbol{2.11}$ & $\boldsymbol{22.+}$ & $\boldsymbol{22.-}$  & $\boldsymbol{.22}$  & $\boldsymbol{31.}$ &  $\boldsymbol{.31}$&
 $\boldsymbol{4.+}$ & $\boldsymbol{4.-}$ \\
 \hline \hline
$\boldsymbol{11^+}$ &3& -1& 3& -1& 1& -1& 3& -1& -1& 0& 0& -1& 1  \\ 
$\boldsymbol{11^-}$  &   3& -1& 3& -1& 1& -1& -1& 3& -1& 0& 0& 1& -1  \\ 
$\boldsymbol{1.111}$  &   4& 0& -4& -2& 0& 2& 0& 0& 0& 1& -1& 0& 0  \\ 
$\boldsymbol{.1111}$   &  1& 1& 1& -1& -1& -1& 1& 1& 1& 1& 1& -1& -1  \\ 
$\boldsymbol{11.2}$   &  6& -2& 6& 0& 0& 0& -2& -2& 2& 0& 0& 0& 0  \\ 
$\boldsymbol{1.21}$   &  8& 0& -8& 0& 0& 0& 0& 0& 0& -1& 1& 0& 0  \\ 
$\boldsymbol{ .211}$  &  3& 3& 3& -1& -1& -1& -1& -1& -1& 0& 0& 1& 1  \\ 
 $\boldsymbol{2^+}$  &  3& -1& 3& 1& -1& 1& 3& -1& -1& 0& 0& 1& -1  \\ 
$\boldsymbol{2^-}$   &  3& -1& 3& 1& -1& 1& -1& 3& -1& 0& 0& -1& 1  \\ 
$\boldsymbol{.22}$   &  2& 2& 2& 0& 0& 0& 2& 2& 2& -1& -1& 0& 0  \\ 
 $\boldsymbol{1.3}$  &  4& 0& -4& 2& 0& -2& 0& 0& 0& 1& -1& 0& 0  \\ 
$\boldsymbol{.31}$   &  3& 3& 3& 1& 1& 1& -1& -1& -1& 0& 0& -1& -1  \\ 
$\boldsymbol{.4}$   &  1& 1& 1& 1& 1& 1& 1& 1& 1& 1& 1& 1& 1 \\ \hline
\end{tabular}}
\bk
\caption{Character table of the Weyl subgroup $W_{\boldsymbol{\mathfrak c}_1}$ (of type $D_4$).}
\label{Table:Wc1}
\end{table}

Our goal is now to determine the character, denoted by $\chi_{\boldsymbol{\mathfrak c}_{1}}$, of the $W_{\boldsymbol{\mathfrak c}_1}$-representation 
$\mathbf Z^{\boldsymbol{\mathfrak c}_{1}^{red}}$. To this end, one first computes the permutation matrix $M_i$ of 
$s_i$ viewed as an endomorphism of  $\mathbf Z^{\boldsymbol{\mathfrak c}_{1}^{red}}$ when 
expressed in the basis $\big(\boldsymbol{c}_{1,i}\big)_{i=2}^5$. Straightforward computations give us the following formulas: 
$$ M_2=\left[\begin{array}{cccc}
0 & 1 & 0 & 0 
\\
 1 & 0 & 0 & 0 
\\
 0 & 0 & 1 & 0 
\\
 0 & 0 & 0 & 1 
\end{array}\right], 
\hspace{0.3cm}
 M_3=
\left[\begin{array}{cccc}
1 & 0 & 0 & 0 
\\
 0 & 0 & 1 & 0 
\\
 0 & 1 & 0 & 0 
\\
 0 & 0 & 0 & 1 
\end{array}\right],
\hspace{0.3cm}
 M_4=
\left[\begin{array}{cccc}
1 & 0 & 0 & 0 
\\
 0 & 1 & 0 & 0 
\\
 0 & 0 & 0 & 1 
\\
 0 & 0 & 1 & 0 
\end{array}\right], 
\hspace{0.3cm}
 M_5=
\left[\begin{array}{cccc}
0 & 1 & 0 & 0 
\\
 1 & 0 & 0 & 0 
\\
 0 & 0 & 1 & 0 
\\
 0 & 0 & 0 & 1 
\end{array}\right]\,. $$

In view of determining $\chi_{\boldsymbol{\mathfrak c}_{1}}=
\chi_{\mathbf C^{\boldsymbol{\mathfrak c}_{1}^{red}}}
$, we need to construct explicit representatives of the conjugacy classes in $W_{\boldsymbol{\mathfrak c}_1}$. This can be achieved in GAP3 by means of the following command: 
\begin{lstlisting}[language=GAP]
   gap> List(ConjugacyClasses(Wc1),x->CoxeterWord(Wc1,Representative(x)));
\end{lstlisting}
%
%
which returns the following string of 13 sequences: 
{\begin{lstlisting}[language=GAP]
     [ [], [1, 2], [1, 2, 3, 1, 2, 3, 4, 3, 1, 2, 3, 4], [1], [1, 2, 3], 
     [1, 2, 4], [1, 4], [2, 4], [1, 3, 1, 2, 3, 4], [1, 3], [1, 2, 3, 4], 
     [1, 4, 3], [2, 4, 3] ]
\end{lstlisting}}

The $k$-th element of the above list corresponds to an explicit representative of the conjugacy class associated to the $k$-th column of Table \ref{Table:CharTableWD5} in terms of the generators of {\ttfamily  Wc1} the software GAP is dealing with. More precisely, if this string is $[i_1,\ldots,i_m]$, the  corresponding representative is $\sigma_{i_1}\cdots \sigma_{i_m}$, where $\sigma_1,\ldots,\sigma_5$ stand for the generators of  {\ttfamily  WD5} implicitly defined when the latter group was defined.
More explicitly, let $\nu$ be the permutation such that the relations \eqref{Eq:sigma-i-s-i} are equivalent to $\sigma_i=s_{\nu(i)}$ for $i=1,\ldots,5$ (namely one has $\nu=(1425)$). For $[i_1,\ldots,i_m]$ in the list above, 
the value taken by the character $\chi_{\boldsymbol{\mathfrak c}_{1}}$ when evaluated on the corresponding conjugacy class 
is given by the trace 
$${\rm Tr}\Big( M_{\nu(i_1)}\cdots M_{\nu(i_m)}
\Big)\in \mathbf Z\, . $$

From the material above, it is then just a (straightforward but lengthy) computational matter to determine 
$\chi_{ {\boldsymbol{\mathfrak c}_1} }$ explicitly: one gets 
$\chi_{ {\boldsymbol{\mathfrak c}_1} }={\bf 1}
+ \chi_{ {\bf H}_{{\boldsymbol{\mathfrak c}_1}} }$
where ${\bf 1}$ denotes  the trivial character and where 
$\chi_{ {\bf H}_{{\boldsymbol{\mathfrak c}_1}} }$
 is given by 
$$
 \chi_{ {\bf H}_{{\boldsymbol{\mathfrak c}_1}} } = 
(3, -1, 3, 1, -1, 1, -1, 3, -1, 0, 0, -1, 1)\ . 
$$ 
%
%
Then using Table \ref{Table:Wc1}, one finally gets that as $W_{\boldsymbol{\mathfrak c}_1}$-representation, one has 
\begin{equation}
\label{Eq:Hc1}
\mathbf Z^{  \boldsymbol{\mathfrak c}_1^{red}}= {\bf 1}\oplus {\bf H}_{ \boldsymbol{\mathfrak c}_1} 
\qquad \mbox{ with } \qquad {\bf H}_{ \boldsymbol{\mathfrak c}_1}\simeq 
V^3_{[2.2]^-}\, .
\end{equation}

Because one obviously has $\oplus_{ \boldsymbol{\mathfrak c}\in \boldsymbol{\mathcal K} } {\bf H}_{ \boldsymbol{\mathfrak c}}\simeq {\rm Ind}_{W_{\boldsymbol{\mathfrak c}_1}}^W\big(  {\bf H}_{ \boldsymbol{\mathfrak c}_1}\big)$ as $W$-representations, one can use 
\eqref{Eq:Hc1} to determine the decomposition of $\oplus_{ \boldsymbol{\mathfrak c}\in \boldsymbol{\mathcal K} } {\bf H}_{ \mathfrak c}$ in irreducibles. For that purpose, one only needs the induction table from $W_{\boldsymbol{\mathfrak c}_1}$ to $W$ which can be obtained via the following GAP command: 
\begin{lstlisting}[language=GAP]
   gap>Display( InductionTable(  Wc1, WD5) );
\end{lstlisting}

We get the table below which has to be used as follows:  each column corresponds to 
the $W$-induction of the $W_{\boldsymbol{\mathfrak c}_1}$-irrep labeled by the bipartition of 4 given in the top entry of the column.  
Denoting by ${[\lambda.\lambda']}$ this bipartition, the coefficients of the column correspond to the multiplicity of the $W$-irreducibles appearing in the decomposition of the induction of $V_{[\lambda.\lambda']}$ from $W_{\boldsymbol{\mathfrak c}_1}$ to $W$.
\begin{table}[!h]
\resizebox{5.8in}{1.9in}
{
\begin{tabular}{|l|c|c|c|c|c|c|c|c|c|c|c|c|c|}
\hline
       & $\boldsymbol{[11]^+}$ &  $\boldsymbol{[11]^-}$  & $\boldsymbol{[1.1^3]}$ & $ \boldsymbol{[.1111]}$ &$\boldsymbol{[11.2]}$ &$\boldsymbol{[1.21]}$ & $\boldsymbol{[.211]}$ & $\boldsymbol{[2.2]^+}$ & 
       $\boldsymbol{[2.2]^-}$ & $\boldsymbol{[.22]}$  & $\boldsymbol{[1.3]}$ & $\boldsymbol{[.31]}$ & $\boldsymbol{[.4]}$
 \\ \hline \hline 
$\boldsymbol{[1^2.1^3]}$ &  1 &  1  &   1    &   &   &   & & &  &  &  & & 
 \\ \hline
$\boldsymbol{[1.1^4]}$ & &  &     1  &   1   &   &   & & &  &  &  & & 
 \\ \hline
$\boldsymbol{[.1^5]}$ & &  &    &     1   &   &   & & &  &  &  & & 
 \\ \hline
$\boldsymbol{[1^3.2]}$ &  &  &     1    &    & 1   &   & & &  &  &  & & 
 \\ \hline
$\boldsymbol{[1^2.21]}$ &   1 &  1    &  &   &    1  &  1   & & &  &  &  & & 
 \\ \hline
$\boldsymbol{[1.21^2]}$  & &  &     1    &   &  &   1  &  1 & &  &  &  & & 
 \\ \hline
$\boldsymbol{[.21^3]}$  & &  &    &     1   &    & &    1 & &  &  &  & & 
 \\ \hline
$\boldsymbol{[1.2^2]}$   & &  &    &    &   &    1   & &  & &   1  &  & & 
 \\ \hline
$\boldsymbol{[2.21]}$   & &  &    &    &    1  &  1   & &  1 & 1  &  &  & & 
 \\ \hline
$\boldsymbol{[.2^21]}$   & &  &    &    &   &   &    1 &  & &   1  &  & & 
 \\ \hline
$\boldsymbol{[1^2.3]}$   & &  &    &    &    1   &   & & &  & &   1  & & 
 \\ \hline
$\boldsymbol{[1.31]}$   & &  &    &    &   &    1   & & &  & &   1    & 1 &  
 \\ \hline
$\boldsymbol{[.31^2]}$   & &  &    &    &   &   &    1 & &  &  &  &   1 & 
 \\ \hline
$\boldsymbol{[2.3]}$   & &  &    &    &   &   &   &  1  & 1  &  &  1  & & 
 \\ \hline
$\boldsymbol{[.32]}$    & &  &    &    &   &   &   & & &   1  &   & 1 &  
 \\ \hline
$\boldsymbol{[1.4 ]}$   & &  &    &    &   &   &   & & &  &   1  & &   1 
 \\ \hline
$\boldsymbol{[.41]}$    & &  &    &    &   &   &   & & &  &  & 1   &  1  
\\ \hline
$\boldsymbol{[.5]}$     & &  &    &    &   &   &   & & &  &  &  &  1 
\\ \hline
\end{tabular}}
\bk
\caption{Induction table from $W_{\boldsymbol{\mathfrak c}_1}$ to $W$ (of type $D_4$ and $D_5$ respectively).}
\label{Table:Ind-WD4-WD5}
\end{table}

\vspace{-0.3cm}
Because $H_{ \boldsymbol{\mathfrak c}_1}$ is isomorphic to the  $W_{\boldsymbol{\mathfrak c}_1}$-representation $V_{[2.2]^-}$, the above table gives us that 
$$
\oplus_{ \boldsymbol{\mathfrak c}\in \boldsymbol{\mathcal K} } {\bf H}_{ \boldsymbol{\mathfrak c}}\simeq {\rm Ind}_{W_{\boldsymbol{\mathfrak c}_1}}^W\big(  V^3_{[22]^-}\big)\simeq V_{[2.21]}^{20}\oplus V_{[2.3]}^{10}\, .
$$

\newpage

 \subsubsection*{\bf The weight 2 hyperlogarithmic ARs of ${\mathcal W} \hspace{-0.46cm}{\mathcal W}_{ {\rm dP}_4
 \hspace{-0.4cm}
  {\rm dP}_4}$}
One has  ${\bf H}_{{\rm dP}_4}\simeq V_{[2,3]}^{10}$ (cf.\,\eqref{Eq:HLogAR1-W-structure}) hence the associated $W$-character $\chi_{{\bf H}_{{\rm dP}_4}}$ as well as  $\wedge^2 \chi_{{\bf H}_{{\rm dP}_4}}$ 
can be defined and the latter can be determined using the following commands in GAP3: 
\begin{lstlisting}[language=GAP]
  gap>chi_HdP4:=TableWD5.irreducibles[14];
  gap>Wedge2_chi_HdP4:=AntiSymmetricParts(TableWD5,[chi_HdP4],2)[1]; 
  [ 45, -3, -3, 3, 3, -5, 3, -3, 1, 1, 0, 0, 0, 0, 0, -1, 1, 0 ]
\end{lstlisting}

The multiplicities appearing in the decomposition of $\wedge^2 {{\bf H}_{{\rm dP}_4}}$ in irreducibles are obtained by typing  
\begin{lstlisting}[language=GAP]
    gap>List(TableWD5.irreducibles,t->ScalarProduct(TableWD5,t,Wedge2_chi_HdP4[1])); 
  \end{lstlisting}
  which returns the string  
  \begin{lstlisting}[language=GAP]
    [ 0, 0, 0, 0, 1, 0, 0, 0, 0, 0, 1, 1, 0, 0, 0, 0, 0, 0 ]
  \end{lstlisting}
  from which one deduces that,  
%
 as $W$-modules, one has
\begin{equation}
\label{rrr1}
\wedge^2  {{\bf H}_{{\rm dP}_4}}\simeq V_{[11.21]}^{20}\oplus 
V_{[11.3]}^{10}\oplus 
V_{[1.31]}^{15}\,.
\end{equation}

Working similarly, one gets  $
\wedge^2 
 {{\bf H}_{\boldsymbol{\mathfrak c}_1}}\simeq V_{[1.1]^-}^{3}
$ as $W_{ \boldsymbol{\mathfrak c}_1}$-modules. Using Table\ref{Table:Ind-WD4-WD5}, one deduces that as $W$-modules, one has
\begin{equation}
\label{rrr2}
\oplus_{\boldsymbol{ \mathfrak c} \in \boldsymbol{\mathcal K} } \wedge^2 {\bf H}_{\boldsymbol{\mathfrak c}}\simeq 
 {\rm Ind}_{W_{\boldsymbol{\mathfrak c}_1}}^W\big( \wedge^2 {\bf H}_{\boldsymbol{\mathfrak c}_1} \big) 
 \simeq V^{10}_{[11.111]}\oplus V_{[11.21]}^{20}\, .
\end{equation}
 Because the map $\oplus_{ \boldsymbol{\mathfrak c} \in \boldsymbol{\mathcal K} } \wedge^2 {\bf H}_{\boldsymbol{\mathfrak c} }\rightarrow \wedge^2  {{\bf H}_{{\rm dP}_4}}$ is non zero, it follows from 
 \eqref{rrr1} and \eqref{rrr2}  that as a $W$-representation, one has necessarily
$${\bf HLogAR}^2_{\rm asym} \simeq 
V^{10}_{[11.111]}
\, .$$
 
 We now consider the representations involved in the description of ${\bf HLogAR}^2_{\rm sym}$ as a $W$-module. First we construct the character of the second symmetric product of ${\bf H}_{{\rm dP}_4}$, then determine 
 its decomposition in irreducibles $W$-representations by typing
\begin{lstlisting}[language=GAP]
  gap> Sym2_chi_HdP4:=SymmetricParts(TableWD5,[chi_HdP4],2)[1]; 
  gap> ScalarProduct(TableWD5,t,Sym2_chi_HdP4));
 \end{lstlisting}
 The last command returns the following string
\begin{lstlisting}[language=GAP]
  [ 0, 0, 0, 0, 0, 0, 0, 1, 1, 0, 0, 0, 0, 1, 1, 1, 1, 1 ]
 \end{lstlisting}
  from which we deduce (using  Schur orthogonality relations) the decomposition we are looking for: 
$$
\big(  {{\bf H}_{{\rm dP}_4}}\big)^{\odot 2}\simeq 
{\bf 1} \oplus 
V_{[.41]}^{4}\oplus 
V_{[1.4]}^{5}\oplus 
V_{[.32]}^{5}\oplus 
V_{[2.3]}^{10}\oplus 
V_{[1.22]}^{10}\oplus 
V_{[2.21]}^{20}
\,.
$$

We proceed similarly as in the antisymmetric case for determining the decomposition of 
$\oplus_{ \mathfrak c \in \boldsymbol{\mathcal K} }
\big(  {{\bf H}_{{\boldsymbol{\mathfrak c}}}}\big)^{\odot 2}$.  First we get  the 
decomposition $\big(   {{\bf H}_{\boldsymbol{\mathfrak c}_1}} \big)^{\odot 2}$ as a $W_{\boldsymbol{\mathfrak c}_1}$-representation by means of the following commands: 
\begin{lstlisting}[language=GAP]
  gap> chi_Hc1:=TableWc1.irreducibles[9];
  gap> Sym2_chi_Hc1:=SymmetricParts(TableWc1,[chi_Hc1],2)[1]:
  gap> List(TableWc1.irreducibles,t->ScalarProduct(TableWc1,t,Sym2_chi_Hc1));
            [ 0, 0, 0, 0, 0, 0, 0, 0, 1, 1, 0, 0, 1 ]
\end{lstlisting}
Hence  we obtain that 
 $
\big(   {{\bf H}_{\boldsymbol{\mathfrak c}_1}} \big)^{\odot 2}\simeq 
V_{[2.2]^-}^{3}\oplus 
V_{[.22]}^{2}\oplus 
{\bf 1}$ as a $W_{\boldsymbol{\mathfrak c}_1}$-representation. Then as explained above in \eqref{Eq:III}, one deduces that  ${\bf HLogAR}^2_{\rm sym}\simeq V^5_{[.221]}$ as a $W$-representation. 
\mk

\noindent 
{\bf The $\boldsymbol{W}_{\boldsymbol{\mathfrak c}_1}$-representation ${\bf H}_{\boldsymbol{\mathfrak c}_1}$  when $\boldsymbol{r=6}$.}
First, in order to define the Weyl group we will work with in this case 
and return the corresponding Dynkin diagram within GAP, we run the following commands: 
\begin{lstlisting}[language=GAP]
   gap> WE6:=CoxeterGroup("E",6);
   gap> PrintDiagram( WE6 );
\end{lstlisting}
The last command returns the following diagram : 
\begin{equation}
\label{Eq:Diag-Gap3-WE6}
  \xymatrix@R=0.3cm@C=0.4cm{ 
 &  & 2\ar@{-}[d]   & &   &
 \\
1 \ar@{-}[r] & 3  \ar@{-}[r] & 4   \ar@{-}[r]  & 5\ar@{-}[r]  & 6 
  }
\end{equation}

We define and get the corresponding characters table by typing 
\begin{lstlisting}[language=GAP]
   gap> TableWE6:=CharTable(WE6);
   gap> Display(TableWE6);
\end{lstlisting}

The following diagram illustrates the labelling of the fundamental roots considered when $r=6$ (see \eqref{Eq:Fundamental-Roots}): 
\begin{equation}
\label{Eq:Diag-Gap3-WE6-rho}
  \xymatrix@R=0.3cm@C=0.4cm{ 
 &  & \boldsymbol{\rho}_6\ar@{-}[d]   & &   &
 \\
\boldsymbol{\rho}_1 \ar@{-}[r] &\boldsymbol{\rho}_2  \ar@{-}[r] & \boldsymbol{\rho}_3   \ar@{-}[r]  & \boldsymbol{\rho}_4\ar@{-}[r]  & \boldsymbol{\rho}_5
  }
\end{equation}

Let us denote by $\sigma_1,\ldots,\sigma_6$ the generators of the group 
{\ttfamily  WE6}
in GAP. Comparing   \eqref{Eq:Diag-Gap3-WE6} and \eqref{Eq:Diag-Gap3-WE6-rho}, one deduces 
the following relations between the generators $s_i=s_{\boldsymbol{\rho}_i}$ ($i=1,\ldots,r$) of the Weyl group considered in 
\S\ref{SSS} and the  $\sigma_i$'s implemented in GAP by means of the command 
{\ttfamily  WE6:=CoxeterGroup("E",6)}: 
\begin{equation}
\label{Eq:ss-sigma}
\sigma_1=s_1\, , \quad \sigma_3=s_2\, , \quad  \sigma_4=s_3 
\, , \quad \sigma_5=s_4\, , \quad \sigma_6=s_5 \quad 
\mbox{and} \quad \sigma_2=s_6
\end{equation}

The stabilizer $W_{\boldsymbol{\mathfrak c}_1}$ of $\boldsymbol{\mathfrak c}_1$ in $W$ is the subgroup spanned by the $s_j$'s for $j=2,\ldots,6$. It is a Weyl group of type $D_5$:
\begin{equation}
\label{Eq:WWcc11}
W_{\boldsymbol{\mathfrak c}_1}=\big\langle  s_2,\ldots,s_6 \big\rangle \simeq W(D_5)\, .
\end{equation}
The elements of the set 
$\boldsymbol{\mathfrak c}_1^{red}$ of  reducible conics  in the pencil $\lvert \, \boldsymbol{\mathfrak c}_1\, \lvert$ 
are the
$\boldsymbol{c}_{1,i}=(\boldsymbol{e}_i)+(\boldsymbol{h}-\boldsymbol{e}_1-\boldsymbol{e}_i)$'s  for 
$i=2,\ldots,6$.  The action of the $s_k$'s on $\boldsymbol{\mathfrak c}_1^{red}$ for $k=2,\ldots,6$  are as follows: for $k=2,\ldots,5$, $s_k$ acts as the transposition exchanging $\boldsymbol{c}_{1,k}$ and $\boldsymbol{c}_{1,k+1}$ while $s_6$ acts as $s_2$.  The corresponding matrices in the basis $\mathfrak  C_1=(\boldsymbol{c}_{1,i})_{i=2}^6$ of $\mathbf C^{\boldsymbol{\mathfrak c}_1^{red}}\simeq \mathbf C^5$ are the corresponding transposition matrices and will be denoted by 
$$ 
M_k={\rm Mat}_{{\mathfrak C}_1}( s_k)\in {\rm GL}_5(\mathbf Z)  \qquad 
k=2,\ldots,6.
$$

From \eqref{Eq:ss-sigma}, the matrices $\mathcal M_k
={\rm Mat}_{{\mathfrak C}_1}( \sigma_k) $ for $k=2,\ldots,6$ 
 are the following ones: 
\begin{equation}
\mathcal M_1 = M_1 \, , \quad 
\mathcal M_2 = M_6 \, , \quad 
\mathcal M_3 = M_2 \, , \quad 
\mathcal M_4 = M_3 \, , \quad 
\mathcal M_5 = M_4 \, , \quad 
\mathcal M_6 = M_5 \, .
\end{equation}

From \eqref{Eq:ss-sigma} and 
\eqref{Eq:WWcc11}, it comes that  the group 
{\ttfamily  Wc1} corresponding to $W_{\boldsymbol{\mathfrak c}_1}$ 
as well as its characters table 
are defined as follows in GAP: 
\begin{lstlisting}[language=GAP]
   gap> Wc1 := ReflectionSubgroup( WE6, [2,3,4,5,6] );
   gap> TableWc1:=CharTable(Wc1); Display(TableWc1);
\end{lstlisting}

In view of determining the character $\chi_{\boldsymbol{\mathfrak c}_{1}}$
 of the $W_{\boldsymbol{\mathfrak c}_{1}}$-representation $H_{\boldsymbol{\mathfrak c}_{1}}$, we need to construct explicit representatives of the conjugacy classes in $W_{\boldsymbol{\mathfrak c}_1}$. This can be achieved in GAP3 by means of the following command: 
\begin{lstlisting}[language=GAP]
   gap> List(ConjugacyClasses(Wc1),x->CoxeterWord(Wc1,Representative(x)));
\end{lstlisting}
which returns the following collection of strings: 
\begin{lstlisting}[language=GAP]
 [ [  ], [ 2, 3 ], [ 2, 3, 4, 2, 3, 4, 5, 4, 2, 3, 4, 5 ], [ 2 ], [ 2, 3, 4 ], 
   [ 2, 3, 5 ], [ 2, 3, 4, 2, 3, 4, 5, 4, 2, 3, 4, 5, 6 ], [ 2, 5 ], 
   [ 2, 4, 2, 3, 4, 5 ], [ 2, 3, 4, 6 ], [ 2, 4 ], [ 2, 3, 4, 5 ], 
   [ 2, 3, 5, 6 ], [ 2, 4, 6 ], [ 2, 4, 2, 3, 4, 5, 6 ], [ 2, 5, 4 ], 
   [ 2, 3, 4, 5, 6 ], [ 2, 5, 4, 6 ] ]
\end{lstlisting}

It is then just a matter of elementary computations with the matrices $\mathcal M_2,\ldots,\mathcal M_6$ to get the character of $\mathbf C^{ \boldsymbol{\mathfrak c}_1^{red} }$ as a $W_{\boldsymbol{\mathfrak c}_1}$-representation.  For instance, the value of this character evaluated on the conjugacy class encoded by the string {\ttfamily  [ 2, 3, 5 ]}  is the trace 
of the matricial product $\mathcal M_2\mathcal M_3\mathcal M_5$, etc. 
Proceeding in this way, one easily obtains that 
$$
\chi_{\mathbf C^{ \boldsymbol{\mathfrak c}_1^{red}}} = 
\big( 5, 5, 5, 3, 3, 3, 3, 1, 1, 1, 2, 2, 2, 0, 0, 1, 1, 0\big) \, .
$$  
This character can be decomposed as the sum of the trivial character with the one associated with the $W_{\boldsymbol{\mathfrak c}_1}$-representation $V_{[.41]}^4$.
It follows that 
\begin{equation}
\label{Eq:Hc1-dP3}
{\bf H}_{ \boldsymbol{\mathfrak c}_1 }= V_{[.41]}^4\, . 
\end{equation}
\begin{center}
$\star$
\end{center}

In fact,  if one assumes that ${\bf H}_{ \boldsymbol{\mathfrak c}_1 }$ is irreducible, there is another way for establishing \eqref{Eq:Hc1-dP3}. Indeed, there are only two irreducible representations of $W_{ \boldsymbol{\mathfrak c}_1}\simeq W(D_5)$ of degree 4, the two associated   with the bipartitions $[.2111]$ and $[.41]$ respectively.  Moreover, we know that ${\rm Ind}_{W_{ \boldsymbol{\mathfrak c}_1}}^{W(E_6)}( {\bf H}_{ \boldsymbol{\mathfrak c}_1 })$ must contain ${\bf H}_{X_6}\simeq V^{20,2}$ as an irreducible $W(E_6)$-subrepresentation (see \S\ref{SSS:Weyl-group-action-on-ARs}). One obtains the induction table from $W_{ \boldsymbol{\mathfrak c}_1}$ to $W(E_6)$ by typing the following command in GAP3 : 
\begin{lstlisting}[language=GAP]
   gap> Display( InductionTable(  Wc1, WE6) );
\end{lstlisting}
We obtain easily that 
$$
{\rm Ind}_{W_{ \boldsymbol{\mathfrak c}_1}}^{W_6}\Big(V_{[.41]}^4\Big)= V^{20,2}\oplus V^{24,6}\oplus V^{64,4} 
\qquad \mbox{ while } 
\qquad 
{\rm Ind}_{W_{ \boldsymbol{\mathfrak c}_1}}^{W_6}\Big(V_{[.2111]}^4\Big)=V^{20,20}\oplus V^{24,12}\oplus V^{64,13} 
\, . $$
Because ${\bf H}_{X_6}={\bf H}^0\big( X_6, \Omega^1_{X_6}\big( {\bf Log}\, L_6 \big) \big) $ is $V^{20,2}$ according to Corollary \ref{Cor:H-Xr},  only the first case can occur hence  necessarily \eqref{Eq:Hc1-dP3} holds true.


\newpage

\bigskip\bigskip


\vfill 
{\small  ${}^{}$ \hspace{-0.6cm} {Luc Pirio} ({\tt luc.pirio@uvsq.fr})\\
 LMV -- Universit\'e Paris-Saclay UVSQ \& CNRS  (UMR 8100)\\

\end{document}